\documentclass[draft]{memo-l}

\usepackage{amsmath,amssymb,amscd,mathrsfs,epic,wasysym,latexsym,tikz,mathrsfs,cite,hyperref,array,amsthm,hhline,stmaryrd,latexsym}
\usepackage{pb-diagram}

\newtheorem{Theorem}{Theorem}[section]
\newtheorem{Proposition}[Theorem]{Proposition}
\newtheorem{Lemma}[Theorem]{Lemma}
\newtheorem{Corollary}[Theorem]{Corollary}

\newtheorem{Example}[Theorem]{Example}
\newtheorem{Remark}[Theorem]{Remark}

\newtheorem{Conjecture}[Theorem]{Conjecture}

\makeatletter \@addtoreset{equation}{chapter}

\makeatother

\numberwithin{section}{chapter}
\numberwithin{equation}{section}

\let\<\langle
\let\>\rangle

\newcommand\Comment[2][\relax]{\space\par\medskip\noindent%
   \fbox{\begin{minipage}{\textwidth}\textbf{Comment\ifx\relax#1\else---#1\fi}\newline%
        #2\end{minipage}}\medskip
}

\def\bi{\text{\boldmath$i$}}
\def\bj{\text{\boldmath$j$}}

\def\bk{\text{\boldmath$k$}}
\def\bl{\text{\boldmath$l$}}

\def\b1{\text{\boldmath$1$}}
\def\gg{\text{\boldmath$g$}}

\def\bn{\text{\boldmath$n$}}
\def\bh{\text{\boldmath$h$}}

\def\pmod#1{\text{ }(\text{\rm mod } #1)\,}

\newcommand{\Hom}{\operatorname{Hom}}

\newcommand{\ext}{\operatorname{ext}}
\newcommand{\EXT}{\operatorname{EXT}}
\newcommand{\End}{\operatorname{End}}
\newcommand{\ind}{\operatorname{ind}}
\newcommand{\im}{\operatorname{im}}
\newcommand{\id}{\operatorname{id}}
\def\sgn{\mathtt{sgn}}
\newcommand{\res}{\operatorname{res}}
\newcommand{\soc}{\operatorname{soc}}
\newcommand{\head}{\operatorname{head}}

\newcommand{\cha}{\operatorname{char}}

\newcommand{\infl}{\operatorname{infl}}
\newcommand{\trun}{\operatorname{trun}}

\newcommand{\Z}{\mathbb{Z}}

\def\eps{{\epsilon}}
\def\phi{{\varphi}}

\newcommand{\ga}{\gamma}
\newcommand{\Ga}{\Gamma}
\newcommand{\la}{\lambda}
\newcommand{\La}{\Lambda}
\newcommand{\al}{\alpha}
\newcommand{\be}{\beta}

\def\Si{\mathfrak{S}}
\newcommand{\si}{\sigma}

\newcommand{\de}{\delta}
\newcommand{\De}{\Delta}

\def\triv#1{\O_{#1}}

\newcommand{\Ze}{{\mathbb Z}}

\def\id{\mathop{\mathrm {id}}\nolimits}

\newcommand{\Ind}{{\mathrm {Ind}}}
\newcommand{\Coind}{{\mathrm {Coind}}}

\newcommand{\tr}{{\mathrm {tr}}}

\newcommand{\pr}{{\mathrm {pr}}}

\newcommand{\Res}{{\mathrm {Res}}}
\newcommand{\Ann}{{\mathrm {Ann}}}
\newcommand{\C}{{\mathbb C}}

\newcommand{\A}{{\mathscr A}}
\newcommand{\D}{{\mathscr D}}

\renewcommand{\mod}{\bmod \,}

\def\h{{\mathfrak h}}
\def\g{{\mathfrak g}}
\def\n{{\mathfrak n}}

\def\Par{{\mathscr P}}
\def\ula{{\underline{\lambda}}}
\def\umu{{\underline{\mu}}}
\def\unu{{\underline{\nu}}}

\def\T{{\mathtt T}}

\def\spa{\operatorname{span}}
\def\height{{\operatorname{ht}}}

\def\wt{{\operatorname{wt}}}

\def\surj{{\twoheadrightarrow}}

\def\op{{\mathrm{op}}}
\def\re{{\mathrm{re}}}
\def\im{{\mathrm{im}}}
\def\onto{{\twoheadrightarrow}}
\def\into{{\hookrightarrow}}
\def\Mod#1{#1\!\operatorname{-Mod}}
\def\mod#1{#1\!\operatorname{-mod}}

\renewcommand\O{\mathcal O}
\def\HCI{I}
\def\HCR{{}^*I}
\def\iso{\stackrel{\sim}{\longrightarrow}}
\def\B{B}
\def\Mde{M}
\def\Sde{\operatorname{Sym}}
\def\Zde{Z}
\def\Lade{\La}
\def\Yde{Y}
\def\Sdot{\operatorname{Sym}}
\def\Ladot{\La}
\def\Zdot{Z}
\def\Lde{L}
\def\Dede{\De}
\def\nade{\nabla}
\def\Tde{T}
\def\Pde{P}
\def\triv{{\tt triv}}
\def\ch{\operatorname{ch}}
\def\lan{\langle}
\def\ran{\rangle}

\def\Stand{\Delta}
\def\Perm{\operatorname{Per}}
\def\SPerm{\operatorname{SPer}}
\def\CH{{\operatorname{ch}_q\:}}
\def\DIM{{\operatorname{dim}_q\:}}
\def\co{{\operatorname{col}}}
\def\words{{\langle I\rangle}}
\def\shift{{\tt sh}}

\def\Seq{{\tt Se}}
\def\Car{{\tt C}}
\def\ImS{{\mathscr S}}
\def\JTD{{\mathscr D}}
\def\cc{{\tt c}}

\makeatletter
\newcommand*{\rom}[1]{\expandafter\@slowromancap\romannumeral #1@}
\makeatother

{\catcode`\|=\active
  \gdef\set#1{\mathinner{\lbrace\,{\mathcode`\|"8000%
  \let|\midvert #1}\,\rbrace}}
}
\def\midvert{\egroup\mid\bgroup}

\colorlet{darkgreen}{green!50!black}
\tikzset{dots/.style={very thick,loosely dotted},
         greendot/.style={fill,circle,color=darkgreen,inner sep=1.5pt,outer sep=0},
         blackdot/.style={fill,circle,color=black,inner sep=1.1pt,outer sep=0},
         graydot/.style={fill,circle,color=gray,inner sep=1.1pt,outer sep=0}
}
\def\greendot(#1,#2){\node[greendot] at(#1,#2){}}
\def\blackdot(#1,#2){\node[blackdot] at(#1,#2){}}
\def\graydot(#1,#2){\node[graydot] at(#1,#2){}}

\newenvironment{braid}{
  \begin{tikzpicture}[baseline=6mm,blue,line width=1pt, scale=0.4,
                      draw/.append style={rounded corners},
                      every node/.append style={font=\fontsize{5}{5}\selectfont}]%
  }{\end{tikzpicture}
}

\def\Grid(#1,#2){
  \draw[very thin,gray,step=2mm] (0,0)grid(#1,#2);
  \draw[very thin,darkgreen,step=10mm] (0,0)grid(#1,#2);
}

\newcommand\Tableau[2][\relax]{
  \begin{tikzpicture}[scale=0.5,draw/.append style={thick,black}]
    \ifx\relax#1\relax%
    \else 
      \foreach\box in {#1} { \filldraw[blue!30]\box+(-.5,-.5)rectangle++(.5,.5); }
    \fi
    \newcount\row\newcount\col
    \row=0
    \foreach \Row in {#2} {
       \col=1
       \foreach\k in \Row {
          \draw(\the\col,\the\row)+(-.5,-.5)rectangle++(.5,.5);
          \draw(\the\col,\the\row)node{\k};
          \global\advance\col by 1
       }
       \global\advance\row by -1
    }
  \end{tikzpicture}
}

\newcommand\YoungDiagram[2][\relax]{
  \begin{tikzpicture}[scale=0.5,draw/.append style={thick,black}]
    \ifx\relax#1\relax%
    \else 
    \foreach\box in {#1} {
      \filldraw[blue!30]\box rectangle ++(1,1);
    }
    \fi
    \newcount\row
    \row=0
    \foreach \col in {#2} {
       \draw(1,\the\row)grid ++(\col,1);
       \global\advance\row by -1
    }
  \end{tikzpicture}
}

\makeindex

\begin{document}


\frontmatter

\title[Imaginary Schur-Weyl duality]{{\bf Imaginary Schur-Weyl duality}}

\author{\sc Alexander Kleshchev}
\address{Department of Mathematics\\ University of Oregon\\
Eugene\\ OR 97403, USA}
\email{klesh@uoregon.edu}

\author{\sc Robert Muth}
\address{Department of Mathematics\\ University of Oregon\\
Eugene\\ OR 97403, USA}
\email{muth@uoregon.edu}


\subjclass[2000]{20C08, 20C30, 05E10}

\thanks{
Research supported in part by the NSF grant no. DMS-1161094, the Humboldt Foundation, and the Hausdorff Institute for Mathematics.}

\begin{abstract}
We study imaginary representations of the Khovanov-Lauda-Rouquier algebras of affine Lie type. Irreducible modules for such algebras arise as simple heads of standard modules. In order to define standard modules one needs to have a cuspidal system for a fixed convex preorder. A cuspidal system consists of irreducible cuspidal modules---one for each real positive root for the corresponding affine root system ${\tt X}_l^{(1)}$, as well as irreducible imaginary modules---one for each $l$-multipartition. 
We study  imaginary modules by means of `imaginary Schur-Weyl duality'. We introduce an imaginary analogue  of  tensor space and the imaginary Schur algebra.
 We construct a projective generator for the imaginary Schur algebra, which yields a Morita equivalence between the imaginary and the classical Schur algebra. We construct imaginary analogues of Gelfand-Graev representations, Ringel duality and the Jacobi-Trudy formula. 
\end{abstract}

\maketitle

\setcounter{page}{4}

\tableofcontents

\chapter{Introduction}

\section{Convex preorders and cuspidal systems}\label{SConCus}
Let $F$ be an arbitrary ground field. The KLR algebra $R_\al=R_\al(\Car,F)$, whose precise definition is given in in Section~\ref{SSDefKLR}, is a graded unital associative $F$-algebra depending on a Lie type $\Car$ and an element $\al$ of the positive part $Q_+$ of the corresponding root lattice. 
A natural approach to representation theory of $R_\al$ is provided by the theory of standard modules. For KLR algebras of {\em finite} Lie type such a theory was first described in \cite{KRbz}, see also \cite{HMM,BKOP,BKM}, and especially \cite{McN}. For KLR algebras of {\em affine}\, Lie type two different approaches to  the theory of standard modules were proposed in  \cite{TW} and \cite{Kcusp}. This paper is a continuation of \cite{Kcusp}. 

To describe the main results, let the Cartan matrix $\Car$ be of arbitrary {\em untwisted affine type}. In particular, the  simple roots $\al_i$ are labeled by  $i\in I=\{0,1,\dots,l\}$, where $0$ is the affine vertex of the corresponding Dynkin diagram. We have an (affine) root system $\Phi$ and the corresponding finite root subsystem $\Phi' =\Phi\cap \Z\,\text{-}\spa(\al_1,\dots,\al_l)$. Denote by $\Phi'_+$ and $\Phi_+$ the sets of {\em positive}\, roots in $\Phi'$ and $\Phi$, respectively. Then 
$\Phi_+=\Phi_+^\im\sqcup \Phi_+^\re
$, where
$
\Phi_+^\im=\{n\de\mid n\in\Z_{>0}\}
$
for the null-root $\de$, and
$$
\Phi_+^\re=\{\be+n\de\mid \be\in  \Phi'_+,\ n\in\Z_{\geq 0}\}\sqcup \{-\be+n\de\mid \be\in  \Phi'_+,\ n\in\Z_{> 0}\}. 
$$


As in \cite{BKT}, a {\em convex preorder} on $\Phi_+$ is a preorder $\preceq$ such that the following three conditions hold for all $\be,\ga\in\Phi_+$:
\begin{eqnarray}
\label{EPO1}
&\be\preceq\ga \ \text{or}\ \ga\preceq \be;
\\
\label{EPO2}
&\text{if $\be\preceq \ga$ and $\be+\ga\in\Phi_+$, then $\be\preceq\be+\ga\preceq\ga$};
\\
&\label{EPO3}
\text{$\be\preceq\ga$ and $\ga\preceq\be$ if and only if $\be$ and $\ga$ are proportional}.
\end{eqnarray}
Convex preorders are known to exist. Let us fix an arbitrary convex preorder $\preceq$ on $\Phi_+$. From (\ref{EPO3}) we have  that $\be\preceq\ga$ and $\ga\preceq\be$ happens for $\be\neq\ga$ if and only if both $\be$ and $\ga$ are imaginary. We write $\be\prec\ga$ if $\be\preceq\ga$ but $\ga\not\preceq\be$. 
The following set is {\em totally ordered}\, with respect to $\preceq$:
\begin{equation}\label{EPsi}
\Psi:=\Phi_+^\re\cup\{\de\}.
\end{equation}
It is easy to see that the set of real roots splits into two disjoint infinite sets
$$
\Phi^\re_{\succ}:=\{\be\in \Phi_+^\re\mid \be\succ\de\}\ \text{and}\ 
\Phi^\re_{\prec}:=\{\be\in \Phi_+^\re\mid \be\prec\de\}. 
$$


If $\mu$ is a partition of $n$ we write $\mu\vdash n$ and $n=|\mu|$. 
By an {\em $l$-multipartition} of $n$, we mean a tuple $\umu=(\mu^{(1)},\dots,\mu^{(l)})$ of partitions such that $|\mu^{(1)}|+\dots+|\mu^{(l)}|=n$. The set of all $l$-multipartitions of $n$ is denoted by $\Par_n$, and $\Par:=\sqcup_{n\geq 0}\Par_n$. 
A {\em root partition of $\al$} is a pair $(M,\umu)$, where $M$ is a tuple $(m_\rho)_{\rho\in\Psi}$ of non-negative integers such that $\sum_{\rho\in\Psi}m_\rho\rho=\al$, and $\umu$ is an $l$-multipartition of $m_\de$. 
It is clear that all but finitely many integers $m_\rho$ are zero, so we can always choose a finite subset
$$
\rho_1>\dots>\rho_s>\de>\rho_{-t}>\dots>\rho_{-1}
$$
of $\Psi$ such that $m_\rho=0$ for $\rho$ outside of this subset. Then, denoting $m_u:=m_{\rho_u}$, we can 
 write any root partition of $\al$ in the form
\begin{equation}\label{ERP}
(M,\umu)=(\rho_1^{m_1},\dots,\rho_s^{m_s},\umu,\rho_{-t}^{m_{-t}},\dots,\rho_{-1}^{m_{-1}}),
\end{equation}
where all $m_u\in\Z_{\geq 0}$, $\umu\in\Par$, and 
$$\sum_{u=1}^s m_u\rho_u+|\umu|\de+\sum_{u=1}^{t} m_{-u}\rho_{-u}=\al.$$
We write $\Pi(\al)$ for the set of root partitions of~$\al$. 
The set $\Pi(\al)$ has a natural partial order `$\leq$', see Section~\ref{SSConPreRP}. 

There is an induction functor $\Ind_{\al,\be}$, which associates to an $R_\al$-module $M$ and and $R_\be$-module $N$ the $R_{\al+\be}$-module $M\circ N:=\Ind_{\al,\be} M\boxtimes N$ for $\al,\be\in Q_+$. We refer to this operation as the {\em induction product}. The functor $\Ind_{\al,\be}$ has a right adjoint $\Res_{\al,\be}$.

A {\em cuspidal system} (for a fixed convex preorder) is the following data:
\begin{enumerate}
\item[{\rm (Cus1)}] A {\em cuspidal} irreducible $R_\rho$-module $L_\rho$ assigned to every $\rho\in \Phi_+^\re$, with the following property: if $\be,\ga\in Q_+$ are non-zero elements such that $\rho=\be+\ga$ and $\Res_{\be,\ga}L_\rho\neq 0$, then $\beta$ is a sum of roots less than $\rho$ and $\ga$ is a sum of roots greater than $\rho$. 
\item[{\rm (Cus2)}] An irreducible {\em imaginary}  $R_{n\de}$-module $L(\umu)$ assigned to every $\umu\in\Par_n$, with the following property: if $\be,\ga\in Q_+\setminus\Phi_+^\im$ are non-zero elements such that $n\de=\be+\ga$ and $\Res_{\be,\ga}L(\umu)\neq 0$, then $\beta$ is a sum of real roots less than $\de$ and $\ga$ is a sum of real roots greater than $\de$. It is required that $L(\ula)\not\simeq L(\umu)$ unless $\ula=\umu$. 
\end{enumerate}

It is proved in \cite{Kcusp} that (for a fixed convex preorder) cuspidal modules exist and are determined uniquely up to an  isomorphism.

Given a root partition 
$
\pi=(\rho_1^{m_1},\dots,\rho_s^{m_s},\umu,\rho_{-t}^{m_{-t}},\dots,\rho_{-1}^{m_{-1}})\in\Pi(\al)
$ as above, 
the corresponding {\em standard module}  is: 
\begin{equation}\label{EStandIntro}
\Stand(\pi):=q^{\shift(\pi)}L_{\rho_1}^{\circ m_1} \circ\dots\circ L_{\rho_s}^{\circ m_s}\circ  L(\umu)\circ L_{\rho_{-t}}^{\circ m_{-t}}\circ\dots\circ L_{\rho_{-1}}^{m_{-1}} ,
\end{equation}
where $q^{\shift(\pi)}$ means that grading is shifted by an integer $\shift(\pi)$ defined in (\ref{EShift}). 

\vspace{2mm}
\noindent
{\bf Theorem 1. (Cuspidal Systems)} 
{\em {\rm \cite[Main Theorem]{Kcusp}} 
For any convex preorder there exists a cuspidal system $$\{L_\rho\mid \rho\in \Phi_+^\re\}\cup\{L(\umu)\mid \umu\in\Par\}.$$ Moreover: 
\begin{enumerate}
\item[{\rm (i)}] For every root partition $\pi$, the standard module  
$
\Stand(\pi)
$ has irreducible head; denote this irreducible module $L(\pi)$. 

\item[{\rm (ii)}] $\{L(\pi)\mid \pi\in \Pi(\al)\}$ is a complete and irredundant system of irreducible $R_\al$-modules up to isomorphism and degree shift.

\item[{\rm (iii)}] For every root partition  $\pi$, we have $L(\pi)^\circledast\cong L(\pi)$.  

\item[{\rm (iv)}] For all root partitions $\pi,\si\in\Pi(\al)$, we have that $[\Stand(\pi):L(\pi)]_q=1$, and $[\Stand(\pi):L(\si)]_q\neq 0$ implies $\si\leq \pi$. 


\item[{\rm (v)}] The induced module $L_\rho^{\circ n}$ is irreducible for all $\rho\in\Phi^\re_+$ and $n\in\Z_{>0}$.
\end{enumerate}
}
\vspace{2mm}

We note that in the definition (\ref{EStandIntro}) of standard modules $\De(\pi)$, instead of $L(\umu)$ we could take
\begin{equation}\label{EStandDelta}
\De(\umu):=\De_1(\mu^{(1)})\circ \dots\circ \De_l(\mu^{(l)}),
\end{equation}
with $\De_i(\mu^{(i)})$ being the {\em imaginary standard module of color $i$} defined in (\ref{EDeIntro}) below. By \cite[Theorem 4.7]{Kcusp}, the new standard module 
\begin{equation}\label{EStand'Intro}
\Stand'(\pi):=q^{\shift(\pi)}L_{\rho_1}^{\circ m_1} \circ\dots\circ L_{\rho_s}^{\circ m_s}\circ  \De(\umu)\circ L_{\rho_{-t}}^{\circ m_{-t}}\circ\dots\circ L_{\rho_{-1}}^{m_{-1}} 
\end{equation}
has all the properties of $\De(\pi)$ listed in Theorem 1. If $\cha F=0$ or $\cha F>m_\de$, we always have $\De(\pi)=\De'(\pi)$. The {\em disadvantage} of $\De'(\pi)$ is that it is in general larger than $\De(\pi)$, and so it is `further' from the irreducible module $L(\pi)$. The {\em advantage} is that $\De'(\pi)$ has formal character which does not depend on the characteristic of the ground field, and which in many important cases is known.

\section{Imaginary representations}\label{SImRep}
Theorem 1 gives a `rough classification' of irreducible $R_\al$-modules. The main problem is that we did not give a canonical definition of individual irreducible imaginary modules $L(\umu)$. We just know that the amount of such modules for $R_{n\de}$ is equal to the number of $l$-multipartitions of $n$, and so we have labeled them by such multipartitions in an arbitrary way. 

In this paper we address this problem. Our approach relies on the so-called {\em imaginary Schur-Weyl duality}. This theory in particular allows us to construct an equivalence between an appropriate category of imaginary representations of KLR algebras and the category of representations of the classical Schur algebra. 

Let us make an additional assumption that the convex preorder is {\em balanced}, which means that
\begin{equation}\label{EBalanced}
\Phi_{\succ}^\re=\{\be+n\de\mid \be\in  \Phi'_+,\ n\in\Z_{\geq 0}\}. 
\end{equation}
This is equivalent to  
\begin{equation}\label{EBalancedEquiv}
\al_i\succ n\de\succ\al_0 \qquad(i\in I',n\in\Z_{>0}).
\end{equation}
Of course, we then also have $\Phi_{\prec}^\re=\{-\be+n\de\mid \be\in  \Phi'_+,\ n\in\Z_{> 0}\}$. 
Balanced convex preorders always exist, see for example \cite{BCP}. 

The first steps towards imaginary Schur-Weyl duality have already been made in \cite{Kcusp}. First of all, we have described explicitly  {\em minuscule} representations---the irreducible imaginary representations which correspond to $l$-multipartitions of~$1$. There are of course exactly $l$ such multipartitions, namely $\umu(1),\dots,\umu(l)$, where 
$$
\umu(i):=(\emptyset,\dots,\emptyset,(1),\emptyset,\dots,\emptyset)
$$
with the partition $(1)$ in the $i$th position. For each $i=1,\dots,l$,  we have defined an irreducible $R_\de$-module $L_{\de,i}$, see \cite[Section 5]{Kcusp}, and set
$$
L(\umu(i)):=L_{\de,i}\qquad(1\leq i\leq l).
$$ 
Further, we have defined the {\em imaginary tensor space of color $i$}  to be the $R_{n\de}$-module
$$
M_{n,i}:=L_{\de,i}^{\circ n}\qquad(1\leq i\leq l). 
$$
and proved \cite[Lemma 5.7]{Kcusp} that any composition factor of a {\em mixed tensor space} 
$$M_{n_1,1}\circ \dots\circ M_{n_l,l}$$ 
is imaginary. We call composition factors of $M_{n,i}$ {\em irreducible imaginary  modules of  color $i$.} The following theorem reduces the study of irreducible imaginary modules to irreducible imaginary modules of a fixed color:

\vspace{2mm}
\noindent
{\bf Theorem 2. (Reduction to One Color)} 
{\em {\rm \cite[Theorem 5.10]{Kcusp}} 
Suppose that for each $n\in\Z_{\geq 0}$ and $i\in I'$, we have an irredundant family $\{L_i(\la)\mid\la\vdash n\}$ of irreducible imaginary $R_{n\de}$-modules of color $i$. For a multipartition $\ula=(\la^{(1)},\dots,\la^{(l)})\in\Par_n$, define $$L(\ula):=L_1(\la^{(1)})\circ\dots\circ L_l(\la^{(l)}).$$ 
Then $\{L(\ula)\mid\ula\in\Par_n\}$ is a complete and irredundant system of irreducible imaginary $R_{n\de}$-modules. In particular, the given modules $\{L_i(\la)\mid\la\vdash n\}$ give all the irreducible imaginary modules of color $i$ up to isomorphism. 
}
\vspace{2mm}

Theorem 2 can be strengthened as a certain category equivalence, see Proposition~\ref{PT2'}. 

\section{Imaginary Schur-Weyl duality}
From now on we fix a color $i\in I'$. In view of Theorem 2, we need to construct  irreducible imaginary $R_{n\de}$-modules $L_i(\la)$ of color $i$. Since $i$ is now fixed we are going to drop the index $i$ from our notation. So we need to describe the composition factors of the imaginary tensor space $\Mde_n=\Mde_{n,i}$ and show that they are naturally labeled by the partitions $\la$ of $n$. 


The $R_{n\de}$-module structure on the imaginary tensor space $M_n$ yields an algebra homomorphism $R_{n\de}\to \End_F(M_n)$. Define the {\em imaginary Schur algebra} $\ImS_n
$ as the image of $R_{n\de}$ under this homomorphism. In other words,  
$$\ImS_n=R_{n\de}/\Ann_{R_{n\de}}(M_n). 
$$
Modules over $R_{n\de}$ which factor through to $\ImS_n$ will be called {\em imaginary modules} (of color $i$). Thus the {\em category of imaginary $R_{n\de}$-modules} is the same as the category of $\ImS_n$-modules. The main goal of this paper is to understand this category.

It is clear that $M_n$ and its composition factors are imaginary modules. Conversely, any irreducible $\ImS_n$-module appears as  a composition factor of $M_n$. So our new notion of an imaginary module fits with the old notion of an irreducible imaginary module in the sense of cuspidal systems.  

Our first new result is as follows:

\vspace{2mm}
\noindent
{\bf Theorem 3. (Imaginary Schur-Weyl Duality)}
{\em
\begin{enumerate}
\item[{\rm (i)}] $M_n$ is a projective $\ImS_n$-module. 
\item[{\rm (ii)}] The endomorphism algebra $\End_{R_{n\de}}(\Mde_n)=\End_{\ImS_n}(\Mde_n)$ of the imaginary tensor space $\Mde_n$ is isomorphic to the group algebra $F\Si_n$ of the symmetric group $\Si_n$ (concentrated in degree zero). Thus $M_n$ can be considered as a right $F\Si_n$-module. 
\item[{\rm (iii)}] $\End_{F\Si_n}(M_n)=\ImS_n$. 
\end{enumerate}

}
\vspace{2 mm}
 
 
Parts (i) and (ii) of Theorem 3  are  Theorem~\ref{TSchub}, and part (iii) is Theorem~\ref{4.5e}(ii). 

In view of Theorem 3, we have an exact functor 
\begin{equation}\label{EGaFun}
\ga_{n}: \mod{\ImS_n}\to\mod{F\Si_n},\quad V\mapsto \Hom_{\ImS_n}(M_n,V).
\end{equation}
Unfortunately, $\ga_n$ is not an equivalence of categories, unless the characteristic of the ground field is zero or greater than $n$, cf. Theorem~\ref{TChar0}, since in general the $\ImS_n$-module $M_n$ is not a projective generator. In order to resolve this problem, we need to upgrade from imaginary Schur-Weyl duality to imaginary Howe duality.

\section{Imaginary Howe and Ringel dualities}
Let 
$
{\mathtt x}_{n}:=\sum_{g\in \Si_n}g.
$
In view of Theorem 3, $M_n$ is a right $F\Si_n$-module. Define the {\em imaginary 
divided} and {\em exterior} powers respectively as follows:
\begin{align*}
\Zde_n&:=\{m\in \Mde_n\mid mg-\sgn(g)m=0\ \text{for all $g\in\Si_n$}\},
\\
\Lade_n&:=\Mde_n {\mathtt x}_n.
\end{align*} 
Note that the divided power in our definition corresponds to the sign representation of the symmetric group, while the exterior power corresponds to the trivial representation. This set up is similar to \cite[\S3.3]{BDK}. 

For $h\in\Z_{>0}$, denote by $X(h,n)$ the set of all compositions of $n$ with $h$ parts:
$$
X(h,n):=\{(n_1,\dots,n_h)\in\Z_{\geq 0}^h\mid n_1+\dots+n_h=n\}.
$$
The corresponding set of partitions is
$$
X_+(h,n):=\{(n_1,\dots,n_h)\in X(h,n)\mid n_1\geq\dots\geq n_h\}.
$$
For a composition $\nu=(n_1,\dots,n_h)\in X(h,n)$, we define the functors of {\em imaginary induction} and {\em imaginary restriction} as
$$
\HCI_\nu^n:=\Ind_{n_1\de,\dots,n_h\de}: \mod{R_{n_1\de,\dots,n_h\de}}\to \mod{R_{n\de}}
$$
and
$$
\HCR_\nu^n:=\Res_{n_1\de,\dots,n_h\de}: \mod{R_{n\de}}\to \mod{R_{n_1\de,\dots,n_h\de}}.
$$
These functors `respect' the categories of imaginary representations. For example, given imaginary $R_{n_b\de}$-modules $V_b$ for $b=1,\dots,h$, the module $\HCI_\nu^n(V_1\boxtimes\dots\boxtimes V_h)$ is also imaginary. Define
\begin{align*}
\Zde^\nu&:=\HCI_\nu^n(\Zde_{n_1}\boxtimes\dots\boxtimes \Zde_{n_h}),
\\
\Lade^\nu&:=\HCI_\nu^n(\Lade_{n_1}\boxtimes\dots\boxtimes \Lade_{n_h}).
\end{align*} 

Now, let $S_{h,n}$ be the classical Schur algebra, whose representations are the same as the degree $n$ polynomial representations of the general linear group $GL_h(F)$, see \cite{Green}. Well-known basic facts on $S_{h,n}$ are collected in Sections~\ref{SSSA}--\ref{SSIRSA}. In particular, it is a finite dimensional quasi-hereditary algebra with irreducible, standard, costandard, and indecomposable tilting modules
$$
L_h(\la),\ \De_h(\la),\ \nabla_h(\la),\ T_h(\la)\qquad(\la\in X_+(h,n)). 
$$

\vspace{2mm}
\noindent
{\bf Theorem 4. (Imaginary Howe Duality)}
{\em
\begin{enumerate}
\item[{\rm (i)}] For each $\nu\in X(h,n)$ the $\ImS_n$-module $\Zdot^\nu$ is  projective. Moreover, for any $h\geq n$, we have that $Z:=\bigoplus_{\nu\in X(h,n)}\Zdot^\nu$ is a projective generator for $\ImS_n$.  
\item[{\rm (ii)}] The endomorphism algebra $\End_{\ImS_n}(Z)$ is isomorphic to the classical Schur algebra $S_{h,n}$. Thus $Z$ can be considered as a right $S_{h,n}$-module. 
\item[{\rm (iii)}] $\End_{S_{h,n}}(Z)=\ImS_n$. 
\end{enumerate}
}
\vspace{2 mm}

Part (i) of Theorem 4 is Theorem~\ref{3.4g}(iii), part (ii) is Theorem~\ref{3.4c}, while part (iii) follows from (i) and (ii) and general Morita theory, see for example \cite[Theorem 3.54(iv)]{CRI}. 

Theorem 4 allows us to plug in Morita theory to define mutually inverse  equivalences of categories 
\begin{align}
\al_{h,n}&: \mod{\ImS_n}\to\mod{S_{h,n}},\quad V\mapsto \Hom_{\ImS_n}(Z,V)
\label{EAlpha}
\\
\be_{h,n}&:\mod{S_{h,n}}\to \mod{\ImS_n},\quad W\mapsto Z\otimes_{S_{h,n}}W.
\label{EBeta}
\end{align}
Denoting by $f_{h,n}$ the usual Schur functor, as for example in \cite[\S6]{Green}, by definitions we then have a commutative triangle (up to isomorphism of functors):
\begin{equation}\label{EFunTriangle}
\begin{picture}(100,30)
\put(0,-30){
\begin{tikzpicture}
\node at (0,0.5) {$\mod{S_{h,n}}$};
\node at (1.5,-1) {$\mod{\ImS_n}$}; 
\node at (-1.5,-1) {$\mod{F\Si_n}$};

\draw [->] (-0.5,0.2) -- (-1.2,-0.8);
\draw [<-] (-0.5,-1) -- (0.6,-1);
\draw [->] (0.5,0.2) -- (1.3,-0.7);
\draw [<-] (0.3,0.2) -- (1.1,-0.7);

\node at (-1.3,-0.2) {${\scriptstyle f_{h,n}}$};
\node at (0.25,-0.25) {${\scriptstyle \al_{h,n}}$};
\node at (1.35,-0.25) {${\scriptstyle \be_{h,n}}$};
\node at (0,-0.8) {${\scriptstyle \ga_{n}}$};
\end{tikzpicture}
}
\end{picture}
\end{equation}

\vspace{1cm}
Under the Morita equivalence imaginary induction commutes with tensor products:

\vspace{2mm}
\noindent
{\bf Theorem 5. (Imaginary Induction and Tensor Products)}
{\em
Let $h\geq n$ and $\nu=(n_1,\dots,n_a)\in X(a,n)$. 
The following functors are isomorphic:
\begin{align*}
\HCI_\nu^n (\be_{h,n_1}?\boxtimes\dots\boxtimes \be_{h,n_a} ?) & :\mod{S_{h,n_1}}\times\dots\times \mod{S_{h,n_a}}\to \mod{\ImS_n},
\\
\be_{h,n}(?\otimes\dots\otimes  ?) & :\mod{S_{h,n_1}}\times\dots\times \mod{S_{h,n_a}}\to \mod{\ImS_n}.
\end{align*}
}

Theorem 5 is Theorem~\ref{4.2a}. Since the Schur functor sends tensor products to the induction for symmetric groups, cf. \cite[Theorem 4.13]{BKlr}, we deduce from Theorem 5 and (\ref{EFunTriangle}):

\vspace{2mm}
\noindent
{\bf Corollary.}
{\em
Let $\nu=(n_1,\dots,n_a)\in X(a,n)$. We have a functorial isomorphism  
$$
\ga_n\big(\HCI_\nu^n (V_1\boxtimes\dots\boxtimes  V_a)\big)\cong \ind_{F\Si_\nu}^{F\Si_n}
\big(\ga_{n_1}(V_1)\boxtimes\dots\boxtimes \ga_{n_a}(V_a)\big)
$$
for $V_1\in\mod{\ImS_{n_1}},\dots,V_a\in\mod{\ImS_{n_a}}$.
}

\vspace{2mm}
Let $\la \in X_+(n,n)$ and $h\geq n$. We can also consider $\la$ as an element of $X_+(h,n)$. Define the graded $\ImS_n$-modules (hence, by inflation, also graded $R_{n\de}$-modules):
\begin{eqnarray}
\label{ELIntro}
\Lde(\la) &:= \be_{h,n}(L_h(\la)),
\\
\label{EDeIntro}
\Dede(\la) &:= \be_{h,n}(\Delta_h(\la)),
\\
\label{ENaIntro}
\nade(\la) &:= \be_{h,n}(\nabla_h(\la)),
\\
\label{ETIntro}
T(\la) &:= \be_{h,n}(T_h(\la)).
\end{eqnarray}
These definitions turn out to be independent of the choice of $h \geq n$, see Lemma~\ref{3.5c}. 
We get:

\vspace{2mm}
\noindent
{\bf Theorem 6. (Imaginary Schur Algebra)}
{\em
The imaginary Schur algebra $\ImS_n$ is a finite dimensional quasi-hereditary algebra with irreducible, standard, costandard, and indecomposable tilting modules
$$
L(\la),\ \De(\la),\ \nabla(\la),\ T(\la)\qquad(\la\in X_+(h,n)). 
$$ 
}

We also have 

\vspace{2mm}
\noindent
{\bf Theorem 7. (Imaginary Ringel Duality)}
{\em
\begin{enumerate}
\item[{\rm (i)}] Let $h\geq n$. The $\ImS_n$-module $\bigoplus_{\nu\in X(h,n)} \Lade^\nu$ is a full tilting module. 
\item[{\rm (ii)}] We have isomorphisms of endomorphism algebras 
$$\End_{\ImS_n}\Big(\bigoplus_{\nu\in X(h,n)} \Lade^\nu \Big)\cong S_{h,n}\quad\text{and}\quad   \End_{S_{h,n}}\Big(\bigoplus_{\nu\in X(h,n)} \Lade^\nu \Big)\cong \ImS_{n}.$$
\end{enumerate}
}
\vspace{2 mm}

Part (i) of Theorem 7 is Theorem~\ref{4.5d}, while part (ii) is Theorem~\ref{4.5e}(i). 

\section{Gelfand-Graev words and representations}
An important role in the paper is played by an analogue of {\em Gelfand-Graev representation}, cf. \cite[Corollary 2.5e and \S2.4]{BDK}. 
Recall that the algebra $R_\al$ comes with a family of idempotents $\{1_\bi\mid\bi\in\words_\al\}$ labeled by the set $\words_\al$ of words or weight $\al$, see Section~\ref{SSDefKLR}. For a finite dimensional graded $R_\al$-module $V$ we have the corresponding {\em word spaces}
$$
V_\bi:=1_\bi V\qquad(\bi\in\words_\al)
$$
and the {\em formal character}  
$$
\CH=\sum_{\bi\in\words_\al}(\DIM V_\bi)\bi\in \Z[q,q^{-1}]\cdot \words_\al,
$$
where $\DIM V_\bi\in\Z[q,q^{-1}]$ is the graded dimension of the word space $V_\bi$, see Section~\ref{SSBasicRep}.

Recall that $M_1=L_{\de}$ is a minuscule representation (whose formal character is well understood). Let $e:=\height(\de)$. Fix an extremal word $\bi=i_1\dots i_e$ appearing in the formal character of $M_1$ as in Section~\ref{SSImTenSp}, and write it in the form
$$
\bi=j_1^{m_1}\dots j_r^{m_r}
$$
with $j_k\neq j_{k+1}$ for all $1\leq k<r$. 
Define the corresponding {\em Gelfand-Graev  words} 
$$
\gg^{(t)}=\gg^{(t)}_\bi:=i_1^t i_2^t\dots i_e^t=j_1^{m_1t}\dots j_r^{m_rt}\in\words_{t\de}
$$
for all $t\in\Z_{>0}$, and, more generally, 
\begin{equation}\label{EGGW}
\gg^\mu=\gg^\mu_\bi:=\gg^{(\mu_1)}\dots\gg^{(\mu_n)}\in\words_{n\de}
\end{equation}
for any composition $\mu=(\mu_1,\dots,\mu_n)\in X(n, n)$. 

For $\mu=(\mu_1,\dots,\mu_n)\in X(n,n)$, set 
\begin{equation}\label{EmlaIntro}
c(\mu)=c(\mu)_\bi:=\prod_{m=1}^n\prod_{k=1}^r [m_k\mu_m]^!_{j_k}\in\Z[q,q^{-1}], 
\end{equation}
where the quantum factorials are defined as usual, see (\ref{EQuantumFactorials}). If $\Car$ is symmetric then all $m_k=1$, and so $c(\mu)=([\mu_1]_q^!\dots [\mu_n]_q^!)^e$. 

Let $V$ be a finite dimensional graded $R_{n\de}$-module. 
It is known that  
$$\DIM V_{\gg^\mu}=c(\mu)m_\mu(V)\qquad(\mu\in X(n,n))$$
for some $m_\mu(V)\in\Z[q,q^{-1}]$. 
In (\ref{EGGId}), we define explicit {\em Gelfand-Graev idempotents} $\ga_{\mu,\de}\in R_{n\de}$ such that $\ga_{\mu,\de}=1_{\gg^\mu}\ga_{\mu,\de}1_{\gg^\mu}$. In the special case $\mu=(n,0,\dots,0)$, we denote $\ga_{\mu,\de}$ by $\ga_{n,\de}$ and define the {\em Gelfand-Graev module} to be the projective module 
$$
\Ga_n:=\big(\prod_{k=1}^rq_{j_k}^{-nm_k(nm_k-1)/2}\big)R_{n\de}\ga_{n,\de}.
$$
Recall that a power of $q$ means a degree shift. 

\vspace{2mm}
\noindent
{\bf Theorem 8. (`Gelfand-Graev Properties')}
{\em
\begin{enumerate}
\item[{\rm (i)}] $\Hom_{R_{n\de}}(\Ga_{n},\Mde_n)\cong F$.  
\item[{\rm (ii)}] The submodule $\Zde_n\subseteq \Mde_n$ coincides with the image of any non-zero homomorphism from $\Ga_n$ to $M_n$. 
\item[{\rm (iii)}] for any $V\in\mod{R_{n\de}}$ and $\mu\in X(n,n)$, we have 
$$
m_\mu(V)=\DIM\Hom_{R_{n\de}}(\Ga_{\mu_1}\circ \dots\circ \Ga_{\mu_n},V).
$$
\end{enumerate}
}
\vspace{2 mm}

Part (i) of Theorem 8 is proved in Proposition~\ref{2.5e}(ii),  part (ii) is proved in Theorem~\ref{3.4g}(i), and part (iii) is Lemma~\ref{LDimGG}. 

Since we have equivalences of categories (\ref{EAlpha}) and (\ref{EBeta}), every finite dimensional graded $\ImS_n$-module $V$ can be written as $\be_{n,n}(W)$ (up to degree shift). Suppose we know the usual formal character of the module $W$ over the classical Schur algebra $S_{n,n}$, i.e. the dimensions of all usual weight spaces $W_\mu=e(\mu)W$ for $\mu\in X(n,n)$, see Section~\ref{SSSchurRT}. Then, recalling  $c(\mu)$ from (\ref{EmlaIntro}), we can also describe the {\em Gelfand-Graev fragment} of the (graded) formal character of the imaginary representation $V$ as follows:

\vspace{2mm}
\noindent
{\bf Theorem 9. (Gelfand-Graev Fragment of Graded Character)}
{\em
Let 
$\mu\in X(n, n)$, $W\in\mod{S_{h,n}}$ and $V=\beta_{n,n}(W)\in\mod{\ImS_n}$. Then 
$$\DIM V_{\gg^\mu}=c(\mu)\dim e(\mu)W.$$ 
}
\vspace{.1 mm}

Theorem 9 is proved in Corollary~\ref{CCharLLa}. 

For $\la\in X_+(n,n)$ and $\mu\in X(n, n)$, we now define 
\begin{align*}
k_{\la,\mu}&:=\dim e(\mu)L_n(\la),\\
K_{\la,\mu}&:=\dim e(\mu)\De_n(\la)
\end{align*}
Of course, $K_{\la,\mu}$ is the classical Kostka number, i.e. $K_{\la,\mu}$ equals the number of semistandard $\la$-tableaux of type $\mu$. 
If $\cha F=0$, we have $k_{\la,\mu}=K_{\la,\mu}$. 

\vspace{2mm}
\noindent
{\bf Theorem 10. (Gelfand-Graev Fragment for Irreducible and Standard Modules)}
{\em
Let $\la\in X_+(n,n)$ and $\mu\in X(n, n)$. Then we have:
\begin{enumerate}
\item[{\rm (i)}] $\DIM L(\la)_{\gg^\mu}=c(\mu)k_{\la,\mu}.$
\item[{\rm (ii)}] $\DIM \Dede(\la)_{\gg^\mu}=c(\mu)K_{\la,\mu}.$   
\end{enumerate}
}
\vspace{2 mm}

Part (i) of Theorem 10 is proved in Theorem~\ref{TGGFragmDede}, and part (ii) in Corollary~\ref{CCharLLa}. We note that the whole graded  character of $\Dede(\la)$ can be found using the {\em imaginary Jacobi-Trudi formula} of Theorem~\ref{TDet}. 

For $\la\in X_+(n,n)$, let $P(\la)$ be the projective cover of $L(\la)$ as an $\ImS_n$-module. 
The multiplicities $k_{\la,\mu}$ can  be used to strengthen Theorem 4(i) and Theorem 6(i) as follows:

\vspace{2mm}
\noindent
{\bf Theorem 11. (PIMs and Indecomposable Tilting Multiplicities)}
{\em
For $\mu\in X(n,n)$ we have:
\begin{enumerate}
\item[{\rm (i)}] $\Zde^\mu\cong \bigoplus_{\la\vdash n}\Pde(\la)^{\oplus k_{\la,\mu}}$;
\item[{\rm (ii)}] $\Lade^\nu\cong \bigoplus_{\la\vdash n}\Tde(\la^\tr)^{\oplus k_{\la,\mu}}$.
\end{enumerate}
}
\vspace{2 mm}

Note that $\la^\tr$ stands for the partition conjugate (or transpose) to $\la$. 
Theorem 11 is proved in Theorem~\ref{4.5f}. 

\section{Example: type ${\tt A}_1^{(1)}$}
For illustration, we now specialize to the case $\Car={\tt A}_1^{(1)}$ and describe some of our results again in this special case. Then the set of the real roots is: 
$$
\Phi_+^\re=\{n\de+\al_1,n\de+\al_0\mid n\in\Z_{\geq 0}\}. 
$$
There are only two convex preorders:
$$
\al_1\succ\al_1+\de\succ\al_1+2\de\succ\dots\succ \de,2\de,3\de,\dots\succ \dots\succ \al_0+2\de\succ\al_0+\de\succ\al_0
$$
and its opposite:
$$
\al_0\succ\al_0+\de\succ\al_0+2\de\succ\dots\succ \de,2\de,3\de,\dots\succ \dots\succ \al_1+2\de\succ\al_1+\de\succ\al_1
$$
For the second preorder, the theory is completely similar to and is obtained from the theory for the first preorder by switching the roles of $1$ and $0$ (which are  symmetric). So we just work with the first preorder, which is balanced. 

Let us notice that there is an obvious {\em trivial} module $\triv_n$ over $R_{n\de}$ with the the formal character 
$$\bi_{n\de}:=0101\dots 01=(01)^n.$$ 
It is defined as the $1$-dimensional module on which all standard generators of $R_{n\de}$, except $1_{\bi_{n\de}}$ act as $0$. When $n=0$, this is interpreted as the trivial $R_0$-module $\b1$, which is just the trivial $F$-module $F$. 
It will turn out that $\triv_n$ is the imaginary module $L((1^n))$ using the notation of Theorem~6. At any rate, it is immediate  that $\triv_1=L_{\de,1}=L_\de$. 

The cuspidal modules can now be described rather explicitly:

\vspace{2mm}
\noindent
{\bf Theorem A. (Cuspidal modules in type ${\tt A}_1^{(1)}$)}
{\em
Let $\al=\al_1+n\de$ or $\al=\al_0+n\de$. If $n>0$, set $\be=\al-\de$. We have: 

\begin{enumerate}
\item[{\rm (i)}] Let $\al=\al_1+n\de$.

\begin{enumerate}
\item[{\rm (a)}] $0^n1^{n+1}$ is an extremal word of $L_\al$; in particular, 
$
L_{\al}=\tilde f_0^n\tilde f_1^{n+1}\b1. 
$

\item[{\rm (b)}] The standard module 
$\De(\be,(1))= L_\be\circ \triv_1$
is uniserial of length two with head $L(\be,(1))$ and socle $qL_\al$. 
\item[{\rm (c)}] $\displaystyle\CH L_\al=\frac{1}{q-q^{-1}}\big((\CH L_\be)\circ (01)-(01)\circ (\CH L_\be)\big)$. 
\item[{\rm (d)}] $\displaystyle\CH L_\al=\frac{1}{q-q^{-1}}\sum_{m=0}^n (-1)^m(01)^{\circ m}\circ (1) \circ (01)^{\circ (n-m)}.$
\end{enumerate}
\item[{\rm (ii)}] Let $\al_0+n\de$. 
\begin{enumerate}
\item[{\rm (a)}] $0^{n+1}1^{n}$ is an extremal word of $L_\al$; in particular, 
$
L_\al=\tilde f_0^n\tilde f_1^{n+1}\b1. 
$

\item[{\rm (b)}] The standard module 
$\De((1),\be)= \triv_1\circ L_\be$ 
is uniserial of length two with head $L((1),\be)$ and socle $qL_\al$. 
\item[{\rm (c)}] $\displaystyle\CH L_\al=\frac{1}{q-q^{-1}}\big((01)\circ (\CH L_\be)-(\CH L_\be)\circ (01)\big)$. 
\item[{\rm (d)}] $\displaystyle\CH L_\al=\frac{1}{q-q^{-1}}\sum_{m=0}^n (-1)^{n-m}(01)^{\circ m}\circ (0) \circ (01)^{\circ (n-m)}.$
\end{enumerate}
\end{enumerate}
}
\vspace{2 mm}

Theorem A is a special case of \cite[Propositions 6.7, 6.8]{Kcusp}. Let us order the words lexicographically so that $0<1$. Then it follows easily from Theorem A that the word $\bi_\al$ is the highest word of $L_\al$ and $\bj_\al$ is the lowest word, where
$$
\bi_\al:=
\left\{
\begin{array}{ll}
(01)^n1 &\hbox{if $\al=\al_1+n\de$,}\\
0(01)^n &\hbox{if $\al=\al_0+n\de$}
\end{array}
\right.
\quad\text{and}
\quad 
\bj_\al:=
\left\{
\begin{array}{ll}
0^{n}1^{n+1} &\hbox{if $\al=\al_1+n\de$,}\\
0^{n+1}1^{n} &\hbox{if $\al=\al_0+n\de$.}
\end{array}
\right.
$$

Since $l=1$, we only have one color, and only one tensor space
$
M_n=\triv_1^{\circ n}. 
$
Then Theorem 2 is not needed, and Theorems 3-11 hold as stated. Note that the Gelfand-Graev word $\gg^{(n)}$ is simply
$
\gg^{(n)}=0^n1^n.
$
It is easy to check that
$$
\De(1^n)=L(1^n)=\triv_n. 
$$
So the Imaginary Jacobi-Trudi formula of Theorem~\ref{TDet} takes an especially nice form. 

Let $\la=(l_1,\dots,l_a)\vdash n$. Note that the quantum shuffle product is commutative on words of the form $\bi_{r\de}$, i.e. $\bi_{r\de}\circ\bi_{s\de}=\bi_{s\de}\circ\bi_{r\de}$. So we can use the quantum shuffle product to make sense of the  following determinant as an element of $\A \words_{n\de}$: 
$$
D(\la):=\det(\bi_{(l_r-r+s)\de})_{1\leq r,s\leq a} \in \A \words_{n\de}.
$$
where $\bi_{0\cdot\de}$ is interpreted as (a multiplicative) identity, and $\bi_{m\de}$ is interpreted as (a multiplicative) $0$ if $m<0$. 
For example, if $\la=(3,1,1)$, we get
\begin{align*}
D((1,1))&=\det 
\left(
\begin{matrix}
\bi_{\de} & \bi_{2\de}  \\
1 & \bi_{\de}
\end{matrix}
\right)
=\bi_{\de}\circ\bi_{\de}-\bi_{2\de}=(0101)+(q+q^{-1})^2 (0011),
\\
D((3,1,1))&=\det 
\left(
\begin{matrix}
\bi_{3\de} & \bi_{4\de}& \bi_{5\de}  \\
1 & \bi_{\de}& \bi_{2\de}\\
0 & 1 & \bi_\de
\end{matrix}
\right)
=\bi_{3\de}\circ\bi_{\de}\circ \bi_{\de}+\bi_{5\de}-\bi_{4\de}\circ\bi_{\de}-\bi_{3\de}\circ \bi_{2\de}.
\end{align*}

\vspace{2mm}
\noindent
{\bf Theorem B. (Imaginary Jacobi-Trudi Formula in type ${\tt A}_1^{(1)}$)}
{\em
Let $\la\vdash n$. Then
$
\CH\Dede(\la)=D(\la^\tr). 
$
}
\vspace{2 mm}

\section{Structure of the paper}

We now outline the structure of the paper. In Chapter~\ref{SPrel}, we set up basic notation concerning partitions and compositions, recall some useful facts on shortest coset representatives in symmetric groups, and then review standard theory of classical Schur algebras and their representations. 

In Chapter~\ref{ChKLR}, we review the definition and standard properties on KLR algebras,  their basic representation theory, induction and restriction functors, crystal operators, and the Mackey Theorem. 
Then we review standard module theory from \cite{Kcusp}. We discuss  convex preorders, cuspidal systems, colored imaginary tensor spaces and reduction to one color. 

New material begins in Chapter~\ref{ITTCA}. Here we study some fundamental properties of imaginary tensor spaces, deduce imaginary Schur-Weyl duality, define imaginary Schur algebras, and study imaginary induction and restriction. 

The key Chapter~\ref{SISDEP} studies properties of 
Gelfand-Graev representations and their relation to imaginary 
tensor spaces. 
We define imaginary symmetric, divided and exterior powers and describe a convenient projective generator for imaginary Schur algebra. This projective generator is used to obtain a Morita equivalence between imaginary and classical Schur algebras. 

This Morita equivalence is further studied in Chapter~\ref{SSIM}, where we describe and study standard, irreducible and tilting  modules over imaginary Schur algebras, investigate connections between imaginary induction and tensor products, obtain Ringel duality for imaginary Schur algebras, and complete the proofs of double centralizer properties. 

In Chapter~\ref{SFCIM}, we study formal characters of imaginary representations, describe their Gelfand-Graev fragments and obtain an analogue of the Jacobi-Trudy Formula. 

Chapter~\ref{ChLast} is devoted to construction of minuscule imaginary representations for non-simply-laced types and the proof of the corresponding Schur-Weyl duality. We use the Kang-Kashiwara-Kim deformation approach, but we cannot use \cite{KKK} directly, since that paper assumes that the Cartan matrix is symmetric. These results are used in Chapters~\ref{ITTCA}--\ref{SFCIM}, but unfortunately, their proofs are quite technical, so we relegated them to the last chapter. 


\section{Acknowledgements}
We are grateful to Jon Brundan for infinite useful discussions and computer support. We are also grateful to Arun Ram for many useful discussions and hospitality---the idea of imaginary Schur-Weyl duality is essentially due to him. Finally, we are grateful to Peter McNamara, who explained to us that the Coxeter relations for the action of the symmetric group on an imaginary tensor space follow readily  from the work of Kang, Kashiwara, and Kim \cite{KKK}---originally these relations were checked by lengthy case by case computations, which also required computer for exceptional types.

\mainmatter
\chapter{Preliminaries}\label{SPrel}

Throughout the paper, $F$ is a 
field 
of arbitrary  characteristic $p\geq 0$. We also often work over a ring $\O$, which is assumed to be either $\Z$ or $F$. Denote the ring of Laurent polynomials in the indeterminate $q$ by $\A:=\Z[q,q^{-1}]$. 
We use quantum integers 
$$[n]_q:=(q^n-q^{-n})/(q-q^{-1})\in\A\qquad(n\in Z),
$$ 
and the quantum factorials $[n]^!_q:=[1]_q[2]_q\dots[n]_q$. 


\section{Partitions and compositions}\label{SSParComp}
We denote by $X(h,n)$ the set of all compositions of $n$ with $h$ parts (some of which could be zero), $X_+(h,n)$ the set of al partitions of $n$ with at most $h$ parts, and $X_+(n)=X_+(n,n)$ the set of partitions of $n$. Sometimes we write $\la\vdash n$ to indicate that $\la\in X_+(n)$ and $\la\vDash n$ to indicate that $\la\in X(n,n)$. The standard {\em dominance order} on $X(h,n)$ is denoted by ``$\leq$". 


We will use the special elements $\eps_1,\dots,\eps_h\in X(h,1)$, where $$\eps_m=(0,\dots,0,1,0,\dots,0)$$ 
with $1$ in the $m$th position. 
For a composition $\mu\vDash n$ we denote by $\mu^+\vdash n$ the unique partition obtained from $\mu$ by a permutation of its parts. 
For $\la \vdash n$, we have its {\em transpose} partition $\la^\tr\vdash n$. 

If $p>0$, then $\la\in X_+(h,n)$ is {\em $p$-restricted} if $\la_r-\la_{r+1}<p$ for all $r=1,2,\dots,h-1$. 
A {\em $p$-adic expansion}\, of $\la$ is some (non-unique) way of writing
$\la=\la(0) +p\la(1) +p^2\la(2) +\dots$ 
such that each $\la(i) \in X_+(h,n(i))$ is $p$-restricted. 
This can be applied to a partition $\la\vdash n$ considered as an element of $X_+(n,n)$, in which case the $n$th part $\la_n\leq 1$, and so the $p$-adic expansion is unique.

\section{Coset representatives}\label{SSCosetRep}
Let $\la=(\la_1,\dots,\la_a)\vDash n$, and let $\Si_\la$ be the corresponding standard parabolic subgroup of $\Si_n$, i.e. $\Si_\la$ is the row stabilizer in $\Si_n$ of the {\em row leading tableau}\, $\T^\la$ obtained by allocating the numbers $1,\dots,n$ into the boxes of $\la$ from left to right in each row starting from the first row and going down. The {\em column leading tableau}\, $\T_\la$ obtained by allocating the numbers $1,\dots,n$ into the boxes of $\la$ from top to bottom in each column starting from the first column and going to the right. 
Denote
$$
{\mathtt x}_\la=\sum_{w\in\Si_\la}w,\quad {\mathtt y}_\la=\sum_{w\in\Si_\la}\sgn(w)w\qquad(\la\vDash n),
$$
where $\sgn(w)$ is the sign of the permutation $w\in\Si_n$. Sometimes we also use the notation $\sgn_{\Si_n}$ and $\sgn_{\Si_\la}$ to denote the sign representations of the corresponding groups.

Let $x\in \Si_n$. 
If 
$
x\Si_\la x^{-1}
$ is a standard parabolic subgroup, say $\Si_\mu$ for some composition $\mu$, we write  $\mu=:x\la$ and say that $x$ {\em permutes the parts of $\la$}, i.e. in that case we have 
$$
x\Si_\la x^{-1}=\Si_{x\la}. 
$$


We recall some standard facts on minimal length coset representatives in symmetric groups, see e.g. \cite[Section 1]{DJ}. 
For $\la\vDash n$, denote by $\D^\la_n$ (resp. ${}^\la\D_n$) the set of the minimal length left (resp. right) coset representatives of $\Si_\la$ in $\Si_n$. 
Note that the {\em permutation module} 
$$
\Perm^\la:=\ind_{\O\Si_\la}^{\O\Si_n}{\tt triv}_{\Si_\la}\simeq \O\Si_n{\mathtt x}_\la
$$
has an $\O$-basis $\{g\otimes 1\mid g\in \D^\la_n\}$, and similarly for the {\em signed permutation module}
$$
\SPerm^\la:=\ind_{\O\Si_\la}^{\O\Si_n}{\sgn}_{\Si_\la}\simeq \O\Si_n{\mathtt y}_\la\simeq \Perm^\la\otimes \sgn_{\Si_n}.
$$

More generally, if $\nu\vDash n$ and $\la$ is a refinement of $\nu$, denote $\D^\la_\nu:=\D^\la_n\cap\Si_\nu$ and 
${}^\la\D_\nu:={}^\la\D_n\cap\Si_\nu$. Then $\D^\la_\nu$ (resp.  
${}^\la\D_\nu$) is set of the minimal length left (resp. right) coset representatives of $\Si_\la$ in $\Si_\nu$. Moreover,
\begin{equation}\label{E12411}
\D_n^\la=\{xy\mid x\in\D_n^\nu,\ y\in \D_\nu^\la\}\ \text{and}\  {}^\la\D_n=\{yx\mid x\in{}^\nu\D_n,\ y\in {}^\la\D_\nu\}.
\end{equation}

For two compositions $\la,\mu\vDash n$ set ${}^\la\D_n^\mu:=\D^\mu_n\cap{}^\la\D_n$. Then ${}^\la\D_n^\mu$ is the set of the minimal length $(\Si_\la,\Si_\mu)$-double coset representatives in $\Si_n$. If $x\in {}^\la\D_n^\mu$, then $\Si_\la\cap x\Si_\mu x^{-1}$ is a standard parabolic in $\Si_n$. This standard parabolic corresponds to certain composition of $n$, which we denote $\la\cap x\mu$. Similarly, $x^{-1}\Si_\la x\cap\Si_\mu$ is the  standard parabolic corresponding to a composotion $x^{-1}\la\cap \mu$. Thus:
\begin{equation}\label{EDJ1}
\Si_\la\cap x\Si_\mu x^{-1}=\Si_{\la\cap x\mu},\quad 
x^{-1}\Si_\la x\cap\Si_\mu=\Si_{x^{-1}\la\cap \mu}\qquad(x\in{}^\la\D_n^\mu).
\end{equation}
Moreover, $x$ permutes the parts of $x^{-1}\la\cap \mu$, and $x(x^{-1}\la\cap \mu)= \la\cap x\mu$, so
$$
x\Si_{x^{-1}\la\cap \mu}x^{-1}=\Si_{\la\cap x\mu}.
$$


For $\la\vdash n$ define $u_\la$ to be the unique element of ${}^{\la^\tr}\D^\la$ such that $\Si_{\la^\tr} \cap u_\la\Si_\la u_\la^{-1} = \{1\}$; in other words, $u_\la$ is defined from $u_\la\T^\la=\T_\la$.

\begin{Lemma} \label{1.1d} {\rm \cite[Lemma 4.1]{DJ}} 
If $\la\vdash n$, then ${\mathtt y}_{\la^\tr}\O\Si_n {\mathtt x}_\la$ is an $\O$-free $\O$-module of rank one, generated by the element ${\mathtt y}_{\la^\tr}u_\la {\mathtt x}_\la$.
\end{Lemma}

\section{Schur algebras}\label{SSchurPrel}\label{SSSA}
The necessary information on Schur algebras is conveniently gathered in \cite[Section 1]{BDK}. We recall only some most often needed facts for reader's convenience. The {\em Schur algebra} $S_{h,n}=S_{h,n,\O}$ is defined to the endomorphism algebra 
$$
S_{h,n}:=\End_{\O\Si_n}\Big(\bigoplus_{\nu\in X(h,n)}\Perm^\nu\Big),
$$
writing endomorphisms commuting with the left action of $\O\Si_n$ on the right.

Let $\la,\mu\in X(h,n)$ and $u\in \Si_n$. The right multiplication in $\O\Si_n$ by 
\begin{equation}\label{EGUMuLa}
g_{\mu,\la}^u:=\sum_{w\in \Si_\mu u\,\Si_\la\,\cap\,{}^\mu D}w
\end{equation}
induces a well-defined homomorphism of left $\O\Si_n$-modules 
$$
\phi^u_{\mu,\la} : \Perm^\mu\to \Perm^\la.
$$
Extending $\phi^u_{\mu,\la}$ to all of $\bigoplus_{\nu\in X(h,n)}\Perm^\nu$ by letting it act as zero on $\Perm^\nu$ for $\nu\neq\mu$, we obtain a well-defined element 
\begin{equation}\label{EPhiULaMu}
\phi^u_{\mu,\la}\in S_{h,n}.
\end{equation}

\begin{Lemma} \label{1.2a} 
$S_{h,n}$ is $\O$-free with basis $\{\phi^u_{\mu,\la}\mid \mu,\la\in X(h,n),\ u\in {}^\mu\D^\la\}.$
\end{Lemma}

\begin{Lemma} \label{1.2b}	
For $h \geq n$, the $\O$-linear map $\kappa: F\Si_n\to S_{h,n}$, defined on a basis element $w\in \Si_n$ by $\kappa(w) := \phi^w_{(1^n),(1^n)}$, is a (unital) ring embedding.
\end{Lemma}

One can also define the Schur algebra using the
signed permutation modules. 
So consider instead the algebra 
$$
\End_{\O\Si_n}\Big(\bigoplus_{\nu\in X(h,n)}\SPerm^\nu\Big).
$$
For $\la,\mu\in X(h,n)$ and $u\in \Si_n$ set  
\begin{equation}\label{ESUMuLa}
s_{\mu,\la}^u:=\sum_{w\in \Si_\mu u\,\Si_\la\,\cap\,{}^\mu D}\sgn(w)w.
\end{equation}

\begin{Lemma} \label{1.2c} 
The algebras $S_{h,n}$ and $\End_{\O\Si_n}\Big(\bigoplus_{\nu\in X(h,n)}\SPerm^\nu\Big)$ are isomorphic, the natural basis element $\phi^u_{\mu,\la}$ of $S_{h,n}$ corresponding under the isomorphism to the endomorphism which is zero on $\SPerm^\nu$ for $\nu\neq\mu$ and sends $\SPerm^\mu$ into $\SPerm^\la$ via the homomorphism induced by right multiplication in $\O\Si_n$ by $s_{\mu,\la}^u$.
\end{Lemma}

\section{Representation theory of Schur algebras} \label{SSSchurRT}
We recall some facts about the representation theory of $S_{h,n}$, assuming now that $\O=F$. All the results gathered here are explained in detail and properly referenced in \cite[Section 1]{BDK}. First of all, it is known that the elements
\begin{equation}\label{EIdSchur}
e(\mu):=\phi^1_{\mu,\mu}\in S_{h,n} \qquad(\mu\in X(h,n))
\end{equation}
are idempotents. We have a {\em weight space decomposition} for $W\in\mod{S_{h,n}}$:
$$
W=\bigoplus_{\mu\in X(h,n)} e(\mu)W.
$$
The subspaces $e(\mu)W$ are the {\em weight spaces of $W$}.

The irreducible $S_{h,n}$-modules are parametrized by the elements of  $X_+(h,n)$. We write $L_h(\la)$ for the irreducible $S_{h,n}$-module corresponding to $\la\in X_+(h,n)$. In particular, $L_h(\la)$ has highest weight $\la$, i.e. $e(\la)L_h(\la)\neq 0$ and $e(\mu)L_h(\la)= 0$ for all $\mu\in X(h,n)$ with $\mu\not\leq\la$. 
It is known that $S_{h,n}$ is a quasi-hereditary algebra with weight poset $(X_+(h,n),\leq)$. 
In particular, we have associated to $\la\in X_+(h,n)$ the standard and costandard modules $\De_h(\la)$ and $\nabla_h(\la)$ such that $\De_h(\la)$ (resp. $\nabla_h(\la)$) has simple head (resp. socle) isomorphic to $L_h(\la)$, and all other composition factors are of the form $L_h(\mu)$ with $\mu<\la$.

For $\la\in X_+(h,n)$ and $\nu\in X(h,n)$, denote by $k_{\la,\nu}$ the dimension of the $\nu$-weight space of $L_h(\la)$:
\begin{equation}\label{EWtSp}
k_{\la,\nu}:=\dim e(\nu)L_h(\la).
\end{equation}
In particular, if $\cha F=0$, then $L_h(\la)=\De_h(\la)$ and so it is well-known that $k_{\la,\nu}=K_{\la,\nu}$, where 
\begin{equation}\label{EKostka}
K_{\la,\nu}:=\sharp\{\text{semistandard $\la$-tableaux of type $\nu$}\},
\end{equation}
also known as the $(\la,\nu)$-Kostka number. To give the necessary definitions, we consider $\la$ as a partition, and so we can speak of the corresponding Young diagram. A {\em $\la$-tableau} is an allocation of numbers from the set $\{1,\dots,h\}$ (possibly with repetitions) into the boxes of the Young diagram $\la$. 
A $\la$-tableau is of type $\nu$ if each $1\leq k\leq n$ appears in it exactly $\nu_k$ times. 
A $\la$-tableau is {\em column strict} if its entries increase down the columns.  A $\la$-tableau is {\em row weak} if its entries weakly increase from left to right along the rows. A $\la$-tableau is {\em semistandard} if it is row weak and column strict. 

The algebra $S_{h,n}$ possesses an anti-automorphism $\tau$ defined on the standard basis elements	by $\tau(\phi^u_{\mu,\la})=\phi^{u^{-1}}_{\la,\mu}$.  Using this, we define the contravariant dual $M^\tau$ of an
$S_{h,n}$-module $M$ to be the dual vector space $M^*$ with action defined by $(s\cdot f)(m) = f(\tau(s)m)$ for all $s \in S_{h,n}, m \in M,f \in M^*$. We have $L_h(\la)^\tau\simeq L_h(\la)$ and $\De_h(\la)^\tau\simeq \nabla_h(\la)$
for all $\la\in  X_+(h, n)$. 

Given a left $S_{h,n}$-module $M$, we write $\tilde M$ for the right $S_{h,n}$-module equal to $M$ as a vector space with right action defined by $ms = \tau(s)m$ for $m \in M,s \in S_{h,n}$. This gives us modules $\tilde L_h(\la), \tilde\De_h(\la)$ and $\tilde\nabla_h(\la)$ for each $\la\in X_+(h, n)$.

\begin{Lemma} \label{1.2e} 
$S_{h,n}$	has	a	filtration	as	an $(S_{h,n},S_{h,n})$-bimodule with factors	isomorphic to $\De_h(\la) \otimes \tilde\De_h(\la)$, each appearing once for each $\la\in X_+(h, n)$ and ordered in any way refining the dominance order on partitions so that factors corresponding to more dominant $\la$ appear lower in the filtration.
\end{Lemma}

We have an algebra map $S_{h,n+l} \to S_{h,n} \otimes S_{h,l}$, 
which enables us to view the tensor product $M \otimes M'$ of an $S_{h,n}$-module $M$  and an $S_{h,l}$-module $M'$ as an $S_{h,n+l}$-module. 
Let $V_h = L_h((1))$ be the {\em natural module}. The $n$th tensor power $V_h^{\otimes n}$ can
be regarded as an $S_{h,n}$-module. We also have the {\em symmetric}, {\em divided} and {\em exterior powers}: 
$S^n(V_h) = \nabla_h((n))$, $Z^n(V_h) = \De_h((n))$, $\La^n(V_h) = L_h((1^n))$.
More generally, given $\nu = (n_1, \dots , n_a) \in X(h,n)$, define 
\begin{eqnarray}
\label{1.3.2}
S^\nu(V_h)&:=S^{n_1}(V_h)\otimes\dots\otimes S^{n_a}(V_h),
\\
\label{1.3.3}
Z^\nu(V_h)&:=Z^{n_1}(V_h)\otimes\dots\otimes Z^{n_a}(V_h),
\\
\label{1.3.4}
\La^\nu(V_h)&:=\La^{n_1}(V_h)\otimes\dots\otimes \La^{n_a}(V_h).
\end{eqnarray}
all of which can be regarded as $S_{h,n}$-modules.

\begin{Lemma} \label{1.3b}
For $\nu  \in X(h,n)$ we have:
\begin{enumerate}
\item[{\rm (i)}]  the left ideal $S_{h,n}e(\nu)$ of $S_{h,n}$ is isomorphic to $Z^\nu(V_h)$ as an $S_{h,n}$-module; 
\item[{\rm (ii)}] providing $h \geq n$, the left ideal $S_{h,n}\kappa({\mathtt y}_\nu)$ of $S_{h,n}$ is isomorphic to $\La^\nu(V_h)$ as
an $S_{h,n}$-module, where $\kappa : F\Si_n(V)\to S_{h,n}$ is the embedding of Lemma~\ref{1.2b}.
\end{enumerate}
\end{Lemma}

A finite dimensional $S_{h,n}$-module M has a {\em standard} (resp.  {\em costandard}) filtration if $M$  has a filtration $0 = M_0 \subset M_1 \subset \dots\subset M_r=M$ such that each factor $M_i/M_{i-1}$ is isomorphic to a direct sum of copies of $\De_h(\la)$ (resp. $\nabla_h(\la)$) for some fixed $\la\in X_+(h,n)$ (depending on $i$). 

\begin{Lemma} \label{1.3c} 
If $M,M'$ are modules with standard (resp. costandard) filtrations 
then so is $M \otimes  M'$. 
\end{Lemma}

In particular, Lemma~\ref{1.3c} implies that for any $\nu  \in X(h,n)$, the modules $S^\nu(V_h)$, and $\La^\nu(V_h)$ have costandard filtrations, while $Z^\nu(V_h)$ and $\La^\nu(V_h)$ have standard filtrations. 

\begin{Lemma} \label{1.3d} 
Let $h\geq n$ and $\la\in X_+(h,n)$, we have that the homomorphism space $\Hom_{S_{h,n}}(Z^\la(V_h), \La^{\la^\tr}(V_h))$ is one-dimensional, and the image of any non-zero such homomorphism is isomorphic to $\De_h(\la)$.
\end{Lemma}

If $p>0$, then for and $h,n,r\in\Z_{\geq 0}$, there is a {\em Frobenius homomorphism} 
$$
F_r:S_{h,np^r}\to S_{h,n}, 
$$
twisting with which one gets the {\em Frobenius twist} functor
$$
\mod{S_{h,n}}\to\mod{S_{h,np^r}},\ M\mapsto M^{[r]}.
$$
For example $L(\la)^{[r]}\simeq L(p^r\la)$. The Steinberg tensor product theorem is:

\begin{Lemma} \label{1.3e} 
 Suppose that $\la\in X_+(h,n)$ has $p$-adic expansion
 $\la=\la(0) +p\la(1) +p^2\la(2) +\dots$. 
Then, $L_h(\la) \simeq L_h(\la(0))\otimes L_h(\la(1))^{[1]} \otimes L_h(\la(2))^{[2]} \otimes \dots$.
\end{Lemma}

\section{Induction and restriction for Schur algebras}\label{SSIRSA}

For a composition $\chi = (h_1,\dots,h_a)\vDash h$ there is a natural {\em Levi Schur subalgebra} 
\begin{equation}\label{ELevi}
S_{\chi,n}\simeq \bigoplus_{n_1+\dots+n_a=n} S_{h_1,n_1}\otimes\dots\otimes S_{h_a,n_a}\subseteq S_{h,n},
\end{equation}
and the usual restriction, and  induction functors:
$$
\res^{S_{h,n}}_{S_{\chi,n}} :\mod{S_{h,n}} \to \mod{S_{\chi,n}}, \quad 
\ind^{S_{h,n}}_{S_{\chi,n}} :\mod{S_{\chi,n}} \to \mod{S_{h,n}}.
$$

Moreover, fix $l \leq h$ and embed $X(l, k)$ into $X(h, k)$ in the natural way. Let $e$ be the idempotent 
\begin{equation}\label{1.5.1}
e = e_{h,l}: =\sum_{\mu\in X(l, k)}e(\mu) \in S_{h,n}.
\end{equation}

\begin{Lemma} \label{LTruncSchur} 
We have $S_{l,n}\simeq eS_{h,n}e$. 
\end{Lemma}

Then, we have the Schur functor 
\begin{equation}\label{1.5.2}
\trun^{S_{h,n}}_{S_{l,n}} :\mod{S_{h,n}}\to\mod{S_{l,n}},\ M\mapsto eM
\end{equation}
and its left adjoint 
\begin{equation}\label{1.5.3}
\infl^{S_{h,n}}_{S_{l,n}} :\mod{S_{l,n}}\to\mod{S_{h,n}},\ N\mapsto S_{h,n}e\otimes_{eS_{h,n}e} N.
\end{equation}

\begin{Lemma} \label{1.5a} 
If $n\leq l\leq h$, then the functors $\trun^{S_{h,n}}_{S_{l,n}}$	and	$\infl^{S_{h,n}}_{S_{l,n}}$ are mutually quasi-inverse equivalences of categories.
\end{Lemma}

\begin{Lemma} \label{1.5b} 
Let $l<h$ and $\mu=(\mu_1,\dots,\mu_h)\in X_+(h,n)$.
\begin{enumerate}
\item[{\rm (i)}] If $\mu_{l+1}\neq 0$ then $\trun^{S_{h,n}}_{S_{l,n}} L_h(\mu) = \trun^{S_{h,n}}_{S_{l,n}} \De_h(\mu) = \trun^{S_{h,n}}_{S_{l,n}} \nabla_h(\mu) = 0$. 
\item[{\rm (ii)}] If $\mu_{l+1} = 0$, we may regard $\mu$  as an element of $X_+(l, n)$, and then we have $\trun^{S_{h,n}}_{S_{l,n}} L_h(\mu)\simeq L_l(\mu)$, $\trun^{S_{h,n}}_{S_{l,n}} \De_h(\mu)\simeq \De_l(\mu)$, and\, $\trun^{S_{h,n}}_{S_{l,n}} \nabla_h(\mu)\simeq \nabla_l(\mu)$.
\end{enumerate}
\end{Lemma}

\begin{Lemma} \label{1.5d} 
If $\chi = (h_1,\dots,h_a) \vDash h$ and $\nu = (n_1,\dots,n_a) \vDash n$, with $h_r\geq n_r$ for all $r=1,\dots,a$, then 
$$
\ind_{S_{\chi,n}}^{S_{h,n}}(? \boxtimes\dots\boxtimes \,?)\qquad  \text{and}\qquad 
(\infl_{S_{h_1,n_1}}^{S_{h,n_1}} ?)\otimes\dots\otimes (\infl_{S_{h_a,n_a}}^{S_{h,n_a}} ?)
$$
are isomorphic functors from 
$\mod{S_{h_1,n_1}}\times\dots\times\mod{S_{h_a,n_a}}$ to $\mod{S_{h,n}}$.
\end{Lemma}
 
 \section{Schur functors}\label{SSSF}
 Here we review the material of \cite[Section~ 3.1]{BDK} for future references. Let $\ImS$ be a finite dimensional $F$-algebra, $P\in\mod{\ImS}$ be a projective module, and $H = \End_\ImS(P)$, writing endomorphisms commuting with the left $\ImS$-action on the right. Define the functors: 
\begin{align*}
 \al&:=\Hom_{\ImS}(P,?):\mod{\ImS}\to \mod{H},
 \\
 \be&:=P\otimes_H?:\mod{H}\to \mod{\ImS}.
 \end{align*}
The $\al$ is exact, and $\be$ is left adjoint to $\al$.

Given an $\ImS$-module $V$, let $O_P(V )$ denote the largest submodule $V'$ of $V$  such that $\Hom_\ImS(P,V')=0$. Let $O^P(V)$ denote the submodule of $V$ generated by the images of all $\ImS$-homomorphisms from $P$ to $V$. 
Any $\ImS$-module homomorphism $V \to W$ sends $O_P(V )$ into $O_P(W )$ and $O^P(V)$ into $O^P(W)$, so we can view $O_P$ and $O^P$ as functors $\mod{\ImS}\to\mod{\ImS}$. Finally, any homomorphism $V \to W$ induces a well-defined $\ImS$-module homomorphism $V/O_P (V ) \to W/O_P (W)$. We thus obtain an exact functor $A_P : \mod{\ImS} \to \mod{\ImS}$ defined on objects by $V \mapsto V/O_P(V)$. The following two lemmas can be found for example in \cite[3.1a, 3.1c]{BDK}:

\begin{Lemma} \label{L3.1a}
The functors $\al\circ \be$ and $\al\circ A_P\circ\be$ are both isomorphic to the identity.
\end{Lemma}

\begin{Lemma} \label{L3.1c} 
If $V, W \in \mod{\ImS}$ satisfy $O^P(V) = V$ and $O_P(W) = 0$, then $\Hom_{\ImS}(V , W ) \simeq \Hom_H(\al( V ) , \al( W ) )$.
\end{Lemma}

The main result on the functors $\al,\be$, proved for example in \cite[3.1d]{BDK},  is:

\begin{Theorem} \label{T3.1d}
The functors $\al$ and $A_P \circ\be$ induce mutually inverse equivalences of categories between $\mod{H}$ and the full subcategory 
of $\mod{\ImS}$ consisting of all $V \in\mod{\ImS}$ such that $O_P(V)=0$ and $O^P(V)=V$.
\end{Theorem}

An easy consequence is the following relation between the irreducible modules, see \cite[3.1e]{BDK}:

\begin{Lemma} \label{L3.1e} 
Let $\{E_m \mid m \in M\}$ be a complete set of non-isomorphic irreducible $\ImS$-modules appearing in the head of $P$. For all $m\in M$, set $D_m := \al(E_m)$. Then, $\{D_m \mid m\in M\}$ is a complete irredundant set of irreducible $H$-modules, and $A_P \circ \beta(D_m) \simeq E_m$.
\end{Lemma}

Finally, we will make use of the following more explicit description of the effect of the composite functor $A_P \circ \beta$ on left ideals of H, see \cite[3.1f]{BDK}:

\begin{Lemma} \label{L3.1f}
Suppose that every composition factor of the socle of $P$ also appears in its head. Then for any left ideal $J$ of $H$, we have $A_P \circ \beta(J) \simeq PJ$.
\end{Lemma}

\chapter{Khovanov-Lauda-Rouquier algebras}\label{ChKLR}
\section{Lie theoretic notation}\label{SSLTN}
Throughout the paper 
$$\Car=(\cc_{ij})_{i,j\in I}$$ 
is a {\em Cartan matrix} of {\em untwisted affine type}, see \cite[\S 4, Table Aff 1]{Kac}. 
We have 
$$I=\{0,1,\dots,l\},$$ 
where $0$ is the affine vertex. 
Following \cite[\S 1.1]{Kac}, let $(\h,\Pi,\Pi^\vee)$ be a realization of the Cartan matrix $\Car$, so we have simple roots $\{\al_i\mid i\in I\}$, simple coroots $\{\al_i^\vee\mid i\in I\}$, and a bilinear form $(\cdot,\cdot)$ on $\h^*$ such that 
$$\cc_{ij}=\frac{2(\al_i,\al_j)}{(\al_i,\al_i)}$$ for all $i,j\in I$. We normalize $(\cdot,\cdot)$ so that $(\al_i,\al_i)=2$ if $\al_i$ is a short simple root. 

The {\em fundamental dominant weights} $\{\La_i\mid i\in I\}$ have the property $\lan\La_i,\al_j^\vee\ran=\de_{i,j}$, where $\lan\cdot,\cdot\ran$ is the natural pairing between $\h^*$ and $\h$. We have the {\em integral weight lattice} $P=\oplus_{i\in I}\Z\cdot\La_i$ and the set of {\em dominant weights} $P_+=\sum_{i\in I}\Z_{\geq 0}\cdot\La_i$. 
For $i\in I$ we define
\begin{equation}\label{EQuantumFactorials}
q_i:=q^{(\al_i,\al_i)/2},\quad   [n]_i:=[n]_{q_i},\quad [n]_i^!:=[n]_{q_i}^!.
\end{equation}

Denote 
$Q_+ := \bigoplus_{i \in I} \Z_{\geq 0} \cdot \al_i$. For $\alpha \in Q_+$, we write $\height(\alpha)$ for the sum of its 
coefficients when expanded in terms of the $\al_i$'s.


Let $\g'=\g(\Car')$ be the finite dimensional simple Lie algebra  whose Cartan matrix $\Car'$ corresponds to the subset of vertices $I':=I\setminus\{0\}$. The affine Lie algebra $\g=\g(\Car)$ is then obtained from $\g'$ by a procedure described in \cite[Section 7]{Kac}. We denote by $W$ (resp. $W'$) the corresponding {\em affine Weyl group} (resp. {\em finite Weyl group}). It is a Coxeter group with standard generators $\{r_i\mid i\in I\}$ (resp. $\{r_i\mid i\in I'\}$), see \cite[Proposition~3.13]{Kac}.

Let $\Phi'$ and $\Phi$ be the root systems of $\g'$ and $\g$ respectively. Denote by $\Phi'_+$ and $\Phi_+$ the set of {\em positive}\, roots in $\Phi'$ and $\Phi$, respectively, cf. \cite[\S 1.3]{Kac}. 
Let 
\begin{equation}\label{EDelta}
\de=a_0\al_0+a_1\al_1+\dots+a_l\al_l.
\end{equation}
By \cite[Table Aff 1]{Kac}, we always have 
\begin{equation}\label{Ea_0}
a_0=1.
\end{equation}
We have 
\begin{equation}\label{EHRoot}
\de-\al_0=\theta,
\end{equation} 
where $\theta$ is the highest root in the finite root system $\Phi'$. 
Finally,  
$\Phi_+=\Phi_+^\im\sqcup \Phi_+^\re
$, where
$$
\Phi_+^\im=\{n\de\mid n\in\Z_{>0}\}
$$
and 
\begin{align*}
\Phi_+^\re=\{\be+n\de\mid \be\in  \Phi'_+,\ n\in\Z_{\geq 0}\}\sqcup \{-\be+n\de\mid \be\in  \Phi'_+,\ n\in\Z_{> 0}\}.
\end{align*}

\section{The definition and first properties}\label{SSDefKLR}
Define the polynomials in the variables $u,v$ 
$$\{Q_{ij}(u,v)\in F[u,v]\mid i,j\in I\}$$ 
as follows. For the case where the Cartan matrix $\Car\neq {\tt A}_1^{(1)}$, 
choose signs $\eps_{ij}$ for all $i,j \in I$ with $\cc_{ij}
< 0$  so that $\eps_{ij}\eps_{ji} = -1$.
Then set: 
\begin{equation}\label{EArun}
Q_{ij}(u,v):=
\left\{
\begin{array}{ll}
0 &\hbox{if $i=j$;}\\
1 &\hbox{if $\cc_{ij}=0$;}\\
\eps_{ij}(u^{-\cc_{ij}}-v^{-\cc_{ji}}) &\hbox{if $\cc_{ij}<0$.}
\end{array}
\right.
\end{equation}
For type $A_1^{(1)}$ we define
\begin{equation}\label{EArun1}
Q_{ij}(u,v):=
\left\{
\begin{array}{ll}
0 &\hbox{if $i=j$;}\\
(u-v)(v-u) &\hbox{if $i\neq j$.}
\end{array}
\right.
\end{equation}

Fix 
$\al\in Q_+$ of height $d$. Let
\begin{equation}\label{EWords}
\words_\al=\{\bi=(i_1,\dots,i_d)\in I^d\mid \al_{i_1}+\dots+\al_{i_d}=\al\}. 
\end{equation}
The {\em KLR-algebra} $R_\al=R_\al(\O)$ is an associative graded unital $\O$-algebra, given by the generators
\begin{equation}\label{EKLGens}
\{1_\bi\mid \bi\in \words_\al\}\cup\{y_1,\dots,y_{d}\}\cup\{\psi_1, \dots,\psi_{d-1}\}
\end{equation}
and the following relations for all $\bi,\bj\in \words_\al$ and all admissible $r,t$:
\begin{equation}
1_\bi  1_\bj = \de_{\bi,\bj} 1_\bi ,
\quad{\textstyle\sum_{\bi \in \words_\alpha}} 1_\bi  = 1;\label{R1}
\end{equation}
\begin{equation}\label{R2PsiY}
y_r 1_\bi  = 1_\bi  y_r;\quad y_r y_t = y_t y_r;
\end{equation}
\begin{equation}
\psi_r 1_\bi  = 1_{s_r\bi} \psi_r;\label{R2PsiE}
\end{equation}
\begin{equation}
(y_t\psi_r-\psi_r y_{s_r(t)})1_\bi  
= \de_{i_r,i_{r+1}}(\de_{t,r+1}-\de_{t,r})1_\bi;
\label{R6}
\end{equation}
\begin{equation}
\psi_r^21_\bi  = Q_{i_r,i_{r+1}}(y_r,y_{r+1})1_\bi 
 \label{R4}
\end{equation}
\begin{equation} 
\psi_r \psi_t = \psi_t \psi_r\qquad (|r-t|>1);\label{R3Psi}
\end{equation}
\begin{equation}
\begin{split}
&(\psi_{r+1}\psi_{r} \psi_{r+1}-\psi_{r} \psi_{r+1} \psi_{r}) 1_\bi  
\\=
&
\de_{i_r,i_{r+2}}\frac{Q_{i_r,i_{r+1}}(y_{r+2},y_{r+1})-Q_{i_r,i_{r+1}}(y_r,y_{r+1})}{y_{r+2}-y_r}1_\bi.
\end{split}
\label{R7}
\end{equation}
The {\em grading} on $R_\al$ is defined by setting:
$$
\deg(1_\bi )=0,\quad \deg(y_r1_\bi )=(\al_{i_r},\al_{i_r}),\quad\deg(\psi_r 1_\bi )=-(\al_{i_r},\al_{i_{r+1}}).
$$

\vspace{2 mm}

It is pointed out in \cite{KL2} and \cite[\S3.2.4]{R} that up to isomorphism $R_\al$ depends only on the Cartan matrix and $\al$. 

Fix in addition a dominant weight $\La\in P_+$. The corresponding {\em cyclotomic KLR algebra} $R_\al^\La$ is the quotient of $R_\al$ by the following ideal:
\begin{equation}\label{ECyclot}
J_\al^\La:=(y_1^{\lan\La,\al_{i_1}^\vee\ran}1_\bi \mid \bi=(i_1,\dots,i_d)\in\words_\al). 
\end{equation}

For each element $w\in S_d$ fix a reduced expression $w=s_{r_1}\dots s_{r_m}$ and set 
$$
\psi_w:=\psi_{r_1}\dots \psi_{r_m}.
$$
In general, $\psi_w$ depends on the choice of the reduced expression of $w$. 

\begin{Theorem}\label{TBasis}{\rm \cite[Theorem 2.5]{KL1}, \cite[Theorem 3.7]{R}}  
The elements 
$$ \{\psi_w y_1^{m_1}\dots y_d^{m_d}1_\bi \mid w\in S_d,\ m_1,\dots,m_d\in\Z_{\geq 0}, \ \bi\in \words_\al\}
$$ 
form an $\O$-basis of  $R_\al$. 
\end{Theorem}

There exists a homogeneous algebra anti-involution 
\begin{equation}\label{ECircledast}
\tau:R_\al\longrightarrow R_\al,\quad 1_\bi\mapsto 1_\bi,\quad y_r\mapsto y_r,\quad \psi_s\mapsto \psi_s  
\end{equation}
for all $\bi\in \words_\al,\ 1\leq r\leq d$, and $1\leq s<d$. If $M=\bigoplus_{d\in\Z}M_d$ is a finite dimensional 
graded $R_\al$-module, then the {\em graded dual}
$M^\circledast$ is the graded $R_\al$-module such that $(M^\circledast)_n:=\Hom_\O(M_{-n},\O)$, for all
$n\in\Z$, and the $R_\al$-action is given by $(xf)(m)=f(\tau(x)m)$, for all $f\in M^\circledast, m\in M, x\in
R_\al$.

\section{Basic representation theory of $R_\al$}\label{SSBasicRep} 
Let $H$ be any ($\Z$-)graded $F$-algebra. By a module $V$ 
over $H$, we {\em always} mean a graded left $H$-module.
We denote by $\Mod{H}$ the abelian category of all graded left $H$-modules, with 
morphisms being {\em degree-preserving} module homomorphisms, which we denote by $\Hom$.
Let $\mod{H}$ denote
the abelian subcategory of all
{\em finite dimensional}\, graded $H$-modules, and 
 $[\mod{H}]$ be  the corresponding Grothendieck group. Then $[\mod{H}]$ is an $\A$-module via 
$
q^m[M]:=[q^m M],
$ 
where $q^m M$ denotes the module obtained by 
shifting the grading up by $m$, i.e. 
$
(q^mM)_n:=M_{n-m}.
$ 
We denote by $\hom_H(M,N)$ the space of morphism in $\Mod{H}$, i.e. degree zero homogeneous $H$-module homomorphisms. Similarly we have $\ext_H^m(M,N)$. 

For $n \in \Z$, let
$
\Hom_H(M, N)_n := \hom_H(q^nM , N)
$
denote the space of all homomorphisms
that are homogeneous of degree $n$.
Set
$$
\Hom_H(M,N) := \bigoplus_{n \in \Z} \Hom_H(M,N)_n. 
$$
For graded $H$-modules $M$ and $N$ we write $M\cong N$ to mean that $M$ and $N$ are isomorphic as graded modules and $M\simeq N$ to mean that they are isomorphic as $H$-modules after we forget the gradings. 

For a finite dimensional 
graded vector space $V=\oplus_{n\in \Z} V_n$, its {\em graded dimension} is $\DIM \, V:=\sum_{n \in \Z}  (\dim V_n)q^n\in\A$. 
Given $M, L \in \mod{H}$ with $L$ irreducible, 
we write $[M:L]_q$ for the corresponding {\em  graded composition multiplicity},
i.e. 
$
[M:L]_q := \sum_{n \in \Z} a_n q^n,
$ 
where $a_n$ is the multiplicity
of $q^nL$ in a graded composition series of $M$.

Going back to the algebras $R_\al=R_\al(F)$, every irreducible graded $R_\al$-module is finite dimensional \cite[Proposition 2.12]{KL1}, and  there are finitely many irreducible modules in $\mod{R_\al}$ up to isomorphism and grading shift \cite[\S 2.5]{KL1}. A prime field is a splitting field for $R_{\al}$, see \cite[Corollary 3.19]{KL1}, so working with irreducible $R_\al$-modules we do not need to assume that $F$ is algebraically closed. 
Finally, for every irreducible module $L$, there is a unique choice of the grading shift so that we have $L^\circledast \iso L$ \cite[Section 3.2]{KL1}. When speaking of irreducible $R_\al$-modules we often assume by fiat that the shift has been chosen in this way. 

For $\bi\in \words_\al$ and $M\in\mod{R_\al}$, the {\em $\bi$-word space} of $M$ is
$
M_\bi:=1_\bi M.
$
We have 
$
M=\bigoplus_{\bi\in \words_\al}M_\bi.
$
We say that $\bi$ is a {\em word of $M$} if $M_\bi\neq 0$. A non-zero vector $v\in M_\bi$ is called a {\em vector of word $\bi$}. 
Note from the relations that 
$
\psi_r M_\bi\subset M_{s_r \bi}.
$

Let $M$ be a finite dimensional graded $R_\al$-module. Define the {\em $q$-character} of $M$ as follows: 
\begin{equation*}
\CH M:=\sum_{\bi\in \words_\al}(\DIM M_\bi) \bi \in \A\words_\al.
\end{equation*}
The $q$-character map $\CH: \mod{R_\al}\to \A\words_\al$ factors through to give an {\em injective} $ \A$-linear map from the Grothendieck group
$
\CH: [\mod{R_\al}]\to  \A\words_\al, 
$ see \cite[Theorem 3.17]{KL1}.

\section{Induction, coinduction, and duality for KLR algebras}

Given $\alpha, \beta \in Q_+$, we set $
R_{\alpha,\beta} := R_\alpha \otimes 
R_\beta$.  
Let $M \boxtimes N$ be 
the outer tensor product of the $R_\alpha$-module $M$ and the $R_\beta$-module 
$N$.
There is an injective homogeneous (non-unital) algebra homomorphism 
$R_{\alpha,\beta}\,\into\, R_{\alpha+\beta}$ 
mapping $1_\bi \otimes 1_\bj$ to $1_{\bi\bj}$,
where $\bi\bj$ is the concatenation of the two sequences. The image of the identity
element of $R_{\alpha,\beta}$ under this map is
$$
1_{\alpha,\beta}:= \sum_{\bi \in \words_\alpha,\ \bj \in \words_\beta} 1_{\bi\bj}.
$$ 

Let $\Ind_{\alpha,\beta}^{\alpha+\beta}$ and $\Res_{\alpha,\beta}^{\alpha+\beta}$
be the corresponding induction and restriction functors: 
\begin{align*}
\Ind_{\alpha,\beta}^{\alpha+\beta} &:= R_{\alpha+\beta} 1_{\alpha,\beta}
\otimes_{R_{\alpha,\beta}} ?:\mod{R_{\alpha,\beta}} \rightarrow \mod{R_{\alpha+\beta}},\\
\Res_{\alpha,\beta}^{\alpha+\beta} &:= 1_{\alpha,\beta} R_{\alpha+\beta}
\otimes_{R_{\alpha+\beta}} ?:\mod{R_{\alpha+\beta}}\rightarrow \mod{R_{\alpha,\beta}}.
\end{align*}
We often omit upper indices and write simply $\Ind_{\alpha,\beta}$ and $\Res_{\alpha,\beta}$. 

Note that $\Res_{\alpha,\beta}$ is just left multiplication by
the idempotent $1_{\alpha,\beta}$, so it is exact and sends finite dimensional modules to
finite dimensional modules. 
By \cite[Proposition 2.16]{KL1},
$1_{\alpha,\beta} R_{\alpha+\beta}$ is a free left $R_{\alpha,\beta}$-module of finite rank,
so $\Res_{\alpha,\beta}$ also sends finitely generated projectives to finitely generated projectives.
Similarly, $R_{\alpha+\beta} 1_{\alpha,\beta}$ is a  
free right $R_{\alpha, \beta}$-module of finite rank, so
$\Ind_{\alpha,\beta}$ is exact and sends finite dimensional modules to finite dimensional modules.
The functor $\Ind_{\alpha,\beta}$ is left adjoint to $\Res_{\alpha,\beta}$, and it sends finitely generated projectives to finitely generated projectives.

These functors have obvious generalizations to $n\geq 2$ factors: 
\begin{align*}
\Ind_{\ga_1,\dots,\ga_n}
:\mod{R_{\ga_1,\dots,\ga_n}} \rightarrow \mod{R_{\ga_1+\dots+\ga_n}},\\
\Res_{\ga_1,\dots,\ga_n}
:\mod{R_{\ga_1+\dots+\ga_n}}\rightarrow \mod{R_{\ga_1,\dots,\ga_n}}.
\end{align*}


If $M_a\in\Mod{R_{\ga_a}}$, for $a=1,\dots,n$, we define 
\begin{equation}\label{ECircProd}
M_1\circ\dots\circ M_n:=\Ind_{\ga_1,\dots,\ga_n}
M_1\boxtimes\dots\boxtimes M_n. 
\end{equation}
In view of \cite[Lemma 2.20]{KL1}, we have
\begin{equation}\label{ECharShuffle}
\CH(M_1\circ\dots\circ M_n)=\CH(M_1)\circ\dots\circ \CH(M_n).
\end{equation}

Finally, the functors of induction and restriction have parabolic analogues. For example, given a family $(\al^a_b)_{1\leq a\leq n,\ 1\leq b\leq m}$ of elements of $Q_+$, set 
$\sum_{a=1}^n\al^{a}_b=:\be_b$ for all $1\leq b\leq m$. Then we have obvious functors
$$
\Ind_{(\al^{1}_1,\dots,\al^{n}_{1});\dots;(\al^{1}_{m},\dots,\al^{n}_{m})}^{\be_1;\dots;\be_m}\qquad \text{and}\qquad \Res_{(\al^{1}_1,\dots,\al^{n}_{1});\dots;(\al^{1}_{m},\dots,\al^{n}_{m})}^{\be_1;\dots;\be_m}.
$$

While the induction functor 
$
\Ind_{\ga_1,\dots,\ga_n}
$ is left adjoint to the functor $\Res_{\ga_1,\dots,\ga_n}
$,  the {\em right}\, adjoint is given by the coinduction:
$$
\Coind_{\ga_1,\dots,\ga_n}=\Coind_{\ga_1,\dots,\ga_n}^{\ga_1+\dots+\ga_n}:=\Hom_{R_{\ga_1,\dots,\ga_n}}(1_{\ga_1,\dots,\ga_n}R_{\ga_1+\dots+\ga_n},\,?) 
$$
Induction and coinduction are related as follows:

For $\underline{\ga}:=(\ga_1,\dots,\ga_n)\in Q_+^n$, we denote
$$
d(\underline{\ga}):=\sum_{1\leq m<k\leq n}(\ga_m,\ga_k).
$$

\begin{Lemma} \label{LLV} {\rm \cite[Theorem 2.2]{LV}} 
Let $\underline{\ga}:=(\ga_1,\dots,\ga_n)\in Q_+^n$, and $V_m$ be a finite-dimensional $R_{\ga_m}$-module for $m=1,\dots,n$. 
Then 
$$
\Ind_{\ga_1,\dots,\ga_n}
V_1\boxtimes\dots\boxtimes V_n\cong 
q^{d(\underline{\ga})}\,
\Coind_{\ga_n,\dots,\ga_1}
V_n\boxtimes\dots\boxtimes V_1.
$$
\end{Lemma}

\begin{Lemma} \label{LDualInd}
Let $\underline{\ga}:=(\ga_1,\dots,\ga_n)\in Q_+^n$, and $V_m$ be a finite dimensional $R_{\ga_m}$-module for $m=1,\dots,n$. Then 
$$
(V_1\circ\dots\circ V_n)^\circledast\cong 
q^{ d(\underline{\ga})}
(V_n^\circledast\circ\dots\circ V_1^\circledast).
$$
\end{Lemma}
\begin{proof}
Follows from Lemma~\ref{LLV} by uniqueness of adjoint functors as in \cite[Theorem 3.7.5]{Kbook}
\end{proof}

\section{Crystal operators and extremal words}\label{SSCryst}
The theory of crystal operators has been developed in \cite{KL1}, \cite{LV} and \cite{KK} following ideas of Grojnowski \cite{Gr}, see also \cite{Kbook}. We review necessary facts for the reader's convenience. 

Let $\al\in Q_+$ and $i\in I$. It is known that $R_{n\al_i}$ is a nil-Hecke algebra with unique (up to a degree shift) irreducible module
$$
L(i^n)=q_i^{n(n-1)/2}L(i)^{\circ n}.
$$ 
Moreover, $\DIM L(i^n)=[n]^!_i$ 
We have functors 
\begin{align*}
&e_i: \mod{R_\al}\to\mod{R_{\al-\al_i}},\ M\mapsto \Res^{R_{\al-\al_i,\al_i}}_{R_{\al-\al_i}}\circ \Res_{\al-\al_i,\al_i}M,
\\
&f_i: \mod{R_\al}\to\mod{R_{\al+\al_i}},\ M\mapsto \Ind_{\al,\al_i}M\boxtimes L(i).
\end{align*}
If $L\in\mod{R_\al}$ is irreducible, we define
$$
\tilde f_i L:=\head (f_i L),\quad \tilde e_i L:=\soc (e_i L).
$$
A fundamental fact is that $\tilde f_i L$ is again irreducible and $\tilde e_i L$ is irreducible or zero. 
We refer to $\tilde e_i$ and $\tilde f_i$ as the crystal operators. These are operators on $B\cup\{0\}$, where $B$ is the set of isomorphism classes of the irreducible $R_\al$-modules for all $\al\in Q_+$.  
Define
$
\wt:B\to P,\ [L]\mapsto -\al
%
$
if $L\in\mod{R_\al}$.

\begin{Theorem} \label{TLV} {\rm \cite{LV}} 
$B$ with the operators $\tilde e_i,\tilde f_i$ and the function $\wt$ is the crystal graph of the negative part $U_q(\n_-)$ of the quantized enveloping algebra of $\g$.  
\end{Theorem}

For $M\in\mod{R_\al}$, define
$$
\eps_i(M)=\max\{k\geq 0\mid e_i^k(M)\neq 0\}.
$$
Then $\eps_i(M)=\max\{\eps_i(\bj)\mid\text{$\bj$ is a word of $M$}\}$, where for $\bj=(j_1,\dots,j_d)\in\words$, 
\begin{equation}\label{ETail}
\eps_i(\bj):=\max\{k\geq 0\mid j_{d-k+1}=\dots=j_d=i\}
\end{equation}
is the length of the longest $i$-tail of $\bj$. Define also
$$
\eps_i^*(M):=\max\{k\geq 0\mid j_1=\dots=j_k=i\ \text{for a word $\bj=(j_1,\dots,j_d)$ of $M$}\}
$$
to be the length of the longest $i$-head of the words of $M$. 

\begin{Proposition} \label{PCryst1} {\rm \cite{LV,KL1}} 
Let $L$ be an irreducible $R_\al$-module, $i\in I$, and $\eps=\eps_i(L)$. 
\begin{enumerate}
\item[{\rm (i)}] $e_iL$ is either zero or it has a simple socle; denote this socle $\tilde e_i L$ interpreted as $0$ if $e_i L=0$;
\item[{\rm (ii)}] $f_iL$ has simple head denoted $\tilde f_i L$;
\item[{\rm (iii)}] $\tilde e_i\tilde f_iL\simeq L$ and if $\tilde e_i L\neq 0$ then $\tilde f_i\tilde e_iL\simeq L$; 
\item[{\rm (iv)}] $\eps=\max\{k\geq 0\mid \tilde e_i^k(L)\neq 0\}$;
\item[{\rm (v)}] $\Res_{\al-\eps\al_i,\eps\al_i}L\simeq \tilde e_i^\eps L\boxtimes L(i^n)$.
\end{enumerate}
\end{Proposition}

Recall from (\ref{ECyclot}) the cyclotomic ideal $J_\al^\La$. We have an obvious  functor of inflation $\infl^\La:\mod{R_\al^\La}\to \mod{R_\al}$ and its left adjoint 
$$
\pr^\La:\mod{R_\al}\to \mod{R_\al^\La},\ M\mapsto M/J_\al^\La M.
$$

\begin{Lemma} \label{LPr} {\rm \cite[Proposition 2.4]{LV}} 
Let $L$ be an irreducible $R_\al$-module. Then $\pr^\La L\neq 0$ if and only if $\eps_i^*(L)\leq \lan\La,\al_i^\vee\ran$ for all $i\in I$. 
\end{Lemma}

Let $M\in\mod{R_\al}$ and $\bi= i_1^{a_1}\dots i_b^{a_b}$, with  $a_1,\dots,a_b\in\Z_{> 0}$, be a word of $M$. Then $\bi$ is {\em extremal} for $M$ if 
$$a_b=\eps_{i_b}(M),\ a_{b-1}=\eps_{i_{b-1}}(\tilde e_{i_b}^{a_b}M)\ ,\ \dots\ ,\  a_1=\eps_{i_1}(\tilde e_{i_2}^{a_2}\dots\tilde e_{i_b}^{a_b}M).
$$   
It follows that $i_k\neq i_{k+1}$ for all $k=1,\dots,b-1$.

\begin{Lemma} \label{LMultOneWeight} {\rm \cite[Lemma 2.10]{Kcusp}} 
Let $L$ be an irreducible $R_\al$-module, and $\bi=i_1^{a_1}\dots i_b^{a_b}\in\words_\al$ be an extremal word for $L$ with $i_k\neq i_{k+1}$. Set 
$N:=\sum_{m=1}^b a_m(a_m-1)(\al_{i_m},\al_{i_m})/4.$ 
Then 
$$\DIM L_\bi=\prod_{k=1}^b[a_k]^!_{i_k}\quad 
\text{and}\quad 
\dim 1_\bi L_N=\dim 1_\bi L_{-N}=1.$$ 
\end{Lemma}

\section{Mackey Theorem}\label{SSMackey}
We state a slight generalization of Mackey Theorem of Khovanov and Lauda \cite[Proposition~2.18]{KL1}. First some notation. Given $\underline{\kappa}=(\kappa_1,\dots,\kappa_N)\in Q_+^N$, and a permutation $x\in \Si_N$, we denote 
$$
x\underline{\kappa}:=(\kappa_{x^{-1}(1)},\dots,\kappa_{x^{-1}(N)}). 
$$
Correspondingly, define the integer
$$
s(x,\underline{\kappa}):=-\sum_{1\leq m<k\leq N,\ x(m)>x(k)}(\kappa_m,\kappa_k).
$$

Writing $R_{\underline{\kappa}}$ for $R_{\kappa_1,\dots,\kappa_N}$, there is an obvious natural algebra isomorphism 
$$
\phi^x:R_{x\underline{\kappa}}\to R_{\underline{\kappa}}
$$
permuting the components. Composing with this isomorphism, we get a functor
$$
\mod{R_{\underline{\kappa}}}\to \mod{R_{x\underline{\kappa}}},\  M\mapsto {}^{\phi^x}M.
$$
Making an additional shift, we get a functor 
$$
\mod{R_{\underline{\kappa}}}\to \mod{R_{x\underline{\kappa}}},\  M\mapsto {}^xM:=
q^{ s(x,\underline{\kappa})}
({}^{\phi^x}M).
$$

For the purposes of the following theorem, let us fix 
$$\underline{\ga}=(\ga_1,\dots,\ga_n)\in Q_+^n\quad \text{and}\quad \underline{\be}=(\be_1,\dots,\be_m)\in Q_+^m$$ 
with 
$$
\ga_1+\dots+\ga_n=\be_1+\dots+\be_m=:\al.
$$ 
Denote $d_b^a:=\height(\al_b^a)$ and $d:=\height(\al)$. 

Let $\D(\underline{\be},\underline{\ga})$ be the set of all tuples $\underline{\al}=(\al^a_b)_{1\leq a\leq n,\ 1\leq b\leq m}$ of elements of $Q_+$ such that  $\sum_{b=1}^m\al^{a}_b=\ga^a$ for all $1\leq a\leq n$ and $\sum_{a=1}^n\al^{a}_b=\be_b$ for all $1\leq b\leq m$. 

For each $\underline{\al}\in \D(\underline{\be},\underline{\ga})$, we define permutations $x(\underline{\al})\in\Si_{mn}$ and $x(\underline{\al})\in\Si_{d}$. The permutation $x(\underline{\al})$ maps  
$$
(\al^{1}_1,\dots,\al^{1}_{m},\al^{2}_1,\dots,\al^{2}_{m}\dots,\al^{n}_{1},\dots,\al^{n}_{m})
$$
to
$$
(\al^{1}_1,\dots,\al^{n}_{1},\al^{1}_2,\dots,\al^{n}_{2},\dots,\al^{1}_{m},\dots,\al^{n}_{m}).
$$
On the other hand, $w(\underline{\al})$ is the corresponding permutation of  the blocks of sizes $d_b^a$.

\begin{Example} 
{\rm 
Assume that $n=2$, $m=3$, and all $d_b^a=2$. Then $x(\underline{\al})\in \Si_6$ is the permutation which maps $1\mapsto 1, 2\mapsto 3,3\mapsto 5,4\mapsto 2, 5\mapsto 4, 6\mapsto 6$, which can be illustrated by the following picture:
$$
x(\underline{\al})
=\begin{braid}\tikzset{baseline=7mm}
  \draw (0,3)node[above]{$\al^1_1$}--(0,-3);
  \draw (1,3)node[above]{$\al^1_2$}--(2,-3);
  \draw (2,3)node[above]{$\al^1_3$}--(4,-3);
 \draw (3,3)node[above]{$\al^2_1$}--(1,-3);
  \draw (4,3)node[above]{$\al^2_2$}--(3,-3);
  \draw (5,3)node[above]{$\al^2_3$}--(5,-3);
\end{braid}.
$$
On the other hand, $w(\underline{\al})\in\Si_{12}$ is the corresponding block permutation:
$$
w(\underline{\al})
=\begin{braid}\tikzset{baseline=7mm}
  \draw (0,3)node[above]{$1$}--(0,-3);
  \draw (1,3)node[above]{$2$}--(1,-3);
  \draw (2,3)node[above]{$3$}--(4,-3);
 \draw (3,3)node[above]{$4$}--(5,-3);
  \draw (4,3)node[above]{$5$}--(8,-3);
  \draw (5,3)node[above]{$6$}--(9,-3);
\draw (6,3)node[above]{$7$}--(2,-3);
  \draw (7,3)node[above]{$8$}--(3,-3);
  \draw (8,3)node[above]{$9$}--(6,-3);
  \draw (9,3)node[above]{$10$}--(7,-3);
  \draw (10,3)node[above]{$11$}--(10,-3);
  \draw (11,3)node[above]{$12$}--(11,-3);
\end{braid}.
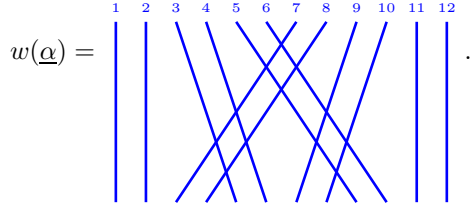
$$
}
\end{Example}


Let $M\in\mod{R_{\underline{\ga}}}$. We can now consider the $R_{\al^{1}_1,\dots,\al^{n}_{1};\dots;\al^{1}_{m},\dots,\al^{n}_{m}}$-module 
$${}^{x(\underline{\al})}\big(\Res_{\al^{1}_1,\dots,\al^{1}_{m};\dots;\al^{n}_{1},\dots,\al^{n}_{m}}^{\ga_1;\dots;\ga_n}
\,M \big).$$
Finally, let $\leq$ be a total order refining the Bruhat order on $\Si_d$, and 
for $\underline{\al}\in\D(\underline{\be},\underline{\ga})$, consider the submodules 
\begin{align*}
F_{\leq \underline{\al}}(M)&:=\sum_{w\in \D(\underline{\be},\underline{\ga}),\ w\leq w(\underline{\al})}R_{\underline{\be}}\psi_w\otimes 1_{\underline{\al}}M
\subseteq \Res^\al_{\underline{\be}}\Ind^\al_{\underline{\ga}} M,
\\
F_{< \underline{\al}}(M)&:=\sum_{w\in \D(\underline{\be},\underline{\ga}),\ w< w(\underline{\al})}R_{\underline{\be}}\psi_w\otimes 1_{\underline{\al}}M
\subseteq \Res^\al_{\underline{\be}}\Ind^\al_{\underline{\ga}} M.
\end{align*}

\begin{Theorem} \label{TMackeyKL}
Let $$\underline{\ga}=(\ga_1,\dots,\ga_n)\in Q_+^n\quad \text{and}\quad \underline{\be}=(\be_1,\dots,\be_m)\in Q_+^m$$ 
with 
$$
\ga_1+\dots+\ga_n=\be_1+\dots+\be_m=:\al,
$$ 
and $M\in\mod{R_{\underline{\ga}}}$. With the notation as above, the filtration $(F_{\leq \underline{\al}}(M))_{\underline{\al}\in \D(\underline{\be},\underline{\ga})}$ is a filtration of  $\Res_{\underline{\be}}\,\Ind_{\underline{\ga}} M$ as an $R_{\underline{\be}}$-module. Moreover, the subquotients of the filtration are:  
\begin{align*}
F_{\leq \underline{\al}}(M)/F_{<\underline{\al}}(M)
&\cong
\Ind_{x(\underline{\al})\cdot\underline{\al}}^{\underline{\be}}
\big({}^{x(\underline{\al})}\big(\Res_{\underline{\al}}^{\underline{\ga}}
\,M \big)\big).
\\
&= \Ind_{\al^{1}_1,\dots,\al^{n}_{1};\dots;\al^{1}_{m},\dots,\al^{n}_{m}}^{\be_1;\dots;\be_m}
{}^{x(\underline{\al})}\big(\Res_{\al^{1}_1,\dots,\al^{1}_{m};\dots;\al^{n}_{1},\dots,\al^{n}_{m}}^{\ga_1;\dots;\ga_n}
\,M \big).
\end{align*}
\end{Theorem}
\begin{proof}
For $m=n=2$ this follows from \cite[Proposition~2.18]{KL1}. The general case can be proved by the same argument or deduced from the case $m=n=2$ by induction. 
\end{proof}

\section{Convex preorders and root partitions}\label{SSConPreRP} 
We now describe the theory of cuspidal systems from \cite{Kcusp}.  
Recall the notion of a convex preorder on $\Phi_+$ from (\ref{EPO1})--(\ref{EPO3}). 
General theory of cuspidal systems is valid for an arbitrary convex preorder, but for the theory of imaginary representations we will need an additional assumption that the preorder is balanced, see (\ref{EBalanced}), (\ref{EBalancedEquiv}).

Recall that $I'=\{1,\dots,l\}$. We will consider the set $\Par$ of $l$-multipartitions 
$$\ula=(\la^{(1)},\dots,\la^{(l)}),$$ 
where each $\la^{(i)}=(\la^{(i)}_1,\la^{(i)}_2,\dots)$ is a usual partition. 
We denote 
$$|\ula|:=\sum_{i\in I'}|\la^{(i)}|.$$ 
For $n\in \Z_{\geq 0}$, the set of all $\ula\in\Par$ such that $|\ula|=n$ is denoted $\Par_n$. 

Recall the totally ordered set $\Psi$ defined in  (\ref{EPsi}). 
Denote by $\Seq$ the set of all finitary  tuples  $M=(m_\rho)_{\rho\in\Psi}\in \Z_{\geq 0}^\Psi$ of non-negative integers. The left lexicographic order on $\Seq$ is denoted $\leq_l$ and the right lexicographic order on $\Seq$ is denoted $\leq_r$. We will use the following {\em bilexicographic} partial order on $\Seq$: 
$$
M\leq N\qquad\text{if and only if}\qquad M\leq_l N \ \text{and}\ M\geq_r N.
$$

Let  
$$\pi=(M,\umu)=(\rho_1^{m_1},\dots,\rho_s^{m_s},\umu,\rho_{-t}^{m_{-t}},\dots,\rho_{-1}^{m_{-1}})
$$ 
be a root partition as in (\ref{ERP}), so that $M\in \Seq$ and $\umu\in\Par_{m_\de}$. For  $\rho\in\Psi$, we define
$
M_\rho:=m_\rho\rho$, and consider a tuple   
$
|M|=(M_\rho)_{\rho\in\Psi}\in Q_+^\Psi$. Ignoring trivial terms, we can also write   
$$|M|=  (m_1\rho_1,\dots,m_s\rho_s,m_\de\de,m_{-t}\rho_{-t},\dots,m_{-1}\rho_{-1}).
$$ 
Then we have a parabolic subalgebra 
$$
R_{|M|}=R_{m_1\rho_1,\dots,m_s\rho_s,m_\de\de,m_{-t}\rho_{-t},\dots,m_{-1}\rho_{-1}}\subseteq R_\al.
$$
We will use the following partial order on the set $\Pi(\al)$ of root partitions of $\al$:
\begin{equation}\label{EBilex}
(M,\umu)\leq (N,\unu)\ \text{if and only if}\  M\leq N\ \text{and if $M=N$ then}\ \umu=\unu.
\end{equation}

\section{Cuspidal systems and standard modules}
Let $\preceq$ be an arbitrary convex preorder on $\Phi_+$. 
Recall the definition of a cuspidal system 
$$
\{L_\rho\mid \rho\in \Phi_+^\re\}\cup\{L(\umu)\mid \umu\in\Par\}
$$
from $\S\ref{SConCus}$.

For every $\al\in Q_+$ and $\pi=(M,\umu)\in\Pi(\al)$ as in (\ref{ERP}), we define an integer
\begin{equation}\label{EShift}
\shift(\pi)=\shift(M,\umu):=\sum_{\rho\in \Phi_+^\re} (\rho,\rho)m_\rho(m_\rho-1)/4,
\end{equation}
the irreducible $R_{|M|}$-module
\begin{equation}
L_{\pi}=L_{M,\umu}:=
q^{\shift(\pi)}\,
L_{\rho_1}^{\circ m_1} \boxtimes \dots\boxtimes L_{\rho_s}^{\circ m_s}\boxtimes  L(\umu)\boxtimes L_{\rho_{-t}}^{\circ m_{-t}}\boxtimes L_{\rho_{-1}}^{m_{-1}} , 
\end{equation}
and the {\em standard module}
\begin{equation}\label{EStand}
\Stand(\pi)=\Stand(M,\umu):=
q^{\shift(\pi)}\,
L_{\rho_1}^{\circ m_1} \circ \dots\circ L_{\rho_s}^{\circ m_s}\circ L(\umu)\circ L_{\rho_{-t}}^{\circ m_{-t}}\circ L_{\rho_{-1}}^{m_{-1}}.
\end{equation}
Note that $\Stand(M,\umu)=\Ind_{|M|} L_{M,\umu}$.

\begin{Theorem} \label{THeadIrr} {\rm \cite{Kcusp}} 
Given a convex preorder there exists a corresponding cuspidal system $\{L_\rho\mid \rho\in \Phi_+^\re\}\cup\{L(\ula)\mid \ula\in\Par\}$. Moreover: 
\begin{enumerate}
\item[{\rm (i)}] For every root partition $\pi$, the standard module  
$
\Stand(\pi)
$ has irreducible head; denote this irreducible module $L(\pi)$. 

\item[{\rm (ii)}] $\{L(\pi)\mid \pi\in \Pi(\al)\}$ is a complete and irredundant system of irreducible $R_\al$-modules up to isomorphism and degree shift.

\item[{\rm (iii)}] For every root partition  $\pi$, we have $L(\pi)^\circledast\cong L(\pi)$.  

\item[{\rm (iv)}] For all $\pi,\si\in\Pi(\al)$, we have that $[\Stand(\pi):L(\pi)]_q=1$, and $[\Stand(\pi):L(\si)]_q\neq 0$ implies $\pi\leq \si$. 

\item[{\rm (v)}] For all $(M,\umu),(N,\unu)\in\Pi(\al)$, we have that $\Res_{|M|}L(M,\umu)\cong L_{M,\umu}$ and $\Res_{|N|}L(M,\umu)\neq 0$ implies $N\leq M$.  

\item[{\rm (vi)}] The induced module $L_\rho^{\circ n}$ is irreducible for all $\rho\in\Phi^\re_+$ and $n\in\Z_{>0}$. 
\end{enumerate}
\end{Theorem}

\section{Colored imaginary tensor spaces}\label{SMinusc}
{\em From now on we assume that the fixed convex preorder we are working with is balanced}, so that $\al_i\succ n\de\succ\al_0$ for all $i\in I'$ and $n\in\Z_{>0}$. It turns out that the theory of imaginary representations is independent of the choice of a {\em balanced} convex preorder.  
Denote 
$$
e:=\height(\de).
$$
Recall the irreducible imaginary representations of $R_{n\de}$ defined by the property (Cus2) in \S\ref{SConCus}. 
The irreducible imaginary representations of $R_\de$ are called {\em minuscule imaginary representations}.
The minuscule imaginary representations can be {\em canonically} labeled by the elements of $\Par_1$ as explained below. 

\begin{Lemma} \label{L150813}%
{\rm \cite[Lemma 5.1]{Kcusp}} 
Let $L$ be an irreducible $R_\de$-module. The following are equivalent: 
\begin{enumerate}
\item[{\rm (i)}] $L$ is minuscule imaginary; 
\item[{\rm (ii)}] $L$ factors through to the cyclotomic quotient $R_\de^{\La_0}$;
\item[{\rm (iii)}] we have $i_1=0$ for any word $\bi=(i_1,\dots,i_e)$ of $L$.
\end{enumerate}
\end{Lemma}

We always consider $R_\al^{\La_0}$-modules as $R_\al$-modules via $\infl^{\La_0}$. 

\begin{Proposition} \label{L3912} {\rm \cite[Lemma 5.2, Corollary 5.3]{Kcusp}} 
Let $i\in I'$. 
\begin{enumerate}
\item[{\rm (i)}] The cuspidal module $L_{\de-\al_i}$ factors through $R_{\de-\al_i}^{\La_0}$ and it is the only irreducible $R_{\de-\al_i}^{\La_0}$-module. 
\item[{\rm (ii)}] The minuscule imaginary modules are exactly 
$$\{L_{\de,i}:=\tilde f_i L_{\de-\al_i}\mid i\in I'\}.$$ 
\item[{\rm (iii)}] $e_j L_{\de,i}=0$ for all $j\in I\setminus\{ i\}$. Thus, for each $i\in I'$, the minuscule imaginary module $L_{\de,i}$ can be characterized uniquely up to isomorphism as the irreducible $R_\de^{\La_0}$-module such that $i_e=i$ for all words $\bi=(i_1,\dots,i_e)$ of $L_{\de,i}$. 
\end{enumerate}
\end{Proposition}

For each $i\in I'$, we refer to the minuscule module $L_{\de,i}$ described in Proposition~\ref{L3912} as the minuscule module of {\em color~$i$}. Let 
\begin{equation}\label{EMuJ}
\umu(i):=(\emptyset,\dots,\emptyset,(1),\emptyset,\dots,\emptyset)\in\Par_1\qquad(i\in I')
\end{equation} 
be the $l$-multipartition of $1$ with $(1)$ in the $i$th component. 
We associate to it the minuscule module $L_{\de,i}$:
\begin{equation}\label{EMinLabel}
L(\umu(i)):=L_{\de,i}\qquad(i\in I'). 
\end{equation}

\begin{Lemma} \label{LEpsLDe} {\rm \cite[Lemma 5.4]{Kcusp}} 
Let $i\in I'$. Then $\eps_i(L_{\de,i})=1$.
\end{Lemma}

The minuscule modules are defined over $\Z$, see \cite[Remark 5.5]{Kcusp}. To be more precise, for each $i\in I'$, there exists an $R_\de(\Z)$-module $L_{\de,i,\Z}$ which is free finite rank over $\Z$ and such that $L_{\de,i,\Z}\otimes F$ is the minuscule imaginary module $L_{\de,i,F}$ over $R_\de(F)$ for any ground field $F$. In particular,
\begin{equation}\label{EEndL}
\End_{R_\de}(L_{\de,i,\O})\cong \O.
\end{equation}


The {\em imaginary tensor space of color $i$} is the $R_{n\de}$-module 
\begin{equation}\label{EMNI}
M_{n,i}:=L_{\de,i}^{\circ n}.
\end{equation}
In this definition we allow $n$ to be zero, in which case $M_{0,i}$ is the trivial module over the trivial algebra $R_0$. A composition factor of $M_{n,i}$ is called an {\em irreducible imaginary module of \em color $i$}. Color is well-defined in the following sense: if $n>0$ and $L$ is an  irreducible imaginary  $R_{n\de}$-module of color $i$, then $L$ cannot be irreducible imaginary of color $j\in I'$. Indeed, every word appearing in the character of $M_{n,i}$, and hence in the character of $L$, ends on $i$.

\begin{Lemma} \label{LD0} {\rm \cite[Lemma 5.7]{Kcusp}} 
Any composition factor of $M_{n_1,1}\circ\dots\circ M_{n_l,l}$ is imaginary.
\end{Lemma}

The following theorem provides a `reduction to one color':

\begin{Theorem} \label{TD1} {\rm \cite[Theorem 5.10]{Kcusp}} 
Suppose that for each $n\in\Z_{\geq 0}$ and $i\in I'$, we have an irredundant family $\{L_i(\la)\mid\la\vdash n\}$ of irreducible imaginary $R_{n\de}$-modules of color $i$. For a multipartition $\ula=(\la^{(1)},\dots,\la^{(l)})\in\Par_n$, define $$L(\ula):=L_1(\la^{(1)})\circ\dots\circ L_l(\la^{(l)}).$$ 
Then $\{L(\ula)\mid\ula\in\Par_n\}$ is a complete and irredundant system of irreducible imaginary $R_{n\de}$-modules. In particular, the given modules $\{L_i(\la)\mid\la\vdash n\}$ give all the irreducible imaginary modules of color $i$ up to isomorphism. 
\end{Theorem}

\begin{Corollary} \label{CColorUB} 
Suppose that for each $n\in\Z_{\geq 0}$ and for each $i=1,\dots,l$, we have an irredundant family $\{L_i(\la)\mid\la\vdash n\}$ of  irreducible imaginary $R_{n\de}$-modules of color $i$. Then each irreducible imaginary $R_{n\de}$-module of color $i$ is isomorphic to one of the modules $L_i(\la)$ for some $\la\vdash n$.  
\end{Corollary}
\begin{proof}
Let $L$ be an irreducible imaginary $R_{n\de}$-module of color $i$. By Theorem~\ref{TD1}, 
we must have $L\simeq  L_1(\mu^{(1)})\circ\dots\circ L_l(\mu^{(l)})$ for some multipartition $(\mu^{(1)},\dots,\mu^{(l)})\in\Par_n$. It remains to note that $\mu^{(j)}=\emptyset$ for all $j\neq i$, for otherwise $j$ would arise as a last letter of some word arising in the character of $L$, giving a contradiction.  
\end{proof}

If the Cartan matrix $\Car$ is symmetric, then the minuscule representations can be described very explicitly as certain special homogeneous representations, see \cite[Sections 5.4,5.5]{Kcusp}.

\begin{Lemma} \label{Lw(i)} {\rm \cite[Lemma 5.16]{Kcusp}} 
Let $i\in I'$. Then we can write $\La_0-\de+\al_i=w(i)\La_0$ for a unique $w(i)\in W$ which is $\La_0$-minuscule.
\end{Lemma}

By the theory of homogeneous representations \cite[Sections 5.4,5.5]{Kcusp}, the minuscule element $w(i)$ constructed in Lemma~\ref{Lw(i)} is of the form $w_{C(i)}$ for some strongly homogeneous component $C(i)$ of $G_{\de-\al_i}$. 

\begin{Lemma} \label{L3912_2} {\rm \cite[Lemma 5.17]{Kcusp}} 
Let $i\in I'$, $d:=e-1=\height(\de-\al_i)$ and $\bj=(j_1,\dots,j_{d})\in C(i)$. Then the cuspidal module $L_{\de-\al_i}$ is the homogeneous module $L(C(i))$, and we have:
\begin{enumerate}
\item[{\rm (i)}] $j_1=0$;
\item[{\rm (ii)}] $j_d$ is connected to $i$ in the Dynkin diagram, i.e. $a_{j_d,i}=-1$
\item[{\rm (iii)}] if $j_b=i$ for some $b$, then there are at least three indices $b_1,b_2,b_3$ such that $b<b_1<b_2<b_3\leq d$ such that $a_{i,b_1}=a_{i,b_2}=a_{i,b_3}=-1$. 
\end{enumerate}  
\end{Lemma}

Now we can describe the minuscule modules as homogeneous modules:

\begin{Proposition} \label{PMinSL} {\rm \cite[Proposition 5.19]{Kcusp}} 
Let $i\in I'$. The set of concatenations
$$
C_i:=\{\bj i\mid \bj\in C(i)\}
$$ 
is a homogeneous component of $G_\de$, and the corresponding homogeneous $R_{\de}$-module $L(C_i)$ is isomorphic to the minuscule imaginary module $L_{\de,i}$. 
\end{Proposition}

\begin{Example} \label{ETypeA}
{\rm 
Let  $\Car={\tt A}_l^{(1)}$ and $i\in I'$. Then $L_{\de,i}$ is the homogeneous irreducible $R_{\de}$-module 
with character 
$$
\CH L_{\de,i}=0\big((1,2,\dots, i-1)\circ (l,l-1,\dots,i+1)\big)i.
$$
For example, $L_{\de,1}$ and $L_{\de,l}$ are $1$-dimensional with characters
$$
\CH L_{\de,1}=(0,l,l-1,\dots,1),\quad
\CH L_{\de,l}=(01\dots l),
$$
while for $l\geq 3$, the module $L_{\de,l-1}$ is $(l-2)$-dimensional with character
$$
\CH L_{\de,l-1}=\sum_{r=0}^{l-3}(0,1,\dots,r, l,r+1,\dots,l-1).
$$
}
\end{Example} 

\begin{Example} \label{ExTypeD}
{\rm 
Let  $\Car={\tt D}_l^{(1)}$ and $i\in I'$. By Proposition~\ref{PMinSL}, we have that $L_{\de,i}$ is the homogeneous module $L(C_i)$, where $C_i$ is the connected component in $G_\de$ containing the following word:
\begin{equation}\label{EIBTypeD}
\left\{
\begin{array}{ll}
(0,2,3,\dots, l-2,l,l-1,l-2,\dots,i+1,1,2,\dots,i) &\hbox{if $i\leq l-1$,}\\
(0,2,3,\dots, l-2,l-1, 1,2,\dots,l-2,l) &\hbox{if $i=l$.}
\end{array}
\right.
\end{equation}

}
\end{Example}

If the Cartan matrix $\Car$ is non-symmetric, the explicit construction of the minuscule representations $L_{\de,i}$ is more technical. It is explained in Chapter~\ref{ChLast}.

\chapter{Imaginary Schur-Weyl duality}\label{ITTCA}
Fix $i\in I'$, and recall from (\ref{EMNI}) the  imaginary tensor space $M_{n,i}=L_{\de,i}^{\circ n}$ of color $i$. We are going to study irreducible imaginary $R_{n\de}$-modules of color $i$, i.e. composition factors of $M_{n,i}$. Since $i$ is going to be fixed throughout, we usually simplify our notation and write $M_n$ for $M_{n,i}$, $L_\de$ for $L_{\de,i}$, etc. Recall that we denote by $e$ the height of null-root~$\de$. 

\section{Imaginary tensor space and its parabolic analogue}\label{SSImTenSp}
{\em Throughout}\, we fix an extremal word 
\begin{equation}\label{EBWt}
\bi:=(i_1,\dots,i_e)
\end{equation}
of $L_\de$ 
so that the top degree component $(1_\bi L_\de)_N$ of the word space $1_\bi L_\de$ is $1$-dimensional, see Lemma~\ref{LMultOneWeight}. 
To be more precise, for a symmetric Cartan matrix $\Car$, the module $L_\de$ is homogeneous by Proposition~\ref{PMinSL}, i.e. all its word spaces are $1$-dimensional, and we can take $\bi$ to be an arbitrary word of $L_\de$. For non-symmetric $\Car$, we make a specific choice of $\bi$ as in (\ref{EITypeB}), (\ref{EITypeC}), (\ref{EITypeF}), (\ref{EITypeG}) in types ${\tt B_l^{(1)},C_l^{(1)},F_4^{(1)},G_2^{(1)}}$ respectively. 
 
Pick a non-zero vector $v\in (1_\bi L_\de)_N$.  
Recall that $L_\de$ is defined over $\Z$, so we may assume that $(1_\bi L_\de)_N=\O\cdot v$. Denote 
\begin{equation}\label{Ev_n}
v_n:=v\otimes\dots\otimes v\in L_\de^{\boxtimes n}.
\end{equation}
We identify $L_\de^{\boxtimes n}$ with the submodule 
$1\otimes L_\de^{\boxtimes n}\subseteq M_n=\Ind_{\de,\dots,\de}L_\de^{\boxtimes n}$, so $v_n$ can will be considered as an element of $M_n$. 
Note that $v_n$ generates $M_n$ as an $R_{n\de}$-module, and that $$v_n\in (1_{\bi^n}M_n)_{nN}.$$ 
By degrees,
$$
y_rv_n=0\qquad(1\leq r\leq n). 
$$

More generally, let $\nu=(n_1,\dots,n_a)\vDash n$. 
Consider the parabolic subalgebra 
\begin{equation}\label{ERNuDe}
R_{\nu,\de}:=R_{n_1\de,\dots,n_a\de}\subseteq R_{n\de},
\end{equation}
and consider the $R_{\nu,\de}$-module 
$$
\Mde_\nu:=\Mde_{n_1}\boxtimes\dots\boxtimes \Mde_{n_a},
$$
with generator 
$$v_\nu:=v_{n_1}\otimes \dots\otimes v_{n_a}.$$ 
By transitivity of induction this module embeds naturally into $M_n$ as an $R_{\nu,\de}$-submodule. 

\begin{Lemma} \label{LMSelfDual}
$
\Mde_\nu^\circledast\cong \Mde_\nu. 
$ In particular, every composition factor of the socle of $\Mde_\nu$ appears in its head. 
\end{Lemma}
\begin{proof}
This is \cite[Lemma 5.6]{Kcusp}. 
\end{proof}

We denote by $e\nu$ the composition 
$$e\nu:=(en_1,\dots,en_a)\vDash en.$$ 
The following lemma immediately follows from the Basis Theorem~\ref{TBasis}:

\begin{Lemma} \label{LMBasis}
Let $\nu\vDash n$. Then 
$$M_\nu=\bigoplus_{w\in \D_{e\nu}^{(e^n)}}\psi_{w}\otimes L_\de^{\boxtimes n}$$ 
as  $\O$-modules. In particular, 
$$M_n=\bigoplus_{w\in \D_{en}^{(e^n)}}\psi_{w}\otimes L_\de^{\boxtimes n}$$ 
as  $\O$-modules.  
\end{Lemma}

Define 
\begin{equation}\label{EV_n}
V_n:=\Res_{\de,\dots,\de}M_n.
\end{equation}
More generally, for a composition $\nu=(n_1,\dots,n_a)\vDash n$, set 
$$V_\nu:=\Res^{n_1\de;\dots;n_a\de}_{\de,\dots,\de}M_\nu\cong V_{n_1}\boxtimes \dots\boxtimes V_{n_a}.$$  
Clearly $v_\nu\in V_\nu$. 

To describe $V_n$ and $V_\nu$, we introduce 
the {\em block permutation group} $B_n$ as the subgroup of $\Si_{en}$ generated by the {\em block permutations} $w_1,\dots,w_{n-1}$, where $w_r$ is the product of transpositions 
\begin{equation}\label{Ew_r}
w_r:=\prod_{b=re-e+1}^{re}(b,b+e)\qquad\qquad
\qquad(1\leq r<n).
\end{equation}
The group $B_n$ is isomoprphic to the symmetric group $\Si_n$ via 
$$
\iota:\Si_n\iso \B_n,\ s_r\mapsto w_r \qquad(1\leq r<n).
$$
Note that each element $\iota(w)\in B_n$ belongs to $\D_{en}^{(e^n)}$. 
For example, if $n=2$ then, in terms of Khovanov-Lauda diagrams \cite{KL1} we have 
$$
\psi_{w_1}1_{\bi^{2}}=\begin{braid}\tikzset{baseline=7mm}
  \draw (0,4)node[above]{$i_1$}--(6,-4);
  \draw (1,4)node[above]{$i_2$}--(7,-4);
  \draw[dots] (1.3,4)--(4.9,4);
  \draw[dots] (7.4,4)--(10.7,4);
  \draw (5,4)node[above]{$i_e$}--(11,-4);
 \draw (6,4)node[above]{$i_1$}--(0,-4);
  \draw (7,4)node[above]{$i_2$}--(1,-4);
 \draw[dots] (1.3,-4)--(4.9,-4);
  \draw[dots] (7.4,-4)--(10.7,-4);
  \draw (11,4)node[above]{$i_e$}--(5,-4);
\end{braid}.
$$
Define 
$$
\si_r:=\psi_{w_{r}}\qquad(1\leq r<n),
$$
and
$$
\si_{w}:=\si_{r_1}\dots \si_{r_m}\qquad(w\in \Si_n),
$$
where we have picked a reduced decomposition $w=s_{r_1}\dots s_{r_m}$. 

Let us write $\de^n$ for $(\de,\dots,\de)$ with $n$ terms. By definition, $V_n=\Res_{\de^n}M_n$ is an $R_{\de^n}$-module. 

\begin{Proposition}\label{PEndNew}
We have: 
\begin{enumerate}
\item[{\rm (i)}] As an $R_{\de^n}$-module, $V_n$ has a filtration with $n!$ composition factors $\cong L_\de^{\boxtimes n}$. 
\item[{\rm (ii)}] As an $\O$-module, $V_n=\bigoplus_{w\in\Si_n} V(w)$, where $V(w):=\si_{w}\otimes L_\de^{\boxtimes n}$. 
\item[{\rm (iii)}] $1_{\bi^n}M_n=\oplus_{w\in\Si_n} (\si_w \otimes (1_\bi L_\de)^{\boxtimes n})$
\item[{\rm (iv)}] $(1_{\bi^n}M_n)_{nN}$ is the top degree component of the weight space $1_{\bi^n}M_n$, and $(1_{\bi^n}M_n)_{nN}=\oplus_{w\in\Si_n}\O\cdot (\si_w \otimes v_n)$. 
\end{enumerate}
\end{Proposition}
\begin{proof}
(i) follows by an application of the Mackey Theorem~\ref{TMackeyKL}, using the property (Cus2) of $L_\de$ and the fact that $(\de,\de)=0$ to deduce that all grading shifts are trivial. 

(ii) is proved by a word argument. Indeed, given words $\bi^{(1)},\dots,\bi^{(n)}$ of $L_\de$, we have $i^{(1)}_1=\dots=i^{(n)}_n=0$ by Lemma~\ref{L150813}(iii). So, the only shuffles of $\bi^{(1)},\dots,\bi^{(n)}$ which lie in $V_n$ are permutations of these words. So the result follows from Lemma~\ref{LMBasis}. 

(iii) follows from (ii), and (iv) follows from (iii).  
\end{proof}

\begin{Corollary} \label{CEndMDegZero} 
All $R_{n\de}$-endomorphisms of $M_n$ are of degree zero, and $\dim\operatorname{End}_{R_{n\de}}(M_n)\leq n!$.
\end{Corollary}
\begin{proof}
This follows from the adjointness of the functors $\Ind$ and $\Res$ and Proposition~\ref{PEndNew}(i). 
\end{proof}

\section{Action of $\Si_n$ on $M_n$}\label{SSMainConj}

In this section, we prove the key fact that $\End_{R_{n\de}}(M_n)$ is isomorphic to the group algebra of the symmetric group $\Si_n$.  We distinguish between the cases where the Cartan matrix $\Car$ is symmetric and non-symmetric. The symmetric case can be handled nicely using the work \cite{KKK}. For the non-symmetric case we have to appeal to the computations made in Chapter~\ref{ChLast}\footnote{Originally, we also handled the simply-laced cases via lengthy case-by-case computations. However, Peter  McNamara has explained to us how to avoid this using \cite{KKK}.}. 

Assume in this paragraph that $\Car$ is symmetric. We review the Kang-Kashiwara-Kim intertwiners \cite{KKK} adapted to our needs. 
Definition~1.4.5 of \cite{KKK} yields a {\em non-zero} $R_{2\de}$-homomorphism 
$$
\tau: M_2\to q^{(\de,\de)-2(\de,\de)_n+2s}M_2,
$$
where $(\cdot,\cdot)_n$ and $s$ are as in \cite[\S1.3,(1.4.8)]{KKK}. (This homomorphism would be denoted $r_{L_\de,L_\de}$ in \cite{KKK}.) Since $(\de,\de)=0$, and all endomorphisms of $M_2$ are of degree zero by Corollary~\ref{CEndMDegZero}, it follows that $s=(\de,\de)_n$ and we actually have $\tau:M_2\to M_2$. 
Now, it follows from \cite[Proposition~1.4.4(iii)]{KKK}, the adjointness of $\Ind$ and $\Res$, and Proposition~\ref{PEndNew}(ii) that 
\begin{equation}\label{ETauCv1v2}
\tau(v^1\otimes v^2)=\si_1\cdot (v^{2}\otimes v^{1})+c(v^1,v^2)\,v^1\otimes v^2\qquad(v^1,v^2\in L_\de)
\end{equation}
for some $c(v^1,v^2)\in\O$. In particular 
\begin{equation}\label{ETauC}
\tau(v_2)=(\si_{1}+c)v_2 
\end{equation}
for some constant $c\in\O$. 

Even if $\Car$ is not symmetric, there is an endomorphism $\tau$ of $M_2$ with the property (\ref{ETauC}), see Chapter~\ref{ChLast}. So from now on we use it in all cases. 
We note that the elements $\tau_r$ go back to \cite{KMR}, where  a special case of Theorem~\ref{TEndMn} below is checked, see \cite[Theorem 4.13]{KMR}. 


Inserting the endomorphism $\tau$ into the $r$th and $r+1$st positions in $M_n=L_\de^{\circ n}$, yields endomorphisms
\begin{equation}\label{ETauPsiC}
\tau_r:M_n\to M_n,\ v_n\mapsto (\si_{r}+c)v_n\qquad(1\leq r<n).
\end{equation}

We always consider the group algebra $\O\Si_n$ as a graded algebra concentrated in degree zero. 

\begin{Theorem} \label{TEndMn} 
The endomorphisms $\tau_r$ satisfy the usual Coxeter relations of the standard generators of the symmetric group $\Si_n$, i.e. $\tau_r^2=1$, $\tau_r\tau_s=\tau_s\tau_r$ for $|r-s|>1$, and $\tau_r\tau_{r+1}\tau_r=\tau_{r+1}\tau_r\tau_{r+1}$. This defines a (degree zero) homomorphism
$$
F\Si_n\to\End_{R_{n\de}}(M_n)^\op
$$
which is an isomorphism. 
\end{Theorem}
\begin{proof}
If $\Car$ is symmetric, we use the elements $\phi_{w}$ from \cite[Lemma~1.3.1(iii)]{KKK}. 
Then $\tau_r$'s satisfy braid relations, as noted in \cite[p.16]{KKK}. For the quadratic relations, by definition, $\tau_r^2$ maps $v_n$ to $((z'-z)^{-2s}\phi_{w_r}^2 v_n)|_{z=z'=0}$, where the action is taking place in $(L_\de)_z\circ (L_\de)_{z'}$ and we consider $v_n$ as a vector of $(L_\de)_z\circ (L_\de)_{z'}$ in the obvious way. 
Since $y_sv_n=0$ in $L_\de$  we have $y_sv_n=zv_n$ in $(L_\de)_z$ for all $s$. So the product in the right hand side of \cite[Lemma~1.3.1(iv)]{KKK} is easily seen to act with the scalar $(z'-z)^{2(\de,\de)_n}$ on the vector $v_n\in (L_\de)_z\circ (L_\de)_{z'}$. Since we already know that $s=(\de,\de)_n$, it follows that $\tau_r^2v_n=v_n$. Since $v_n$ generates $M_n$ as an $R_{n\de}$-module, we deduce that $\tau_r^2=1$. 

If $\Car$ is not symmetric, then we check in Proposition~\ref{NSLbraidrels} that the $\tau_r$ still satisfy the quadratic and braid relations. 

For an arbitrary $\Car$ let $w\in\Si_n$ with reduced decomposition $w=s_{r_1}\dots s_{r_m}$. Then in view of (\ref{ETauPsiC}), for $\tau_w:=\tau_{r_1}\dots\tau_{r_m}$ (the product in $\End_{R_{n\de}}(M_n)^\op$), we have
\begin{equation}\label{ETauAction}
\tau_w(v_n)=(\si_{r_1}+c)\dots(\si_{r_m}+c)v_n.
\end{equation}
It follows that 
\begin{equation}\label{ETauSiTriangular}
\tau_w(v_n)=\si_w(v_n)+\sum_{u<w}c_u\si_uv_n\qquad(c_u\in\O),
\end{equation}
where $<$ is the Bruhat order. In view of Proposition~\ref{PEndNew}(ii), the elements $\{\tau_w\mid w\in\Si_n\}$ are linearly independent, and the result follows from Corollary~\ref{CEndMDegZero}.  
\end{proof}





In view of the theorem, we can now consider $M_n$ as an  $(R_{n\de},\O\Si_n)$-bimodule, with the right action $mw=\tau_w(m)$ for $m\in M_n$ and $w\in\Si_n$, where the linear transform action $\tau_w$ is defined by (\ref{ETauAction}).

\begin{Corollary} \label{CVnDec}
We have:
\begin{enumerate}
\item[{\rm (i)}] As $R_{\de^n}$-modules, 
$V_n \cong  \bigoplus_{w\in\Si_n}L_\de^{\boxtimes n}w\cong (L_\de^{\boxtimes n})^{\oplus n!}.$
\item[{\rm (ii)}] As $\O$-modules, 
$$1_{\bi^n}M_n=\oplus_{w\in\Si_n}(1_\bi L_\de)^{\boxtimes n}w\quad \text{and}\quad  
(1_{\bi^n}M_n)_{nN}=\oplus_{w\in\Si_n}\O\cdot  v_nw.
$$
\end{enumerate}
\end{Corollary}
\begin{proof}
Since $L_\de^{\boxtimes n}$ is irreducible as an $R_{\de^n}$-module, the result now follows from Theorem~\ref{TEndMn} and  Proposition~\ref{PEndNew}.
\end{proof}

Let $u_0\in\Si_{2e}$ be the minimal length such that
$$
u_0\bi^2=\bi^{\{2\}}:=(i_1,i_1,i_2,i_2,\dots,i_e,i_e).
$$

\begin{Example} 
{\rm 
If $\Car$ is symmetric, we know that $L_\de$ is homogeneous, and so $i_m\neq i_{m+1}$ for all $1\leq m<e$. So in that case, we have that 
\begin{equation}\label{Eu_0}
u_0: n\mapsto
\left\{
\begin{array}{ll}
2n-1 &\hbox{if $1\leq n\leq e$,}\\
2(n-e) &\hbox{if $e< n\leq 2e$.}
\end{array}
\right.
\end{equation}
In terms of Khovanov-Lauda diagrams, 
$$
\psi_{u_0}1_{\bi}=\begin{braid}\tikzset{baseline=7mm}
  \draw (0,4)node[above]{$i_1$}--(0,0);
  \draw (1,4)node[above]{$i_2$}--(2,0);
  \draw[dots] (1.3,4)--(4.9,4);
  \draw[dots] (3.2,0)--(9.1,0);
  \draw (5,4)node[above]{$i_{e}$}--(9.3,0);
 \draw (6,4)node[above]{$i_1$}--(1,0);
  \draw (7,4)node[above]{$i_2$}--(2.8,0);
  \draw[dots] (7.1,4)--(9.9,4);
  \draw (10.2,4)node[above]{$i_{e}$}--(10.2,0);
\end{braid}.
$$
}
\end{Example}

Note that in all cases, we can write
\begin{equation}\label{EU'}
\si_1=\psi_{u'}\psi_{u_0}
\end{equation}
for some $u'\in\Si_{2e}$.

\begin{Lemma} \label{Lcd}
The constant $c$ appearing in (\ref{ETauC}) is equal to $\pm 1$. Moreover, $\psi_{u_0}\si_1v_2=-2c\psi_{u_0}v_2$. 
\end{Lemma}
\begin{proof}
By Theorem~\ref{TEndMn}, we have $\tau_1^2=1$. It follows that $(\si_1+c)^2v_2=v_2$ or $\si_1^2v_2=(-2c\si_1+1-c^2)v_2$. 

On the other hand, note that $\psi_{u_0}v_2\neq 0$ spans the top degree component of the word space $(M_2)_{\bi^{\{2\}}}$. It follows that $\psi_{u_0}\si_1v_2=d\psi_{u_0}v_2$ for some constant $d$. Multipluing on the left with $\psi_{u'}$ as in (\ref{EU'}), this yields $\si_1^2v_2=d\si_1v_2$. Comparing with the previous paragraph, we conclude that $1-c^2=0$ and $d=-2c$. 
\end{proof}

\begin{Lemma} \label{LTriang}
Let $w\in\Si_n$, and $v^1,\dots,v^n\in L_\de$. Then 
$$
(v^1\otimes\dots\otimes v^n)w\equiv \si_w(v^{w1}\otimes \dots\otimes v^{wn})\pmod{\sum_{u<w}\si_u\otimes L_\de^{\boxtimes n}}
$$
and 
$$
\si_w(v^1\otimes\dots\otimes v^n)\equiv (v^{w^{-1}1}\otimes \dots\otimes v^{w^{-1}n})w\pmod{\sum_{u<w}L_\de^{\boxtimes n}u}.
$$
\end{Lemma}
\begin{proof}
The second statement follows from the first. Further, note that the first statement in the special case where $v^1=\dots= v^n=v_1$ is contained in (\ref{ETauSiTriangular}). For the general case, write  $v^1=x_1v_1,\dots, v^n=x_nv_1$ for some $x_1,\dots,x_n\in R_\de$. Then, using (\ref{ETauSiTriangular}) and considering $x_1\otimes \dots\otimes x_n\in R_{\de^n}$ as an element of $R_{n\de}\supseteq R_{\de^n}$, we get  
\begin{align*}
(v^1\otimes\dots\otimes v^n)w&=(x_1v_1\otimes\dots\otimes x_nv_1)w
\\
&=(x_1\otimes\dots\otimes x_n)((v_1\otimes\dots\otimes v_1)w)
\\
&=(x_1\otimes\dots\otimes x_n)(\si_w(v_1\otimes \dots\otimes v_1)+\sum_{u<w}c_u\si_u(v_1\otimes \dots\otimes v_1)).
\end{align*}
Using the fact that all our vectors belong to $V_n$, Proposition~\ref{PEndNew}(ii),  and the relations in $R_{n\de}$, we have for any $u\leq w$ that 
\begin{align*}
(x_1\otimes\dots\otimes x_n)\si_u(v_1\otimes \dots\otimes v_1)
&\equiv
\si_u(x_{u1}v_1\otimes\dots\otimes x_{un}v_1)
\pmod{\sum_{x<u}\si_x\otimes L_\de^{\boxtimes n}}
\\
&\equiv
\si_u(v^{u1}\otimes\dots\otimes v^{un})
\pmod{\sum_{x<u}\si_x\otimes L_\de^{\boxtimes n}},
\end{align*}
which proves the lemma. 
\end{proof}

\begin{Remark} 
{\rm 
If $\Car$ is symmetric, then an induction on the Bruhat order and (\ref{ETauCv1v2}) allow us to strengthen Lemma~\ref{LTriang} as follows: for any $v^1,\dots,v^n\in L_\de$, we have 
\begin{equation}\label{ETauSiTriangularv1vn}
(v^1\otimes\dots\otimes v^n)w=\si_w(v^{w1}\otimes \dots\otimes v^{wn})+\sum_{u<w}c_u\si_u(v^{u1}\otimes \dots\otimes v^{un})
\end{equation}
for some scalars $c_u\in\O$ (depending on $v^1,\dots,v^n$). 
}
\end{Remark}

If $c=-1$, it will be convenient to change the sign, so let us redefine $\tau_r$ so that   
\begin{equation}\label{ETau}
\tau_r(v_n):=(c\si_{r}+1)v_n \qquad(1\leq r<n).
\end{equation}

\begin{Remark} 
{\rm 
The constant $c$ in general depends on the choice of the signs $\eps_{i,j}$ in the definition of the KLR algebra. For symmetric $\Car$, it can be proved that $c=\prod_{1 \leq r < s \leq e} \epsilon_{i_r,i_s}.$
We are not going to need this result. 
}
\end{Remark}

\section{Imaginary Schur algebras}

A key role in this paper is played by the {\em imaginary Schur algebra}
$$
\ImS_n=\ImS_{n,\O}:=R_{n\de}/\operatorname{Ann}_{R_{n\de}}(\Mde_n),
$$
and its parabolic analogue for $\nu=(n_1,\dots,n_a)\vDash n$:
\begin{equation}\label{ECNu}
\ImS_\nu=\ImS_{\nu,\O}:=R_{\nu,\de}/\operatorname{Ann}_{R_{\nu,\de}}(\Mde_\nu)
\cong \ImS_{n_1}\otimes\dots\otimes \ImS_{n_a}. 
\end{equation}

Modules over $R_{n\de}$ which factor through $\ImS_n$ will be called {\em imaginary modules}. Thus the category of imaginary $R_{n\de}$-modules is the same as the category of $\ImS_n$-modules. 

We make use of the following useful criterion:

\begin{Lemma} \label{LSchub}  {\bf (`Schubert's Criterion')}
Let $A$ be a (graded) algebra and $0 \to Z \to P \to M \to 0$ be a short exact sequence of (graded) $A$-modules with $P$ (graded) projective. If every (degree zero) $A$-module homomorphism from $P$ to $M$ annihilates $Z$, then $M$ is a (graded) projective  $A/\Ann_A(M)$-module.
\end{Lemma}
\begin{proof}
The proof given in \cite[Lemma 3.2a]{BDK} goes through for the graded setting. 
\end{proof}

Now we can prove our first key result. 

\begin{Theorem} \label{TSchub}
$\Mde_\nu$ is a projective $\ImS_\nu$-module, and 
\begin{equation}\label{ESWOne}
\End_{\ImS_{\nu}}
(\Mde_\nu)\cong \O\Si_\nu.
\end{equation}
\end{Theorem}
\begin{proof}
The second statement comes from Theorem~\ref{TEndMn}. It suffices to prove the first statement  for the special case $\nu=(n)$. We will apply Schubert's Criterion to see that $\Mde_n$ is projective as a $\ImS_n$-module. 
Let
$
P:=q^{ nN} R_{n\de}1_{\bi^n}
$
Then we have a (homogeneous) surjection 
$
\pi:P\onto \Mde_n,\ 1_{\bi^n}\mapsto v_n.
$

To verify the assumptions in Lemma~\ref{LSchub}, it suffices to show that every (homogeneous) homomorphism $\phi:P\to M$ can be written as $\phi=f\circ \pi$ for $f\in  \operatorname{end}_{R_{n,\de}}
(\Mde_n)$. Since $P$ is generated by $1_{\bi^n}$, it suffices to prove that $\phi(1_{\bi^n})=f(\pi(1_{\bi^n}))=f(v_n)$. By Proposition~\ref{PEndNew}(iv) and Corollary~\ref{CVnDec}(ii), the vector 
$\phi(1_{\bi^n})\in (1_{\bi^n}M_n)_{nN}$ can be written as a linear combination 
$$\phi(1_{\bi^n})\in (1_{\bi^n}M_n)_{nN}
=\sum_{w\in\Si_n}c_wv_nw
=\sum_{w\in\Si_n}c_w\tau_w(v_n)
\qquad(c_w\in\O). 
$$
So we can take $f=\sum_{w\in\Si_n}c_w\tau_w$. 
Now apply Schubert's Criterion to see that $\Mde_n$ is projective. 
\end{proof}

\section{Characteristic zero theory}
In this section, we assume that $\O=F$. 
If the characteristic of $F$ is zero or greater than $n$, the imaginary Schur algebra is semisimple and Morita equivalent to $F\Si_n$:

\begin{Theorem} \label{TChar0}
Assume that $\cha F=0$ or $\cha F>n$. Then $\ImS_n$ is semisimple, $M_n$ is a projective generator for $\ImS_n$, and $\ImS_n$ is Morita equivalent to $F\Si_n$. 
\end{Theorem}
\begin{proof}
Under the assumptions on the characteristic, the endomorphism algebra of the $\ImS_n$-module $M_n$, which we know is isomorphic to $F\Si_n$, is semisimple. In view of Lemma~\ref{L3.1f}, we conclude that $M_n$ is semisimple as an $\ImS_n$-module. By definition, the imaginary Schur algebra $\ImS_n$ is semisimple, and the theorem now follows from Morita theory. 
\end{proof}

The theorem defines a Morita equivalence 
$$
\ga_n:\mod{\ImS_n}\to\mod{F\Si_n}. 
$$
One can easily show that for $N\in \mod{\ImS_n}$ and $M\in \mod{\ImS_m}$, there is a functorial isomorhism $\ga_{n+m}(N\circ M)\cong \ind_{\Si_n\times \Si_m}^{\Si_{n+m}}\ga_n(N)\boxtimes \ga_m(M)$. We will not do it now, since more general result will be obtained (for an arbitrary ground) field in Section~\ref{SIndMorita} using Schur algebras.


\section{Imaginary induction and restriction}\label{ChImInd}
Throughout the section $\nu=(n_1,\dots,n_a)\vDash n$. Recall the parabolic subalgebra $R_{\nu,\de}\subseteq R_{n\de}$ from (\ref{ERNuDe}).

Consider the functors of {\em imaginary induction} and {\em imaginary restriction}: 
\begin{align*}
\HCI_\nu^n&:=\Ind_{n_1\de,\dots,n_a\de}^{n\de}: \mod{R_{\nu,\de}}\to \mod{R_{n\de}},
\\
\HCR^n_\nu&:=\Res_{n_1\de,\dots,n_a\de}^{n\de}: \mod{R_{n\de}}\to \mod{R_{\nu,\de}}.
\end{align*}
Let $\words_{\nu,\de}\subseteq \words_{n\de}$ be the set of the concatenations  
$\bj=\bj(1)\dots\bj(a)$ such that $\bj(b)\in \words_{n_b\de}$ for all $b=1,\dots,a$. Set
$$
1_{\nu,\de}:=\sum_{\bj\in \words_{\nu,\de}}1_\bj.
$$
Then $1_{\nu,\de}$ is the identity in $R_{\nu,\de}$ and  
$\HCR^n_\nu M=1_{\nu,\de}M$. 
The functor $\HCI_\nu^n$ is left adjoint to the functor $\HCR_\nu^n$. The following result (partially) describes the right adjoint:

\begin{Lemma} \label{LOp}
For a composition $\nu=(n_1,\dots,n_a)\vDash n$ consider the opposite composition $\nu^\op:=(n_a,\dots,n_1)\vDash n$. Let $V$ be an $R_{n\de}$-module,  and $V_b$ be an $R_{\nu_b\de}$-module for $b=1,\dots,a$. Then there is a functorial isomorphism
$$
\Hom_{R_{\nu,\de}}(\HCR^n_\nu V,V_1\boxtimes\dots\boxtimes V_a)
\cong\Hom_{R_{n\de}}(V,\HCI_{\nu^\op}^nV_a\boxtimes\dots\boxtimes V_1).
$$
\end{Lemma}
\begin{proof}
This follows from Lemma~\ref{LLV}.
\end{proof}


In view of Theorem~\ref{TEndMn}, $\Mde_\nu$ is an $(R_{\nu,\de},\O\Si_\nu)$-bimodule, so we can regard 
$\HCI_\nu^n \Mde_\nu$ as an $(R_{n\de},\O\Si_\nu)$-bimodule. Similarly, $\Mde_n$ is an $(R_{n\de},\O\Si_n)$-bimodule, so we can regard 
$\HCR_\nu^n \Mde_n$ as an $(R_{\nu,\de},\O\Si_n)$-bimodule. 
Recall that we identify $L_\de^{\boxtimes n}$ with a natural $R_{\de^n}$-submodule of $M_n$.

\begin{Lemma} \label{3.2f}
We have:
\begin{enumerate}
\item[{\rm (i)}] $\HCI_\nu^n \Mde_\nu\cong \Mde_n$ as  $(R_{n\de},\O\Si_\nu)$-bimodules. 
\item[{\rm (ii)}] $\HCR_\nu^n \Mde_n\cong \Mde_\nu\otimes_{\O\Si_\nu}\O\Si_n$ as  $(R_{\nu,\de},\O\Si_n)$-bimodules. 
\item[{\rm (iii)}] 
We have the following decompositions of $\O$-modules: 
$$
\HCR_\nu^n \Mde_n=\bigoplus_{x\in \D_{e\nu}^{(e^n)},\ y\in {}^\nu\D_n} \psi_{x}\si_y L_\de^{\boxtimes n}=\bigoplus_{x\in \D_{e\nu}^{(e^n)},\ y\in {}^\nu\D_n} \psi_{x} L_\de^{\boxtimes n}y
$$
\end{enumerate}
\end{Lemma}
\begin{proof}
(i) By transitivity of induction, $\HCI_\nu^n \Mde_\nu\cong \Mde_n$ as  $R_{n\de}$-modules. By definition, the isomorphism is compatible with the right $\O\Si_\nu$-module structures. 

(iii) Recall the decomposition
$M_n=\bigoplus_{w\in \D_{en}^{(e^n)}}\psi_{w} L_\de^{\boxtimes n}$ from Lemma~\ref{LMBasis}. Note, using word argument, that $\psi_{w} L_\de^{\boxtimes n}\subseteq \HCR_\nu^n \Mde_n=1_{\nu,n}\Mde_n$ for $w\in \D_{en}^{(e^n)}$ if and only if $w$ can be written (uniquely) as $w=x\iota(y)$, where $x\in \D_{e\nu}^{(e^n)}$, $y\in {}^\nu\D_n$, and $\ell(w)=\ell(x)+\ell(\iota(y))$, and otherwise 
$\psi_{w} L_\de^{\boxtimes n}\cap 1_{\nu,n}\Mde_n=0$. 
This gives the first decomposition.

To deduce the second decomposition from the first, observe by a  word argument, that each $\psi_x L_\de^{\boxtimes n}y\subseteq \HCR_\nu^n \Mde_n$. Next, note using Lemma~\ref{LMBasis}  that each $\psi_x\si_y L_\de^{\boxtimes n}\cong L_\de^{\boxtimes n}$ as vector spaces. As $\psi_x L_\de^{\boxtimes n}\to \psi_x L_\de^{\boxtimes n} y$ is an invertble linear transformation, we also have that $\psi_x L_\de^{\boxtimes n}y\cong L_\de^{\boxtimes n}$ as vector spaces. Now, by dimensions, it suffices to prove that the sum $\sum_{x\in \D_{e\nu}^{(e^n)}, y\in {}^\nu\D_n} \psi_{x} L_\de^{\boxtimes n}y$ is direct. Well, if 
\begin{equation}\label{E111213}
\sum_{x\in \D_{e\nu}^{(e^n)}, y\in {}^\nu\D_n} \psi_{x} v_{x,y}y=0
\end{equation} 
with $v_{x,y}\in L_\de^{\boxtimes n}$, let $x,y$ be chosen so that $xy\in\Si_{en}$ is Bruhat maximal with $v_{x,y}\neq 0$. Rewriting the left hand side of (\ref{E111213}) 
using Lemma~\ref{LTriang}, gives   
$
\psi_x\si_y v_{x,y}+(*)=0$,
where 
$$(*)\in \sum_{x'\in \D_{e\nu}^{(e^n)},\ y'\in {}^\nu\D_n,\ x'y'\not\geq xy}\psi_{x'}\si_{y'}L_\de^{\boxtimes n}.
$$
We get a contradiction. 

(ii) follows from (iii). 
\end{proof}

In view of Lemma~\ref{3.2f}(ii), the $R_{\nu,\de}$-action on $\HCR_\nu^n \Mde_n$ factors through the quotient $\ImS_\nu$, so $\HCR_\nu^n \Mde_n$ is a $\ImS_\nu$-module in a natural way. 
In Corollary~\ref{CHCIC} we will prove a stronger result that 
the functor $\HCI_\nu^n$ sends $\ImS_\nu$-modules to $\ImS_n$-modules and the functor $\HCR_\nu^n$ sends $\ImS_n$-modules to $\ImS_\nu$-modules.

\begin{Corollary} \label{CFunPerm}
The following pairs of functors are isomorphic:
\begin{enumerate}
\item[{\rm (i)}] $\HCI_\nu^n\circ (\Mde_\nu\otimes_{\O\Si_\nu}\, ?)$ and $(\Mde_n\otimes_{\O\Si_n}\, ?)\circ \ind_{\Si_\nu}^{\Si_n}\,:\, \mod{\O\Si_\nu}\to \mod{R_{n,\de}}$. 
\item[{\rm (ii)}] $\HCR_\nu^n\circ (\Mde_n\otimes_{\O\Si_n}\, ?)$ and $(\Mde_\nu\otimes_{\O\Si_\nu}\, ?)\circ \res_{\Si_\nu}^{\Si_n}\,:\, \mod{\O\Si_n}\to \mod{R_{\nu,\de}}$. 
\end{enumerate}
\end{Corollary}
\begin{proof}
(i) Take $N\in \mod{\O\Si_\nu}$. Using Lemma~\ref{3.2f}(i), we have natural isomorphisms
\begin{align*}
\Mde_n\otimes_{\O\Si_n} \ind_{\Si_\nu}^{\Si_n}N&=\Mde_n\otimes_{\O\Si_n} \O\Si_n\otimes_{\O\Si_\nu}N
\cong \Mde_n\otimes_{\O\Si_\nu}N
\\
&\cong (\HCI^n_\nu \Mde_\nu)\otimes_{\O\Si_\nu}N
\cong \HCI^n_\nu (\Mde_\nu\otimes_{\O\Si_\nu}N),
\end{align*}
as required.

(ii) 
Using Lemma~\ref{3.2f}(ii), for an $\O\Si_n$-module $N$, we have natural isomorphisms
\begin{align*}
\HCR_\nu^n (\Mde_n\otimes_{\O\Si_n} N)&\cong(\HCR_\nu^n \Mde_n)\otimes_{\O\Si_n} N
\cong (\Mde_\nu\otimes_{\O\Si_\nu}\O\Si_n)\otimes_{\O\Si_n} N
\\
&\cong \Mde_\nu\otimes_{\O\Si_\nu}  \res_{\Si_\nu}^{\Si_n} N,
\end{align*}
as required.
\end{proof}

We need a version of the Mackey Theorem for imaginary induction and restriction. Recall the notation from Section~\ref{SSCosetRep}. In particular, given two compositions $\la,\mu\vDash n$ and $x\in {}^\la\D^\mu_n$ we have compositions $\la\cap x\mu$ and $x^{-1}\la\cap\mu$. Moreover, the corresponding parabolic algebras $R_{\la\cap x\mu,\de}$ and $R_{x^{-1}\la\cap\mu,\de}$ are naturally isomorphic  via an isomorphism 
\begin{equation}\label{EPiX}
\Pi_x:R_{\la\cap x\mu,\de}\iso R_{x^{-1}\la\cap\mu,\de},
\end{equation}
which permutes the components. Composing with this isomorphism we get a functor 
$$
\mod{R_{x^{-1}\la\cap\mu,\de}}\to \mod{R_{\la\cap x\mu,\de}},\ M\mapsto {}^xM.
$$
Note that we do not need any grading shifts. With this notation, we have:

\begin{Theorem} \label{TMackey} 
{\bf (Imaginary Mackey Theorem)}
Let $\la,\mu\vDash n$, and $M$ be an $\ImS_{\mu}$-module. Then there is a filtration of $\HCR^n_\la\HCI^n_\mu M$ with subfactors 
$$ 
\HCI_{\la\cap x\mu}^{\la}{}^{x}(\HCR^\mu_{x^{-1}\la\cap\mu} M)
\qquad(x\in{}^\la\D^\mu_n).
$$
\end{Theorem}
\begin{proof}
This follows from the usual Mackey Theorem~\ref{TMackeyKL} using the fact that all composition factors of $M$ are imaginary in the sense of (Cusp2). 
\end{proof}

We conjecture that instead of a filtration, in the Imaginary Mackey Theorem we actually have a {\em direct sum}:

\begin{Conjecture} \label{CMackey} 
{\bf (Strong Imaginary Mackey)}
Let $\la,\mu\vDash n$, and $M$ be an $\ImS_{\mu}$-module. Then 
$$
\HCR^n_\la\HCI^n_\mu M\cong \bigoplus_{x\in{}^\la\D^\mu_n} 
\HCI_{\la\cap x\mu}^{\la}{}^{x}(\HCR^\mu_{x^{-1}\la\cap\mu} M).
$$
\end{Conjecture}

\chapter{Imaginary Howe duality}\label{SISDEP}
The imaginary Schur-Weyl duality described in the previous section is not quite sufficient to describe the composition factors of $M_n$ at least when the characteristic of the ground field is positive. The problem is that even though $M_n$ is a projective module over the imaginary Schur algebra, it is not in general a projective generator. 
We construct the desired projective generator as a direct sum  $Z=\bigoplus_{\nu\in X(h,n)} Z^\nu$ of `imaginary divided powers' modules, and the endomorphism algebra of $Z$ turns out to be the classical Schur algebra $S_{h,n}$. This leads to an equivalence of module categories for the imaginary and the classical Schur algebras. First, we need to develop a theory of ``Gelfand-Grave modules''.

Throughout the chapter, $\nu=(n_1,\dots,n_a)\vDash n\in \Z_{>0}$. 

\section{Gelfand-Graev modules}\label{SIGGR}

Denote by $w_0$ the longest element of $\Si_n$, and for $i\in I$,  consider the element
$$
\ga_{n,i}:=\psi_{w_0}\prod_{m=1}^ny_m^{m-1}\in R_{n\al_i},
$$
and the $R_{n\al_i}$-module 
$$
\Ga_{n,i}:=q_i^{-n(n-1)/2}R_{n\al_i}\ga_{n,i}.
$$

The following is well-known:

\begin{Proposition} \label{PGGId} 
The algebra $R_{n\al_i}$ is isomorphic to the affine nil-Hecke algebra and has unique (up to isomorphism and degree shift) irreducible module, denoted  $L(i^n)$ with formal character $[n]_i^!(i^n)$. Moreover:
\begin{enumerate}
\item[{\rm (i)}] $\ga_{n,i}$ is a primitive idempotent in $R_{n\al_i}$. In particular, 
$
\Ga_{n,i}
$
is a projective indecomposable $R_{n\al_i}$-module. In fact, $\Ga_{n,i}$ is the projective cover of $L(i^n)$. 
\item[{\rm (ii)}] $\Ga_{n,i}$  is isomorphic to the polynomial representation of the affine nil-Hecke algebra $R_{n\al_i}$ (with degree shifted down by $(\al_i,\al_i)n(n-1)/4)$; in particulr, $\Ga_{n,i}$ has an $\O$-basis 
$$\{y_1^{b_1}\dots y_n^{b_n}\ga_{n,i}\mid b_1,\dots, b_n\in\Z_{\geq 0} \},$$ and the formal character $q_i^{-n(n-1)/2}(1-q_i^2)^{-n}(i^n)$. 
\item[{\rm (iii)}] Let $(m_1,\dots,m_s)\vDash n$. Then 
$$\Res_{m_1\al_i,\dots,m_s\al_i}\Ga_{n,i}\cong 
q_i^{-n(n-1)/2+\sum_{i=1}^s m_i(m_i-1)/2}
\Ga_{m_1,i}\,\boxtimes\dots\boxtimes\, \Ga_{m_s,i}.$$ 
\end{enumerate}
\end{Proposition}
\begin{proof}
For (i) and (ii), see for example \cite[section 2.2]{KL1} or \cite[Theorem~4.12]{KLM}.  Part (iii) follows easily from (i) and (ii) by characters.  
\end{proof}

Now, recall the word $\bi=(i_1,\dots,i_e)$ from (\ref{EBWt}). We rewrite:
\begin{equation}\label{EGGNewWay}
\bi=j_1^{m_1}\dots j_r^{m_r}
\end{equation}
with $j_k\neq j_{k+1}$ for all $r=1,2,\dots,r-1$. 
Define 
the {\em Gelfand-Graev idempotent}: 
$$\ga_{n,\de}:=\ga_{nm_1,j_1}\otimes \ga_{nm_2,j_2}\otimes\dots\otimes \ga_{nm_r,j_r}\in R_{nm_1\al_{j_1},nm_2\al_{j_2},\dots,nm_r\al_{j_r}}\subseteq R_{n\de}.$$
By Proposition~\ref{PGGId}(i), 
$$\Ga_{n}:=
\prod_{k=1}^r q_{j_k}^{-nm_k(nm_k-1)/2}
R_{n\de}\ga_{n,\de}\cong\Ga_{nm_1,j_1}\circ \Ga_{nm_2,j_2}\circ\dots\circ \Ga_{nm_r,j_r}
$$
is a projective $R_{n\de}$-module which we refer to as the {\em Gelfand-Graev}\, module. 
By Proposition~\ref{PGGId}(ii),
$$
\CH\Ga_{n}=\prod_{k=1}^r q_{j_k}^{-nm_k(nm_k-1)/2}(1-q_{j_k}^2)^{-nm_k}\, j_1^{nm_1}\circ\, j_2^{nm_2}\,\circ\,\dots\,\circ\, j_r^{nm_r}.
$$

More generally, 
we consider the {\em (parabolic) Gelfand-Graev idempotent} 
\begin{equation}\label{EGGId}
\ga_{\nu,\de}:=\ga_{n_1,\de}\otimes\dots\otimes \ga_{n_a,\de}\in R_{\nu,\de}
\subseteq R_{n\de},
\end{equation} 
and the projective $R_{\nu,\de}$-module 
$$
\Ga_{\nu}:= \prod_{b=1}^a\prod_{k=1}^r q_{j_k}^{-n_bm_k(n_bm_k-1)/2}R_{\nu,\de}\ga_{\nu,\de}
\cong\Ga_{n_1}\boxtimes \dots\boxtimes \Ga_{n_a}
$$
with  character
\begin{align*}
\CH \Ga_{\nu}=&
\prod_{b=1}^a\prod_{k=1}^r q_{j_k}^{-n_bm_k(n_bm_k-1)/2}(1-q_{j_k}^2)^{-n_bm_k}
\\
&
\times 
(j_1^{n_1m_1} 
\circ \dots \circ  j_r^{n_1m_r})\dots(j_1^{n_am_1}  
\circ \dots \circ  j_r^{n_am_r}).
\end{align*}


\begin{Lemma} \label{LMMult1} 
We have  
$$\Res_{nm_1\al_{j_1},\dots,nm_r\al_{j_r}}M_n\simeq L(j_1^{nm_1})\boxtimes\dots\boxtimes L(j_r^{nm_r}).$$
\end{Lemma}
\begin{proof}
The lemma is obtained by an application of the Mackey Theorem or a character computation. 
\end{proof}

\begin{Proposition} \label{2.5e} 
We have:
\begin{enumerate}
\item[{\rm (i)}] $\HCR^n_\nu \Ga_{n}\cong \Ga_{\nu}\oplus X$, where $X$ is a projective module over $R_{\nu,\de}$ such that $\Hom_{R_{\nu,\de}}(X,M_\nu)=0$.
\item[{\rm (ii)}] $\Hom_{R_{n\de}}(\Ga_{n},\Mde_n)\cong \O$. 
\item[{\rm (iii)}] We have an isomorphism of right modules over $\O\Si_n=\End_{R_{n\de}}(\Mde_n)$:  
$$\Hom_{R_{n\de}}(\Ga_{n},\Mde_n)\cong \sgn_{\Si_n}.$$ 
\end{enumerate} 
\end{Proposition}
\begin{proof}
(ii) By Frobenius Reciprocity, $\Hom_{R_{n\de}}(\Ga_{n},\Mde_n)$ is isomorphic to  
$$
\Hom_{R_{nm_1\al_{j_1},\dots,nm_r\al_{j_r}}}(\Ga_{nm_1,j_1}\boxtimes\dots\boxtimes\Ga_{nm_r,j_r},\Res_{nm_1\al_{j_1},\dots,nm_r\al_{j_r}}\Mde_n).
$$
By Proposition~\ref{PGGId}, $\Ga_{nm_1,j_1}\boxtimes\dots\boxtimes\Ga_{nm_r,j_r}$ is the projective cover of $L:=L(j_1^{nm_1})\boxtimes\dots\boxtimes L(j_r^{nm_r})$. The result now follows from (\ref{EEndL}) and  Lemma~\ref{LMMult1}. 

(i) By the Mackey Theorem and Proposition~\ref{PGGId}(iii), the module  
$$
\HCR^n_\nu \Ga_{n}=\Res_{n_1\de,\dots,n_a\de}\,\Ind_{nm_1\al_{j_1},\dots,nm_r\al_{j_r}}\Ga_{nm_1,j_1}\boxtimes\dots\boxtimes \Ga_{nm_r,j_r}
$$
has filtration with factors of the form 
$$
(\Ga_{m_{1,1},j_1}\circ \Ga_{m_{1,2},j_2}\circ\dots\circ \Ga_{m_{1,r},j_r})\boxtimes\dots\boxtimes (\Ga_{m_{a,1},j_1}\circ \Ga_{m_{a,2},j_2}\circ\dots\circ \Ga_{m_{a,r},j_r}),
$$
where $\sum_{s=1}^r m_{t,s}\al_{j_s}=n_t\de$ for all $t=1,\dots,a$, $\sum_{t=1}^am_{t,s}=nm_s$ for all $s=1,\dots,r$, 
and we ignore grading shifts. All of this modules are projective, so we actually have a direct sum. One of the terms is $\Ga_\nu$---it corresponds to taking $m_{t,s}=n_tm_s$ for all $1\leq t\leq a,1\leq s\leq r$. 

Now, note, using Lemma~\ref{LOp} that 
\begin{align*}
\Hom_{R_{\nu,\de}}(\HCR^n_\nu \Ga_n,M_\nu)=
\Hom_{R_{\nu,\de}}(\HCR^n_\nu \Ga_n,M_{n_1}\boxtimes\dots\boxtimes M_{n_a})
\\
\cong \Hom_{n\de}(\Ga_n,\HCI^n_{\nu^\op}(M_{n_a}\boxtimes\dots\boxtimes M_{n_1}))
\cong
\Hom_{n\de}(\Ga_n,M_n).
\end{align*}
By part (ii), the latter Hom-space is isomorphic to $\O$. On the other hand, again by (ii), we have that 
$
\Hom_{\nu,\de}(\Ga_\nu,M_\nu)\cong \O.
$ 
We conclude that $\Ga_\nu$ appears in $\HCR^n_\nu \Ga_n$ with graded multiplicity $1$, and other projective summands do not have non-trivial homomorphisms to $M_\nu$, as required.


(iii) Note that $\Ga_{n}$ is generated by a vector of the word $\bj:=(j_1^{nm_1},\dots,j_r^{nm_r})$. Under any homomorphism from $\Ga_{n}$ to $\Mde_n$, this generating vector is mapped to a vector in the word space $1_\bj\Mde_n$. So, it suffices to show that an arbitrary $w\in \Si_n$ acts on the whole word space $1_\bj\Mde_n$ with the scalar $\sgn(w)$. Let $u$ be the shortest element of $\Si_{ne}$ such that $u\cdot \bi^n=\bj$. Then any other vector in $1_\bj\Mde_n$ can be written in the form 
$
\{\psi_x\psi_uv_n\}
$
for some $x$. So it suffices to prove that $\psi_uv_ns_r=-\psi_uv_n$ for an arbitrary simple generator $s_r$ of $\Si_n$ with $1\leq r<n$.  

Recall the definition of $u_0\in\Si_{2e}$ from Lemma~\ref{Lcd}. For $1\leq r<n$, let 
$$\phi_r:\Si_{2e}\simeq\Si_{1^{(r-1)e}}\times \Si_{2e}\times\Si_{1^{(n-r-1)e}}\into \Si_{ne}$$ 
be the natural embedding and  $u_0(r):=\phi_r(u_0)$. 

There exists $u'\in\Si_d$ such that 
$\psi_u=\psi_{u'}\psi_{u_0(r)}.$ 
 So by Lemma~\ref{Lcd}, we have 
\begin{align*}
\psi_u\si_r(v_n)=\psi_{u'}\psi_{u_0(r)}\si_{r}v_n=-2c\psi_{u'}\psi_{u_0(r)}=-2c\psi_uv_n.
\end{align*}
Therefore, using (\ref{ETau}), we get 
\begin{align*}
\psi_uv_ns_r&=\psi_u\tau_r(v_n)=\psi_u(c\si_r+1)(v_n)
\\&=c\psi_u\si_r(v_n)+\psi_uv_n=-2\psi_uv_n+\psi_uv_n=-\psi_uv_n,
\end{align*}
completing the proof. 
\end{proof}

\begin{Remark} 
{\rm 
In type ${\tt A}_l^{(1)}$, we can strengthen Proposition~\ref{2.5e}(i) to claim that 
$\HCR^n_\nu \Ga_{n}\cong \Ga_{\nu}$.  Indeed, in this case each simple root appears in $\de$ with multiplicity one, from which one can easily deduce that $\CH\HCR^n_\nu \Ga_{n}=\CH \Ga_{\nu}$.
}
\end{Remark}

\section{Imaginary symmetric, divided,  and exterior powers}
Let 
$${\mathtt x}_n:=\sum_{g\in\Si_n}g\qquad\text{and}\qquad {\mathtt y}_n:=\sum_{g\in\Si_n}\sgn(g)g.
$$
Define  {\em imaginary symmetric, divided, and exterior powers}\, as the following $R_{n\de}$-modules: 
\begin{align*}
\Sde_n&:=\Mde_n/\spa\{mg-\sgn(g)m\mid g\in\Si_n,\ m\in \Mde_n\},
\\
\Zde_n&:=\{m\in \Mde_n\mid mg-\sgn(g)m=0\ \text{for all $g\in\Si_n$}\},
\\
\Lade_n&:=\Mde_n {\mathtt x}_n.
\end{align*} 
These $R_{n\de}$-modules factor through the quotient $\ImS_n$ to induce well-defined $\ImS_n$-modules. It is perhaps unfortunate that our symmetric powers correspond to the sign representation and our exterior powers correspond to the trivial representation; curiously this is the same phenomenon as for finite $GL_n$, cf. \cite[section 3.3]{BDK}. 

Note that $\Lade_n=\Mde_n{\mathtt x}_n\neq 0$ and $\Mde_n{\mathtt y}_n\neq 0$ for example by Theorem~\ref{TEndMn}. Finally, by definition, $\Mde_n {\mathtt y}_n$ is 
a submodule of $\Zde_n$. 
Recall the word $\bi$ from (\ref{EBWt}). 

\begin{Lemma} \label{LLaIrr}
We have $1_{\bi^n}\Lade_n=(1_\bi L_\de)^{\boxtimes n}{\mathtt x}_n$ and $1_{\bi^n}\Mde_n{\mathtt y}_n=(1_\bi L_\de)^{\boxtimes n}{\mathtt y}_n$. Moreover, if $\O=F$ (i.e. $\O$ is a field) then $\Lade_n$ and $\Mde_n{\mathtt y}_n$ 
are irreducible $R_{n\de}$-modules. 
\end{Lemma}
\begin{proof}
We prove the lemma for $\Lade_n$, the argument for $\Mde_n{\mathtt y}_n$ being similar. 

By Corollary~\ref{CVnDec}, the word space $1_{\bi^n}M_n$ is isomorphic to free right module over $\O\Si_n$ with basis $(1_\bi L_\de)^{\boxtimes n}$, and under this isomorphism the generator $v_n\in M_n$ corresponds to $1\in \O\Si_n$. Therefore $1_{\bi^n}\Lade_n=1_{\bi^n}M_n{\mathtt x}_n=(1_\bi L_\de)^{\boxtimes n}{\mathtt x}_n$. 

We know that $\Mde_n$ is a projective $\ImS_n$-module and every composition factor of its socle appears in its head, see Theorem~\ref{TSchub} and Lemma~\ref{LMSelfDual}. Also, left ideal $F\Si_n{\mathtt x}_n$ is an irreducible $F\Si_n$-module. Using these remarks, the irreducibility of $\Lade_n=\Mde_n{\mathtt x}_n$ follows from Lemmas~\ref{L3.1f} and \ref{L3.1e}. 
\end{proof}

\begin{Lemma} \label{LSZ}
We have $\Sde_n\cong \Mde_n\otimes_{\O\Si_n}\sgn_{\Si_n}$. Moreover, if $\O=F$ then:
\begin{enumerate}
\item[{\rm (i)}] $\Sde_n$ has simple head isomorphic to $\Mde_n {\mathtt y}_n$, 
and no other composition factors of $\Sde_n$ are isomorphic to quotients of $\Mde_n$.
\item[{\rm (ii)}] $\Zde_n\cong (\Sde_n)^\circledast$. 
\item[{\rm (iii)}] $\Zde_n$ has simple socle isomorphic to $\Mde_n {\mathtt y}_n$, 
and no other composition factors of $\Zde_n$ are isomorphic to submodules of $\Mde_n$.
\end{enumerate}
\end{Lemma}
\begin{proof}
Write $M:= \Mde_n$ for short. By definition, $\Mde\otimes_{\O\Si_n}\sgn_{\Si_n}$ is the quotient of $\Mde\otimes_{\O}\sgn_{\Si_n}$ by 
$
\spa\{mg\otimes 1-\sgn(g)m\otimes 1\mid g\in \Si_n,\ m\in \Mde\}. 
$
If we identify $\Mde\otimes_{\O}\sgn_{\Si_n}$	and $\Mde$ as $\O$-modules in the natural way, this immediately gives the first statement.

(i) Let $\al$, $\beta$ and $A_P$ be the functors defined in Section~\ref{SSSF}, taking the projective module $P$ to be the $\ImS_n$-module $\Mde$ of Theorem~\ref{TSchub}. Then, the previous paragraph shows that $\Sde_n\cong \beta(\sgn_{\Si_n})$. 
As 
every composition factor of the socle of $\Mde$ appears in its head by  
Lemma~\ref{LMSelfDual}, we conclude using 
Lemmas~\ref{L3.1f} and \ref{LLaIrr} that 
$$
A_P\circ\be(\sgn_{\Si_n})\cong A_P(\Sde_n)\cong \Mde {\mathtt y}_n
$$
is an irreducible $\ImS_n$-module. We deduce that $\Mde {\mathtt y}_n$ appears in the head of $\Sde_n$ and no other composition factors of $\Sde_n$ appear in the head of $\Mde$. Since $\Sde_n$ is a quotient of $\Mde$, this means that $\Sde_n$    has simple head.

(ii) In view of Lemma~\ref{LMSelfDual}, we choose some isomorphism $\phi : M \to M^\circledast$ of $R_{n\de}$-modules. This choice induces an isomorphism $\kappa : \End_{R_{n\de}}(M) \to \End_{R_{n\de}}(M^\circledast)$ with 
$f\kappa(\theta) = \phi((\phi^{-1}f)\theta)$ for all $f \in M^\circledast$ and $\theta\in \End_{R_{n\de}}(M)$, writing endomorphisms on the right. 
On	the	other hand, there is a natural anti-isomorphism $\sharp : \End_{R_{n\de}}(M) \to \End_{R_{n\de}}(M^\circledast)$ defined by letting $\theta^\sharp$ be the dual map to $\theta \in \End_{R_{n\de}}(M)$. 
Now if we set $\si:=\kappa^{-1}\circ \sharp$, we have defined an anti-automorphism of $F\Si_n = \End_{R_{n\de}}(M)$. Define a non-degenerate bilinear form on $M$ by 
$(v,w):=\phi(v)(w)$ for $v,w\in M$. For any $h\in F\Si_n$ we have 
$$
(v\si(h),w) = (\phi^{-1}(\phi(v)h^\sharp),w) = (\phi(v)h^\sharp)(w) = \phi(v)(wh) = (v,wh). 
$$
By definition, $\Sde_n=M/\spa\{vh-\sgn(h)v\mid h\in F\Si_n,\ v\in M\}$. So
\begin{align*}
\Sde_n^\circledast &\cong\{w\in M\mid (w,vh-\sgn(h)v)=0\ \text{for all $v\in M,h\in F\Si_n$}\}
\\
&=\{w\in M\mid (w\si(h)-\sgn(h)w,v)=0\ \text{for all $v\in M,h\in F\Si_n$}\}
\\
&=\{w\in M\mid wh=\sgn(\si(h))w\ \text{for all $h\in F\Si_n$}\}.
\end{align*}

To complete the proof, it remains to show that $\sgn(\si(h)) = \sgn(h)$ for all $h\in F\Si_n$. 
We can consider $\sgn \circ\si$ as a linear representation of $F\Si_n^\op$, so we either have $\sgn \circ\si = \sgn$ as required, or  $\sgn \circ\si = \id$. In the latter case, $\Sde_n^\circledast$ contains $\Lade_n$ as an irreducible submodule, see Lemma~\ref{LLaIrr}, whence $\Sde_n$ contains $M{\mathtt x}_n$ in its head. But this is not so by (i), unless  $M{\mathtt x}_n \cong M{\mathtt y}_n$, in which case, applying Lemma~\ref{L3.1e}, the sign representation of $F\Si_n$ is isomorphic to its trivial representation and we are done. 

(iii) This follows from (ii) by dualizing, using (i).
\end{proof}

\section{Parabolic analogues}
For $\nu=(n_1,\dots,n_a)\vDash n$, let 
$${\mathtt x}_\nu:=\sum_{g\in\Si_\nu}g\qquad\text{and}\qquad {\mathtt y}_\nu:=\sum_{g\in\Si_\nu}\sgn(g)g.
$$
We have the parabolic analogues of symmetric, divided and exterior powers, namely the  $\ImS_\nu$-modules 
\begin{align*}
\Sde_\nu&:=\Mde_\nu/\spa\{mg-\sgn(g)m\mid g\in\Si_\nu,\ m\in \Mde_\nu\},
\\
\Zde_\nu&:=\{m\in \Mde_\nu\mid mg-\sgn(g)m=0\ \text{for all $g\in\Si_\nu$}\},
\\
\Lade_\nu&:=\Mde_\nu {\mathtt x}_\nu.
\end{align*} 
In view of (\ref{ECNu}), if $\nu=(n_1,\dots,n_a)$ then $\Sde_\nu\cong \Sde_{n_1}\boxtimes\dots\boxtimes \Sde_{n_a}$, and similarly for $Z,\La$. In view of this observation, the basic properties of $\Sde_\nu, \Zde_\nu$ and $\Lade_\nu$ follow directly from Lemmas \ref{LLaIrr} and \ref{LSZ}.

\begin{Lemma} \label{L3.3c}
For any $\nu\vDash n$, we have $\HCR^n_\nu \Sde_n\cong \Sde_\nu$ and $\HCR^n_\nu \Zde_n\cong \Zde_\nu$. 
\end{Lemma}
\begin{proof}
We prove the first statement, the second one then follows from Lemma \ref{LSZ}(ii) since the restriction functor $\HCR^n_\nu$ commutes with duality. By Lemma~\ref{LSZ}, we have $\Sde_n \cong \Mde_n \otimes_{F\Si_n} \sgn_{\Si_n}$ and $\Sde_\nu \cong \Mde_\nu \otimes_{F\Si_\nu}\nolinebreak \sgn_\nu$. Now, using Corollary~\ref{CFunPerm}(ii), we get
$$ \HCR^n_\nu \Sde_n
\cong \HCR^n_\nu(\Mde_n \otimes_{F\Si_n} \sgn_{\Si_n})
\cong \Mde_\nu \otimes_{F\Si_\nu} (\res^{\Si_n}_{\Si_\nu} \sgn_{\Si_n}) \cong \Mde_\nu\otimes_{F\Si_\nu}\sgn_{\Si_\nu},
$$
which is isomorphic to $\Sde_\nu$, as required. 
\end{proof}

Define $R_{n\de}$-modules 
\begin{align*}
\Sdot^\nu&:=\Mde_n/\spa\{mg-\sgn(g)m\mid g\in\Si_\nu,\ m\in \Mde_n\},
\\
\Zdot^\nu&:=\{m\in \Mde_n\mid mg=\sgn(g)m\ \text{for all $g\in\Si_\nu$}\},
\\
\Ladot^\nu&:=\Mde_n {\mathtt x}_\nu.
\end{align*}
If we identify $\Mde_n=\HCI^n_\nu \Mde_\nu$ as $(R_{n\de},\O\Si_\nu)$-bimodules as in Lemma~\ref{3.2f}(i), it is easy to check that the quotient $\Sdot^\nu$ of $\Mde_n$ is identified with the quotient $\HCI^n_\nu \Sde_\nu$ of $\HCI^n_\nu 
\Mde_\nu$. Similarly we get the analogous results for $\Zde$ and $\Lade$. Thus:
\begin{equation}\label{EUpperNu}
\Sdot^\nu\cong \HCI_\nu^n \Sde_\nu,
\quad 
\Zdot^\nu\cong \HCI_\nu^n \Zde_\nu,
\quad
\Ladot^\nu\cong \HCI_\nu^n \Lade_\nu.
\end{equation}

Note that $\Zdot^\nu$ contains $\Mde_n {\mathtt y}_\nu$ as a submodule. 

\begin{Lemma} \label{3.3d}
We have 
$\Sdot^\nu\cong \Mde_n\otimes_{\O\Si_n}(\ind^{\Si_n}_{\Si_\nu}\sgn_{\Si_\nu}).$ 
Moreover, if $\O=F$ then:
\begin{enumerate}
\item[{\rm (i)}] $(\Sdot^\nu)^\circledast\cong \Zdot^{\nu^\op}$, where $\nu^\op=(n_a,\dots,n_1)$ is the opposite composition; 
\item[{\rm (ii)}] No composition factors of $\Zdot^\nu/\Mde_n {\mathtt y}_\nu$ are isomorphic to submodules of $\Mde_n$. 
\end{enumerate}
\end{Lemma}
\begin{proof}
By Lemma~\ref{LSZ}, we have $\Sde_\nu \cong \Mde_\nu \otimes_{F\Si_\nu}\nolinebreak \sgn_{\Si_\nu}$. Therefore, using Corollary~\ref{CFunPerm}(i) and (\ref{EUpperNu}) , we get that 
$$\Mde_n\otimes_{\O\Si_n}(\ind^{\Si_n}_{\Si_\nu}\sgn_{\Si_\nu})\cong \HCI_\nu^n \Sde_\nu\cong 
\Sdot^\nu.
$$

(i) Using (\ref{EUpperNu}), Lemma~\ref{LDualInd}, and Lemma~\ref{LSZ}(ii), we have 
\begin{align*}
(\Sde^\nu)^\circledast
&\cong 
I_\nu^n(\Sde_{n_1}\boxtimes\dots\boxtimes \Sde_{n_a})^\circledast
\\
&\cong I_{\nu^\op}^n(\Sde_{n_a}^\circledast\boxtimes\dots\boxtimes \Sde_{n_1}^\circledast)
\\
&\cong I_{\nu^\op}^n(\Zde_{n_a}\boxtimes\dots\boxtimes \Zde_{n_1})
\cong \Zde^{\nu^\op}.
\end{align*}

(ii) Let $\al$, $\beta$ and $A_P$ be the functors defined in Section~\ref{SSSF}, taking the projective module $P$ to be the $\ImS_n$-module $\Mde_\nu$ of Theorem~\ref{TSchub}. Now, $\ind^{\Si_n}_{\Si_\nu}\sgn_{\Si_\nu}$ is the left ideal $F\Si_n {\mathtt y}_\nu$  of $F\Si_n$. So, by Lemmas~\ref{LMSelfDual} and~\ref{L3.1f}, we get $A_P \circ\beta(\ind^{\Si_n}_{\Si_\nu}\sgn_{\Si_\nu})\cong\Mde_n {\mathtt y}_\nu$. 
Using the first statement of the lemma and the definition of the functor $A_P \circ \beta$, we see that $\Sdot^\nu\cong \beta(\ind^{\Si_n}_{\Si_\nu}\sgn_{\Si_\nu})$ is an extension of $\Mde_n {\mathtt y}_\nu$ and a module having no composition factors in common with the head (or equivalently by Lemma~\ref{LMSelfDual} the socle) of $\Mde_n$. Now (ii) follows on dualizing using (i) and Lemma~\ref{LDualInd}.
\end{proof}

\section{Schur algebras as endomorphism algebras}\label{SEA}
Recalling the theory of Schur algebras from Section~\ref{SSchurPrel}, 
fix an integer $h \geq n$ and let $S_{h,n}=S_{h,n,\O}$ denote classical the Schur algebra, always considered as a graded algebra in a {\em trivial} way, i.e. concentrated in degree zero. 

Recall the elements $\phi^u_{\mu,\la}$ from (\ref{EPhiULaMu}) and $g_{\mu,\la}^u$ from (\ref{EGUMuLa}). 
Our first connection between $R_{n\de}$ and the Schur algebra arises as follows:

\begin{Theorem} \label{3.4a}
Let $\O=F$. Then there is an algebra isomorphism 
$$
S_{h,n}\iso \End_{\ImS_n}\left(\textstyle\bigoplus_{\nu\in X(h,n)}\Ladot^\nu\right),
$$
under which the natural basis element $\phi^u_{\mu,\la}$ of $S_{h,n}$ maps to the endomorphism which is zero on the summands $\Ladot^\nu$ for $\nu\neq\mu$ and sends $\Ladot^\mu$ into $\Ladot^\la$ via the homomorphism induced by right multiplication in $\Mde_n$ by $g_{\mu,\la}^u$. 
\end{Theorem}
\begin{proof}
Let $A_P \circ\beta$ denote the equivalence of categories from Theorem~\ref{T3.1d}, for the projective $\ImS_n$-module $P=\Mde_n$, see Theorem~\ref{TSchub}. By Lemmas~\ref{LMSelfDual} and \ref{L3.1f}, we have
$$
\textstyle\bigoplus_{\nu\in X(h,n)}\Ladot^\nu\cong A_P\circ\beta\,\left(\textstyle\bigoplus_{\nu\in X(h,n)}\Perm^\nu\right). 
$$
So the endomorphism algebras of $\bigoplus_{\nu\in X(h,n)}\Ladot^\nu$ and $\bigoplus_{\nu\in X(h,n)}\Perm^\nu$ are isomorphic. The latter is $S_{h,n}$ by definition. It remains to check that the image of $\phi^u_{\mu,\la}$ under the functor $A_P \circ\beta$  is precisely the endomorphism described. This follows using Lemma~\ref{L3.1f} one more time.
\end{proof}

Note in the theorem above and in the similar results below that the algebra $S_{h,n}$ acts with degree zero homogeneous endomorphisms, so in particular we have 
$$
\textstyle\End_{\ImS_n}\left(\bigoplus_{\nu\in X(h,n)}\Ladot^\nu\right)=\operatorname{end}_{\ImS_n}\left(\bigoplus_{\nu\in X(h,n)}\Ladot^\nu\right).
$$

Recalling that $S_{h,n}$ can also be described as the endomorphism algebra 
$$\End_{F\Si_n}(\bigoplus_{\nu\in X(h,n)}\SPerm^\nu),$$ and the elements (\ref{ESUMuLa}), the same argument as in the proof of Theorem~\ref{3.4a} shows:

\begin{Proposition} \label{3.4b}
Let $\O=F$. Then there is an algebra isomorphism 
$$
\textstyle S_{h,n}\iso \End_{\ImS_n}\left(\bigoplus_{\nu\in X(h,n)}\Mde_n {\mathtt y}_\nu\right),
$$
under which the natural basis element $\phi^u_{\mu,\la}$ of $S_{h,n}$ maps to the endomorphism which is zero on the summands $\Mde_n {\mathtt y}_\nu$ for $\nu\neq\mu$ and sends $\Mde_n {\mathtt y}_\mu$ into $\Mde_n {\mathtt y}_\la$ via the homomorphism induced by right multiplication in $\Mde_n$ by	$s_{\mu,\la}^u$. 
\end{Proposition}

Our final endomorphism algebra result is as follows: 

\begin{Theorem} \label{3.4c}
Let $\O=F$. Then there is an algebra isomorphism 
$$
\textstyle S_{h,n}\iso \End_{\ImS_n}\left(\bigoplus_{\nu\in X(h,n)}\Zdot^\nu\right),
$$
under which the natural basis element $\phi^u_{\mu,\la}$ of $S_{h,n}$ maps to the endomorphism which is zero on the summands  $\Zdot^\nu$ for $\nu\neq\mu$ and sends $\Zdot^\mu$ into $\Zdot^\la$ via the homomorphism induced by right multiplication in $\Mde_n$ by	$s_{\mu,\la}^u$. 
\end{Theorem}
\begin{proof}
First, we check that the endomorphisms in the statement of the theorem are well-defined.
For this we need to see that, as submodules of $\Mde_n$, $\Zdot^\mu s_{\mu,\la}^u\subset \Zdot^\la$. To prove this, it suffices by definition of $\Zdot^\la$ to prove that $\Zdot^\mu s_{\mu,\la}^u(s_r-1)=0$ for all simple transpositions $s_r\in \Si_\la$. Right multiplication by $s_{\mu,\la}^u(s_r-1)$ yields an $R_{n\de}$-module homomorphism from $\Zdot^\mu$ to $\Mde_n$. Considering two cases: where $\mu=(1^n)$ and $\mu\neq (1^n)$, we see that the element $s_{\mu,\la}^u(s_r-1)$ always annihilates the submodule $\Mde_n {\mathtt y}_\mu$ of $\Zdot^\mu$. So in fact, $s_{\mu,\la}^u(s_r-1)$ must annihilate all of $\Zdot^\mu$ by Lemma~\ref{3.3d}(ii).

Let $S$ be the subalgebra of $\End_{R_{n\de}}\Big(\bigoplus_{\nu\in X(h,n)}\Zdot^\nu\Big)$ consisting of all endomorphisms which preserve the subspace $\bigoplus_{\nu\in X(h,n)}\Mde_n{\mathtt y}_\nu\subseteq \bigoplus_{\nu\in X(h,n)}\Zdot^\nu$. Restriction gives an algebra homomorphism 
$$
\textstyle S\to \End_{R_{n\de}}\left(\bigoplus_{\nu\in X(h,n)}\Mde_n{\mathtt y}_\nu\right),
$$
which is injective by Lemma~\ref{3.3d}(ii) and surjective by the previous paragraph and Proposition~\ref{3.4b}. This shows in particular that the endomorphisms of the module	$\bigoplus_{\nu\in X(h,n)}\Zdot^\nu$ defined in the statement of the theorem are linearly independent and span $S$. It remains to check using dimensions that $S$  equals all of $\End_{R_{n\de}}\Big(\bigoplus_{\nu\in X(h,n)}\Zdot^\nu\Big)$. On expanding the direct sums, this will follow if we can show that 
$$
\dim\Hom_{R_{n\de}}(\Zdot^\mu,\Zdot^\la) = \dim\Hom_{F\Si_n}(\SPerm^\la,\SPerm^\mu)  
$$
for all $\la,\mu\in X(h, k)$. We calculate using Lemmas~\ref{3.3d} and \ref{L3.1a}: 
\begin{align*}
\Hom_{R_{n\de}}(\Zdot^\mu,\Zdot^\la)&\cong \Hom_{R_{n\de}}(\Sdot^{\la^\op},\Sdot^{\mu^\op})
\\
&\cong\Hom_{R_{n\de}}(\beta(\SPerm^{\la^\op}),\beta(\SPerm^{\mu^\op}))
\\
&\cong\Hom_{F\Si_n}(\SPerm^{\la^\op},\al\circ\beta(\SPerm^{\mu^\op}))
\\
&\cong\Hom_{F\Si_n}(\SPerm^{\la^\op},\SPerm^{\mu^\op})
\\
&\cong\Hom_{F\Si_n}(\SPerm^{\la},\SPerm^{\mu}),
\end{align*}
as desired.
\end{proof}

\section{Projective generator for imaginary Schur algebra}\label{SIP}
Recalling the idempotents $\ga_{n,\de}$ and $\ga_{\nu,\de}$ from Section~\ref{SIGGR}, we  introduce the following temporary notation:  
\begin{equation}\label{EY}
\Yde_n:=R_{n\de}\ga_{n}\Mde_n,\quad \Yde_\nu:=R_{\nu,\de}\ga_{\nu}\Mde_\nu.
\end{equation}
Later it will turn out that $\Yde_n=\Zde_n$ and $\Yde_\nu=\Zde_\nu$. 
It easy to see that 
$$\Yde_\nu\cong \Yde_{n_1}\boxtimes\dots\boxtimes \Yde_{n_a}.$$

Recall for the next lemma that by definition, $\Zde_\nu$ is a submodule of $\Mde_\nu$.

\begin{Lemma} \label{L34d}
If $\O=F$, then:  
\begin{enumerate}
\item[{\rm (i)}] $\Yde_\nu$ is the image of any non-zero element of the one dimensional space $\Hom_{R_{\nu,\de}}(\Ga_{\nu},\Zde_\nu)$. Moreover, the latter $\Hom$-space is concentrated in degree zero. 
\item[{\rm (ii)}] $\HCR^n_\nu \Yde_n\cong \Yde_\nu$. 
\end{enumerate}
\end{Lemma}
\begin{proof}
(i) By Proposition~\ref{2.5e}, we have $\Hom_{R_{\nu,\de}}(\Ga_{\nu},\Mde_\nu)\cong F$, and the image of any non-zero map in this homomorphism space is contained in $\Zde_\nu$. 

(ii) 
By (i), $\Yde_n$ is a non-zero submodule of $\Mde_n$, so $\Hom_{R_{n\de}}(\Yde_n,\Mde_n)\neq 0$. Now, using Lemmas~\ref{3.2f}(i) and~\ref{LOp}, we get 
$$0\neq \Hom_{R_{n\de}}(\Yde_n,\Mde_n)\cong 
\Hom_{R_{n\de}}(\Yde_n,\HCI^n_{\nu^\op}\Mde_{\nu^\op})
\cong \Hom_{R_{\nu,\de}}(\HCR^n_\nu \Yde_n,\Mde_\nu).
$$
In particular, $\HCR^n_\nu \Yde_n\neq 0$. 

Now let $\theta : \Ga_n \to \Zde_n$ be a non-zero homomorphism. By (i), we have  $\im \,\theta=\Yde_n$. By Proposition~\ref{2.5e}, we have $\HCR^n_\nu\Ga_n\cong \Ga_\nu\oplus X$ for $X$ with $\Hom_{R_{\nu,\de}}(X,\Mde_\nu)=0$, and by Lemma~\ref{L3.3c}, we have $\HCR^n_\nu\Zde_n\cong \Zde_\nu$.
So, applying the exact functor $\HCR^n_\nu$ to $\theta$ and restricting to $\Ga_\nu$, 
we obtain a homomorphism $\bar\theta : \Ga_\nu \to \Zde_\nu$  with image $\HCR^n_\nu \Yde_n$, which is non-zero by the previous paragraph. By (i), the image of $\bar\theta$ is $\Yde_\nu$.
\end{proof}

By Proposition~\ref{2.5e}, we have that $Y_\nu=\im\, {\mathtt p}_\nu$ for a map
$
{\mathtt p}_\nu:\Ga_\nu\to M_\nu
$
which spans the one-dimensional space $\Hom_{R_{\nu,\de}}(\Ga_{\nu},\Mde_\nu)$. By functoriality, this map induces a map
$$
{\mathtt p}^\nu:\HCI_\nu^n\Ga_\nu\to \HCI_\nu^n Y_\nu.
$$

\begin{Lemma} \label{LGaY}
Any map $f:\HCI_\nu^n\Ga_\nu\to M_n$ factors through ${\mathtt p}^\nu$, i.e. there exists a unique map $\bar f:\HCI_\nu^n Y_\nu\to M_n$ such that $f=\bar f\circ {\mathtt p}^\nu$. In particular, for any submodule $N\subseteq M$, we have $\DIM \Hom_{R_{n\de}}(\HCI_\nu^n\Ga_\nu, N)=\DIM \Hom_{R_{n\de}}(\HCI_\nu^n Y_\nu, N)$. 
\end{Lemma}
\begin{proof}
By adjointness of $\HCI_\nu^n$ and $\HCR_\nu^n$, the map $f$ is functorially induced by a map $f_\nu:\Ga_\nu\to\HCR_\nu^n M_n$. By Lemma~\ref{3.2f}(ii), we have $\HCR_\nu^n M_n\cong M_\nu\otimes_{\O\Si_\nu}\O\Si_n$. By Proposition~\ref{2.5e}(ii), the map $f_\nu$ factors through ${\mathtt p}_\nu$, i.e. there exists $\bar f_\nu: Y_\nu\to \HCR_\nu^nM_n$ such that $f_\nu=\bar f_\nu\circ {\mathtt p}_\nu$. Now take $\bar f$ to be functorially induced by $\bar f_\nu$.  
\end{proof}

\begin{Lemma} \label{L3.4f}
Let $\O=F$. For any $\la,\mu\vDash n$, all of the spaces
\begin{align}
\label{E341}&\Hom_{R_{n\de}}(\HCI_\la^n\Yde_\la, \HCI_\mu^n\Yde_\mu),
\\
\label{E342}
&\Hom_{R_{n\de}}(\HCI_\la^n\Ga_\la, \HCI_\mu^n\Yde_\mu),
\\
\label{E343}
&\Hom_{R_{n\de}}(\HCI_\la^n\Ga_\la, \HCI_\mu^n\Zde_\mu)
\end{align}
have (graded) dimension equal to $|{}^\la\D^\mu_n|$.
\end{Lemma}
\begin{proof}
We consider only (\ref{E342}), since by Lemma~\ref{LGaY}, the result for (\ref{E341}) follows from that for (\ref{E342}), and the proof for (\ref{E343}) is similar to that for (\ref{E342}), using that $\HCR^n_\la\Zde_n\cong \Zde_\la$ 
according to Lemma~\ref{L3.3c}. 


By adjointness of $\HCI_\la^n$ and $\HCR^n_\la$, we have
$$
\Hom_{R_{n\de}}(\HCI_\la^n\Ga_\la, \HCI_\mu^n\Yde_\mu)\cong
\Hom_{R_{\la,\de}}(\Ga_\la, \HCR_\la^n\HCI_\mu^n\Yde_\mu).
$$
Since $\Ga_\la$ is projective, the Imaginary Mackey Theorem~\ref{TMackey}
show that 
$$
\DIM \Hom_{R_{\la,\de}}(\Ga_\la, \HCR_\la^n\HCI_\mu^n\Yde_\mu)=\sum_{x\in {}^\la\D^\mu_n} \DIM H_x,$$
where 
$$
H_x=\Hom_{R_{\la,\de}}(\Ga_\la, \HCI_{\la\cap x\mu}^{\la}{}^{x}(\HCR^\mu_{x^{-1}\la\cap\mu} \Yde_\mu))
$$
By Lemma~\ref{L34d}, we have $\HCR^\mu_{x^{-1}\la\cap\mu} \Yde_\mu\cong \Yde_{x^{-1}\la\cap\mu}$, so 
$$
H_x
\cong \Hom_{R_{\la,\de}}(\Ga_\la, \HCI_{\la\cap x\mu}^{\la}\Yde_{\la\cap x\mu}). 
$$
Note that the composition $\la\cap x\mu$ is a refinement of $\la$. Denote by $\nu$ the composition obtained by from $\la\cap x\mu$  by taking the parts of this refinement within each part $\la_m$ of $\la$ in the opposite order. 
By Lemmas~\ref{LOp} and \ref{L34d} and Proposition~\ref{2.5e}(i)(ii), we now have
\begin{align*}
H_x
 \cong\Hom_{R_{\nu,\de}}(\HCR^\la_{\nu}\Ga_\la, \Yde_{\nu})
  \cong\Hom_{R_{\nu,\de}}(\Ga_\nu, \Yde_{\nu})\cong F. 
\end{align*}
This completes the proof. 
\end{proof}

Recall that $X_+(n)$ can be identified with the set of the partitions of $n$. The following theorem is the main result of the chapter:

\begin{Theorem} \label{3.4g}
If $\O=F$ 
then: 
\begin{enumerate}
\item[{\rm (i)}] The submodules $\Zde_n$ and $\Yde_n$ of $\Mde_n$ coincide. So $\Zde_n$ can be characterized as the image of any non-zero homomorphism from $\Ga_n$ to $\Mde_n$.

\item[{\rm (ii)}] The number of non-isomorphic composition factors  of the $R_{n\de}$-module $\Mde_n$ is equal to $|X_+(n)|$.

\item[{\rm (iii)}] $\Zdot^\nu$ is a projective $\ImS_n$-module, for all $\nu\vDash n$. Moreover, for any $h\geq n$, we have that $\bigoplus_{\nu\in X(h,n)}\Zdot^\nu$ is a projective generator for $\ImS_n$. 
\end{enumerate}
\end{Theorem}
\begin{proof}
Fix some $h\geq n$ and set
\begin{align*}
\Zde&:=\textstyle\bigoplus_{\nu\in X(h,n)}\HCI_\nu^n\Zdot_\nu=\bigoplus_{\nu\in X(h,n)}\Zdot^\nu, 
\\ 
\Yde&:=\textstyle\bigoplus_{\nu\in X(h,n)}\HCI_\nu^n\Yde_\nu,
\\ 
\Ga&:=\textstyle\bigoplus_{\nu\in X(h,n)}\HCI_\nu^n\Ga_\nu.
\end{align*}
As $\Yde_\nu$ is a non-zero submodule of $\Zde_\nu$, it contains the simple socle $\Mde_\nu {\mathtt y}_\nu$ of $\Zde_\nu$ as a submodule, see Lemma~\ref{LSZ}(iii). 
Applying $\HCI_\nu^n$ to the inclusions $\Mde_\nu {\mathtt y}_\nu\subseteq \Yde_\nu\subseteq \Zde_\nu$, we see that
$
\Mde_n {\mathtt y}_\nu\subseteq \HCI_\nu^n\Yde_\nu\subseteq \HCI_\nu^n\Zde_\nu
$
as naturally embedded submodules of $\Mde_n$. So 
$$\textstyle\bigoplus_{\nu\in X(h,n)}\Mde_n {\mathtt y}_\nu\subseteq \Yde\subseteq \Zde.$$

Also observe that $\Yde$ is a quotient of $\Ga$, since each $\Yde_\nu$ is a quotient of $\Ga_\nu$.
By Lemmas~\ref{L3.4f} and~\ref{1.2a}, we have
\begin{align*}
\dim\Hom_{R_{n\de}}(\Yde,\Yde)=&\dim \Hom_{R_{n\de}}(\Ga,\Yde)=\dim \Hom_{R_{n\de}}(\Ga,Z)
\\=&\sum_{\la,\mu\in X(h,n)}|{}^\la\D^\mu_n|=\dim S_{h,n}.
\end{align*}

Since $\Ga$ is projective it contains the projective cover $P$ of $\Yde$ as a summand. Now the equality 
$\dim\Hom_{R_{n\de}}(\Yde,\Yde) = \dim\Hom_{R_{n\de}}(\Ga,\Yde)$ implies that  
$$\dim\Hom_{R_{n\de}}(\Yde,\Yde)=\dim\Hom_{R_{n\de}}(P,\Yde).$$ 
This verifies the condition in Lemma~\ref{LSchub}, so $\Yde$ is a projective $R_{n\de}/ \Ann_{R_{n\de}} (\Yde )$-module.

As $\Hom_{R_{n\de}}(\Ga, \Yde)$ and $\Hom_{R_{n\de}}(\Ga, \Zde)$ have the same dimension, every $R_{n\de}$-homomorphism from $\Ga$ to $\Zde$ has image lying in $\Yde$ . So since $\Yde$ is a quotient of $\Ga$, we can describe $\Yde$ alternatively as the subspace of $\Zde$ spanned by the images of all $R_{n\de}$-homomorphisms from $\Ga$ to $\Zde$. This description implies that $\Yde$ is stable under all $R_{n\de}$-endomorphisms of $\Zde$. So, restriction gives a well-defined map
$\End_{R_{n\de}}(\Zde) \to \End_{R_{n\de}}(\Yde ).
$
It is injective since we know from Theorem~\ref{3.4c} and Proposition~\ref{3.4b}  that the homomorphism 
$\End_{R_{n\de}}(\Zde) \to \End_{R_{n\de}}\left(\textstyle\bigoplus_{\nu\in X(h,n)} M_n{\mathtt y}_\nu\right)$ 
induced by restriction is injective. Since $\End_{R_{n\de}}(\Zde)\cong S_{h,n}$  and $\End_{R_{n\de}}(\Yde )$ has the same dimension as $S_{h,n}$, we deduce that $\End_{R_{n\de}}(\Yde )\cong S_{h,n}$.

For $h \geq n$, the number of irreducible representations of $S_{h,n}$ is equal to $|X_+(n)|$. Combining what we have already proved with Fitting's lemma \cite[1.4]{Land}, we deduce that $\Yde$ has the same amount of non-isomorphic irreducible modules appearing in its head. It follows in particular that $\Mde$ has at least $|X_+(n)|$ non-isomorphic composition factors. Since $n$ and $i$ are arbitrary, we can now apply Corollary~\ref{CColorUB} to conclude that $\Mde$ has exactly $|X_+(n)|$ non-isomorphic composition factors, and we have proved (ii). 

Since $\Yde$ is a direct sum of submodules of $\Mde_n$, 
the assumption on the number of composition factors 
now implies that every irreducible constituent of $\Mde_n$ appears in the head of $\Yde$. Hence every irreducible constituent of $\Mde_n$ appears in the head of the projective $R_{n\de}$-module $\Ga$. Now we know that every homomorphism from $\Ga$ to $\Zde$ has image lying in $\Yde$, while every composition factor of $\Zde/\Yde$ appears in the head of the projective module $\Ga$. This shows $\Zde = \Yde$, and we have proved (i). 

Further, observe that $\Mde_n\cong \Zdot^{(1^n)}$ is a summand of $\Yde = \Zde$, hence $\Ann_{R_{n\de}}(\Yde) = \Ann_{R_{n\de}}(\Mde)$.
In other words, $R_{n\de}/\Ann_{R_{n\de}}(\Yde) = \ImS_n$. We have already shown that $\Yde$ is a projective $R_{n\de}/\Ann_{R_{n\de}}(\Yde)$-module, which means that $\Zde$ and all its summands are projective $\ImS_n$-modules. Taking $h$ large enough, this shows in particular that $\HCI_\nu^n\Zde_\nu = \Zdot^\nu$ is projective for each $\nu\vDash n$.

Finally, to show that every irreducible $\ImS_n$-module appears in the head of $\Zde$, note that $\Mde_n$ is a faithful $\ImS_n$-module, and so every irreducible $\ImS_n$-module appears as some composition factor of $\Mde_n$, and we have seen that a copy of every composition factor of $\Mde_n$ does appear in the head of $\Yde=\Zde$.
\end{proof}

\begin{Corollary} \label{CHCIC}
The functor $\HCI_\nu^n$ sends $\ImS_\nu$-modules to $\ImS_n$-modules and the functor $\HCR_\nu^n$ sends $\ImS_n$-modules to $\ImS_\nu$-modules.
\end{Corollary}
\begin{proof}
We prove the first statement, the second statement is proved similarly. 
Since $\HCI_\nu^n$ is exact, it suffices to check the first statement on projective $\ImS_\nu$-modules. In turn, since
according to the parabolic analogue of Theorem~\ref{3.4g}, every indecomposable projective $\ImS_\nu$-module is a submodule	of $M_\nu$, we just	need to check that $\HCI_\nu^n	M_\nu$ is	a $\ImS_n$-module. But this is clear since $\HCI_\nu^n	M_\nu\cong M_n$ by Lemma~\ref{3.2f}. 
\end{proof}

\chapter{Morita equaivalence}\label{SSIM}
In this chapter we begin to exploit the Morita context provided by Theorem~\ref{3.4g}. Throughout the chapter, except in \S\ref{SBaseChange}, we assume that $\O=F$. 

\section{Morita equivalence functors}
Let $h\geq n$, and $\Zde=\bigoplus_{\nu\in X(h,n)}\Zdot^\nu$. We always regard $\Zde$ as a $(\ImS_n,S_{h,n})$-bimodule, with $S_{h,n}$ acting as in Theorem~\ref{3.4c}. Define the functors
\begin{align*}
\al_{h,n}&: \mod{\ImS_n}\to\mod{S_{h,n}},\quad V\mapsto \Hom_{\ImS_n}(Z,V)
\\
\be_{h,n}&:\mod{S_{h,n}}\to \mod{\ImS_n},\quad W\mapsto Z\otimes_{S_{h,n}}W.
\end{align*}

\begin{Proposition}
The functors $\al_{h,n}$ and $\beta_{h,n}$ are mutually inverse equivalences of categories between $\mod{\ImS_n}$ and $\mod{S_{h,n}}$. 
\end{Proposition}
\begin{proof}
This follows from the fact that $Z$ is a projective generator for $\ImS_n$ proved in Theorem~\ref{3.4g}(iii). 
\end{proof}

Recall from Chapter~\ref{SPrel} that $S_{h,n}$ is a quasi-hereditary algebra with weight poset $X_+(h, n)$ partially ordered by the dominance order $\leq$. We can identify $X_+(h, n)$ with the set $X_+(n)$ of partitions of $n$, since $h\geq n$. 
Also, for $\la\in X_+(n)$, the algebra $S_{h,n}$ has the irreducible module $L_h(\la)$, the standard module $\De_h(\la)$  and the costandard module $\nabla_h(\la)$. 
For all $\la \in X_+(n)$, define the (graded) $\ImS_n$-modules: 
\begin{eqnarray}
\Lde(\la) &:= \be_{h,n}(L_h(\la)),
\label{3.5.3}
\\
\Dede(\la) &:= \be_{h,n}(\Delta_h(\la)),
\label{3.5.4}
\\
\nade(\la) &:= \be_{h,n}(\nabla_h(\la)).
\label{3.5.5}
\end{eqnarray}

Since $\beta_{h,n}$ is a Morita equivalence, the imaginary Schur algebra $\ImS_n$ is a quasi-hereditary algebra with weight poset $X_+(n)$ partially ordered by $\leq$. 
Moreover, 
$\{\Lde(\la)\}$, $\{\Dede(\la)\}$ and $\{\nade(\la)\}$ for all $\la\in X_+(n)$ give the irreducible, standard and costandard $\ImS_n$-modules. 
The following 
facts follow 
from  Morita equivalence. 

\begin{Lemma} \label{3.5b} 
Let $\la,\mu\in X_+(n)$. Then:
\begin{enumerate} 
\item[{\rm (i)}] $\Dede(\la)$ has simple head isomorphic to $\Lde(\la)$, and all other composition factors are of the form $\Lde(\nu)$ for $\nu<\la$.
\item[{\rm (ii)}]  $[\Dede(\la) : \Lde(\mu)] = [\De_h(\la) : L_h(\mu)]$. 
\end{enumerate}
\end{Lemma}

We next explain why the definitions (\ref{3.5.3})-(\ref{3.5.5}) are independent of the choice of $h \geq n$. Take $h \geq l \geq n$ and, using Lemma~\ref{LTruncSchur}, identify $S_{l,n}$ with the subalgebra $eS_{h,n}e$ of $S_{h,n}$, where $e$ is the idempotent of (\ref{1.5.1}). 
Recall an equivalence of categories from Lemma~\ref{1.5a}:
$$\infl_{S_{l,n}}^{S_{h,n}}	: \mod{S_{l,n}} \to \mod{S_{h,n}}: M\mapsto S_{h,n}e \otimes_{eS_{h,n}e} M.
$$

\begin{Lemma} \label{3.5c}
The functors $\be_{h,n} \circ \infl_{S_{l,n}}^{S_{h,n}}$ and $\be_{l,n}$ from $\mod{S_{l,n}}$ to $\mod{\ImS_n}$ are isomorphic.
\end{Lemma}
\begin{proof}
The module $\bigoplus_{\la\in X(l,n)}\Zde^\la$ is precisely the $(\ImS_n,S_{l,n})$-subbimodule $Ze$ of $Z$. So $\be_{l,n}$ is isomorphic to $Ze\otimes_{eS_{h,n}e} ?$. Now we have the functorial isomorphisms 
$$
Z \otimes_{S_{h,n}} (S_{h,n}e \otimes_{eS_{h,n}e} M) \cong (Z \otimes_{S_{h,n}} S_{h,n}e) \otimes_{eS_{h,n}e} M \cong Ze \otimes_{eS_{h,n}e} M
$$ 
for any $M \in \mod{S_{l,n}}$. 
\end{proof}

By Lemmas~\ref{1.5a} and \ref{1.5b}, we have that $L_h(\la) \cong \infl^{S_{h,n}}_{S_{l,n}} L_l(\la)$. Hence Lemma~\ref{3.5c} yields $\beta_{l,n}(L_l(\la)) \cong\be_{h,n}(L_h(\la))$ as $\ImS_n$-modules. So, the definition (\ref{3.5.3}) is independent of the choice of $h\geq n$, and a similar argument gives independence of $h$ for (\ref{3.5.4}) and (\ref{3.5.5}).

In conclusion of this section, we make a small detour to mixed imaginary tensor spaces of \S\ref{SImRep}. For $\bn=(n_1,\dots,n_l)\in\Z_{\geq 0}^l$, the corresponding mixed tensor space is defined as 
$$M_\bn:=M_{n_1,1}\circ\dots\circ M_{n_l,l},$$ 
where $M_{n_i,i}$ is a colored space of color $i$ for each $i\in I'$, which can be considered as a module over the (color $i$) imaginary Schur algebra 
$\ImS_{n_i,i}:=R_{n_i\de}/\operatorname{Ann}_{R_{n_i\de}}(M_{n_i,i})$. 
Let $n=n_1+\dots+n_l$, and define the {\em mixed imaginary Schur algebra}
$$\ImS_\bn:=R_{n\de}/\operatorname{Ann}_{R_{n\de}}(M_\bn). 
$$
Moreover, for $h_i\geq n_i$ and $\nu\in X(n_i,h_i)$ we have defined modules
$Z^\nu_i$ (previously denoted for brevity $Z^\nu$ since we had $i\in I'$ fixed), and 
$$
Z(i,n_i)=\bigoplus_{\nu\in X(n_i,n_i)} Z^\nu_i\qquad(i\in I').
$$
We have functors $\al_{h_i,n_i,i}: = \Hom_{\ImS_{n_i,i}}(Z(n_i,i),?)$ and $\be_{h_i,n_i,i}=Z(i,n_i)\otimes_{S_{h_i,n_i}}?$. 
Set 
$$
Z_\bn=Z(1,n_1)\circ\dots\circ Z(l,n_l).
$$
The following result strengthens Theorem 2 from the Introduction. 

\begin{Proposition} \label{PT2'}
We have
\begin{enumerate}
\item[{\rm (i)}] $Z_\bn$ is a projective generator for $\ImS_\bn$. 
\item[{\rm (ii)}] $\End_{\ImS_\bn}(Z_\bn)\cong \End_{\ImS_{n_1,1}}(Z(n_1,1))\otimes\dots\otimes \End_{\ImS_{n_l,l}}(Z(n_l,l))\cong S_{\bh,\bn}:=
S_{h_1,n_1}\otimes\dots\otimes S_{h_l,n_l}$. 
\item[{\rm (iii)}] The functors $\al_{\bh,\bn}: = \Hom_{\ImS_\bn}(Z_\bn,?):\mod{\ImS_\bn}\to\mod{S_{\bh,\bn}}$ and $\be_{\bh,\bn}= Z_\bn\otimes_{S_{\bh,\bn}} ?:\mod{S_{\bh,\bn}}\to \mod{\ImS_\bn},$ are mutually inverse equivalences of categories between $\mod{\ImS_\bn}$ and $\mod{S_{\bh,\bn}}$. 
\item[{\rm (iv)}] There is a functorial isomorphism
$$
\be_{\bh,\bn}(W_1\otimes \dots\otimes W_l)\cong  \be_{h_1,n_1,1}(W_1)\circ\dots\circ  \be_{h_l,n_l,l}(W_l)
$$
for $W_1\in\mod{S_{h_1,n_1}},\dots,W_l\in\mod{S_{h_l,n_l}} $. 
\end{enumerate}
\end{Proposition}
\begin{proof}
Part (i) follows from Theorem~\ref{3.4g}(iii). Part (ii) follows from Theorem~\ref{3.4c} and Mackey Theorem. Part (iii) follows from part (i). Part (iv) follows from the definitions and transitivity of induction.
\end{proof}


\section{Induction and Morita equivalence}\label{SIndMorita}
In this section we prove a key result that our Morita equivalence `intertwines' imaginary induction and tensor products for usual Schur algebras. 
We fix an integer $h\geq n$, and a composition $\nu=(n_1,\dots, n_a)\vDash n$. 
Recall the Morita equivalence $\be_{h,n}:\mod{S_{h,n}}\to \mod{\ImS_n}$.

\begin{Theorem} \label{4.2a}
We have an isomorphism of functors from $\mod{S_{h,n_1}}\times\dots\times \mod{S_{h,n_a}}$ to 
$\mod{\ImS_n}$: 
\begin{align*}
\HCI_\nu^n (\be_{h,n_1}?\boxtimes\dots\boxtimes \be_{h,n_a} ?) \cong 
\be_{h,n}(?\otimes\dots\otimes  ?) .
\end{align*}
\end{Theorem}
\begin{proof}
Choose $\chi = (h_1,...,h_a) \vDash h$ with $h_k \geq n_k$ for $k=1,\dots,a$, and denote $S_{\chi,\nu}:=S_{h_1,n_1} \otimes\dots\otimes  S_{h_a,n_a}.$ 
Write $X(\chi,\nu)$ for the set of all compositions $\ga = (g_1,\dots,g_h) \in X(h,n)$ such that 
$$
\ga_1\hspace{-1mm}:= (g_1,\dots,g_{h_1}), \ga_2\hspace{-1mm}:= (g_{h_1+1},\dots,g_{h_1+h_2}),  \dots, 
 \ga_a\hspace{-1mm} := (g_{h_1+\dots+h_{a-1}+1},\dots,g_h)
$$ 
satisfy $\ga_k \in X(h_k,n_k)$  for each $k = 1,\dots,a$. 

Consider the set of triples:
$$
\Omega = \{ (\ga ,\de ,u) \mid \ga , \de \in X (\chi ,\nu),\ u \in {}^\ga \D_\nu^\de\} .
$$
For a triple $(\ga ,\de ,u)\in\Omega$, we have $\ga_k,\de_k \in X(h_k,n_k)$ for each $k = 1,\dots,a$, and $u = (u_1,\dots,u_a) \in \Si_\nu = \Si_{n_1} \times\dots\times\Si_{n_a}$  with each $u_k \in {}^{\ga_k}\D^{\de_k}_{n_k}$. So we have
the element
$$\bar\phi^u_{\ga,\de}:= \phi^{u_1}_{\ga_1,\de_1}\otimes\dots\otimes \phi^{u_a}_{\ga_a,\de_a}\in S_{\chi,\nu}.$$ 
Then $\{\bar\phi^u_{\ga,\de} \mid (\ga ,\de ,u)\in\Omega\}$ is a basis for $S_{\chi,\nu}$.

Recall the $(\ImS_n,S_{h,n})$-bimodule 
$\Zde=\textstyle \bigoplus_{\la\in X(h,n)}I_\la^n \Zde_\la= \bigoplus_{\la\in X(h,n)} \Zde^\la,$ and define also the $(\ImS_\nu,S_{\chi,\nu})$-bimodule
$$
\Zde^{\chi,\nu}:=\textstyle\bigoplus_{\la\in X(\chi,\nu)}I_\la^\nu \Zde_\la\cong\left(\bigoplus_{\la_1\in X(h_1,n_1)} \Zde^{\la_1}\right)\boxtimes\dots\boxtimes \left(\bigoplus_{\la_a\in X(h_a,n_a)} \Zde^{\la_a}\right).
$$
Then $\HCI_\nu^n \Zde^{\chi,\nu}$ is an $(\ImS_n,S_{\chi,\nu})$-bimodule in a natural way. Moreover, by transitivity of induction, we have $\HCI_\nu^n \Zde^{\chi,\nu}\cong \bigoplus_{\la\in X(\chi,\nu)} \Zde^\la$, 
so $\HCI_\nu^n \Zde^{\chi,\nu}$ can be identified with the summand $Ze_\nu$ of the $(\ImS_n,S_{h,n})$-bimodule $\Zde$, where $e_\nu$ is the idempotent
$$
e_\nu:= \sum_{\la\in X(\chi,\nu)}e(\la)\in S_{h,n}. 
$$

Identifying $e_\nu S_{h,n}e_\nu$ with $\End_{\ImS_n}(\Zde e_\nu)$, we obtain an algebra embedding of $S_{\chi,\nu}$ into $e_\nu S_{h,n}e_\nu$. By definition of the actions of $S_{\chi,\nu}$ and $e_\nu S_{h,n}e_\nu$ on $\Zde e_\nu$ and Lemma~\ref{3.2f}(i), this embedding maps the basis element $\bar\phi^u_{\ga,\de}\in S_{\chi,\nu}$ to $\phi^u_{\ga,\de}\in e_\nu S_{h,n}e_\nu$, for all $(\ga ,\de ,u)\in\Omega$. In other words:

\vspace{2 mm}
\noindent
{\bf Claim 1.} {\em 
Identifying $S_{\chi,\nu}$ with a subalgebra $e_\nu S_{h,n}e_\nu$ via the map $\bar\phi^u_{\ga,\de}\mapsto\phi^u_{\ga,\de}$, the $(\ImS_n,S_{\chi,\nu})$-bimodule $\HCI_\nu^n \Zde^{\chi,\nu}$ is isomorphic to $\Zde e_\nu$, regarding the latter as a $(\ImS_n,S_{\chi,\nu})$-bimodule by restricting the natural action of $e_\nu S_{h,n}e_\nu$ to $S_{\chi,\nu}$. 
}
\vspace{2mm}

Now let $S_{\chi,n}$ be the Levi subalgebra of $S_{h,n}$ as in (\ref{ELevi}). Then $e_\nu$ is the central idempotent of $S_{\chi,n}$ such that $e_\nu S_{\chi,n}e_\nu\cong S_{\chi,\nu}$. So in fact, the embedding of $S_{\chi,\nu}$ into $e_\nu S_{\chi,n}e_\nu$ from Claim 1 identifies $S_{\chi,\nu}$ with the summand $e_\nu S_{\chi,n}e_\nu$ of $S_{\chi,n}$. Making this identification, define the functor
$$
I = S_{h,n}e_\nu\otimes_{e_\nu S_{\chi,n}e_\nu} ?: \mod{S_{\chi,\nu}} \to \mod{S_{h,n}}.
$$
Using associativity of tensor product, the functor $I$ can be thought of as the composite of the natural inflation functor $S_{\chi,n}e_\nu\otimes_{e_\nu S_{\chi,n}e_\nu} ?: \mod{S_{\chi,\nu}} \to \mod{S_{\chi,n}}$ followed by ordinary induction 
$\ind^{S_{h,n}}_{S_{\chi,n}} : \mod{S_{\chi,n}} \to \mod{S_{h,n}}$ as defined in Section~\ref{SSIRSA}. In view of this, the following fact follows immediately from Lemma~\ref{1.5d}:

\vspace{2 mm}
\noindent
{\bf Claim 2.} {\em 
The following functors from $\mod{S_{h_1,n_1}}\times\dots\times\mod{S_{h_a,n_a}}$ to $\mod{S_{h,n}}$ are isomorphic:
\begin{equation*}
I(? \boxtimes\dots\boxtimes\,?) 
\qquad\text{and}\qquad
\infl_{S_{h_1,n_1}}^{S_{h,n_1}}?\otimes\dots\otimes \infl_{S_{h_a,n_a}}^{S_{h,n_a}}? .
\end{equation*}
}
\vspace{2mm}

Now note that from Claim 1 and associativity of tensor product we have  the natural isomorphisms
\begin{align*}
\HCI_\nu^n (\Zde^{\chi,\nu}\otimes_{S_{\chi,\nu}} N)
&\cong 
(\HCI_\nu^n \Zde^{\chi,\nu})\otimes_{S_{\chi,\nu}} N
\cong Ze_\nu\otimes_{S_{\chi,\nu}} N
\\
&\cong Ze_\nu\otimes_{e_\nu S_{\chi,n}e_\nu} N
\cong Z\otimes_{S_{h,n}}S_{h,n}e_\nu\otimes_{e_\nu S_{\chi,n}e_\nu} N,
\end{align*}
which is $\beta_{h,n}(I(N))$.
Thus we have proved:

\vspace{2 mm}
\noindent
{\bf Claim 3.} {\em 
The functors $\HCI_\nu^n\circ (\Zde^{\chi,\nu}\otimes_{S_{\chi,\nu}} ?)$ and $\beta_{h,n} \circ I$ from $\mod{S_{\chi,\nu}}$ to $\mod{\ImS_n}$ are isomorphic.
}
\vspace{2mm}

We have the isomorphism of functors from $\mod{S_{h_1,n_1}}\times\dots\times \mod{S_{h_a,n_a}}$ to $\mod{\ImS_\nu}$: 
$$
\beta_{h_1,n_1} ?\boxtimes\dots\boxtimes \beta_{h_a,n_a} ?\cong Z_{\chi,\nu}\otimes_{S_{\chi,\nu}}(?\boxtimes\dots\boxtimes \,?).
$$
So, in view of Claims 2 and 3, we have 

\vspace{2 mm}
\noindent
{\bf Claim 4.} {\em 
There is an isomorphism of functors 
$$
\HCI_\nu^n(\beta_{h_1,n_1} ?\boxtimes\dots\boxtimes \beta_{h_a,n_a} ?)\cong \be_{h,n}(\infl_{S_{h_1,n_1}}^{S_{h,n_1}}?\,\otimes\dots\otimes\, \infl_{S_{h_a,n_a}}^{S_{h,n_a}}?)
$$
from $\mod{S_{h_1,n_1}}\times\dots\times \mod{S_{h_a,n_a}}$ to $\mod{\ImS_n}$.
}
\vspace{2mm}

Finally, by Lemma~\ref{1.5a}, the functors 
$$\infl_{S_{h_k,n_k}}^{S_{h,n_k}}: \mod{S_{h_k,n_k}} \to \mod{S_{h,n_k}}\qquad(1\leq k\leq a)
$$ are equivalences of categories. By Lemma~\ref{3.5c}, the functors $\beta_{h,n_k}\circ \infl_{S_{h_k,n_k}}^{S_{h,n_k}}$ and $\beta_{h,n_k}$ are isomorphic. The theorem follows on combining these statements and Claim 4.
\end{proof}

As a first application, we get the commutativity of induction product on the category of imaginary representations:

\begin{Corollary} \label{CCircComm}
Let $M\in\mod{\ImS_m}$ and $N\in\mod{\ImS_n}$. Then $M\circ N\cong N\circ M$. 
\end{Corollary}
\begin{proof}
For sufficiently large $h$ we have $M=\be_{h,m}(V)$, $N=\be_{h,n}(W)$, 
$M\circ N=\HCI^{m+n}_{(m,n)}(M\boxtimes N)$, and $N\circ M=\HCI^{m+n}_{(n,m)}(N\boxtimes M)$. Now the result follows from $V\otimes W\cong W\otimes V$ and the theorem. 
\end{proof} 

As a second application we establish a version of Steinberg Tensor Product Theorem. 
Let $p=\cha F>0$ and $\la\vdash n$. Considered $\la$ as an element of $X_+(n,n)$. Recall from Section~\ref{SSParComp}, that there exists a unique $p$-adic expansion 
$$
\la=\la(0)+p\la(1)+p^2\la(2)\dots
$$
such that the partitions $\la(0)\vdash m_0,\la(1)\vdash m_1,\la(2)\vdash m_2,\dots$ are all $p$-restricted. With this notation we have: 

\begin{Theorem} 
{\bf (Imaginary Steinberg Tensor Product Theorem)} 
Let 
$n_r:=p^rm_r$ for $r=0,1,2,\dots$, and consider the composition  $$\nu=(n_0,n_1,n_2,\dots)\vDash n.$$ 
Then 
$$
L(\la)=\HCI_{\nu}^n\big(L(\la(0))\boxtimes L(p\la(1))\boxtimes L(p^2\la(2))\boxtimes\dots\big).
$$ 
\end{Theorem}
\begin{proof}
This comes from Theorem~\ref{4.2a} and Lemma~\ref{1.3e}. 
\end{proof} 

As a third application, we prove that imaginary induction and restriction respect standard and costandard filtrations. 
A filtration $0=M_0 \subset M_1 \subset\dots\subset M_b =M$ 
of an $\ImS_n$-module $M$ is called {\em standard} (resp. {\em costandard})
if for each $k = 1, \dots , b$, the quotient $M_k /M_{k-1}$ is isomorphic to $\De(\la)$  (resp. $\nabla(\la)$) for some $\la\vdash n$ (depending on $k$). Similarly, a filtration $0=M_0 \subset M_1 \subset\dots\subset M_b =M$ 
of an $\ImS_{\nu}$-module $M$ is called {\em standard} (resp. {\em costandard})
if for each $k = 1, \dots , b$, the quotient $M_k /M_{k-1}$ is isomorphic to $\De(\la_1)\boxtimes\dots\boxtimes\De(\la_a)$  (resp. $\nabla(\la)(\la_1)\boxtimes\dots\boxtimes\nabla(\la_a)$) for some $\la_1\vdash n_1,\dots \la_a\vdash n_a$ (depending on $k$).

\begin{Theorem} \label{4.2f} 
We have
\begin{enumerate}
\item[{\rm (i)}] The functor $\HCI_\nu^n$	sends $\ImS_\nu$-modules with standard (resp. costandard) filtrations to $\ImS_n$-modules with standard (resp. costandard) filtrations.
\item[{\rm (ii)}] The functor $\HCR_\nu^n$	sends $\ImS_n$-modules with standard (resp. costandard) filtrations to $\ImS_\nu$-modules with standard (resp. costandard) filtrations. 
\end{enumerate}
\end{Theorem}
\begin{proof}
(i) It suffices to check that  
$\HCI_\nu^n(\De(\la_1)\boxtimes\dots\boxtimes\De(\la_a))$ 
has a  standard filtration, for arbitrary $\la_1\vdash n_1,\dots, \la_a\vdash n_a$. By  Theorem~\ref{4.2a},
$$
\HCI_\nu^n(\De(\la_1)\boxtimes\dots\boxtimes\De(\la_a))\cong \be_{h,n}(\De_h(\la_1)\otimes\dots\otimes\De_h(\la_a)). 
$$
So the result follows since $\De_h(\la_1)\otimes\dots\otimes\De_h(\la_a)$ has a standard filtration as an $S_{h,n}$-module by Lemma~\ref{1.3c}.
This proves (i) in the case of standard filtrations; the result for costandard 
filtrations is proved similarly. 

(ii) We prove (ii) in the case of costandard filtrations; the analogous result for standard 
filtrations follows by dualizing. Take $N \in\mod{\ImS_n}$ with a costandard filtration. Using the cohomological criterion for costandard filtrations \cite[A2.2(iii)]{Donkin}, we need to show that $\EXT^1_{\ImS_\nu}(M,\HCR_\nu^n N)=0$ for all $M \in\mod{\ImS_\nu}$ with a standard filtration. For such $M$, by (i) and the cohomological criterion for costandard filtrations, we have $\EXT^1_{\ImS_\nu}(\HCI_\nu^n M, N)=0$. So the result follows from the following:

\vspace{2 mm}
\noindent
{\bf Claim.} {\em 
For $M \in\mod{\ImS_\nu}$ and $N \in\mod{\ImS_n}$, we have 
$
\EXT^1_{\ImS_\nu}(M,\HCR_\nu^n N)\cong \EXT^1_{\ImS_\nu}(\HCI_\nu^n M, N).
$
}

To prove the claim, the adjoint functor property gives us an isomorphism
of functors $\Hom_{\ImS_\nu} (M, ?) \circ \HCR_\nu^n	\cong \Hom_{\ImS_n}	(\HCI_\nu^n M, ?).$ 
Since $\HCR_\nu^n$ is exact and sends injectives to injectives (being adjoint to the exact functor $\HCI_\nu^n$), an application of \cite[I.4.1(3)]{Jantzen} completes the proof of the claim. 
\end{proof}

\begin{Corollary} \label{4.2h} 
Let $n=(n_1,\dots,n_a)\vDash n$, $\la\vdash n$, and $\la_1 \vdash n_1, \dots , \la_a\vdash n_a$.  Then both of
\begin{enumerate}
\item[{\rm (i)}] the multiplicity of $\De(\la)$ in a standard filtration of $\HCI_\nu^n(\De(\la_1) \boxtimes\dots\boxtimes \De(\la_a))$,
\item[{\rm (ii)}] the multiplicity of $\De(\la_1) \boxtimes\dots\boxtimes \De(\la_a)$ in a standard filtration of $\HCR^n_\nu\De(\la)$ 
\end{enumerate}
are given by the Littlewood-Richardson rule.
\end{Corollary}
\begin{proof}
The modules in (i) and (ii) have standard filtrations by Theorem~\ref{4.2f}. Now (i) follows from Theorem~\ref{4.2a} and the classical fact about tensor product multiplicities over the Schur algebra, and (ii) follows from (i) and adjointness, together with the usual properties of standard and costandard filtrations.
\end{proof}

\section{Alternative definitions of standard modules}
Our goal now is to give two alternative definitions of the standard module $\Dede(\la)$  without reference to the Schur algebra and Morita equivalence. Recall from (\ref{1.3.3}) and (\ref{1.3.4}) the modules $Z^\nu(V_h)$ and $\La^\nu(V_h)$ for the classical Schur algebra $S_{h,n}$.

\begin{Lemma} \label{3.5d} 
For $\nu\vDash n$, we have $\Zde^\nu\cong \be_{h,n}(Z^\nu(V_h))$ and $\Lade^\nu \cong \beta_{h,n}(\La^\nu(V_h))$.
\end{Lemma}
\begin{proof}
By Lemma~\ref{1.3b}(i), we have $Z^\nu(V_h)\cong S_{h,n}e(\nu)$. So, by  Lemma~\ref{L3.1f}, we get  
$\beta_{h,n}(S_{h,n}e(\nu)) \cong \Zde e(\nu)$, which is precisely the summand $\Zde^\nu$ of $\Zde$ by the definition of the action  from Theorem~\ref{3.4c}. 

For the second statement, using the embedding $\kappa$ from Lemma~\ref{1.2b} and Lemma~\ref{1.3b}(ii), we have $\La^\nu(V_h)\cong S_{h,n}\kappa({\mathtt y}_\nu)$. So, by Lemma~\ref{L3.1f}, we get  
$\beta_{h,n}(\La^\nu(V_h)) \cong \Zde\kappa({\mathtt y}_\nu)$. By definition of  $\kappa$, together with Theorem~\ref{3.4c}, we have $\Zde\kappa({\mathtt y}_\nu)=M_n\sgn({\mathtt y}_\nu) =M_n {\mathtt x}_\nu =\Lade^\nu$, as required.
\end{proof}

In view of Lemma~\ref{LMSelfDual}, the antiautomorphism $\tau$ of $R_{n\de}$ factors through to give a (homogeneous) antiautomorphism 
$$\tau:\ImS_n\to \ImS_n,$$ 
which leads to the notion of contravariant duality $\circledast$ on $\mod{\ImS_n}$.  
We have a (not necessarily degree zero) isomorphism $L(\la)^\circledast\cong L(\la)$ for each $\la\in X_+(n)$, since this is true even	 as $R_{n\de}$-modules. We now prove a stronger result:

\begin{Lemma} \label{LL(la)SelfDual}
For all $\la\in X_+(n)$ we have $\Lde(\la)^\circledast \cong \Lde(\la)$ and $\Dede(\la)^\circledast \cong \nade(\la)$. 
\end{Lemma}
\begin{proof}
By Lemma~\ref{3.5d}, $M_n=Z^{(1^n)}\cong \be_{h,n}(V_h^{\otimes n})$. 
So the only (graded) composition factors of $M_n$ are of the form $L(\la)$ for $\la\in X_+(n)$ and each such $L(\la)$ occurs with some non-zero {\em graded}\, multiplicity $m_\la\in \Z_{>0}$. The formal characters of the modules $L(\la)$ are linearly independent. Hence, since $\CH M_n=\sum_{\la\in X_+(n)}m_\la\CH L(\la)$ is bar-invariant, by Lemma~\ref{LMSelfDual}, we conclude that each $\CH L(\la)$ is bar-invariant, which immediately implies that $L(\la)^\circledast\cong L(\la)$. It follows from a general theory of quasi-hereditary algebras that $\De(\la)^\circledast$ is the costandard module $\nabla(\la)$ up to a degree shift, and now the first statement of the lemma implies the second one.  
\end{proof}

Now we obtain the desired characterizations of $\Dede(\la)$. Recall the element $u_\la\in\Si_n$ from Lemma~\ref{1.1d}.

\begin{Theorem} \label{3.5e} 
Let $\la\vdash n$. Then: 
\begin{enumerate}
\item[{\rm (i)}] 
$\Hom_{\ImS_n}(\Zdot^{\la},\Ladot^{\la^\tr})\cong F$, and the image of any non-zero homomorphism in $\Hom_{\ImS_n}(\Zdot^{\la},\Ladot^{\la^\tr})$ is isomorphic to $\Dede(\la)$; 
\item[{\rm (ii)}] $\Dede(\la)$ is isomorphic to the submodule $\Zdot^{\la}u_{\la^\tr} {\mathtt x}_{\la^\tr}$ of $M_n$.
\end{enumerate}
\end{Theorem}
\begin{proof}
(i) follows from Lemma~\ref{3.5d}, the definition (\ref{3.5.4}), and Lemma~\ref{1.3d}, since $\be_{h,n}$ is an equivalence of categories.

(ii) Note that  $\Zdot^{\la}$ contains $M_n{\mathtt y}_{\la}$ as a submodule. Moreover, as $M_n$ is a faithful $F\Si_n$-module and ${\mathtt y}_{\la}u_{\la^\tr} {\mathtt x}_{\la^\tr}\neq 0$ by Lemma~\ref{1.1d}, we conclude that $M_n{\mathtt y}_{\la}u_{\la^\tr} {\mathtt x}_{\la^\tr}\neq 0$. Hence $\Zdot^{\la}u_{\la^\tr} {\mathtt x}_{\la^\tr}\neq 0$. Finally, observe that $\Zdot^{\la}u_{\la^\tr} {\mathtt x}_{\la^\tr}$ is both a homomorphic image of $\Zdot^{\la}$ and a submodule of $\Ladot^{\la^\tr}$. So the result follows from (i). 
\end{proof}

We will write $\tilde M$ for the {\em right}\, $\ImS_n$-module obtained from $M \in\mod{\ImS_n}$ by twisting the left action into a right action using the antiautomorphism $\tau$ of $\ImS_n$. In this way, we	obtain right $\ImS_n$-modules $\tilde\Lde(\la)$, $\tilde\Dede(\la)$, and $\tilde\nade(\la)$.

\begin{Theorem} \label{3.5g} 
We have 
\begin{enumerate}
\item[{\rm (i)}] $\ImS_n$ has a filtration as a $(\ImS_n,\ImS_n)$-bimodule with factors isomorphic  to $\Dede(\la) \otimes\tilde\Dede(\la)$, each appearing once for each $\la\vdash n$ and ordered in any way refining the dominance order on partitions so that factors  corresponding to most dominant $\la$ appear at the bottom of the filtration.

\item[{\rm (ii)}] $Z$ has a filtration as a $(\ImS_n,S_{h,n})$-bimodule with factors
$\Dede(\la) \otimes\tilde \De_h(\la)$ appearing once for each $\la\vdash n$ and ordered in any way refining the dominance order so that factors corresponding to most dominant $\la$ appear at the bottom of the filtration.
\end{enumerate}
\end{Theorem}
\begin{proof}
(i) This follows from the general theory of quasi-hereditary algebras, as is explained for example after \cite[(1.2e)]{BDK}.

(ii) The functor $Z\otimes_{S_{h,n}} ?$ can be viewed as an exact functor from the category of $(S_{h,n},S_{h,n})$-bimodules to the category of $(\ImS_n,S_{h,n})$-bimodules. We have: 
$$Z\otimes_{S_{h,n}} (\Dede_h(\la) \otimes\tilde \De_h(\la))\cong (Z\otimes_{S_{h,n}} \Dede_h(\la))\otimes \tilde \De_h(\la)\cong \Dede(\la) \otimes\tilde \De_h(\la).$$ 
So applying $Z\otimes_{S_{h,n}} ?$ to the filtration of Lemma~\ref{1.2e} gives the result.
\end{proof}

\section{Base change}\label{SBaseChange}
Recall that $\O$ denotes the ground ring which is always assumed to be either $\Z$ or $F$.   
The algebras $S_{n,h}$, $\ImS_n$, and the modules $\Mde_n, \Zde^\la$, etc. are all defined over $\Z$, although in many results proved in the previous sections we have assumed that $\O$ is a field. We now work over $\O=\Z$, and study the base change from $\Z$ to $F$. 
Throughout this section, it will also be convenient to denote by $K$ a field of characteristic zero and use notation like $\ImS_{n,\Z}$, $\ImS_{n,F}$, $\ImS_{n,K}$, 
etc. to specify the ring over which the objects are considered.

\begin{Lemma} \label{LIntegral}
We have:
\begin{enumerate}
\item[{\rm (i)}] $\Mde_{n,\Z}$ is a $\Z$-free module of finite rank with $ \Mde_{n,\Z}\otimes_\Z F\cong \Mde_{n,F}$; 
\item[{\rm (ii)}] $\Zde_{n,\Z}$ is a $\Z$-free module of finite rank with $\Zde_{n,\Z}\otimes_\Z F\cong \Zde_{n,F}$. Moreover, $\Zde_{n,\Z}=\Yde_{n,\Z}:=R_{n\de}(\Z)\ga_n\Mde_{n,\Z}$. 
\end{enumerate}
\end{Lemma}
\begin{proof}
(i) comes from the Lemma~\ref{LMBasis}. 

(ii) By (i), $\Mde_{n,\Z}$ is a lattice in $\Mde_{n,K}$, and by definition, we have $\Zde_{n,\Z}=\Zde_{n,K}\cap \Mde_{n,\Z}$. So $\Zde_{n,\Z}$ is a lattice in $\Zde_{n,K}$, and also a direct summand of the $\Z$-module $\Mde_{n,\Z}$.  Hence the natural map 
$$i : \Zde_{n,\Z}\otimes_\Z F \to \Mde_{n,\Z}\otimes_\Z F=\Mde_{n,F}$$ is injective. Since the action of $\Si_n$ on $\Mde_n$ is compatible with base change, we have $\im\,i\subseteq\Zde_{n,F}$.

Recall the submodule $\Yde_n=R_{n\de}\ga_n\Mde_n$ from (\ref{EY}). Note that the natural (not necessarily injective) map from $\Yde_{n,\Z}\otimes_\Z F$ to $\Mde_{n,F}$ has image $\Yde_{n,F}$. Now by Propsosition~\ref{2.5e}(iii), $\Yde_{n,\Z} \subseteq \Zde_{n,\Z}$, so $\Yde_{n,F}\subseteq \im \,i$. 
By Theorem~\ref{3.4g}(i), $\Yde_{n,F}=\Zde_{n,F}$, so by the previous paragraph, the map $i:\Zde_{n,\Z}\otimes_\Z F\to \Zde_{n,F}$ is an isomorphism. 

Finally, the embedding $\Yde_{n,\Z}\to \Zde_{n,\Z}$ has to be an isomorphism, since  otherwise for some field $F$ the induced map $\Yde_{n,\Z}\otimes_\Z F\to \Zde_{n,\Z}\otimes_\Z F$ is not surjective, and so the composition $$\Yde_{n,\Z}\otimes_\Z F\to \Zde_{n,\Z}\otimes_\Z F\to \Zde_{n,F}=Y_{n,F}$$ is not surjective, giving a contradiction.  
\end{proof}

Since induction commutes with base change, we deduce: 

\begin{Corollary} \label{4.1b} 
Let $\nu\vDash n$. Then 
$\Zde^\nu_\Z$ is a $\Z$-free module of finite rank	with
$\Zde^\nu_\Z\otimes_\Z F\cong \Zde^\nu_F$.
\end{Corollary}

Now we can define standard modules over $\Z$. For $\la\vdash n$, set 
$$\De_\Ze(\la) =\Zde^{\la}_\Ze u_{\la^\tr} {\mathtt x}_{\la^\tr}.$$
Compare this to Theorem~\ref{3.5e}(ii), in which we worked over a field.

\begin{Theorem} \label{4.1c} 
$\De_\Z(\la)$ is $\Z$-free of finite rank with $\De_\Z(\la)\otimes_\Z F\cong \De_F(\la)$. 
Moreover, the formal characters of $\De_\Ze(\la)$ and $\De_F(\la)$ are the same. 
\end{Theorem}
\begin{proof}
By definition $\Zde^{\la}_\Ze$ is a submodule of $M_{n,\Z}$, and  $\De_\Z(\la)$ is a submodule of the torsion free $\Z$-module $M_{n,\Z}$, so $\De_\Z(\la)$ is torsion free. There is a natural map $\De_\Z(\la)\otimes_\Z K \to M_{n,K}$,  which is injective since  $K$ is flat over $\Z$. It is easy to check that the image of this map is precisely $\De_K(\la)$, proving that $\De_\Z(\la)$ is a $\Z$-lattice in $\De_K(\la)$. In particular, $\De_\Z(\la)$ has rank equal to $\dim \De_K(\la)$. 
By Theorem~\ref{3.5g}(ii) with $h = n$, we have  
$$
\sum_{\nu\in X(n,n)} \dim\Zde^\nu_K = \sum_{\la\vdash n}(\dim \De_K(\la))(\dim \De_{n,K}(\la)),
$$
where $\De_{n,K}(\la)$ denotes the 
standard module for the Schur algebra  $S_{n,n,K}$. In view of Corollary~\ref{4.1b},  
$\dim\Zde^\nu_K = \dim\Zde^\nu_F$, and it is well-known that the dimensions of standard modules for the Schur algebra do not depend on the ground field. So 
$$
\sum_{\nu\in X(n,n)} \dim\Zde^\nu_F = \sum_{\la\vdash n}(\dim \De_K(\la))(\dim \De_{n,F}(\la)).
$$

There is a natural map $i : \De_\Z(\la)\otimes_\Z F \to \Mde_{n,F}$ with image  $\De_F(\la)$ induced by the embedding $\De_\Z(\la) \to \Mde_{n,\Z}$. So $\dim \De_K(\la) \geq \dim \De_F(\la)$ . On the other hand, applying Theorem~\ref{3.5g}(ii) over $F$, we have that
$$
\sum_{\nu\in X(n,n)} \dim\Zde^\nu_F = \sum_{\la\vdash n}(\dim \De_F(\la))(\dim \De_{n,F}(\la)).
$$
Comparing with our previous expression,	we conclude that $\dim \De_F(\la)=\dim \De_K(\la)$  for all $\la\vdash n$. Hence  $i$ is injective. The result about the characters is now clear. 
\end{proof} 

We now show that the imaginary Schur algebra and its Morita equivalence with a classical Schur algebra 
are defined over $\Z$. 

\begin{Theorem} \label{4.1d}
We have:
\begin{enumerate}
\item[{\rm (i)}] $\ImS_{n,\Z}$ is $\Z$-free of finite rank with $\ImS_{n,\Z}\otimes_\Z F\cong \ImS_{n,F}$;
\item[{\rm (ii)}] $Z^\nu_\Z$ is a projective $\ImS_{n,\Z}$-module for each $\nu\vDash n$;
\item[{\rm (iii)}] $\End_{\ImS_{n,\Z}}\big(\bigoplus_{\nu\in X(n,h)} \Zde^\nu_\Z\big) \cong S_{h,n,\Z}$;
\item[{\rm (iv)}] $\bigoplus_{\nu\in X(n,h)} \Zde^\nu_\Z$ is a projective generator for $\ImS_{n,\Z}$, so $\ImS_{n,\Z}$ is Morita equivalent to $S_{h,n,\Z}$ for $h\geq n$.
\end{enumerate}
\end{Theorem}
\begin{proof}
(i) By definition, $\ImS_{n,\Z}$ is the $\Z$-submodule of $\End_\Z(M_{n,\Z})$ spanned by the images of the $\Z$-basis elements of $R_{n\de,\Z}$ which are of the form
$\psi_w y_1^{b_1} \dots y_d^{b_d} 1_\bi$. Since the degree of each $y_r$ is $2$, all but finitely many such elements act as zero, so $\ImS_{n,\Z}$ is finitely generated over $\Z$, whence $\ImS_{n,\Z}$ is a lattice in $\ImS_{n,K}$.

The natural inclusion
$\ImS_{n,\Z}\into \End_\Z(\Mde_{n,\Z})$ yields a map $$\ImS_{n,\Z}\otimes_\Z F\to \End_\Z(\Mde_{n,\Z})\otimes F\cong\End_F(\Mde_n(F))$$ whose image is $\ImS_{n,F}$. This map is injective since $\dim \ImS_{n,K} = \dim \ImS_{n,F}$ by  Theorems~\ref{3.5g}(i) and \ref{4.1c}.
 
 (ii) By Corollary~\ref{4.1b}, we have $\Zde^\nu_\Z\otimes_\Z F\cong \Zde^\nu_F$, which is a projective $\ImS_{n,F}$-module by Theorem~\ref{3.4g}. Therefore, in view of the Universal Coefficients Theorem, $\Zde^\nu_\Z$ is a projective $\ImS_{n,\Z}$-module. 
 
(iii) Denote $$E_\O:= \End_{\ImS_{n,\O}}\big( \bigoplus_{\nu\in X(n,h)} \Zde^\nu_\O\big).$$ By Theorem~\ref{3.4c}, we have $E_F \cong S_{h,n,F}$. Moreover, $E_\Z$ is an $\O$-lattice in $E_K$, and there is a natural embedding $E_\Z\otimes_\Z F \to E_F$, cf. \cite[Lemma~14.5]{Land}. The last embedding is an isomorphism by dimension. 
So we can identify $E_\Z\otimes K$ with $E_K$ and $E_\Z\otimes_\Z F$ with $E_F$. 

Now, the basis element $\phi_{\mu,\la}^u$ of $E_K\cong S_{h,n,K}$ acts as zero on all summands except $\Zde^\mu_K$, on which it is induced by the right multiplication by $s_{\mu,\la}^u$. By definition,  $\Zde^\nu_\Z= \Zde^\nu_K\cap M_{n,\Z}$. Also, $s_{\mu,\la}^u\in \Z\Si_n$, therefore $s_{\mu,\la}^u$ stabilizes $M_{n,\Z}$. Hence $\Zde^\mu_\Z s_{\mu,\la}\subseteq \Zde^\la_\Z$, so each $\phi^u_{\mu,\la} \in E_K$ restricts to give a well-defined element of $E_\Z$. We have constructed a isomorphic copy $S_\Z$ of $S_{h,n,\Z}$ in $E_\Z$, namely, the $\Z$-span of the standard basis elements $\phi^u_{\mu,\la}\in  S_{h,n,K}$. 

It remains to show that $S_\Z=E_\Z$. We have a short exact sequence of $\Z$-modules: 
$$
0\to S_\Z\to E_\Z\to Q_\Z\to 0,
$$ and we need to prove that $Q_\Z=0$, for which it suffices to prove that $Q_\Z\otimes_\Z F=0$. Tensoring with $F$, we have an exact sequence
$$
S_\Z\otimes_\Z F\stackrel{i}{\to} E_F\to Q_\Z\otimes_\Z F\to 0.
$$
The map $i$ sends $1\otimes \phi^u_{\mu,\la}$ to the corresponding endomorphism $\phi^u_{\mu,\la}$ defined as in Theorem~\ref{3.4c}. Hence, $i$ is injective, so $i$ is an isomorphism by dimensions, and $Q_\Z\otimes_\Z F= 0$, as required. 

(iv) By (ii), $\bigoplus_{\nu\in X(n,h)} \Zde^\nu_\Z$ is a projective $\ImS_{n,\Z}$-module. For $h\geq n$, it is a projective generator, because this is so on tensoring with $F$, using (i) and Theorem~\ref{3.4g}. 
\end{proof}

\section{Ringel duality and double centralizer properties}

Let $S$ be a quasi-hereditary algebra with weight poset $(\La_+,\leq)$ and standard modules $\De(\la)$. Recall that a (finite dimensional) $S$-module  is called {\em tilting}\, if it has both a standard and a costandard filtrations. By \cite{Ringel}, for each $\la\in\La_+$, there exists a unique indecomposable tilting module $T(\la)$ such that $[T(\la) : \De(\la)] = 1$ and $[T(\la) : \De(\mu)] = 0$ unless $\mu \leq\la$. Furthermore, every tilting module is isomorphic to a direct sum of  indecomposable tilting modules $T(\la)$. A {\em full tilting}\, module is a tilting module that contains every $T (\la),\ \la\in\La_+,$ as a summand. Given a full tilting module $T$, the {\em Ringel dual}\, of $S$ relative to $T$ is the algebra $S^*:= \End_S(T)^{\op}$. Writing endomorphisms on the right, $T$ is naturally a right $\End_S(T)$-module, hence a left $S^*$-module. Ringel \cite{Ringel} showed that $S^*$ is also a quasi-hereditary algebra with weight poset $\La_+$, but ordered with the opposite order. 
We will need the following known result (for references see \cite[Section 4.5]{BDK}). 

\begin{Lemma} \label{4.5a} 
Regarded as a left $S^*$-module, $T$ is a full tilting module for $S^*$. Moreover, the Ringel dual $\End_{S^*}(T)^\op$ of $S^*$ relative to $T$ is isomorphic to $S$.
\end{Lemma}

Applying Ringel's theorem first to the Schur algebra $S_{h,n}$, we obtain the indecomposable tilting modules $\{T_h(\la)\mid \la \in X_+(h,n)\}$ of $S_{h,n}$. 
For $h\geq n$, define 
\begin{equation}\label{ETilt}
\Tde(\la) := \beta_{h,n}(T_h(\la))\qquad(\la\vdash n).
\end{equation}
Since $\be_{h,n}$ is Morita equivalence, $\{\Tde(\la) \mid \la\vdash n\}$ are the indecomposable tilting modules for $\ImS_n$.

\begin{Lemma} \label{4.5b} 
The indecomposable tilting modules for $\ImS_n$ are precisely the indecomposable summands of $\Lade^\nu$ for all $\nu\vDash n$. Furthermore, for $\la\vdash n$, the module $\Tde(\la)$ occurs exactly once as a summand of $\Lade^{\la^\tr}$, and if $\Tde(\mu)$ is a summand of $\Lade^{\la^\tr}$ for some $\mu\vdash\la$, then $\mu\leq\la$.
\end{Lemma}
\begin{proof}
By \cite[Section 3.3(1)]{Donkin}, 
$T_h(\la)$ occurs exactly once as a summand of $\La^{\la^\tr}(V_h)$, and if $T_h(\mu)$ is a summand of $\La^{\la^\tr}(V_h)$ then $\mu\leq\la$.
Now the result follows on applying the functor $\beta_{h,n}$ and using Lemma~\ref{3.5d}.
\end{proof}

\begin{Corollary} \label{4.5c} 
For $\la\vdash n$, the module $\Tde(\la)$ is the unique
indecomposable summand of $\Lade^{\la^\tr}$ containing a submodule isomorphic to $\De(\la)$.
\end{Corollary}
\begin{proof}
By Theorem~\ref{3.5e}(i), $\Lade^{\la^\tr}$ has a unique submodule isomorphic to $\De(\la)$. By Lemma~\ref{4.5b}, $\Lade^{\la^\tr}$ has a unique summand isomorphic to $\Tde(\la)$ and for any other summand $M$ of $\Lade^{\la^\tr}$, we have $\Hom_{\ImS_n}(\De(\la),M) = 0$. 
\end{proof}

\begin{Theorem} \label{4.5d}
{\bf (Imaginary Ringel Duality)}
Let $h\geq n$. The $\ImS_n$-module $T :=\bigoplus_{\nu\in X(h,n)} \Lade^\nu$ is a full tilting module. 
Moreover, the Ringel dual $\ImS_n^*$ 	of $\ImS_n$ relative to $T$ is precisely the algebra $S_{h,n}^\op$ where $S_{h,n}$ acts on $T$ as in Theorem~\ref{3.4a}.
\end{Theorem}
\begin{proof}
By Lemma~\ref{4.5b}, $T$ is a full tilting module. The second statement is a restatement of Theorem~\ref{3.4a}.
\end{proof}

So far we have only had `halves' of imaginary Schur-Weyl and Howe dualities, namely:  $\End_{\ImS_n}(\Mde_n)\cong F\Si_n$ and $\End_{\ImS_n}\Big(\bigoplus_{\nu\in X(h,n)} \Lade^\nu \Big)\cong S_{h,n}$. Now we can finally establish the `second halves'.

\begin{Theorem} \label{4.5e} 
{\bf (Double Centralizer Properties)} 
Let $h \geq n$. Then: 
\begin{enumerate}
\item[{\rm (i)}] $\End_{\ImS_n}\Big(\bigoplus_{\nu\in X(h,n)} \Lade^\nu \Big)\cong S_{h,n}$  and $\End_{S_{h,n}}\Big(\bigoplus_{\nu\in X(h,n)} \Lade^\nu \Big)\cong \ImS_{n}$,
where the right $S_{h,n}$-action is as in Theorem~\ref{3.4a}; 
\item[{\rm (ii)}] $\End_{\ImS_n}(\Mde_n)\cong F\Si_n$  and $\End_{F\Si_n}(\Mde_n)\cong \ImS_{n}$ where the right $F\Si_n$-action is as in Theorem~\ref{TSchub}. 
\end{enumerate}
\end{Theorem}
\begin{proof}
(i) Combine Theorem~\ref{4.5d} with Lemma~\ref{4.5a}.

(ii) By Theorem~\ref{TSchub}, we know already that $F\Si_n \cong  \End_{\ImS_n}(\Mde_n)$. Let $e:=e((1^n))\in S_{h,n}$, see (\ref{EIdSchur}). 
We know that $T_h(\la)$ is a summand of $\La^{\la^\tr}(V_h)$, for any $\la\in X_+(h,n)$.  So, by Lemma~\ref{1.3b}(ii), $T_h(\la)$ is both a submodule and a quotient of the $S_{h,n}$-module $S_{h,n}e$. Moreover, $S_{h,n}e \cong V_n^{\otimes n}$, so $S_{h,n}e$ is self-dual. From this one deduces a (well-known) fact that every composition factor of both the socle and the head of $T_h(\la)$ belongs to the head of the projective $S_{h,n}$-module $S_{h,n}e$.  

Now let $T :=\bigoplus_{\nu\in X(h,n)} \Lade^\nu$, and set $\tilde T$ to be the left $S_{h,n}$-module obtained
from the right module $T$ by twisting with $\tau$. Then $\tilde T$ is a full tilting module for $S_{h,n}$ by Lemma~\ref{4.5a} and Theorem~\ref{4.5d}. So, by the previous paragraph, every composition factor of the socle and the head of $\tilde T$ belongs to the head of $S_{h,n}e$. By Lemma~\ref{L3.1c}, we deduce that $\End_{S_{h,n}}(\tilde T) \cong \End_{eS_{h,n}e}(e\tilde T)$. Switching to right actions, and using (i), we have now shown that
$\ImS_n \cong \End_{eS_{h,n}e}(Te).$
So, to prove (ii), it suffices to show that $\End_{F\Si_n}(\Mde_n) \cong \End_{eS_{h,n}e}(Te)$. 

As a left $\ImS_n$-module, $M_n \cong Te$. Recall the map $\kappa$ from  Lemma~\ref{1.2b} and  the action of $S_{h,n}$ on $T$ from Theorem~\ref{3.4a}. One now easily checks that the $(\ImS_n , eS_{h,n} e)$-bimodule $T e$ is isomorphic to the $(\ImS_n , F\Si_n)$-bimodule $M_n$, if we identify $F\Si_n$ with $eS_{h,n}e$, so that $w \mapsto \kappa(\sgn(w)w)$ for each $w \in \Si_n$. In view of this, we have 
$\End_{F\Si_n}(\Mde_n) \cong \End_{eS_{h,n}e}(Te)$.  
\end{proof}

We conclude this section with imaginary analogues of well-known results of Donkin and Mathiew-Papadopoulo. 
For $\la\vdash n$ we denote by $\Pde(\la)$ the projective cover of $\Lde(\la)$ in the category $\mod{\ImS_n}$. Similarly, for $\mu \in X_+(h,n)$, let $P_h(\mu)$ denote the projective cover of $L_h(\mu)$ in the category $\mod{S_{h,n}}$. Using Morita equivalence, we get:
\begin{equation}\label{4.5.3}
\Pde(\la)=\beta_{h,n}(P_h(\la))\qquad(\la\vdash n).
\end{equation}

Recall the generalized Kostka numbers $k_{\la,\mu}$ from (\ref{EWtSp}).

\begin{Theorem} \label{4.5f} 
For $\nu\in X(h,n)$ we have:
\begin{enumerate}
\item[{\rm (i)}] $\Zde^\nu\cong \bigoplus_{\la\vdash n}\Pde(\la)^{\oplus k_{\la,\nu}}$;
\item[{\rm (ii)}] $\Lade^\nu\cong \bigoplus_{\la\vdash n}\Tde(\la^\tr)^{\oplus k_{\la,\nu}}$.
\end{enumerate}
\end{Theorem}
\begin{proof}

(i) By \cite[Lemma 3.4(i)]{DoTil}, we have $Z^\nu(V_h) \cong \bigoplus_{\la\vdash n}P_h(\la)^{\oplus k_{\la,\nu}}$. Now apply the equivalence of categories $\beta_{h,n}$, using (\ref{4.5.3}) and Lemmas~\ref{3.5d}.

(ii) By \cite[Corollary 2.3]{MP}, we have $\La^\nu(V_h) \cong \bigoplus_{\la\vdash n}T_h(\la^\tr)^{\oplus k_{\la,\nu}}$. Now apply the equivalence of categories $\beta_{h,n}$, using (\ref{ETilt}) and Lemmas~\ref{3.5d}.
\end{proof}

\chapter{On formal characters of imaginary modules}\label{SFCIM}
Throughout the chapter we assume that $\O=F$ unless otherwise stated, and $h\geq n$.

\section{Weight multiplicities in $\La^{\la^\tr}(V_h)$}
Let $h\geq n$, $\la\in X_+(h,n)$, and $\mu\in X(h,n)$. Recall the notion of a column strict $\la$-tableau from Section~\ref{SSSchurRT}. 
Denote
\begin{equation}\label{EClamu}
\co_{\la,\mu}:=\sharp\{\text{column strict $\la$-tableaux of type $\mu$}\}.
\end{equation}
Recalling the $S_{h,n}$-module $\La^{\la^\tr}(V_h)$ defined in (\ref{1.3.4}), the following equality for its weight multiplicities is clear:
$$
\dim e(\mu)\La^{\la^\tr}(V_h)=\co_{\la,\mu}.
$$


The following combinatorial result is easy to check using the definition of $\co_{\la,\mu}$: 

\begin{Lemma} \label{LLaChar}
Let $\la\in X_+(h,n)$, $\mu\in X(h,n)$. If the last column of the partition $\la$ has height $l$, and $\bar\la\in X_+(h, n-l)$ is the partition obtained from $\la$ by deleting this last column, then 
$$
\co_{\la,\mu}=\sum_{\stackrel{1\leq s_1<\dots<s_l\leq h,}{\mu-\eps_{s_1}-\dots-\eps_{s_l}\in X(h,n)}}\co_{\bar\la,\mu-\eps_{s_1}-\dots-\eps_{s_l}}.
$$
\end{Lemma}

Recall the classical Kostka numbers from (\ref{EKostka}).

\begin{Lemma} \label{LColNK} 
Let $\la,\mu\vdash n$. Then $\co_{\la,\mu}=\sum_{\nu\vdash n}K_{\nu^\tr,\mu}K_{\nu,\la^\tr}$. 
\end{Lemma}
\begin{proof}
We have the well-known fact that over $\C$ the module 
$\La^{\la^\tr}(V_h)$ decomposes as
$\La^{\la^\tr}(V_h)=\bigoplus_{\nu\vdash n}\De_h(\nu^\tr)^{\oplus K_{\nu,\la^\tr}}$. Passing to the dimensions of the $\mu$-weight spaces in the last equality yields the lemma.  
\end{proof}

\section{Gelfand-Graev words and shuffles}\label{SSGGWSh}
Recall from (\ref{EBWt}) that we have fixed an extremal word $\bi=i_1 \dots i_e$ of $L_\de$. Recall that $i_1=0$ and $i_e=i$.
As in (\ref{EGGNewWay}), we also write $\bi$ in the form
$$
\bi=j_1^{m_1}\dots j_r^{m_r}\qquad(\text{with}\ j_k\neq j_{k+1}\ \text{for all $1\leq k<r$}).
$$
Note that always $m_1=1$. 
We define the {\em Gelfand-Graev words} (of {\em type $\bi$}):
$$
\gg^{(n)}=\gg^{(n)}_\bi:=i_1^n i_2^n\dots i_e^n=j_1^{m_1n}\dots j_r^{m_rn}\in\words_{n\de}
$$
for any $n\in\Z_{>0}$ and, more generally, for any composition $\mu\in X(h, n)$, set:
\begin{equation}\label{EGGWcopy}
\gg^\mu:=\gg^{(\mu_1)}\dots\gg^{(\mu_h)}\in\words_{n\de}
\end{equation}

\begin{Lemma} \label{LCombGG}
Let $n=l_1+\dots+l_a$ for some $l_1,\dots,l_a\in\Z_{>0}$. Suppose that for each $1\leq c\leq a$, we are given a word $\bj^{(c)}$ of the imaginary tensor space $\Mde_{l_c}$. Assume that a Gelfand-Graev word $\gg^\mu$ of type $\bi$ appears as a summand in the shuffle product $\bj^{(1)}\circ \dots\circ \bj^{(a)}$. Then $\bj^{(1)}, \dots, \bj^{(a)}$ are all Gefand-Graev words of type $\bi$. 
\end{Lemma}
\begin{proof}
Clearly we may assume that $a=2$. Then we may write $\bj^{(1)}=\bk^{(1)}\bl^{(1)}$ and $\bj^{(2)}=\bk^{(2)}\bl^{(2)}$, so that $\gg^{(\mu_1)}$ appears in the shuffle product $\bk^{(1)}\circ \bk^{(2)}$. Recall that $\gg^{(\mu_1)}=j_1^{m_1\mu_1}\dots j_r^{m_r\mu_1}$. It follows that $\bk^{(1)}=j_1^{a_1}\dots j_r^{a_r}$, $\bk^{(2)}=j_1^{b_1}\dots j_r^{b_r}$ with $a_k+b_k=m_k\mu_1$ for all $k=1,\dots,r$. Note that  $a_2+b_2=m_2\mu_1=m_2(a_1+b_1)$. We claim that  $a_2=a_1m_2$ and $b_2=b_1m_2$. Indeed, otherwise either $a_2>a_1m_2$ or $b_2>b_1m_2$. But the first inequality  contradicts the fact that $\bj^{(1)}=\bk^{(1)}\bl^{(1)}$ is a word of $M_{l_1}$, and the second inequality contradicts the fact that $\bj^{(2)}=\bk^{(2)}\bl^{(2)}$ is a word of $M_{l_2}$. Continuing this way, we see that 
 $a_k=m_ka_1$ and $b_k=m_kb_1$ for all $k=1,\dots,r$. In other words, $\bk^{(1)}=\gg^{(a_1)}$ and $\bk^{(2)}=\gg^{(b_1)}$. Now the Gelfand-Graev word $\gg^{(\mu_2)}\dots\gg^{(\mu_h)}$ appears in the shuffle product $\bl^{(1)}\circ \bl^{(2)}$. Moreover, $\bl^{(1)}$ and $\bl^{(2)}$ are words of $\Mde_{l_1-a_1}$ and $\Mde_{l_2-b_1}$, respectively. So we can apply induction on the length of $\mu$. 
\end{proof}

For $n\in\Z_{\geq 0}$, we denote 
$$
c(n):=\prod_{k=1}^r [m_kn]^!_{j_k}\in\A,
$$
Note by Lemma~\ref{LMultOneWeight} that $\DIM1_\bi L_\de=c(1)$, since we have chosen $\bi$ to be an extremal word in $L_\de$. 
For $\mu\in X(h, n)$, denote 
\begin{equation}\label{Emla}
c(\mu):=c(\mu_1)\dots c(\mu_h)=\prod_{m=1}^h\prod_{k=1}^r [m_k\mu_m]^!_{j_k}\in\A. 
\end{equation}
In the simply laced types all $m_k=1$, and so $c(\mu)=([\mu_1]_q^!\dots [\mu_h]_q^!)^e$.

Recall the numbers $\co_{\la,\mu}$ defined in (\ref{EClamu}).

\begin{Proposition} \label{PGGinShuffle}
Let $\la\in X_+(h,n)$ and $\mu\in X(h,n)$. Set $\la^\tr=(l_1,\dots,l_a)$. Then $\gg^\mu$ appears in the quantum shuffle product
\begin{equation}\label{EPGGinShuffle}
\bi^{l_1}\circ\dots\circ \bi^{l_a}
\end{equation}
with the coefficient $c(\mu)\co_{\la,\mu}/c(1)^n$. 
\end{Proposition}
\begin{proof}
We apply induction on $a$. If $a=1$, then (\ref{EPGGinShuffle}) is just the concatenation $\bi^{l_1}=\gg^{(1^n)}$, and the result is clear since $c((1^n))=c(1)^n$ and $\co_{(1^n),(1^n)}=1$. For the inductive step, let $a>1$ and denote by $\bar\la\in X_+(h,n-l_a)$ the partition obtained from $\la$ by deleting its last column. By the inductive assumption, for any $\nu\in X(h, n-l_a)$ we have that  $\gg^\nu$ appears in the quantum shuffle product
$$
S:=\bi^{l_1}\circ\dots\circ\bi^{l_{a-1}}
$$
with coefficient $c(\nu)\co_{\bar\la,\nu}/c(1)^{n-l_a}$. Now  (\ref{EPGGinShuffle}) is 
$
S\circ \bi^{l_a}.
$
By Lemma~\ref{LCombGG}, if the word $\bj$ appearing in $S$ has the property that some Gelfand-Graev word $\gg^\mu$ appears in the shuffle product $\bj\circ (\bi^{(b)})^{l_a}$, then the word $\bj$ must itself be Gelfand-Graev, i.e. $\bj=\gg^\nu$ for some $\nu\in X(h,n-l_a)$. Moreover, note that $\gg^\mu$ appears in $\gg^\nu\circ \bi^{l_a}$ if and only if $\mu$ is of the form $\mu=\nu+\eps_{s_1}+\dots+\eps_{s_{l_a}}$ for some $1\leq s_1<\dots<s_{l_a}\leq h$, in which case $\gg^\mu$ appears in $\gg^\nu\circ \bi^{l_a}$ with the coefficient $c(\mu)/c(\nu)c(1)^{l_a}$.  Now the result follows in view of Lemma~\ref{LLaChar}. 
\end{proof}

Recall Gelfand-Graev modules $\Ga_n\cong\Ga_{nm_1,j_1}\circ\dots\circ \Ga_{nm_r,j_r}$ from Section~\ref{SIGGR}. 

\begin{Lemma} \label{LDimGG}
Let $M\in\mod{R_{n\de}}$ and $\mu\in X(n,h)$. Then 
$$\DIM M_{\gg^\mu}=c(\mu)\,\DIM\Hom_{R_{n\de}}(\Ga_{\mu_1}\circ \dots\circ \Ga_{\mu_h},M).$$
\end{Lemma}
\begin{proof}
Let 
$$\si=(\mu_1m_1\al_{j_1},\dots,\mu_1m_r\al_{j_r},\dots,\mu_hm_1\al_{j_1},\dots,\mu_hm_r\al_{j_r}).$$ 
Consider the irreducible module
$$
L:=L(j_1^{\mu_1m_1})\boxtimes\dots\boxtimes L(j_r^{\mu_1m_r})\boxtimes  \dots\boxtimes L(j_1^{\mu_hm_1})\boxtimes\dots\boxtimes L(j_r^{\mu_hm_r})
$$
over the parabolic $R_\si$. 
Note that 
$\DIM M_{\gg^\mu}=c(\mu)[\Res_\si M:L]_q.$ 
Since $\Ga_{m,i}$ is the projective cover of $L(i^m)$ by Proposition~\ref{PGGId}(i), $[\Res_\si M:L]_q$ equals
$$\DIM\Hom_{R_\si}(\Ga_{\mu_1m_1,j_1}\boxtimes\dots\boxtimes \Ga_{\mu_1m_r,j_r}\boxtimes  \dots\boxtimes \Ga_{\mu_hm_1,j_1}\boxtimes\dots\boxtimes \Ga_{\mu_hm_r,j_r},\Res_{\si}M),$$
and the result follows by the adjunction of $\Ind$ and $\Res$. 
\end{proof}

\section{Gelfand-Graev fragment of the formal character of $\Dede(\la)$}
Recall Gelfand-Graev words $\gg^\mu=\gg^\mu_\bi$ from  Section~\ref{SSGGWSh}, scalars $c(n),c(\mu)\in\A$ from (\ref{Emla}), and $R_{n\de}$-modules $L(\la), \De(\la),  T(\la)$ from (\ref{3.5.3}), (\ref{3.5.4}), (\ref{ETilt}). 

\begin{Lemma} \label{L1^n}
We have $\Lde((1^n))=\Dede((1^n))=\Tde((1^n))=\Lade_n$, the Gelfand Graev word $\gg^{(1^n)}=\bi^{n}$ appears in $\CH \Dede((1^n))$ with multiplicity $c((1^n))=c(1)^n$, and $\gg^{\mu}$ with $\mu\in X(n,n)\setminus\{(1^n)\}$ does not appear in $\CH \Dede((1^n))$.  
\end{Lemma}
\begin{proof}
It well-known for the usual Schur algebras that $T_n((1^n))=\De_n((1^n))=L_n(1^n)$. By applying the Morita equivalence $\be_{n,n}$ and Lemma~\ref{4.5b}, we now obtain  
$
\Lde((1^n))=\Dede((1^n))=\Tde((1^n))=\La^{(n)}=\La_n,
$
and the first two statements follow from Lemma~\ref{LLaIrr}. 

For the last statement, let $\mu\in X(n,n)\setminus\{(1^n)\}$. We have to prove that $\De((1^n))_{\gg^\mu}=0$. In view of Theorem~\ref{4.1c}, we may assume that $F$ has characteristic zero. By Lemma~\ref{LDimGG}, we need to prove that 
$\Hom_{R_{n\de}}(\Ga_{\mu_1}\circ\dots\circ \Ga_{\mu_n},\Lade_n)=0.$ 
We have by adjunction of $\Ind$ and $\Res$ and Lemma~\ref{3.2f}(ii):
$$\Hom_{R_{n\de}}(\Ga_{\mu_1}\circ\dots\circ \Ga_{\mu_n},\Mde_n)\cong 
\Hom_{R_{\mu,\de}}(\Ga_{\mu_1}\boxtimes\dots\boxtimes \Ga_{\mu_n},\Mde_\mu \otimes_{F\Si_\mu} F\Si_n).
$$
By Proposition~\ref{2.5e}(iii), the latter space is isomorphic as a right $F\Si_n$-module to $\sgn_{\Si_\mu}\otimes_{F\Si_\mu}F\Si_n$. This module is annihilated by the (right) multiplication by ${\mathtt x}_n$. Therefore the right multiplication by ${\mathtt x}_n$ annihilates the image of any non-zero homomorphism from $\Ga_{\mu_1}\circ\dots\circ \Ga_{\mu_n}$ to $\Mde_n$. On the other hand, the right multiplication by ${\mathtt x}_n$ acts as an automorphism of $\Lade_n=\Mde_n{\mathtt x}_n$ since $F$ has characteristic zero. This implies that $\Hom_{R_{n\de}}(\Ga_{\mu_1}\circ\dots\circ \Ga_{\mu_n},\Lade_n)=0$. 
\end{proof}

Recall that for a composition $\mu\vDash n$ we denote by $\mu^+\vdash n$ the unique partition obtained from $\mu$ by a permutation of its parts. 

\begin{Corollary} \label{CGGFrLade}
Let $\la\vdash n$ and $\mu\vDash n$. Then $\gg^\mu$ appears in $
\CH\Lade^{\la^\tr}$ 
with the coefficient $c(\mu)\co_{\la,\mu}$. In particular, $\gg^\mu$ appears in $
\CH\Lade^{\la^\tr}$ 
with the same coefficient as $\gg^{\mu^+}$.
\end{Corollary}
\begin{proof}
Let $\la^\tr=(l_1,\dots,l_a)$. Then in view of Lemma~\ref{L1^n}, we have 
$$\Lade^{\la^\tr}=\Lade_{l_1}\circ\dots\circ \La_{l_a}=\De(1^{l_1})\circ\dots\circ\De(1^{l_a}).$$ 
So if $\gg^\mu$ appears in $
\CH\Lade^{\la^\tr}$, then $\gg^\mu$ appears in the shuffle product $\bj^{(1)}\circ\dots\circ \bj^{(a)}$, where $\bj^{(c)}$ is a word of $\De(1^{l_c})$ for all $c=1,\dots,a$. 
By Lemma~\ref{LCombGG}, each $\bj^{(c)}$ is a Gelfand-Graev word (of type $\bi$). By Lemma~\ref{L1^n}, we have $\bj^{(c)}=\gg^{(1^{l_c})}=\bi^{l_c}$ for $c=1,\dots,a$. 
Now the result follows from  Proposition~\ref{PGGinShuffle}. The second statement comes by noticing that $\co_{\la,\mu}=\co_{\la,\mu^+}$ and $c(\mu)=c(\mu^+)$. 
\end{proof}

Recall the Kostka numbers $K_{\la,\mu}$ from (\ref{EKostka}). The matrix $K:=(K_{\la,\mu})_{\la,\mu\vdash n}$ is unitriangular, in particular it is invertible. Let 
$$N=(N_{\la,\mu})_{\la,\mu\vdash n}:=K^{-1}$$ be the inverse matrix.

\begin{Lemma} \label{LLaDeRel}
Let $\la\vdash n$. We have: 
\begin{enumerate}
\item[{\rm (i)}] $\displaystyle \CH\Lade^{\la^\tr}=\sum_{\mu\vdash n}K_{\mu^\tr,\la^\tr}\,\CH\Dede(\mu)=\CH\Dede(\la)+\sum_{\mu<\la}K_{\mu^\tr,\la^\tr}\,\CH\Dede(\mu)$; 
\item[{\rm (ii)}] $\displaystyle \CH\Dede(\la)=\sum_{\mu\vdash n}N_{\mu^\tr,\la^\tr}\,\CH\Lade^{\mu^\tr}=\CH\Lade^{\la^\tr}+\sum_{\mu<\la}N_{\mu^\tr,\la^\tr}\,\CH\Lade^{\mu^\tr}$.
\end{enumerate}
\end{Lemma}
\begin{proof}
By Lemma~\ref{L1^n} and Theorem~\ref{4.1c}, the characters of $\Lade^\mu$ and $\Dede(\mu)$ are independent of the ground field for all $\mu\vdash n$. So 
we may assume that $F=\C$, in which case $\Dede(\mu)=\nade(\mu)=\Tde(\mu)=\Lde(\mu)$ for all $\mu\vdash n$. Now (i) follows from Theorem~\ref{4.5f}(ii), and (ii) follows from (i). 
\end{proof}

We can now determine the multiplicity of any Gelfand-Graev word in the standard module $\De(\la)$. We refer to this partial character information as  the {\em Gelfand-Graev fragment}\, of the character.

\begin{Theorem} \label{TGGFragmDede}
Let $\la\vdash n$ and $\mu\vdash n$, and $\nu\vDash n$. Then:
\begin{enumerate}
\item[{\rm (i)}] $\DIM \De(\la)_{\gg^\nu}=\DIM \De(\la)_{\gg^{\nu^+}}$;
\item[{\rm (ii)}] $\DIM \De(\la)_{\gg^\mu}=c(\mu)K_{\la,\mu}.$
\end{enumerate} 
\end{Theorem}
\begin{proof}
(i) By Lemma~\ref{LLaDeRel}(ii), we have
$$\DIM \De(\la)_{\gg^\nu}=\sum_{\mu\vdash n}N_{\mu^\tr,\la^\tr}\,\DIM (\Lade^{\mu^\tr})_{\gg^\nu}.$$
So it suffices to prove that $\DIM (\Lade^{\mu^\tr})_{\gg^\nu}=\DIM (\Lade^{\mu^\tr})_{\gg^{\nu^+}}$ for all $\mu\vdash n$. But this is contained in Corollary~\ref{CGGFrLade}.

(ii) Using Lemma~\ref{LLaDeRel}(ii), Corollary~\ref{CGGFrLade} and Lemma~\ref{LColNK}, we get:
\begin{align*}
\DIM \De(\la)_{\gg^\mu}&=\sum_{\nu\vdash n}N_{\nu^\tr,\la^\tr}\,\DIM (\Lade^{\nu^\tr})_{\gg^\mu}
\\
&=\sum_{\nu\vdash n}N_{\nu^\tr,\la^\tr}\,c(\mu)\co_{\nu,\mu}
\\
&=\sum_{\nu,\kappa\vdash n}N_{\nu^\tr,\la^\tr}\,c(\mu)K_{\kappa^\tr,\mu}K_{\kappa,\nu^\tr}
\\
&=\sum_{\kappa\vdash n}\,c(\mu)K_{\kappa^\tr,\mu}\de_{\kappa,\la^\tr}=c(\mu)K_{\la,\mu},
\end{align*}
where we have used that $N=K^{-1}$. 
\end{proof}

We can extend the above result to the Gelfand-Graev fragments of characters of other imaginary modules:

\begin{Corollary} \label{CCharLLa}
Let $\la\vdash n$ and $\nu\vDash n$, and suppose that $h\geq n$, $W\in\mod{S_{h,n}}$, and $M=\beta_{h,n}(W)\in\mod{\ImS_n}$. Then 
$$\DIM M_{\gg^\nu}=c(\nu)\dim e(\nu)W.$$ In particular, 
$\DIM L(\la)_{\gg^\nu}=c(\nu)k_{\la,\nu}.$
\end{Corollary}
\begin{proof}
Note that $\{\CH\De_h(\mu)\mid\mu\vdash n\}$ is a linear basis in the character ring of modules in $\mod{S_{h,n}}$. 
So we can write $\ch W=\sum_{\mu\vdash n}n_\mu\ch\De_h(\mu)$ for some $n_\mu\in\Z$. Applying the Morita-equivalence $\be_{h,n}$, we then also have 
$$\CH M=\sum_{\mu\vdash n}n_\mu\CH \De(\mu).$$ 
So, applying Theorem~\ref{TGGFragmDede}, 
\begin{align*}
\DIM M_{\gg^\nu} &= \sum_{\mu\vdash n}n_\mu\DIM \De(\mu)_{\gg^\nu}
= \sum_{\mu\vdash n}n_\mu c(\nu) \dim e(\nu)\De_h(\mu)
\\
&
=c(\nu)\dim e(\nu)W, 
\end{align*}
as required.
\end{proof}

\section{Imaginary Jacobi-Trudy formula}
The formal characters of standard modules $\De(\la)$ in terms of the characters of the modules $\De(1^m)$ can in principle be found from Lemma~\ref{LLaDeRel}, since the modules $\Lade^\nu$ are just  $\De(1^{n_1})\circ\dots\circ\De(1^{n_a})$ if $\nu=(n_1,\dots,n_a)$. A standard way of dealing with the inverse matrix $N=K^{-1}$ is through Jacobi-Trudy formulas. 

Let $\la=(l_1,\dots,l_a)\vdash n$. Note by Corollary~\ref{CCircComm} that 
$$\CH\De(1^k)\circ\CH\De(1^l)=\CH\De(1^l)\circ \CH\De(1^k)\qquad(k,l\in\Z_{>0}).
$$
So we can use the quantum shuffle product to make sense of the  following determinant as an element of $\A \words_{n\de}$: 
$$
\JTD(\la):=\det\big(\CH\De(1^{l_r-r+s})\big)_{1\leq r,s\leq a} \in \A \words_{n\de}.
$$
where $\CH\De(1^{0})$ is interpreted as (a multiplicative) identity, and $\CH\De(1^{m})$ is interpreted as (a multiplicative) $0$ if $m<0$. 
For example, if $\la=(3,1,1)$, we get
\begin{align*}
\JTD((3,1,1))
=&\det 
\left(
\begin{matrix}
\CH\De(1^{3}) & \CH\De(1^{4})& \CH\De(1^{5})  \\
1 & \CH\De(1)& \CH\De(1^{2})\\
0 & 1 & \CH\De(1)
\end{matrix}
\right)
\\
=&\CH\De(1^{3})\circ\CH\De(1)\circ \CH\De(1)+\CH\De(1^{5})
\\&-\CH\De(1^{4})\circ\CH\De(1)-\CH\De(1^{3})\circ \CH\De(1^{2}).
\end{align*}

\begin{Remark} 
{\rm 
The characters of the modules $\De(1^m)$ are well-understood in many situations. Let $i$ be the color of the tensor space we are working with, and $\bi=(i_1,\dots,i_e)$ so that $i_e=i$. Then $\bi^m$ is a word of $\De(1^m)=L(1^m)$, see Lemma~\ref{L1^n}. Oftentimes the word  $\bi^m$ is homogeneous, and so $\De(1^m)$ is the homogeneous irreducible module associated to the connected component of $\bi^m$ in the word graph. 
For example in Lie type $\Car={\tt A}_l^{(1)}$ this is {\em always}  the case.  
}
\end{Remark}

\begin{Theorem} \label{TDet} 
{\bf (Imaginary Jacobi-Trudy Formula)} 
Let $\la\vdash n$. Then
$
\CH\Dede(\la)=\JTD(\la^\tr).
$
\end{Theorem}
\begin{proof}
Let $\la^\tr=(l_1,\dots,l_a)$ and take $h\geq n$. It follows from the classical Jacobi-Trudy formula that in the Grothendieck group of $\mod{S_{h,n}}$ we have
$$
[\De_h(\la)]=\det([\La^{(l_r-r+s)}(V_h)])_{1\leq r,s\leq a}
$$
with determinant defined using multiplication given by tensor product. Applying the equivalence of categories $\be_{h,n}$, Theorem~\ref{4.2a}, and Lemmas~\ref{3.5d}, \ref{L1^n} we get 
$$
[\De(\la)]=\det([\Lade_{l_r-r+s}])_{1\leq r,s\leq a}=\det([\De(1^{l_r-r+s})])_{1\leq r,s\leq a}
$$
with determinant defined using multiplication given by induction product `$\circ$'.
Passing to the formal characters, we obtain the required formula. 
\end{proof}

\chapter{Imaginary tensor space for non-simply-laced types}\label{ChLast}

In this section we construct the minuscule \(R_{\de}\)-modules \(L_{\de,i}\) of color \(i\) for a Cartan matrix \(\Car\) of non-simply-laced type, along with the endomorphisms \(\tau_r\) of \(M_n = L_{\de,i}^{\circ n}\) that satisfy the Coxeter relations of the symmetric group \(\frak{S}_n\).

\section{Minuscule representations for non-simply-laced types}\label{ConLNS}

Write \(R_\alpha^z = \O[z] \otimes_{\O} R_\alpha\), where \(z\) is an indeterminate element of degree 2. In the spirit of \cite[Section 1.3]{KKK}, we construct in each case an \(R_\delta^z\)-module \(L_{\de,i}^z\) with quotient isomorphic to  \(L_{\de,i}\) as an \(R_\delta\)-module. Kang, Kashiwara and Kim were able to construct this \(z\)-deformation of arbitrary modules for simply-laced types in general, but for the non-simply-laced types considered below we are forced to be more explicit. This approach allows us to construct the nonzero map \(\tau\) out of the map \(R\) defined in \cite[Chapter 1]{KKK} which satisfies Coxeter relations, but unfortunately happens to be zero on \(L_{\de,i} \circ L_{\de,i}\). 

\subsection{Construction of $L_{\de,i}^z$ in type \({\tt B}_l^{(1)}\)}\label{BlCons}
We label the vertices of the diagram \({\tt B}_l^{(1)}\) as shown below:
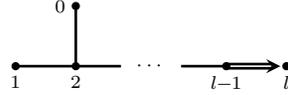
\begin{align*}
\begin{braid}\tikzset{baseline=3mm}
\draw[black](2,2) node[left]{\scriptsize $0$}--(2,0);
\draw[black](0,0) node[below]{\scriptsize $1$}--(3.5,0);
\draw[black](2,0) node[below]{\scriptsize $2$};
\draw[black](4.5,0) node {\scriptsize $\cdots$};
\draw[black](7,0) node[below]{\scriptsize $l\hspace{-0.8mm}-\hspace{-0.8mm}1$}--(5.5,0);
\draw[black,double,->,>=stealth](7,0)--(8.8,0);
\draw[black](9,0) node[below]{\scriptsize $l$};
\blackdot(2,2);
\blackdot(0,0);
\blackdot(2,0);
\blackdot(7,0);
\blackdot(9,0);
\end{braid}
\end{align*}
Fix \(i \in I'\). Define
\begin{align*}
\bi:= \begin{cases}
(0,2,3, \ldots, l, l, l-1, \ldots, i+1, 1, 2, \ldots, i) &\textup{if }i<l;\\
(0,2,3, \ldots, l,1,2,\ldots, l) &\textup{if }i=l.
\end{cases}
\end{align*}
Let \(C_{\bi}\) be the connected component of \(\bi\) in the weight graph \(G_\delta\). For \(\bj \in C_{\bi}\), \(1 \leq k  \leq 2l\), define constants
\begin{align*}
\xi_{\bj,r} = \begin{cases} 1 & \textup{if \(j_r = l\) and \(j_t \neq l\) for all \(t<r\)}\\ 
-1 & \textup{if \(j_k = l\) and \(j_t = l\) for some \(t<r\)}\\
0 & \textup{otherwise}\end{cases}
\end{align*}
Let \(L_{\de,l}^z\) be the graded free \(\O[z]\)-module on basis
 \begin{align*}
 B = \{ v_{\bj} \mid \bj \in C_{\bi}\},
 \end{align*}
where each \(v_{\bj}\) is in degree 0. For \(i \neq l\), let \(L_{\de,i}^z\) be the graded free \(\O[z]\)-module on basis
  \begin{align*}
 B= \{ v_{\bj}^{(c)} \mid c \in \{\pm 1\}, \bj \in C_{\bi}\},
 \end{align*}
 where \(v_{\bj}^{(c)}\) is in degree \(c\).
 
Define an action of generators of \(R_\delta^z\) on \(L_{\de,l}^z\) as follows:
\begin{align*}
1_{\bk}v_{\bj}&:=\delta_{\bj,\bk} v_{\bj}\\
y_r v_{\bj}&:= \begin{cases} z^2v_{\bj} & \textup{if } \xi_{\bj,r}=0\\
\xi_{\bj,r}zv_{\bj} &\textup{if } \xi_{\bj,r} \neq 0
\end{cases}\\
\psi_r v_{\bj}&:= \begin{cases} v_{s_r \bj}& \textup{if } j_r \cdot j_{r+1} = 0\\
0 & \textup{otherwise}.
\end{cases}
\end{align*}
If \(i \neq l\), define an action of generators of  \(R_\delta^z\) on \(L_{\de,i}^z\) as follows:
\begin{align*}
1_{\bk}v_{\bj}^{(c)}&:=\delta_{\bj,\bk} v_{\bj}^{(c)}\\
y_r v_{\bj}^{(c)}&:= \begin{cases} z^2v_{\bj}^{(c)} & \textup{if } \xi_{\bj,r}=0\\
\xi_{\bj,r}zv_{\bj}^{(1)} &\textup{if } \xi_{\bj,r} \neq 0, c = 1\\
-\xi_{\bj,r}\left(zv_{\bj}^{(-1)} + v_{\bj}^{(1)}\right) & \textup{if } \xi_{\bj,r} \neq 0, c =-1\\
\end{cases}\\
\psi_r v_{\bj}^{(c)}&:= \begin{cases} v_{s_r \bj}^{(c)}& \textup{if } j_r \cdot j_{r+1} = 0\\
v_{\bj}^{(-1)} &\textup{if }c=1, j_r = j_{r+1} = l \\
0 & \textup{otherwise}.
\end{cases}
\end{align*}

\begin{Proposition}\label{Bnrels}
The formulas above define a graded action of \(R_\delta^z\) on \(L_{\de,i}^z\), and \(L_{\de, i}^z\) is \(\O[z]\)-free on basis \(B\).
\end{Proposition}

\begin{proof}
Assume \(i < l\). We check that the given action agrees with the algebra relations 
(\ref{R1})--(\ref{R7}). That relations (\ref{R1}), (\ref{R2PsiE}), and (\ref{R3Psi}) are satisfied is clear. For purposes of checking relations we will have cause to combine cases to write the action of \(y_k\) and \(\psi_k\) as
\begin{align*}
y_r v_{\bj}^{(c)}&= [(1-\delta_{j_r,l})z^2 + \xi_{\bj,r}cz]v_{\bj}^{(c)} - \delta_{c,-1}\xi_{\bj,r}v_{\bj}^{(1)}\\
\psi_r v_{\bj}^{(c)}&=\delta_{j_r \cdot j_{r+1},0}v_{s_r\bj}^{(c)} + \delta_{c,1}\delta_{j_r,j_{r+1}}\delta_{j_r,l}v_{\bj}^{(-1)}.
\end{align*}
We omit idempotents \(1_{\bj}\) in the below, considering only \(\bj \in C_{\bi}\), as other idempotents act as zero and thus cause the the remaining relations to be satisfied.

Relation (\ref{R2PsiY}). We have
\begin{align*}
y_ty_r v_{\bj}^{(c)}:=& y_t \left[ [(1-\delta_{j_r,l})z^2 + \xi_{\bj,r}cz]v_{\bj}^{(c)} - \delta_{c,-1}\xi_{\bj,r}v_{\bj}^{(1)} \right]\\
=& [(1-\delta_{j_r,l})z^2 + \xi_{\bj,r}cz] \left[[(1-\delta_{j_t,l})z^2 + \xi_{\bj,t}cz]v_{\bj}^{(c)} - \delta_{c,-1} \xi_{\bj,t}v_{\bj}^{(1)}\right] \\
&\hspace{10mm} - \delta_{c,-1} \xi_{\bj,r} \left[         [(1-\delta_{j_t,l})z^2 + \xi_{\bj,t}cz]v_{\bj}^{(1)}  \right]\\
=&  [(1-\delta_{j_r,l})z^2 + \xi_{\bj,r}cz] [(1-\delta_{j_t,l})z^2 + \xi_{\bj,t}cz]v_{\bj}^{(c)} \\
&\hspace{10mm} - \delta_{c,-1} \left[ (1-\delta_{j_r,l})\xi_{\bj,t}z^2 + (1-\delta_{j_t,l})\xi_{\bj,r}z^2 + 2 \xi_{\bj,r} \xi_{\bj,t}\right]v_{\bj}^{(1)},
\end{align*}
which is invariant under the exchange of \(r\) and \(t\), hence \(y_ty_rv_{\bj}^{(c)} = y_ry_tv_{\bj}^{(c)}\).

Relation (\ref{R6}). Assume \(j_r \cdot j_{r+1} = 0\). Then
\begin{align*}
y_t \psi_r v_{\bj}^{(c)} &= [(1-\delta_{(s_r \bj)_t,l})z^2 + \xi_{s_r \bj, t}cz]v_{s_r \bj}^{(c)} - \delta_{c,-1} \xi_{s_r \bj,t}v_{s_r \bj}^{(1)}\\
\psi_r y_{s_rt}v_{\bj}^{(c)} &= [(1-\delta_{j_{s_rt,l}})z^2 + \xi_{\bj, s_rt} cz]v_{s_r\bj}^{(c)} - \delta_{c,-1}\xi_{\bj, s_rt} v_{s_r\bj}^{(1)}
\end{align*}
We have \((s_r \bj)_t = j_{s_rt}\), and since \(j_r \neq j_{r+1}\), we have \(\xi_{\bj,s_rt} = \xi_{s_r\bj, t}\). Then 
\begin{align*}
(y_t \psi_r - \psi_r y_{s_rt})v_{\bj}^{(c)} = 0 = \delta_{j_r, j_{r+1}}(\delta_{t, r+1} - \delta_{t,r})v_{\bj}^{(c)}.
\end{align*}
Next assume  \(j_r=j_{r+1} = l\), \(c=1\). Then
\begin{align*}
y_t \psi_r v_{\bj}^{(1)}&= [(1- \delta_{j_t,l})z^2 - \xi_{\bj,t}]v_{\bj}^{(-1)} - \xi_{\bj,t}v_{\bj}^{(1)}\\
\psi_r y_{s_rt} v_{\bj}^{(1)} &= [(1- \delta_{j_{s_rt},l})z^2 + \xi_{\bj,s_rt}z]v_{\bj}^{(-1)}.
\end{align*}
There are a few cases to consider:
\begin{enumerate}
\item \(t \neq r, r+1\). Then \(\xi_{\bj,s_rt} = \xi_{\bj,t} = \delta_{j_t,l} = \delta_{j_{s_rt},l} = 0\), so 
\begin{align*}
(y_t \psi_r - \psi_r y_{s_rt})v_{\bj}^{(c)} = 0 = \delta_{j_r, j_{r+1}}(\delta_{t,r+1} - \delta_{t,r})v_{\bj}^{(c)}.
\end{align*}
\item \(t = r\). Then \(\xi_{\bj,t} = \delta_{j_t,l} = \delta_{j_{s_rt},l} = 1\), \(\xi_{\bj,s_rt} = -1\), so
\begin{align*}
(y_t \psi_r - \psi_r y_{s_rt})v_{\bj}^{(c)} = -v_{\bj}^{(1)} = \delta_{j_r, j_{r+1}}(\delta_{t,r+1} - \delta_{t,r})v_{\bj}^{(c)}.
\end{align*}
\item \(t = r+1\). Then \(\xi_{\bj,s_rt} = \delta_{j_t,l} = \delta_{j_{s_rt},l} = 1\), \(\xi_{\bj,t} = -1\), so
\begin{align*}
(y_t \psi_r - \psi_r y_{s_rt})v_{\bj}^{(c)} = v_{\bj}^{(1)} = \delta_{j_r, j_{r+1}}(\delta_{t,r+1} - \delta_{t,r})v_{\bj}^{(c)}.
\end{align*}
\end{enumerate}
Now assume \(j_r = j_{r+1} = l\), \(c=-1\). Then 
\begin{align*}
y_t \psi_r v_{\bj}^{(-1)} &= 0\\
\psi_r y_{s_rt} v_{\bj}^{(-1)} &= - \xi_{\bj, s_r t} v_{\bj}^{(-1)}.
\end{align*}
If \(t \neq r, r+1\), then \(\xi_{\bj, s_rt} = 0\). If \(t = r\), then \(\xi_{\bj, s_rt}=-1\), and \(t = r+1\), then \(\xi_{\bj,s_rt} =1\). In all these cases (3.2.8) is satisfied.\\
\indent This leaves \(j_r \cdot j_{r+1} \neq 0\), with \(j_r \neq 0\) or \(j_{r+1} \neq 0\). Then \(y_t \psi_r v_{\bj}^{(c)} = \psi_r y_{s_rt}v_{\bj}^{(c)} = 0\). For \(\bj \in C_{\bi}\), \(j_r = j_{r+1}\) implies \(j_r = j_{r+1} = l\), so we have \(j_r \neq j_{r+1}\). Then \(\delta_{j_r,j_{r+1}}(\delta_{t,r+1} - \delta_{t,r})v_{\bj}^{(c)} = 0\), so in all cases (\ref{R6}) is satisfied.

Relation (\ref{R4}). If \(j_r \cdot j_{r+1} \geq 0\), then this relation is clearly satisfied. So assume \(j_r \cdot j_{r+1} < 0\). In this case \(\psi_r^2 v_{\bj}^{(c)} = 0\). Then we just check a few cases:
\begin{enumerate}
\item \(j_r, j_{r+1} \neq l\). Then
\begin{align*}
Q_{j_r,j_{r+1}}(y_r,y_{r+1})v_{\bj}^{(c)} = \epsilon_{j_r,j_{r+1}}(y_r - y_{r+1}) v_{\bj}^{(c)} = \epsilon_{j_r,j_{r+1}}[ z^2v_{\bj}^{(c)} - z^2 v_{\bj}^{(c)}]=0.
\end{align*}
\item \(j_r=l-1, j_{r+1} = l\). Then since \(\bj \in C_{\bi}\), we have \(\xi_{\bj,r}=0, \xi_{\bj,r+1}=1\), and
\begin{align*}
Q_{j_r,j_{r+1}}(y_r,y_{r+1})v_{\bj}^{(c)} &= \epsilon_{l-1,l}(y_r - y_{r+1}^2)v_{\bj}^{(c)}\\
&= \epsilon_{l-1,l}[z^2v_{\bj}^{(c)} - y_{r+1}(czv_{\bj}^{(c)} - \delta_{c,-1}v_{\bj}^{(1)})]\\
&= \epsilon_{l-1,l}[ z^2v_{\bj}^{(c)} - (cz(czv_{\bj}^{(c)} - \delta_{c,-1}v_{\bj}^{(1)}) - \delta_{c,-1}zv_{\bj}^{(1)})] = 0.
\end{align*}
\item \(j_r = l, j_{r+1} =l-1\). Then since \(\bj \in C_{\bi}\), we have  \(\xi_{\bj,r}=-1, \xi_{\bj,r+1}=0\), and 
\begin{align*}
Q_{j_r,j_{r+1}}(y_r,y_{r+1})v_{\bj}^{(c)} &= 
\epsilon_{l,l-1}(y_r^2 - y_{r+1})v_{\bj}^{(c)}\\
&= \epsilon_{l,l-1}[ y_r(-czv_{\bj}^{(c)} + \delta_{c,-1}v_{\bj}^{(1)})               -z^2v_{\bj}^{(c)}] \\
&= \epsilon_{l,l-1}[ (-cz(-czv_{\bj}^{(c)} + \delta_{c,-1}v_{\bj}^{(1)}) - \delta_{c,-1}v_{\bj}^{(1)}) - z^2v_{\bj}^{(c)}]=0.
\end{align*}
\end{enumerate}
\noindent Thus relation (\ref{R4}) is satisfied.

Relation (\ref{R7}). We have
\begin{align*}
\psi_r \psi_{r+1} \psi_r v_{\bj}^{(c)}
&= \delta_{j_r \cdot j_{r+1},0} \delta_{j_r\dot j_{r+2}, 0} \delta_{j_{r+1} \cdot j_{r+2},0} v_{s_r s_{r+1}s_r \bj}^{(c)}\\
&\hspace{10mm}+ \delta_{c,1} \delta_{j_r \cdot j_{r+1}, 0} \delta_{j_{r+1},j_{r+2}} \delta_{j_{r+1},l} v_{s_{r+1}s_r \bj}^{(-1)}\\
&\hspace{20mm} + \delta_{c,1}\delta_{j_r \cdot j_{r+1},0}\delta_{j_r,j_{r+2}} \delta_{j_r,l} v_{\bj}^{(-1)}\\
&\hspace{30mm}+ \delta_{c,1}  \delta_{j_r \cdot j_{r+2},0} \delta_{j_r,j_{r+1}}\delta_{j_r,l}v_{s_r s_{r+1} \bj}^{(-1)}\\
\psi_{r+1} \psi_{r} \psi_{r+1} v_{\bj}^{(c)}
&= \delta_{j_r \cdot j_{r+1},0} \delta_{j_r\dot j_{r+2}, 0} \delta_{j_{r+1} \cdot j_{r+2},0} v_{s_{r+1} s_{r}s_{r+1} \bj}^{(c)}\\
&\hspace{10mm}+ \delta_{c,1} \delta_{j_r \cdot j_{r+1}, 0} \delta_{j_{r+1},j_{r+2}} \delta_{j_{r+1},l} v_{s_{r+1}s_r \bj}^{(-1)}\\
&\hspace{20mm} + \delta_{c,1}\delta_{j_r \cdot j_{r+1},0}\delta_{j_r,j_{r+2}} \delta_{j_r,l} v_{\bj}^{(-1)}\\
&\hspace{30mm}+ \delta_{c,1}  \delta_{j_r \cdot j_{r+2},0} \delta_{j_r,j_{r+1}}\delta_{j_r,l}v_{s_r s_{r+1} \bj}^{(-1)}.
\end{align*}
So, since \(s_rs_{r+1}s_r = s_{r+1} s_r s_{r+1}\), we have 
\begin{align*}
(\psi_{r+1}\psi_r \psi_{r+1} - \psi_r \psi_{r+1} \psi_r)v_{\bj}^{(c)} = 0.
\end{align*}
Thus, if \(j_r \neq j_{r+2}\), relation (\ref{R7}) is satisfied. If \(\bj \in C_{\bi}\) is such that \(j_r = j_{r+2}\), it follows that \(j_r = j_{r+2} = l\), and \(j_r \cdot j_{r+1} = 0\). Then \(Q_{j_r,j_{r+1}} = 1\), and again (\ref{R7}) is satisfied.

The relations in case \(i=l\) are more easily verified by similar computations. 
\end{proof}

\subsection{Construction of $L_{\de,i}^z$ in type \({\tt C}_l^{(1)}\)}\label{ClCons}
We label the vertices of the diagram \({\tt C}_l^{(1)}\) as shown below:
\begin{align*}
\begin{braid}\tikzset{baseline=3mm}
\draw[black,double,->,>=stealth](-2,0) node[below]{\scriptsize $0$}--(-0.2,0);
\draw[black](0,0) node[below]{\scriptsize $1$}--(3.5,0);
\draw[black](2,0) node[below]{\scriptsize $2$};
\draw[black](4.5,0) node {\scriptsize $\cdots$};
\draw[black](7,0) node[below]{\scriptsize $l\hspace{-0.8mm}-\hspace{-0.8mm}1$}--(5.5,0);
\draw[black,double,->,>=stealth](9,0) node[below]{\scriptsize $l$}--(7.2,0);;
\blackdot(-2,0);
\blackdot(0,0);
\blackdot(2,0);
\blackdot(7,0);
\blackdot(9,0);
\end{braid}
\end{align*}

Fix \(i \in I'\). Let \(L_{\de,i,0}^z\) be the graded free 1-dimensional \(\O[z]\)-module on generator \(x_0\) (in degree 0). Define an action of \(R^z_{\alpha_0}\) on \(L_{\de,i,0}^z\) by \(1_{(0)}x_0 = x_0\), \(y_1x_0 = z^2x_0\).

Define
\begin{align*}
\bj^{(1)} &= \begin{cases}
(1,\ldots, l-1,l,l-1, \ldots i+1) &\textup{if } i<l;\\
(1,\ldots, l-1) &\textup{if } i=l. 
\end{cases}
\end{align*}
Let \(L_{\delta,i,1}^z\) be the graded free 1-dimensional \(\O[z]\)-module on generator \(x_1\) (in degree 0). Define an action of generators of \(R^z_{\delta - \alpha_0 - \cdots - \alpha_i}\) as follows:
\begin{align*}
1_{\bk} x_1 &:= \delta_{\bj^{(1)},\bk} x_1\\
y_{r}x_1 &:= \begin{cases} 
zx_1 & \textup{if } r < l;\\
z^2x_1 & \textup{if } r =l; \\
-zx_1 & \textup{if } r > l.
\end{cases}\\
\psi_r x_1 &:= 0.
\end{align*}

If \(i>1\), define
\begin{align*}
\bj^{(2)}=(1, \ldots, i-1),
\end{align*}
and let \(L_{\delta,i,2}^z\) be the graded free 1-dimensional \(\O[z]\)-module on generator \(x_2\) (in degree 0). Define an action of generators of \(R_{\alpha_1 + \cdots + \alpha_{i-1}}^z\) on \(L_{\delta,i,2}\) as follows:
\begin{align*}
1_{\bk} x_2 &:= \delta_{\bj^{(2)},\bk} x_2\\
y_{r}x_2 &:= -zx_2 \\
\psi_r x_2 &:= 0.
\end{align*}

Let \(L_{\de,i,3}^z\) be the graded free 1-dimensional \(\O[z]\)-module on generator \(x_3\) (in degree 0). Define an action of \(R^z_{\alpha_i}\) on \(L_{\de,i,3}\) by \begin{align*}
1_{(i)}x_3&= x_3\\
y_1 x_3 &= \begin{cases}
-zx_3 & \textup{if }i <l;\\
z^2x^3 & \textup{if }i=l.
\end{cases}
\end{align*}

\begin{Proposition}\label{Cnrels1}
The formulas above define a graded action of the algebras \(R_{\alpha_0}^z\), \(R_{\delta - \alpha_0 - \cdots - \alpha_i}^z\), \(R_{\alpha_1 + \cdots + \alpha_{i-1}}^z\) and \(R_{\alpha_i}^z\) on the modules \(L_{\de,i,0}^z\), \(L_{\de,i,1}^z\), \(L_{\de,i,2}^z\), and \(L_{\de,i,3}^z\) respectively, which are \(\O[z]\)-free on their respective bases. 
\end{Proposition}

\begin{proof}
It is easily checked that the given action agrees with the algebra relations 
(\ref{R1})--(\ref{R7}).
\end{proof}

Now define the \(R_{\alpha_0, \delta- \alpha_0-\alpha_i, \alpha_i}^z\)-module
\begin{align*}
L_{\de,i}^z = \begin{cases}
L_{\de,i,0} \boxtimes L_{\de,i,1}  \boxtimes L_{\de,i,3} & \textup{if }i=1;\\
L_{\de,i,0} \boxtimes L_{\de,i,1} \circ L_{\de,i,2} \boxtimes L_{\de,i,3} & \textup{if }i>1.
\end{cases}
\end{align*}
If \(i=1\), write \(x=x_0 \otimes x_1 \otimes x_3\) and \(\bj = (0,\bj^{(1)},1)\). Otherwise write \(x= x_0 \otimes x_1 \otimes x_2 \otimes x_3\) and \(\bj= (0,\bj^{(1)}, \bj^{(2)}, i)\). Then \(x\) is a word vector of word \(\bj\), and \(L_{\delta,i}^z\) is \(\O[z]\)-free on basis 
\begin{align*}
B = \{\psi_u x \mid u \in \D^{1,2l-i-1,i-1,1}_{1,2l-2,1}\}.
\end{align*} 
We are most of the way to defining an action of the standard generators of \(R_\de^z\) on \(L_{\de,i}^z\). The generators \(1_{\bk}, y_r\), and \(\psi_r\) (for \(1<r < 2l-1\)) act as already prescribed by membership in the subalgebra \(R_{\alpha_0, \delta- \alpha_0-\alpha_i, \alpha_i}^z\). It remains to define \(\psi_1 v = \psi_{2l-1}v = 0\), and \(1_{\bk}v = 0\) (if \(k_1 \neq 0\) or \(k_{2n} \neq i\)) for all \(v \in L_{\de,i}^z\).

\begin{Proposition}\label{Cnrels}
The description above defines a graded action of \(R_{\delta}^z\) on \(L_{\de,i}^z\), and \(L_{\de,i}^z\) is \(\O[z]\)-free on basis \(B\).
\end{Proposition}
\begin{proof}
We check that the given action agrees with the algebra relations (\ref{R1})--(\ref{R7}) for \(R_\delta^z\). By construction, all relations that do not involve \(\psi_1\) or  \(\psi_{2l-1}\) are easily seen to be satisfied due to the local nature of the relations and the fact that \(L_{\delta,i,1}^z \circ L_{\delta,i,2}^z\) is an \(R_{\delta-\alpha_0 - \alpha_i}^z\)-module. Additionally, all relations that involve \(1_\bk\) with \(k_1 \neq 0\) or \(k_{2l} \neq i\) are trivially satisfied. We check the rest of the relations that are not immediately obvious below, for basis vectors \(\psi_u x \in B\).

Relation (\ref{R6}). There is no \(u \in \D^{1,2l-i-1,i-1,1}_{1,2l-2,1}\) such that \((u\bj)_2 = 0\) or \((u\bj)_{2l-1}=i\), so we just need verify that the left side of the relation is always zero. Indeed, \(\psi_1\) and \(\psi_{2l-1}\) always act as zero, and the action of \(y_1\) and \(y_{2l}\) commute with all \(\psi_r\)'s by construction, giving the result. 

Relation (\ref{R4}). We have \(\psi_1^2 (\psi_u x) = 0\). Note that \((u\bj)_1 = 1\), so we check: 
\begin{align*}
Q_{0,1}(y_1,y_2) \cdot \psi_u x &= \epsilon_{01}(y_1 - y_2^2) \psi_u x
= \epsilon_{01}(z^2\psi_ux - y_2^2\psi_u x).
\end{align*}
Now either \(u(2)=2\), in which case \(y_2^2 \psi_ux = \psi_uy_2^2x = z^2 \psi_ux \), or \(u(2)>2\), in which case \(\psi_u = \psi_u' \psi_2 \cdots \psi_{2l-i} x\) for some \(\psi_{u'}\) that does not involve \(\psi_1\) or \(\psi_2 \). Then
\begin{align*}
y_2^2 \psi_u x &= y_2^2\psi_u' \psi_2 \cdots \psi_{2l-i} x = \psi_{u'} y_2^2\psi_1 \cdots \psi_{2l-i} x.
\end{align*}
Diagrammatically, \(y_2^2\psi_2 \cdots \psi_{2l-i} x\) looks like 
$$
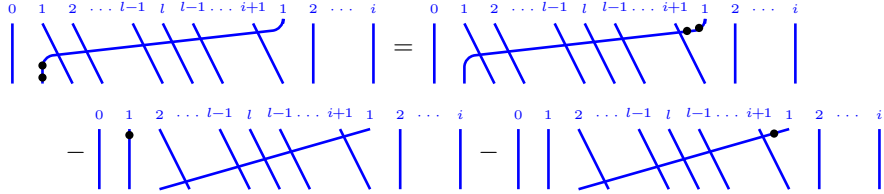
\begin{braid}\tikzset{baseline=0mm}

\draw(0,1) node[above]{$0$}--(0,-1);
\draw(1,1) node[above]{$1$}--(2,-1);
\draw(2,1) node[above]{$2$}--(3,-1);
\draw(3,1) node[above]{$\cdots$};
\draw(4,1) node[above]{$l\hspace{-0.8mm}-\hspace{-0.8mm}1$}--(5,-1);
\draw(5,1) node[above]{$l$}--(6,-1);
\draw(6,1) node[above]{$l\hspace{-0.8mm}-\hspace{-0.8mm}1$}--(7,-1);
\draw(7,1) node[above]{$\cdots$};
\draw(8,1) node[above]{$i\hspace{-0.8mm}+\hspace{-0.8mm}1$}--(9,-1);
\draw(9,1) node[above]{$1$}--(9,0.8)--(1,-0.1)--(1,-1);
\draw(10,1) node[above]{$2$}--(10,-1);
\draw(11,1) node[above]{$\cdots$};
\draw(12,1) node[above]{$i$}--(12,-1);

\blackdot(1,-0.8);
\blackdot(1,-0.4);

\end{braid}
=
\begin{braid}\tikzset{baseline=0mm}

\draw(0,1) node[above]{$0$}--(0,-1);
\draw(1,1) node[above]{$1$}--(2,-1);
\draw(2,1) node[above]{$2$}--(3,-1);
\draw(3,1) node[above]{$\cdots$};
\draw(4,1) node[above]{$l\hspace{-0.8mm}-\hspace{-0.8mm}1$}--(5,-1);
\draw(5,1) node[above]{$l$}--(6,-1);
\draw(6,1) node[above]{$l\hspace{-0.8mm}-\hspace{-0.8mm}1$}--(7,-1);
\draw(7,1) node[above]{$\cdots$};
\draw(8,1) node[above]{$i\hspace{-0.8mm}+\hspace{-0.8mm}1$}--(9,-1);
\draw(9,1) node[above]{$1$}--(9,0.8)--(1,-0.1)--(1,-1);
\draw(10,1) node[above]{$2$}--(10,-1);
\draw(11,1) node[above]{$\cdots$};
\draw(12,1) node[above]{$i$}--(12,-1);

\blackdot(8.8,0.85);
\blackdot(8.4,0.75);

\end{braid}
\hspace{1cm}
$$
$$
\hspace{1cm}
-
\begin{braid}\tikzset{baseline=0mm}

\draw(0,1) node[above]{$0$}--(0,-1);
\draw(1,1) node[above]{$1$}--(1,-1);
\draw(2,1) node[above]{$2$}--(3,-1);
\draw(3,1) node[above]{$\cdots$};
\draw(4,1) node[above]{$l\hspace{-0.8mm}-\hspace{-0.8mm}1$}--(5,-1);
\draw(5,1) node[above]{$l$}--(6,-1);
\draw(6,1) node[above]{$l\hspace{-0.8mm}-\hspace{-0.8mm}1$}--(7,-1);
\draw(7,1) node[above]{$\cdots$};
\draw(8,1) node[above]{$i\hspace{-0.8mm}+\hspace{-0.8mm}1$}--(9,-1);
\draw(9,1) node[above]{$1$}--(2,-1);
\draw(10,1) node[above]{$2$}--(10,-1);
\draw(11,1) node[above]{$\cdots$};
\draw(12,1) node[above]{$i$}--(12,-1);

\blackdot(1,0.8);

\end{braid}
-
\begin{braid}\tikzset{baseline=0mm}

\draw(0,1) node[above]{$0$}--(0,-1);
\draw(1,1) node[above]{$1$}--(1,-1);
\draw(2,1) node[above]{$2$}--(3,-1);
\draw(3,1) node[above]{$\cdots$};
\draw(4,1) node[above]{$l\hspace{-0.8mm}-\hspace{-0.8mm}1$}--(5,-1);
\draw(5,1) node[above]{$l$}--(6,-1);
\draw(6,1) node[above]{$l\hspace{-0.8mm}-\hspace{-0.8mm}1$}--(7,-1);
\draw(7,1) node[above]{$\cdots$};
\draw(8,1) node[above]{$i\hspace{-0.8mm}+\hspace{-0.8mm}1$}--(9,-1);
\draw(9,1) node[above]{$1$}--(2,-1);
\draw(10,1) node[above]{$2$}--(10,-1);
\draw(11,1) node[above]{$\cdots$};
\draw(12,1) node[above]{$i$}--(12,-1);

\blackdot(8.5,0.85);

\end{braid}
$$
where we picture the vector \(x\) as being at the top of each diagram. But then this is
\begin{align*}
z^2\psi_2 \cdots \psi_{2l-i} x - z\psi_3 \cdots \psi_{2l-i}x +z\psi_3 \cdots \psi_{2l-i}x = z^2\psi_2 \cdots \psi_{2l-i} x,
\end{align*}
so \(y_2^2 \psi_ux = z^2 \psi_ux \), and thus \(Q_{0,1}(y_1,y_2) \cdot \psi_u x =0\) in any case.

On the other side, we have \(\psi_{2l-1}^2(\psi_u x)=0\). Note that either \((u \bj)_{2l-1} = i-1\) or \((u \bj)_{2l-1}=i+1\). The case \(i=l\) is handled similarly to the above argument. For simplicity assume \(i \leq l-2\) (again the case \(i = l-1\) is handled similarly). If \((u\bj)_{2l-1} = i-1\), then \(\psi_u\) does not involve \(\psi_{2l-2}\) or \(\psi_{2l-1}\), and thus
\begin{align*}
Q_{i-1,i}(y_{2l-1},y_{2l}) \cdot \psi_u x &= \epsilon_{i-1,i}(y_{2l-1}-y_{2l}) \cdot \psi_u x = \epsilon_{i-1,i}(y_{2l-1} \psi_u x + z  \psi_u x),
\end{align*}
but \(y_{2l-1} \psi_u x = \psi_uy_{2l-1} x = -z \psi_u x\), so this is zero. If \((u\bj)_{2l-1} = i+1\), then 
\begin{align*}
Q_{i+1,i}(y_{2l-1},y_{2l}) \cdot \psi_u x &= \epsilon_{i+1,i}(y_{2l-1}-y_{2l}) \cdot \psi_u x = \epsilon_{i+1,i}(y_{2l-1} \psi_u x + z  \psi_u x),
\end{align*}
and \(\psi_u\) can be written \(\psi_{u'} \psi_{2l-2} \cdots \psi_{2l-i}\) for some \(\psi_{u'}\) that does not involve \(\psi_{2l-2}\) or \(\psi_{2l-1}\). Then
\begin{align*}
y_{2l-1} \psi_u x = y_{2l-1} \psi_{u'} \psi_{2l-2} \cdots \psi_{2l-i} x = \psi_{u'} y_{2l-1}\psi_{2l-2} \cdots \psi_{2l-i} x.
\end{align*}
Diagrammatically, \(y_{2l-1}\psi_{2l-2} \cdots \psi_{2l-i} x\) looks like
$$
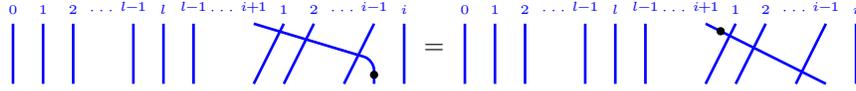
\begin{braid}\tikzset{baseline=0mm}

\draw(0,1) node[above]{$0$}--(0,-1);
\draw(1,1) node[above]{$1$}--(1,-1);
\draw(2,1) node[above]{$2$}--(2,-1);
\draw(3,1) node[above]{$\cdots$};
\draw(4,1) node[above]{$l\hspace{-0.8mm}-\hspace{-0.8mm}1$}--(4,-1);
\draw(5,1) node[above]{$l$}--(5,-1);
\draw(6,1) node[above]{$l\hspace{-0.8mm}-\hspace{-0.8mm}1$}--(6,-1);
\draw(7,1) node[above]{$\cdots$};
\draw(8,1) node[above]{$i\hspace{-0.8mm}+\hspace{-0.8mm}1$}--(12,-0.2)--(12,-1);
\draw(9,1) node[above]{$1$}--(8,-1);
\draw(10,1) node[above]{$2$}--(9,-1);
\draw(11,1) node[above]{$\cdots$};
\draw(12,1) node[above]{$i\hspace{-0.8mm}-\hspace{-0.8mm}1$}--(11,-1);
\draw(13,1) node[above]{$i$}--(13,-1);

\blackdot(12,-0.7);

\end{braid}
=
\begin{braid}\tikzset{baseline=0mm}

\draw(0,1) node[above]{$0$}--(0,-1);
\draw(1,1) node[above]{$1$}--(1,-1);
\draw(2,1) node[above]{$2$}--(2,-1);
\draw(3,1) node[above]{$\cdots$};
\draw(4,1) node[above]{$l\hspace{-0.8mm}-\hspace{-0.8mm}1$}--(4,-1);
\draw(5,1) node[above]{$l$}--(5,-1);
\draw(6,1) node[above]{$l\hspace{-0.8mm}-\hspace{-0.8mm}1$}--(6,-1);
\draw(7,1) node[above]{$\cdots$};
\draw(8,1) node[above]{$i\hspace{-0.8mm}+\hspace{-0.8mm}1$}--(12,-1);
\draw(9,1) node[above]{$1$}--(8,-1);
\draw(10,1) node[above]{$2$}--(9,-1);
\draw(11,1) node[above]{$\cdots$};
\draw(12,1) node[above]{$i\hspace{-0.8mm}-\hspace{-0.8mm}1$}--(11,-1);
\draw(13,1) node[above]{$i$}--(13,-1);

\blackdot(8.5,0.75);

\end{braid}
$$
which is \(-z\psi_{2l-2} \cdots \psi_{2l-i} x\), so \(y_{2l-1}\psi_ux = -z \psi_u x\). Thus \(Q_{i+1,i}(y_{2l-1},y_{2l}) \cdot\psi_u x = 0\).

Relation (\ref{R7}). The only case for which this relation is non-trivial is when \(i=l-1\). Indeed, in all other cases there is no \(\bk = u \bj\) where \(k_1=k_3\) or \(k_{2l-2} = k_{2l}\). When \(i=l-1\), the non-trivial case occurs when \((u \bj)_{2l-2}=l-1\), \((u\bj)_{2l-1} = l\). Then \(\psi_u\) can be written \(\psi_{u'}\psi_{2l-3} \cdots \psi_{l} \psi_{2l-2} \cdots \psi_{l+1}\) for some \(\psi_{u'}\) not involving \(\psi_{2l-3}\), \(\psi_{2l-2}\) or \(\psi_{2l-1}\). Then
\begin{align*}
&\frac{ Q_{l-1,l}(y_{2l},y_{2l-1}) - Q_{l-1,l}(y_{2l-2},y_{2l-1})}{y_{2l}-y_{2l-2}} \psi_ux \\
&= \epsilon_{l-1,l}\frac{(y_{2l}^2 - y_{2l-1}) - (y_{2l-2}^2 - y_{2l-1})}{y_{2l}-y_{2l-2}} \psi_ux\\
&= \epsilon_{l-1,l}(y_{2l} + y_{2l-2}) \cdot \psi_u x\\
&= \epsilon_{l-1,l}(-z \psi_ux + y_{2l-2}\psi_ux)\\
&= \epsilon_{l-1,l}(-z \psi_ux + \psi_{u'}y_{2l-2}\psi_{2l-3} \cdots \psi_{l} \psi_{2l-2} \cdots \psi_{l+1}x)
\end{align*}
Diagrammatically, \(y_{2l-2}\psi_{2l-3} \cdots \psi_{l} \psi_{2l-2} \cdots \psi_{l+1}x\) looks like
$$
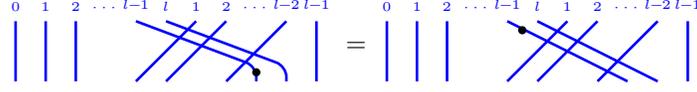
\begin{braid}\tikzset{baseline=0mm}
\draw(0,1) node[above]{$0$}--(0,-1);
\draw(1,1) node[above]{$1$}--(1,-1);
\draw(2,1) node[above]{$2$}--(2,-1);
\draw(3,1) node[above]{$\cdots$};
\draw(4,1) node[above]{$l\hspace{-0.8mm}-\hspace{-0.8mm}1$}--(8,-0.5)--(8,-1);
\draw(5,1) node[above]{$l$}--(9,-0.5)--(9,-1);
\draw(6,1) node[above]{$1$}--(4,-1);
\draw(7,1) node[above]{$2$}--(5,-1);
\draw(8,1) node[above]{$\cdots$};
\draw(9,1) node[above]{$l\hspace{-0.8mm}-\hspace{-0.8mm}2$}--(7,-1);
\draw(10,1) node[above]{$l\hspace{-0.8mm}-\hspace{-0.8mm}1$}--(10,-1);

\blackdot(8,-0.7);

\end{braid}
=
\begin{braid}\tikzset{baseline=0mm}
\draw(0,1) node[above]{$0$}--(0,-1);
\draw(1,1) node[above]{$1$}--(1,-1);
\draw(2,1) node[above]{$2$}--(2,-1);
\draw(3,1) node[above]{$\cdots$};
\draw(4,1) node[above]{$l\hspace{-0.8mm}-\hspace{-0.8mm}1$}--(8,-1);
\draw(5,1) node[above]{$l$}--(9,-1);
\draw(6,1) node[above]{$1$}--(4,-1);
\draw(7,1) node[above]{$2$}--(5,-1);
\draw(8,1) node[above]{$\cdots$};
\draw(9,1) node[above]{$l\hspace{-0.8mm}-\hspace{-0.8mm}2$}--(7,-1);
\draw(10,1) node[above]{$l\hspace{-0.8mm}-\hspace{-0.8mm}1$}--(10,-1);

\blackdot(4.5,0.7);

\end{braid}
$$
But this is \(z\psi_{2l-3} \cdots \psi_{l} \psi_{2l-2} \cdots \psi_{l+1}x\), so \(y_{2l-2} \psi_ux = z \psi_ux\), and we are done.
\end{proof}

\subsection{Construction of $L_{\de,i}^z$ in type ${\tt F}_4^{(1)}$}\label{F4Cons}
We label the vertices of the diagram \({\tt F}_4^{(1)}\) as shown below:
\begin{align*}
\begin{braid}\tikzset{baseline=3mm}
\draw[black](-6,0) node[below]{\scriptsize $0$}--(-4,0);
\draw[black](-4,0) node[below]{\scriptsize $1$}--(-2,0);
\draw[black,double,->,>=stealth](-2,0) node[below]{\scriptsize $2$}--(-0.2,0);
\draw[black](0,0) node[below]{\scriptsize $3$}--(2,0);
\draw[black](2,0) node[below]{\scriptsize $4$};
\blackdot(-2,0);
\blackdot(0,0);
\blackdot(2,0);
\blackdot(-4,0);
\blackdot(-6,0);
\end{braid}
\end{align*}

Fix \(i \in I'\). If \(i \in \{1,2\}\), let \(X\) be the set of the following tuples:
\begin{center}
\begin{tabular}{ccc}
$(1,2,3,4,5,6)$, & $(1,3,2,4,5,6)$, & $(1,2,3,5,4,6)$,\\
$(1,3,2,5,4,6)$, & $(3,1,2,4,5,6)$, & $(1,2,3,5,6,4)$,\\
$(3,1,2,5,4,6)$, & $(1,3,2,5,6,4)$, & $(3,1,2,5,6,4)$,\\
$(3,1,5,2,4,6)$, & $(1,3,5,2,6,4)$, & $(3,1,5,2,6,4)$,\\
$(3,5,1,2,6,4)$, & $(3,1,5,6,2,4)$, & $(3,5,1,6,2,4)$,\\
$(3,5,1,2,4,6)$, & $(1,3,5,6,2,4)$, & $(1,3,5,2,4,6)$.
\end{tabular}
\end{center}
If \(i=3\), let \(X\) be the set of the following tuples:
\begin{center}
\begin{tabular}{cccccc}
$(1,2,3,4,5,6)$, & $(1,3,2,4,5,6)$, & $(1,2,3,5,4,6)$, \\
$(1,3,2,5,4,6)$, & $(3,1,2,4,5,6)$, & $(3,1,2,5,4,6)$,\\
$(3,1,5,2,4,6)$, &$(3,5,1,2,4,6)$, &  $(1,3,5,2,4,6)$.
\end{tabular}
\end{center}
If \(i=4\), let \(X\) be the set of the following tuples:
\begin{center}
\begin{tabular}{cccccc}
$(1,2,3,4,6,5)$, & $(1,3,2,4,6,5)$, & $(3,1,2,4,6,5)$.
\end{tabular}
\end{center}
 For \(a \in \{1, \ldots, 6\}\), define 
\begin{align*}
\omega_{a} &= \begin{cases}
3, & a\in \{1,3,4,6\}\\
4, & a \in \{2,5\}
\end{cases}
\\
\chi_a &= \begin{cases}
1 & a \in \{1,2,4\}\\
-1 & a \in \{3,5,6\}.
\end{cases}
\end{align*}
For \(\nu \in X\), define 
\begin{align*}
\bi_{\nu} = \begin{cases}
(0,1,2,\omega_{\nu_1}, \omega_{\nu_2}, \omega_{\nu_3}, 2, \omega_{\nu_4}, \omega_{\nu_5}, \omega_{\nu_6}, 2,1), & i=1\\
 (0,1,2,\omega_{\nu_1}, \omega_{\nu_2}, \omega_{\nu_3}, 2, \omega_{\nu_4}, \omega_{\nu_5}, \omega_{\nu_6}, 1,2), & i=2\\
  (0,1,2,\omega_{\nu_1}, \omega_{\nu_2}, \omega_{\nu_3}, 2, \omega_{\nu_4}, \omega_{\nu_5}, 1,2,\omega_{\nu_6}), & i=3\\
   (0,1,2,\omega_{\nu_1}, \omega_{\nu_2}, \omega_{\nu_3}, 2, \omega_{\nu_4}, 1,2,\omega_{\nu_5},\omega_{\nu_6}), & i=4.
\end{cases}
\end{align*}
Let \(C_{\bi_{\nu}}\) be the connected \(\bi_{\nu}\)-component of the word graph \(G_{\delta}\). Let \(L_{\delta,i}^z\) be the free graded \(\O[z]\)-module on basis
\begin{align*}
B=\{ v_{\nu, \bj} \mid  \nu \in X, \bj \in C_{\bi_{\nu}} \}.
\end{align*}
If \(i \in \{1,2\}\), the grading is given by 
\begin{align*}
\deg v_{\nu, \bj} \hspace{-0.4mm}= \hspace{-0.4mm}\begin{cases}
3, & \nu  = (1,3,2,5,4,6)\\
2, & \nu \in \{(1,2,3,5,4,6), (1,3,2,4,5,6)\}\\
1, & \nu \in \{(1,2,3,4,5,6), (3,1,2,5,4,6), (1,3,2,5,6,4), (1,3,5,2,4,6)\}\\
0, & \nu \in \{(3,1,2,4,5,6), (1,2,3,5,6,4), (3,5,1,2,4,6), (1,3,5,6,2,4)\}\\
-1, & \nu \in \{(3,1,2,5,6,4), (3,1,5,2,4,6), (1,3,5,2,6,4), (3,5,1,6,2,4)\}\\
-2, & \nu \in \{(3,5,1,2,6,4), (3,1,5,6,2,4)\}\\
-3, & \nu = (3,1,5,2,6,4).
\end{cases}
\end{align*}
If \(i =3\), the grading is given by 
\begin{align*}
\deg v_{\nu, \bj} = \begin{cases}
2, & \nu  = (1,3,2,5,4,6)\\
1, & \nu \in \{(1,2,3,5,4,6), (1,3,2,4,5,6)\}\\
0, & \nu \in \{(1,2,3,4,5,6), (3,1,2,5,4,6), (1,3,5,2,4,6)\}\\
-1, & \nu \in \{(3,1,2,4,5,6), (3,5,1,2,4,6)\}\\
-2, & \nu =(3,1,5,2,4,6).
\end{cases}
\end{align*}
If \(i =4\), the grading is given by 
\begin{align*}
\deg v_{\nu, \bj} = \begin{cases}
1, & \nu =(1,3,2,4,6,5) \\
0, & \nu= (1,2,3,4,6,5)\\
-1, & \nu =(3,1,2,4,6,5).
\end{cases}
\end{align*}
For \(\bj \in C_{\bi_{\nu}}\), let \(\mu(\bj)\) be the list of positions in \(\bj\) occupied by 3 or 4, in increasing order. If \(j_r \in \{3,4\}\), let \(l(r, \bj) \in \{1, \ldots, 6\}\) be such that \((\mu(\bj))_{l(r,\bj)} = r\). For example, if \(\bj = (0,1,2,3,3,2,4,3,4,3,2,1)\), then \(\mu(\bj) = (4,5,7,8,9,10)\), and \(l(8,\bj) =4\). 

We now define an action of generators of \(R_\delta^z\) on \(L_{\delta,i}^z\):
\begin{align*}
1_{\bk} v_{\nu, \bj} &= \delta_{\bj, \bk} v_{\nu, \bj}.\\
y_r v_{\nu, \bj} &= \begin{cases}
z^2v_{\nu,\bj}, & j_r \in \{0,1,2\}\\
\chi_{\nu_{l(r,\bj)}}(zv_{\nu,\bj} + Y(r,\nu,\bj)), & j_r \in \{3,4\}\\
\end{cases}\\
\psi_r v_{\nu,\bj} &= \begin{cases}
v_{\nu,s_r \bj}, & j_r \cdot j_{r+1} = 0\\
\Psi(r, \nu, \bj), & j_r, j_{r+1} \in \{3,4\}\\
0, & \textup{otherwise}.
\end{cases}
\end{align*}
where
\begin{align*}
Y(r,\nu,\bj) &= \begin{cases}
v_{\bj, (1,3,2,4,5,6)}, & \nu = (3,1,2,4,5,6), \; l(r,\bj) \in \{1,2\}\\
v_{\bj, (1,2,3,5,4,6)}, & \nu = (1,2,3,5,6,4), \; l(r, \bj) \in \{5,6\}\\
v_{\bj, (1,3,2,5,4,6)}, & \nu = (3,1,2,5,4,6), \; l(r, \bj) \in \{1,2\}\\
v_{\bj, (1,3,2,5,4,6)}, & \nu = (1,3,2,5,6,4), \; l(r, \bj) \in \{5,6\}\\
v_{\bj, (1,3,2,5,6,4)}, & \nu = (3,1,2,5,6,4), \; l(r, \bj) \in \{1,2\}\\
v_{\bj, (3,1,2,5,4,6)}, & \nu = (3,1,2,5,6,4), \; l(r, \bj) \in \{5,6\}\\
v_{\bj, (1,3,5,2,4,6)}, & \nu = (3,1,5,2,4,6), \; l(r, \bj) \in \{1,2\}\\
v_{\bj, (3,1,2,5,4,6)}, & \nu = (3,1,5,2,4,6), \; l(r, \bj) \in \{3,4\}\\
v_{\bj, (1,3,2,5,6,4)}, & \nu = (1,3,5,2,6,4), \; l(r, \bj) \in \{3,4\}\\
v_{\bj, (1,3,5,2,4,6)}, & \nu = (1,3,5,2,6,4), \; l(r, \bj) \in \{5,6\}\\
v_{\bj, (1,3,5,2,6,4)}, & \nu = (3,1,5,2,6,4), \; l(r, \bj) \in \{1,2\}\\
v_{\bj, (3,1,2,5,6,4)}, & \nu = (3,1,5,2,6,4), \; l(r, \bj) \in \{3,4\}\\
v_{\bj, (3,1,5,2,4,6)}, & \nu = (3,1,5,2,6,4), \; l(r, \bj) \in \{5,6\}\\
\epsilon_{34}v_{\bj, (1,2,3,5,6,4)}, & \nu = (3,5,1,2,6,4), \; l(r, \bj) \in \{1,2,3,4\}\\
v_{\bj, (3,5,1,2,4,6)}, & \nu = (3,5,1,2,6,4), \; l(r, \bj) \in \{5,6\}\\
v_{\bj, (1,3,5,6,2,4)}, & \nu = (3,1,5,6,2,4), \; l(r, \bj) \in \{1,2\}\\
\epsilon_{43}v_{\bj, (3,1,2,4,5,6)}, & \nu = (3,1,5,6,2,4), \; l(r, \bj) \in \{3,4,5,6\}\\
-v_{\bj, (1,2,3,4,5,6)}, & \nu = (3,5,1,6,2,4) \\
\epsilon_{34}v_{\bj, (1,2,3,5,4,6)}, & \nu = (3,5,1,2,4,6), \; l(r, \bj) \in \{1,2,3,4\}\\
\epsilon_{43}v_{\bj, (1,3,2,4,5,6)}, & \nu = (1,3,5,6,2,4), \; l(r, \bj) \in \{3,4,5,6\}\\
v_{\bj, (1,3,2,5,4,6)}, & \nu = (1,3,5,2,4,6), \; l(r, \bj) \in \{3,4\}\\
v_{\bj, (1,3,2,4,6,5)}, & \nu = (3,1,2,4,6,5), \; l(r,\bj) \in \{1,2\}\\
0, & \textup{otherwise}.
\end{cases}
\end{align*}
and
\begin{align*}
\Psi(r,\nu,\bj) &= \begin{cases}
v_{s_{l(r,\bj)}\nu, s_r \bj} & \nu_{l(r,\bj)} < \nu_{l(r,\bj)+1}, \;s_{l(r,\bj)}\nu \in X\\
2 \epsilon_{43}zv_{s_r \bj, (1,2,3,4,5,6)}, & \nu=(1,3,2,4,5,6),\; l(r,\bj)=2\\
2 \epsilon_{34}zv_{s_r \bj, (1,2,3,4,5,6)}, & \nu=(1,2,3,5,4,6),\; l(r,\bj)=4\\
2 \epsilon_{43}zv_{s_r \bj, (1,2,3,5,4,6)}, & \nu=(1,3,2,5,4,6),\; l(r,\bj)=2\\
2 \epsilon_{34}zv_{s_r \bj, (1,3,2,4,5,6)}, & \nu=(1,3,2,5,4,6),\; l(r,\bj)=4\\
\epsilon_{34}v_{s_r \bj, (1,2,3,4,5,6)}, & \nu=(3,1,2,4,5,6),\; l(r,\bj)=2\\
\epsilon_{43}v_{s_r \bj, (1,2,3,4,5,6)}, & \nu=(1,2,3,5,6,4),\; l(r,\bj)=4\\
\epsilon_{34}v_{s_r \bj, (1,2,3,5,4,6)}, & \nu=(3,1,2,5,4,6),\; l(r,\bj)=2\\
2\epsilon_{34}zv_{s_r \bj, (3,1,2,4,5,6)}, & \nu=(3,1,2,5,4,6),\; l(r,\bj)=4\\
2\epsilon_{43}zv_{s_r \bj, (1,2,3,5,6,4)}, & \nu=(1,3,2,5,6,4),\; l(r,\bj)=2\\
\epsilon_{43}v_{s_r \bj, (1,3,2,4,5,6)}, & \nu=(1,3,2,5,6,4),\; l(r,\bj)=4\\
\epsilon_{34}v_{s_r \bj, (1,2,3,5,6,4)}, & \nu=(3,1,2,5,6,4),\; l(r,\bj)=2\\
\epsilon_{43}v_{s_r \bj, (3,1,2,4,5,6)}, & \nu=(3,1,2,5,6,4),\; l(r,\bj)=4\\
\epsilon_{43}v_{s_r \bj, (3,1,2,4,5,6)}, & \nu=(3,1,5,2,4,6),\; l(r,\bj)=4\\
\epsilon_{34}v_{s_r \bj, (1,2,3,5,6,4)}, & \nu=(1,3,5,2,6,4),\; l(r,\bj)=2\\
\epsilon_{34}(2zv_{s_r \bj, (3,1,5,2,6,4)} \\
\hspace{0.5cm}+ v_{s_r \bj, (1,3,5,2,6,4)} \\
\hspace{1cm}+ v_{s_r \bj,(3,1,2,5,6,4)}), & \nu=(3,5,1,2,6,4),\; l(r,\bj)=2\\
\epsilon_{43}(2zv_{s_r \bj, (3,1,5,2,6,4)} \\
\hspace{0.5cm}+ v_{s_r \bj, (3,1,5,2,4,6)} \\
\hspace{1cm}+ v_{s_r \bj,(3,1,2,5,6,4)}), & \nu=(3,1,5,6,2,4),\; l(r,\bj)=4\\
2\epsilon_{34}zv_{s_r \bj, (3,1,5,6,2,4)} \\
\hspace{0.5cm}+ \epsilon_{34}v_{s_r \bj, (1,3,5,6,2,4)} \\
\hspace{1cm}- v_{s_r \bj,(3,1,2,4,5,6)}, & \nu=(3,5,1,6,2,4),\; l(r,\bj)=2\\
2\epsilon_{43}zv_{s_r \bj, (3,5,1,2,6,4)} \\
\hspace{0.5cm}+ \epsilon_{43}v_{s_r \bj, (3,5,1,2,4,6)} \\
\hspace{1cm}- v_{s_r \bj,(1,2,3,5,6,4)}, & \nu=(3,5,1,6,2,4),\; l(r,\bj)=4\\
\epsilon_{34}(2zv_{s_r \bj, (3,1,5,2,4,6)} \\
\hspace{0.5cm}+ v_{s_r \bj, (1,3,5,2,4,6)} \\
\hspace{1cm}+ v_{s_r \bj,(3,1,2,5,4,6)}), & \nu=(3,5,1,2,4,6),\; l(r,\bj)=2\\
-v_{s_r \bj, (1,2,3,4,5,6)}, & \nu=(3,5,1,2,4,6),\; l(r,\bj)=4\\
-v_{s_r \bj, (1,2,3,4,5,6)}, & \nu=(1,3,5,6,2,4),\; l(r,\bj)=2\\
\epsilon_{43}(2zv_{s_r \bj, (1,3,5,2,6,4)} \\
\hspace{0.5cm}+ v_{s_r \bj, (1,3,2,5,6,4)} \\
\hspace{1cm}+ v_{s_r \bj,(1,3,5,2,4,6)}), & \nu=(1,3,5,6,2,4),\; l(r,\bj)=4\\
\epsilon_{34}v_{s_r \bj, (1,2,3,5,4,6)}, & \nu=(1,3,5,2,4,6),\; l(r,\bj)=2\\
\epsilon_{43}v_{s_r \bj, (1,3,2,4,5,6)}, & \nu=(1,3,5,2,4,6),\; l(r,\bj)=4\\
2 \epsilon_{43}zv_{s_r \bj, (1,2,3,4,6,5)}, & \nu=(1,3,2,4,6,5),\; l(r,\bj)=2\\
\epsilon_{34}v_{s_r \bj, (1,2,3,4,6,5)}, & \nu=(3,1,2,4,6,5),\; l(r,\bj)=2\\
0 & \textup{otherwise}.
\end{cases}
\end{align*}

\begin{Proposition}\label{F4rels}
The formulas above define a graded action of \(R_{\delta}^z\) on \(L_{\de,i}^z\), and \(L_{\de,i}^z\) is \(\O[z]\)-free on basis \(B\).
\end{Proposition}
\begin{proof}
That the given action agrees with the algebra relations (\ref{R1})--(\ref{R7}) for \(R_\delta^z\) has been checked via computer.
\end{proof}

\subsection{Construction of $L_{\de,i}^z$ in type ${\tt G}_2^{(1)}$}\label{G2Cons}
We label the vertices of the diagram \({\tt G}_2^{(1)}\) as shown below:
\begin{align*}
\begin{braid}\tikzset{baseline=3mm}
\draw[black](-4,0) node[below]{\scriptsize $0$}--(-2,0);
\draw[black, double distance=3pt, ->, >=stealth](-1.8,0)--(-0.2,0);
\draw[black, ->, >=stealth](-1.8,0) node[below]{\scriptsize $1$}--(-0.2,0);
\draw[black](0,0) node[below]{\scriptsize $2$};
\blackdot(-2,0);
\blackdot(-4,0);
\blackdot(0,0);
\end{braid}
\end{align*}
Fix \(i \in I'\). If \(i =1\), let \(X\) be the set of the following tuples:
\begin{align*}
(1,2,3), \hspace{3mm} (1,3,2), \hspace{3mm}  (2,1,3), \hspace{3mm} 
(2,3,1), \hspace{3mm}  (3,1,2), \hspace{3mm}  (3,2,1).
\end{align*}
If \(i=2\), let \(X\) be the set of the following tuples:
\begin{align*}
(1,2,3), \hspace{3mm} (2,1,3).
\end{align*}
For the type \({\tt G}_2\) case, we assume \(\O = \C\) when defining \(L_{\de,i}^z\). For \(a \in \{1, 2,3\}\), define 
\begin{align*}
\chi_a &= \begin{cases}
1 & a =1\\
\xi & a=2\\
\xi^2 & a = 3.
\end{cases}
\end{align*}
For \(\nu \in X\), define 
\begin{align*}
\bi = \begin{cases}
(0,1,2,2,2,1), & i=1\\
(0,1,2,2,1,2), & i=2.
\end{cases}
\end{align*}
Let \(L_\delta^z\) be the free graded \(\O[z]\)-module on basis
\begin{align*}
B=\{ v_{\nu } \mid  \nu \in X\}.
\end{align*}
If \(i =1\), the grading is given by 
\begin{align*}
\deg v_{\nu} = \begin{cases}
3, & \nu  = (1,2,3)\\
1, & \nu \in \{(2,1,3),(1,3,2)\}\\
-1, & \nu \in \{(2,3,1), (3,1,2)\}\\
-3, & \nu = (3,2,1).
\end{cases}
\end{align*}
If \(i=2\), the grading is given by
\begin{align*}
\deg v_{\nu} = \begin{cases}
1, & \nu=(1,2,3)\\
-1, & \nu =(2,1,3).\\
\end{cases}
\end{align*}

We now define an action of generators of \(R_\delta^z\) on \(L_\delta^z\):
\begin{align*}
1_{\bk} v_{\nu} &= \delta_{\bk, \bi} v_{\nu}.\\
y_r v_{\nu} &= \begin{cases}
z^3v_{\nu}, & j_r \in \{0,1\}\\
\chi_{\nu_{r-2}}zv_{\nu} + Y(r-2,\nu), & r \in \{3,4\}\\
\chi_{\nu_{3}}zv_{\nu} + Y(3,\nu), & r =5, i = 1\\
\xi^2zv_{\nu}, & r =6, i = 2\\
\end{cases}\\
\psi_r v_{\nu} &= \begin{cases}
v_{s_{r-2}\nu}, & r \in \{3,4\}, \nu_{r-2}< \nu_{r-1}, s_{r-2}\nu \in X
\\
0, & \textup{otherwise},
\end{cases}
\end{align*}
where
\begin{align*}
Y(t,\nu) &= \begin{cases}
(-1)^tv_{(1,2,3)} &\nu = (2,1,3), t \in \{1,2\}\\
(-1)^{t+1}v_{(1,2,3)} & \nu = (1,3,2), t \in \{2,3\}\\
-v_{(1,3,2)} &\nu=(2,3,1), t =1\\
-v_{(2,1,3)}&\nu=(2,3,1), t =2\\
v_{(2,1,3)} + v_{(1,3,2)} &\nu=(2,3,1), t =3\\
-v_{(1,3,2)} - v_{(2,1,3)}&\nu=(3,1,2), t =1\\
v_{(1,3,2)}&\nu=(3,1,2), t =2\\
v_{(2,1,3)} &\nu=(3,1,2), t =3\\
-v_{(2,3,1)}&\nu=(3,2,1), t =1\\
-v_{(3,1,2)} + v_{(2,3,1)} &\nu=(3,2,1), t =2\\
v_{(3,1,2)} &\nu=(3,2,1), t =3\\
\end{cases}
\end{align*}

\begin{Proposition}\label{G2rels}
The formulas above define a graded action of \(R_{\delta}^z\) on \(L_{\de,i}^z\), and \(L_{\de,i}^z\) is \(\O[z]\)-free on basis \(B\).
\end{Proposition}
\begin{proof}
That the given action agrees with the algebra relations (\ref{R1})--(\ref{R7}) for \(R_\delta^z\) has been checked via computer.
\end{proof}

\subsection{Construction of $L_{\de,i}$ for non-simply-laced types}

Let \(\Car\) be a Cartan matrix of type \({\tt B}_l^{(1)}, {\tt C}_l^{(1)}, {\tt F}_4^{(1)}\), or \({\tt G}_2^{(1)}\), and fix \(i \in I'\).

\begin{Proposition}\label{zTwistMod}
Viewed as an \(R_\delta\)-module, \(L_{\delta,i}^z\) has submodule \(W = \O\{z^kb \mid k >0, b \in B\}\), and \(L_{\delta,i}^z/ W \cong L_{\de,i}\). 
\end{Proposition}

\begin{proof}
\(R_\delta^z\) has a two-sided ideal \(R_\delta^zz\), with \(R_\delta^z / R_\delta^zz \cong R_\delta\) as \(\O\)-algebras. \(L_{\de,i}^z\) is free as an \(\O\)-module with basis \(\{z^kb \mid k \in \Z_{\geq 0}, b \in B\}\), and has an \(R_\delta^z\)-submodule \(R_\delta^zzL_{\de,i}^z = \O\{ z^k b\mid k \in \Z_{>0}, b \in B_i\}\). Then \(L_{\de,i}':=L_{\de,i}^z/R_\delta^zzL_{\de,i}^z\) is an \(R_\delta^z/R_\delta^zz \cong R_\delta\)-module which is free as an \(\O\)-module with basis \(\{\overline{b} \mid b \in B\}\).\\
\indent All words \(\bj\) of \(L_{\de,i}'\) have \(j_1 = 0\), \(j_{e}=i\) by construction. Thus all composition factors are \(L_{\de,i}\) by \cite[Corollary 5.3]{Kcusp}. By considering the graded dimension of any extremal word space in \(L_{\de,i}'\), it is clear that \(L_{\de,i}\) has composition multiplicity one in \(L_{\de,i}'\), hence \(L_{\de,i}' = L_{\de,i}\). 
\end{proof}

Note that in the case of \({\tt G}_2\), where we previously assumed \(\O=\C\), the coefficients of the action of \(R_\de\) on the basis of \(L_{\de,i}\) are integral, so  we may consider instead the \(\Z\)-form and extend scalars to construct \(L_{\de,i}\) for arbitrary \(\O\).


\section{The endomorphism $\tau_r:M_n \to M_n$ for non-simply-laced types}\label{RNSL}

Let \(z,z'\) be algebraically independent, and write \(R_{\alpha}^{z,z'}\) for \(\O[z,z'] \otimes_{\O} R_\alpha\). Then \(L_{\de,i}^{z'} \circ L_{\de,i}^{z}\) and \(L_{\de,i}^{z} \circ L_{\de,i}^{z'}\) are \(R_{2\delta}^{z,z'}\)-modules that are free as \(\O[z,z]\)-modules. Recall the map \(R_{L_{\delta,i}^{z'}, L_{\delta,i}^{z}}:L_{\delta,i}^{z'} \circ L_{\delta,i}^{z} \to L_{\delta,i}^{z} \circ L_{\delta,i}^{z'}\), and the intertwining elements \(\phi\) defined in \cite{KKK}.

\begin{Proposition}\label{InterReduce}
Let \(v^{z} \in L_{\de,i}^z\) be a word vector of word \(\bi\), with \(y_k v^{z} = c_kz^{a_k}v^z\) for all admissible \(k\). Then \(R_{L_{\de,i}^{z'},L_{\de,i}^{z}}(v^{z'} \otimes v^{z})\) is equal to 
$$ 
\begin{braid}\tikzset{baseline=0mm}
\draw(0,1) node[above]{$0$}--(5,-1);
\draw(1,1) node[above]{$i_2$}--(6,-1);
\draw(2,1) node[above]{$i_3$}--(7,-1);
\draw(3,1) node[above]{$\cdots$};
\draw(4,1) node[above]{$i_e$}--(9,-1);

\draw(5,1) node[above]{$0$}--(0,-1);
\draw(6,1) node[above]{$i_2$}--(1,-1);
\draw(7,1) node[above]{$i_3$}--(2,-1);
\draw(8,1) node[above]{$\cdots$};
\draw(9,1) node[above]{$i_e$}--(4,-1);

\draw(4.5,-1) node[below]{\scriptsize $\displaystyle\prod_{\substack{k,m\\i_k=i_m}} (c_{k}z^{a_k}-c_{m}z'^{a_m})$};

\end{braid}
+
\begin{braid}\tikzset{baseline=0mm}
\draw(0,1) node[above]{$0$}--(0,-1);
\draw(1,1) node[above]{$i_2$}--(6,-1);
\draw(2,1) node[above]{$i_3$}--(7,-1);
\draw(3,1) node[above]{$\cdots$};
\draw(4,1) node[above]{$i_e$}--(9,-1);

\draw(5,1) node[above]{$0$}--(2.5,0)--(5,-1);
\draw(6,1) node[above]{$i_2$}--(1,-1);
\draw(7,1) node[above]{$i_3$}--(2,-1);
\draw(8,1) node[above]{$\cdots$};
\draw(9,1) node[above]{$i_e$}--(4,-1);

\draw(4.5,-1) node[below]{\scriptsize $\displaystyle\prod_{\substack{k,m>1\\i_k=i_m}} (c_{k}z^{a_k}-c_{m}z'^{a_m})$};

\end{braid}.
$$
\end{Proposition}

When we diagrammatically describe module elements, we always picture the vector as being at the top of the diagram, the algebra elements below and acting upwardly, and \(\O[z,z']\)-coefficients at the bottom.
\begin{proof}
Note that \(i_1=0\), and \(i_k \neq 0\) for all \(k>1\). Let \(w_1\) be the block permutation of the tensor factors as in (\ref{Ew_r}). Then by definition \(R_{L_{\de,i}^{z'},L_{\de,i}^{z}}(v^{z'} \otimes v^{z}) = \phi_{w_1}(v^{z} \otimes v^{z'})\). We write this diagrammatically as
$$ 
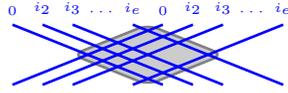
\begin{braid}\tikzset{baseline=0mm}

\draw [rounded corners,color=gray, fill=black!20] (2,0)--(4.5,1)--(7,0)--(4.5,-1)--cycle;

\draw(0,1) node[above]{$0$}--(5,-1);
\draw(1,1) node[above]{$i_2$}--(6,-1);
\draw(2,1) node[above]{$i_3$}--(7,-1);
\draw(3,1) node[above]{$\cdots$};
\draw(4,1) node[above]{$i_e$}--(9,-1);

\draw(5,1) node[above]{$0$}--(0,-1);
\draw(6,1) node[above]{$i_2$}--(1,-1);
\draw(7,1) node[above]{$i_3$}--(2,-1);
\draw(8,1) node[above]{$\cdots$};
\draw(9,1) node[above]{$i_e$}--(4,-1);

\end{braid}
$$
where any strand crossings in the gray section are understood to represent \(\phi\)'s instead of \(\psi\)'s. Assume \(k\) is the smallest such that \(i_k = i_e\). Then the above is equal to 
$$ 
\begin{braid}\tikzset{baseline=0mm}

\draw [rounded corners,color=gray, fill=black!20] (2.25,-0.25)--(5.5,2)--(7.25,0.25)--(5.8,-0.65)--(5.3,-0.1)--(3.5,-1.5)--cycle;

\draw(0.5,2) node[above]{$0$}--(4.5,-2);
\draw(1,2) node[above]{$i_2$}--(5,-2);
\draw(1.7,2) node[above]{$\cdot\cdot$};
\draw(2.5,2) node[above]{$i_k$}--(6.5,-2);
\draw(3.5,2) node[above]{$i_{k\hspace{-0.5mm}+\hspace{-0.5mm}1}$}--(7.5,-2);
\draw(4.4,2) node[above]{$\cdot\cdot$};
\draw(5,2) node[above]{$i_e$}--(9,-2);

\draw(6,2) node[above]{$0$}--(0.5,-2);
\draw(6.5,2) node[above]{$i_2$}--(1,-2);
\draw(7.2,2) node[above]{$\cdot\cdot$};
\draw(8,2) node[above]{$i_{e\hspace{-0.5mm}-\hspace{-0.5mm}1}$}--(2.5,-2);
\draw(9,2) node[above]{$i_e$}--(3.5,-2);

\blackdot(5.05,-0.55);

\end{braid}
-
\begin{braid}\tikzset{baseline=0mm}

\draw [rounded corners,color=gray, fill=black!20] (2.25,-0.25)--(5.5,2)--(7.25,0.25)--(5.8,-0.65)--(5.3,-0.1)--(3.5,-1.5)--cycle;

\draw(0.5,2) node[above]{$0$}--(4.5,-2);
\draw(1,2) node[above]{$i_2$}--(5,-2);
\draw(1.7,2) node[above]{$\cdot\cdot$};
\draw(2.5,2) node[above]{$i_k$}--(6.5,-2);
\draw(3.5,2) node[above]{$i_{k\hspace{-0.5mm}+\hspace{-0.5mm}1}$}--(7.5,-2);
\draw(4.4,2) node[above]{$\cdot\cdot$};
\draw(5,2) node[above]{$i_e$}--(9,-2);

\draw(6,2) node[above]{$0$}--(0.5,-2);
\draw(6.5,2) node[above]{$i_2$}--(1,-2);
\draw(7.2,2) node[above]{$\cdot\cdot$};
\draw(8,2) node[above]{$i_{e\hspace{-0.5mm}-\hspace{-0.5mm}1}$}--(2.5,-2);
\draw(9,2) node[above]{$i_e$}--(3.5,-2);

\blackdot(5.45,-0.55);

\end{braid}
+
\begin{braid}\tikzset{baseline=0mm}

\draw [rounded corners,color=gray, fill=black!20] (2.25,-0.25)--(5.5,2)--(7.25,0.25)--(5.8,-0.65)--(5.3,-0.1)--(3.5,-1.5)--cycle;

\draw(0.5,2) node[above]{$0$}--(4.5,-2);
\draw(1,2) node[above]{$i_2$}--(5,-2);
\draw(1.7,2) node[above]{$\cdot\cdot$};
\draw(2.5,2) node[above]{$i_k$}--(5.25,-0.6)--(3.5,-2);
\draw(3.5,2) node[above]{$i_{k\hspace{-0.5mm}+\hspace{-0.5mm}1}$}--(7.5,-2);
\draw(4.4,2) node[above]{$\cdot\cdot$};
\draw(5,2) node[above]{$i_e$}--(9,-2);

\draw(6,2) node[above]{$0$}--(0.5,-2);
\draw(6.5,2) node[above]{$i_2$}--(1,-2);
\draw(7.2,2) node[above]{$\cdot\cdot$};
\draw(8,2) node[above]{$i_{e\hspace{-0.5mm}-\hspace{-0.5mm}1}$}--(2.5,-2);
\draw(9,2) node[above]{$i_e$}--(5.25,-0.6)--(6.5,-2);

\end{braid}.
$$
By the commuting properties of the block intertwiner (see \cite[Lemma 1.3.1]{KKK}), we can move the beads up to act on the tensor factors, and in the third term the \(i_k\)-strand can be pulled to the left, giving
$$ 
\begin{braid}\tikzset{baseline=0mm}

\draw [rounded corners,color=gray, fill=black!20] (2.25,-0.25)--(5.5,2)--(7.25,0.25)--(5.8,-0.65)--(5.3,-0.1)--(3.5,-1.5)--cycle;

\draw(0.5,2) node[above]{$0$}--(4.5,-2);
\draw(1,2) node[above]{$i_2$}--(5,-2);
\draw(1.7,2) node[above]{$\cdot\cdot$};
\draw(2.5,2) node[above]{$i_k$}--(6.5,-2);
\draw(3.5,2) node[above]{$i_{k\hspace{-0.5mm}+\hspace{-0.5mm}1}$}--(7.5,-2);
\draw(4.4,2) node[above]{$\cdot\cdot$};
\draw(5,2) node[above]{$i_e$}--(9,-2);

\draw(6,2) node[above]{$0$}--(0.5,-2);
\draw(6.5,2) node[above]{$i_2$}--(1,-2);
\draw(7.2,2) node[above]{$\cdot\cdot$};
\draw(8,2) node[above]{$i_{e\hspace{-0.5mm}-\hspace{-0.5mm}1}$}--(2.5,-2);
\draw(9,2) node[above]{$i_e$}--(3.5,-2);

\draw(4.5,-2) node[below]{\scriptsize $ (c_{k}z^{a_k}-c_{e}z'^{a_e})$};

\end{braid}
+
\begin{braid}\tikzset{baseline=0mm}

\draw [rounded corners,color=gray, fill=black!20] (2.25,-0.25)--(5.5,2)--(7.25,0.25)--(5.8,-0.65)--(5.3,-0.1)--(3.5,-1.5)--cycle;

\draw(0.5,2) node[above]{$0$}--(4.5,-2);
\draw(1,2) node[above]{$i_2$}--(5,-2);
\draw(1.7,2) node[above]{$\cdot\cdot$};
\draw(2.5,2) node[above]{$i_k$}--(2,-0.6)--(3.5,-2);
\draw(3.5,2) node[above]{$i_{k\hspace{-0.5mm}+\hspace{-0.5mm}1}$}--(7.5,-2);
\draw(4.4,2) node[above]{$\cdot\cdot$};
\draw(5,2) node[above]{$i_e$}--(9,-2);

\draw(6,2) node[above]{$0$}--(0.5,-2);
\draw(6.5,2) node[above]{$i_2$}--(1,-2);
\draw(7.2,2) node[above]{$\cdot\cdot$};
\draw(8,2) node[above]{$i_{e\hspace{-0.5mm}-\hspace{-0.5mm}1}$}--(2.5,-2);
\draw(9,2) node[above]{$i_e$}--(5.25,-0.6)--(6.5,-2);

\end{braid}.
$$
The latter term is zero however; \(\psi_1L_{\de,i}^z=0\) in all cases. We get a similar result for the next smallest \(k'\) such that \(i_{k'}=i_e\), and we work our way up the \(i_e\)-strand to get 
$$ 
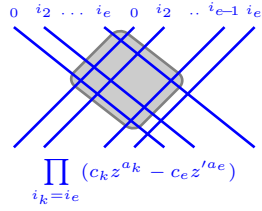
\begin{braid}\tikzset{baseline=0mm}

\draw [rounded corners,color=gray, fill=black!20] (1.75,0.25)--(3.5,2)--(5.75,0.25)--(4,-1.5)--cycle;

\draw(0,2) node[above]{$0$}--(5,-2);
\draw(1,2) node[above]{$i_2$}--(6,-2);
\draw(2,2) node[above]{$\cdots$};
\draw(3,2) node[above]{$i_e$}--(8,-2);

\draw(4,2) node[above]{$0$}--(0,-2);
\draw(5,2) node[above]{$i_2$}--(1,-2);
\draw(6,2) node[above]{$\cdot\cdot$};
\draw(7,2) node[above]{$i_{e\hspace{-0.5mm}-\hspace{-0.5mm}1}$}--(3,-2);
\draw(8,2) node[above]{$i_e$}--(4,-2);

\draw(4,-2) node[below]{\scriptsize $ \displaystyle \prod_{i_k = i_e}(c_{k}z^{a_k}-c_{e}z'^{a_e})$};

\end{braid}.
$$
Then we apply the same arguments to the \(i_{e-1}\)-strand, and work our way up the strands recursively, until we have 
$$ 
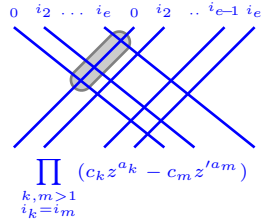
\begin{braid}\tikzset{baseline=0mm}

\draw [rounded corners,color=gray, fill=black!20] (1.75,0.25)--(3.5,2)--(4,1.5)--(2.25,-0.25)--cycle;

\draw(0,2) node[above]{$0$}--(5,-2);
\draw(1,2) node[above]{$i_2$}--(6,-2);
\draw(2,2) node[above]{$\cdots$};
\draw(3,2) node[above]{$i_e$}--(8,-2);

\draw(4,2) node[above]{$0$}--(0,-2);
\draw(5,2) node[above]{$i_2$}--(1,-2);
\draw(6,2) node[above]{$\cdot\cdot$};
\draw(7,2) node[above]{$i_{e\hspace{-0.5mm}-\hspace{-0.5mm}1}$}--(3,-2);
\draw(8,2) node[above]{$i_e$}--(4,-2);

\draw(4,-2) node[below]{\scriptsize $ \displaystyle \prod_{\substack{k,m>1\\i_k = i_m}}(c_{k}z^{a_k}-c_{m}z'^{a_m})$};

\end{braid}.
$$
Then applying the definition of \(\phi\) to the lone \((0,0)\)-crossing gives the result.
\end{proof}


\subsection{The map $R_{L_{\de,i}^{z'},L_{\de,i}^z}$ in type ${\tt B}_l^{(1)}$}\label{RBl}
Fix \(i \in I'\). Recall the construction of \(L_{\de,i}^z\) from section \ref{BlCons}. 
Define the word
\begin{equation}\label{EITypeB}
\bi :=
\left\{
\begin{array}{ll}
(0,2,3, \ldots, l-1,l,l,l-1,\ldots,i+1,1,2,\ldots,i) &\hbox{if $i<l$};\\
(0,2,3, \ldots, l,1,2,\ldots,l) &\hbox{if $i=l$.}
\end{array}
\right.
\end{equation}
Then \(\bi\) is an extremal word for \(L_{\de,i}^z\). Let 
\begin{align*}
v_1^z:=v_{\bi}^{(1)} \;\; (i<l), \hspace{10mm} v_1^z:= v_{\bi}\;\; (i=l).
\end{align*}
Then \(v_1^z\) is a vector of word \(\bi\) that generates \(L_{\de,i}^z\).
\begin{Proposition}\label{Blintertwiner}
\begin{align*}R(v_1^{z'} \otimes v_1^{z}) \in (z^2-z'^2)^{4l-4}\left[\sigma + (-1)^{l+i} + (z-z')R_{\de}^{z} \otimes R_{\de}^{z'}\right](v_1^z \otimes v_1^{z'}).
\end{align*}
\end{Proposition}

\begin{proof}
First, assume \(i<l\). In view of Proposition \ref{InterReduce}, we focus on rewriting the term
$$
C=
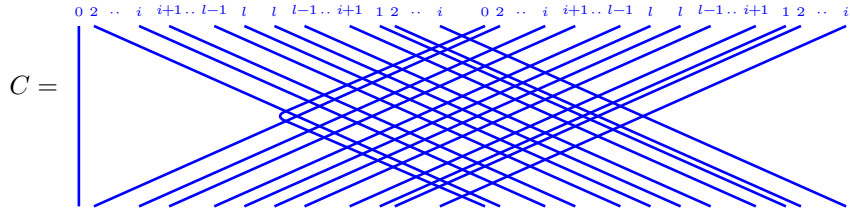
\begin{braid}\tikzset{baseline=3mm}

\draw(1.5,3) node[above]{$0$}--(1.5,-3);
\draw(2,3) node[above]{$2$}--(15.5,-3);
\draw(2.7,3) node[above]{$\cdot\cdot$};
\draw(3.5,3) node[above]{$i$}--(17,-3);
\draw(4.5,3) node[above]{$i\hspace{-0.8mm}+\hspace{-0.8mm}1$}--(18,-3);
\draw(5.2,3) node[above]{$\cdot\cdot$};
\draw(6,3) node[above]{$l\hspace{-0.8mm}-\hspace{-0.8mm}1$}--(19.5,-3);
\draw(7,3) node[above]{$l$}--(20.5,-3);
\draw(8,3) node[above]{$l$}--(21.5,-3);
\draw(9,3) node[above]{$l\hspace{-0.8mm}-\hspace{-0.8mm}1$}--(22.5,-3);
\draw(9.7,3) node[above]{$\cdot\cdot$};
\draw(10.5,3) node[above]{$i\hspace{-0.8mm}+\hspace{-0.8mm}1$}--(24,-3);
\draw(11.5,3) node[above]{$1$}--(25,-3);
\draw(12,3) node[above]{$2$}--(25.5,-3);
\draw(12.7,3) node[above]{$\cdot\cdot$};
\draw(13.5,3) node[above]{$i$}--(27,-3);

\draw(15,3) node[above]{$0$}--(8,0)--(15,-3);
\draw(15.5,3) node[above]{$2$}--(2,-3);
\draw(16.2,3) node[above]{$\cdot\cdot$};
\draw(17,3) node[above]{$i$}--(3.5,-3);
\draw(18,3) node[above]{$i\hspace{-0.8mm}+\hspace{-0.8mm}1$}--(4.5,-3);
\draw(18.7,3) node[above]{$\cdot\cdot$};
\draw(19.5,3) node[above]{$l\hspace{-0.8mm}-\hspace{-0.8mm}1$}--(6,-3);
\draw(20.5,3) node[above]{$l$}--(7,-3);
\draw(21.5,3) node[above]{$l$}--(8,-3);
\draw(22.5,3) node[above]{$l\hspace{-0.8mm}-\hspace{-0.8mm}1$}--(9,-3);
\draw(23.2,3) node[above]{$\cdot\cdot$};
\draw(24,3) node[above]{$i\hspace{-0.8mm}+\hspace{-0.8mm}1$}--(10.5,-3);
\draw(25,3) node[above]{$1$}--(11.5,-3);
\draw(25.5,3) node[above]{$2$}--(12,-3);
\draw(26.2,3) node[above]{$\cdot\cdot$};
\draw(27,3) node[above]{$i$}--(13.5,-3);

\end{braid}
$$

The algebra braid relation implies that 
$ 
\begin{braid}\tikzset{baseline=0mm}
\draw(0,0.3) node[above]{$2$}--(1,-0.3);
\draw(0.5,0.3) node[above]{$0$}--(0,0)--(0.5,-0.3);
\draw(1,0.3) node[above]{$2$}--(0,-0.3);
\end{braid}
$
opens to become
$ 
\begin{braid}\tikzset{baseline=0mm}
\draw(0,0.3) node[above]{$2$}--(0,-0.3);
\draw(0.5,0.3) node[above]{$0$}--(0.5,-0.3);
\draw(1,0.3) node[above]{$2$}--(1,-0.3);
\end{braid}
$
, since if the \(0\)-strand is moved to the right side of the \((2,2)\)-crossing, it slides up to induce a \((0,2)\)-crossing in the second factor or a \((1,2)\)-crossing in the first factor, both of which are zero. In the future we will simply say that a shape such as
$ 
\begin{braid}\tikzset{baseline=0mm}
\draw(0,0.3)--(1,-0.3);
\draw(0.5,0.3)--(0,0)--(0.5,-0.3);
\draw(1,0.3)--(0,-0.3);
\end{braid}
$
`opens' if the the
$ 
\begin{braid}\tikzset{baseline=0mm}
\draw(0,0.3)--(1,-0.3);
\draw(0.5,0.3)--(1,0)--(0.5,-0.3);
\draw(1,0.3)--(0,-0.3);
\end{braid}
$
term in the associated algebra relation induces such strand crossings that act as zero on either of the tensor factors. Next, the braids
$ 
\begin{braid}\tikzset{baseline=0mm}
\draw(0,0.3) node[above]{$3$}--(1,-0.3);
\draw(0.5,0.3) node[above]{$2$}--(0,0)--(0.5,-0.3);
\draw(1,0.3) node[above]{$3$}--(0,-0.3);
\end{braid}
$
through
$ 
\begin{braid}\tikzset{baseline=0mm}
\draw(0,0.3) node[above]{$l\hspace{-1mm}-\hspace{-1mm}1$}--(2,-0.3);
\draw(1,0.3) node[above]{$l\hspace{-1mm}-\hspace{-1mm}2$}--(0,0)--(1,-0.3);
\draw(2,0.3) node[above]{$l\hspace{-1mm}-\hspace{-1mm}1$}--(0,-0.3);
\end{braid}
$
open in succession, giving
$$
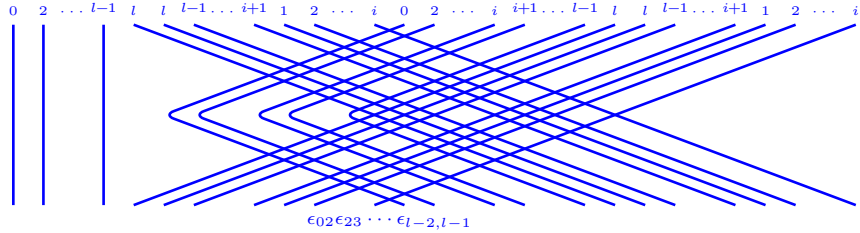
\begin{braid}\tikzset{baseline=3mm}

\draw(4,3) node[above]{$0$}--(4,-3);
\draw(5,3) node[above]{$2$}--(5,-3);
\draw(6,3) node[above]{$\cdots$};
\draw(7,3) node[above]{$l\hspace{-0.8mm}-\hspace{-0.8mm}1$}--(7,-3);
\draw(8,3) node[above]{$l$}--(24,-3);
\draw(9,3) node[above]{$l$}--(25,-3);
\draw(10,3) node[above]{$l\hspace{-0.8mm}-\hspace{-0.8mm}1$}--(26,-3);
\draw(11,3) node[above]{$\cdots$};
\draw(12,3) node[above]{$i\hspace{-0.8mm}+\hspace{-0.8mm}1$}--(28,-3);
\draw(13,3) node[above]{$1$}--(29,-3);
\draw(14,3) node[above]{$2$}--(30,-3);
\draw(15,3) node[above]{$\cdots$};
\draw(16,3) node[above]{$i$}--(32,-3);

\draw(17,3) node[above]{$0$}--(9,0)--(17,-3);
\draw(18,3) node[above]{$2$}--(10,0)--(18,-3);
\draw(19,3) node[above]{$\cdots$};
\draw(20,3) node[above]{$i$}--(12,0)--(20,-3);
\draw(21,3) node[above]{$i\hspace{-0.8mm}+\hspace{-0.8mm}1$}--(13,0)--(21,-3);
\draw(22,3) node[above]{$\cdots$};
\draw(23,3) node[above]{$l\hspace{-0.8mm}-\hspace{-0.8mm}1$}--(15,0)--(23,-3);
\draw(24,3) node[above]{$l$}--(8,-3);
\draw(25,3) node[above]{$l$}--(9,-3);
\draw(26,3) node[above]{$l\hspace{-0.8mm}-\hspace{-0.8mm}1$}--(10,-3);
\draw(27,3) node[above]{$\cdots$};
\draw(28,3) node[above]{$i\hspace{-0.8mm}+\hspace{-0.8mm}1$}--(12,-3);
\draw(29,3) node[above]{$1$}--(13,-3);
\draw(30,3) node[above]{$2$}--(14,-3);
\draw(31,3) node[above]{$\cdots$};
\draw(32,3) node[above]{$i$}--(16,-3);

\draw(16.5,-3) node[below]{\scriptsize $\epsilon_{02} \epsilon_{23} \cdots \epsilon_{l-2,l-1}$};

\end{braid}
$$
This is equal to \(A+B\), where
$$
A=
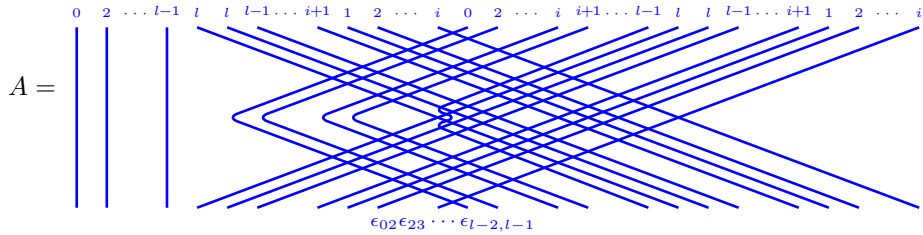
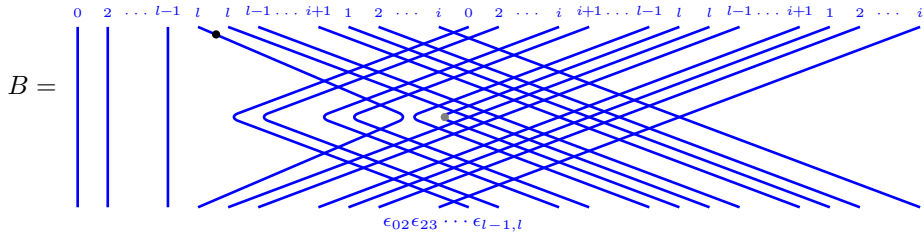
\begin{braid}\tikzset{baseline=3mm}

\draw(4,3) node[above]{$0$}--(4,-3);
\draw(5,3) node[above]{$2$}--(5,-3);
\draw(6,3) node[above]{$\cdots$};
\draw(7,3) node[above]{$l\hspace{-0.8mm}-\hspace{-0.8mm}1$}--(7,-3);
\draw(8,3) node[above]{$l$}--(24,-3);
\draw(9,3) node[above]{$l$}--(25,-3);
\draw(10,3) node[above]{$l\hspace{-0.8mm}-\hspace{-0.8mm}1$}--(26,-3);
\draw(11,3) node[above]{$\cdots$};
\draw(12,3) node[above]{$i\hspace{-0.8mm}+\hspace{-0.8mm}1$}--(28,-3);
\draw(13,3) node[above]{$1$}--(29,-3);
\draw(14,3) node[above]{$2$}--(30,-3);
\draw(15,3) node[above]{$\cdots$};
\draw(16,3) node[above]{$i$}--(32,-3);

\draw(17,3) node[above]{$0$}--(9,0)--(17,-3);
\draw(18,3) node[above]{$2$}--(10,0)--(18,-3);
\draw(19,3) node[above]{$\cdots$};
\draw(20,3) node[above]{$i$}--(12,0)--(20,-3);
\draw(21,3) node[above]{$i\hspace{-0.8mm}+\hspace{-0.8mm}1$}--(13,0)--(21,-3);
\draw(22,3) node[above]{$\cdots$};
\draw(23,3) node[above]{$l\hspace{-0.8mm}-\hspace{-0.8mm}1$}--(15.85,0.25)--(16.65,0)--(15.85,-0.25)--(23,-3);
\draw(24,3) node[above]{$l$}--(8,-3);
\draw(25,3) node[above]{$l$}--(9,-3);
\draw(26,3) node[above]{$l\hspace{-0.8mm}-\hspace{-0.8mm}1$}--(10,-3);
\draw(27,3) node[above]{$\cdots$};
\draw(28,3) node[above]{$i\hspace{-0.8mm}+\hspace{-0.8mm}1$}--(12,-3);
\draw(29,3) node[above]{$1$}--(13,-3);
\draw(30,3) node[above]{$2$}--(14,-3);
\draw(31,3) node[above]{$\cdots$};
\draw(32,3) node[above]{$i$}--(16,-3);

\draw(16.5,-3) node[below]{\scriptsize $\epsilon_{02} \epsilon_{23} \cdots \epsilon_{l-2,l-1}$};

\end{braid}
$$
and
$$
B=
\begin{braid}\tikzset{baseline=3mm}

\draw(4,3) node[above]{$0$}--(4,-3);
\draw(5,3) node[above]{$2$}--(5,-3);
\draw(6,3) node[above]{$\cdots$};
\draw(7,3) node[above]{$l\hspace{-0.8mm}-\hspace{-0.8mm}1$}--(7,-3);
\draw(8,3) node[above]{$l$}--(15,0)--(8,-3);
\draw(9,3) node[above]{$l$}--(25,-3);
\draw(10,3) node[above]{$l\hspace{-0.8mm}-\hspace{-0.8mm}1$}--(26,-3);
\draw(11,3) node[above]{$\cdots$};
\draw(12,3) node[above]{$i\hspace{-0.8mm}+\hspace{-0.8mm}1$}--(28,-3);
\draw(13,3) node[above]{$1$}--(29,-3);
\draw(14,3) node[above]{$2$}--(30,-3);
\draw(15,3) node[above]{$\cdots$};
\draw(16,3) node[above]{$i$}--(32,-3);

\draw(17,3) node[above]{$0$}--(9,0)--(17,-3);
\draw(18,3) node[above]{$2$}--(10,0)--(18,-3);
\draw(19,3) node[above]{$\cdots$};
\draw(20,3) node[above]{$i$}--(12,0)--(20,-3);
\draw(21,3) node[above]{$i\hspace{-0.8mm}+\hspace{-0.8mm}1$}--(13,0)--(21,-3);
\draw(22,3) node[above]{$\cdots$};
\draw(23,3) node[above]{$l\hspace{-0.8mm}-\hspace{-0.8mm}1$}--(15,0)--(23,-3);
\draw(24,3) node[above]{$l$}--(16,0)--(24,-3);
\draw(25,3) node[above]{$l$}--(9,-3);
\draw(26,3) node[above]{$l\hspace{-0.8mm}-\hspace{-0.8mm}1$}--(10,-3);
\draw(27,3) node[above]{$\cdots$};
\draw(28,3) node[above]{$i\hspace{-0.8mm}+\hspace{-0.8mm}1$}--(12,-3);
\draw(29,3) node[above]{$1$}--(13,-3);
\draw(30,3) node[above]{$2$}--(14,-3);
\draw(31,3) node[above]{$\cdots$};
\draw(32,3) node[above]{$i$}--(16,-3);

\draw(16.5,-3) node[below]{\scriptsize $\epsilon_{02} \epsilon_{23} \cdots \epsilon_{l-1,l}$};

\blackdot(8.6,2.75);
\graydot(16.2,0);

\end{braid}
$$
The latter diagram is meant to stand for a sum of two terms; one taken with the black bead in place and one with the gray bead. We focus on \(A\) for now. The upper 
$ 
\begin{braid}\tikzset{baseline=0mm}
\draw(0,0.3) node[above]{$l$}--(2,-0.3);
\draw(1,0.3) node[above]{$l\hspace{-1mm}-\hspace{-1mm}1$}--(0,0)--(1,-0.3);
\draw(2,0.3) node[above]{$l$}--(0,-0.3);
\end{braid}
$
opens, giving beads on the \(l\)-strands which move up to act on the tensor factors, introducing a factor of \((-z +z')\). Next, the lower
$ 
\begin{braid}\tikzset{baseline=0mm}
\draw(0,0.3) node[above]{$l$}--(2,-0.3);
\draw(1,0.3) node[above]{$l\hspace{-1mm}-\hspace{-1mm}1$}--(0,0)--(1,-0.3);
\draw(2,0.3) node[above]{$l$}--(0,-0.3);
\end{braid}
$
braid opens. Pulling strands to the right, the braids
$ 
\begin{braid}\tikzset{baseline=0mm}
\draw(0,0.3) node[above]{$l\hspace{-1mm}-\hspace{-1mm}1$}--(2,-0.3);
\draw(1,0.3) node[above]{$l$}--(0,0)--(1,-0.3);
\draw(2,0.3) node[above]{$l\hspace{-1mm}-\hspace{-1mm}1$}--(0,-0.3);
\end{braid}
$
through
$ 
\begin{braid}\tikzset{baseline=0mm}
\draw(0,0.3) node[above]{$i\hspace{-1mm}+\hspace{-1mm}1$}--(2,-0.3);
\draw(1,0.3) node[above]{$i\hspace{-1mm}+\hspace{-1mm}2$}--(0,0)--(1,-0.3);
\draw(2,0.3) node[above]{$i\hspace{-1mm}+\hspace{-1mm}1$}--(0,-0.3);
\end{braid}
$
open in succession. We now apply the quadratic algebra relation to open
$ 
\begin{braid}\tikzset{baseline=0mm}
\draw(0,0.3) node[above]{$l\hspace{-1mm}-\hspace{-1mm}1$}--(1,0)--(0,-0.3);
\draw(1,0.3) node[above]{$l$}--(0,0)--(1,-0.3);
\end{braid}
$.
One term in this relation is zero, and the other has
$ 
\begin{braid}\tikzset{baseline=0mm}
\draw(0,0.3) node[above]{$l\hspace{-1mm}-\hspace{-1mm}1$}--(1,0)--(0,-0.3);
\draw(1,0.3) node[above]{$l\hspace{-1mm}-\hspace{-1mm}1$}--(0,0)--(1,-0.3);
\blackdot(0.8,0);
\end{braid}
$,
which becomes
$ 
\begin{braid}\tikzset{baseline=0mm}
\draw(0,0.3) node[above]{$n\hspace{-1mm}-\hspace{-1mm}1$}--(1,-0.3);
\draw(1,0.3) node[above]{$n\hspace{-1mm}-\hspace{-1mm}1$}--(0,-0.3);
\end{braid}
$.\\
\indent Next, 
$ 
\begin{braid}\tikzset{baseline=0mm}
\draw(0,0.3) node[above]{$l\hspace{-1mm}-\hspace{-1mm}1$}--(2,-0.3);
\draw(1,0.3) node[above]{$l\hspace{-1mm}-\hspace{-1mm}2$}--(0,0)--(1,-0.3);
\draw(2,0.3) node[above]{$l\hspace{-1mm}-\hspace{-1mm}1$}--(0,-0.3);
\end{braid}
$
opens, and we apply the quadratic relation to
$ 
\begin{braid}\tikzset{baseline=0mm}
\draw(0,0.3) node[above]{$l\hspace{-1mm}-\hspace{-1mm}2$}--(1,0)--(0,-0.3);
\draw(1,0.3) node[above]{$l\hspace{-1mm}-\hspace{-1mm}1$}--(0,0)--(1,-0.3);
\end{braid}
$. Again one term in this relation is zero, and the other has 
$ 
\begin{braid}\tikzset{baseline=0mm}
\draw(0,0.3) node[above]{$l\hspace{-1mm}-\hspace{-1mm}2$}--(1,0)--(0,-0.3);
\draw(1,0.3) node[above]{$l\hspace{-1mm}-\hspace{-1mm}2$}--(0,0)--(1,-0.3);
\blackdot(0.8,0);
\end{braid}
$, 
which becomes 
$ 
\begin{braid}\tikzset{baseline=0mm}
\draw(0,0.3) node[above]{$l\hspace{-1mm}-\hspace{-1mm}2$}--(1,-0.3);
\draw(1,0.3) node[above]{$l\hspace{-1mm}-\hspace{-1mm}2$}--(0,-0.3);
\end{braid}
$. 
This process repeats for braids 
$ 
\begin{braid}\tikzset{baseline=0mm}
\draw(0,0.3) node[above]{$l\hspace{-1mm}-\hspace{-1mm}2$}--(2,-0.3);
\draw(1,0.3) node[above]{$l\hspace{-1mm}-\hspace{-1mm}3$}--(0,0)--(1,-0.3);
\draw(2,0.3) node[above]{$l\hspace{-1mm}-\hspace{-1mm}2$}--(0,-0.3);
\end{braid}
$
through
$ 
\begin{braid}\tikzset{baseline=0mm}
\draw(0,0.3) node[above]{$i\hspace{-1mm}+\hspace{-1mm}1$}--(2,-0.3);
\draw(1,0.3) node[above]{$i$}--(0,0)--(1,-0.3);
\draw(2,0.3) node[above]{$i\hspace{-1mm}+\hspace{-1mm}1$}--(0,-0.3);
\end{braid}
$, 
giving
$$
A=\begin{braid}\tikzset{baseline=0mm}

\draw(4,3) node[above]{$0$}--(4,-3);
\draw(5,3) node[above]{$2$}--(5,-3);
\draw(6,3) node[above]{$\cdots$};
\draw(7,3) node[above]{$l\hspace{-0.8mm}-\hspace{-0.8mm}1$}--(7,-3);
\draw(8,3) node[above]{$l$}--(9,-3);
\draw(9,3) node[above]{$l$}--(8,-3);
\draw(10,3) node[above]{$l\hspace{-0.8mm}-\hspace{-0.8mm}1$}--(10,-3);
\draw(11,3) node[above]{$\cdots$};
\draw(12,3) node[above]{$i\hspace{-0.8mm}+\hspace{-0.8mm}1$}--(12,-3);
\draw(13,3) node[above]{$1$}--(29,-3);
\draw(14,3) node[above]{$2$}--(30,-3);
\draw(15,3) node[above]{$\cdots$};
\draw(16,3) node[above]{$i$}--(32,-3);

\draw(17,3) node[above]{$0$}--(13,2)--(13,-2)--(17,-3);
\draw(18,3) node[above]{$2$}--(13.5,1.75)--(13.5,-1.75)--(18,-3);
\draw(19,3) node[above]{$\cdots$};
\draw(20,3) node[above]{$i$}--(14.5,1.25)--(14.5,-1.25)--(20,-3);
\draw(21,3) node[above]{$i\hspace{-0.8mm}+\hspace{-0.8mm}1$}--(15,1)--(15,-1)--(21,-3);
\draw(22,3) node[above]{$\cdots$};
\draw(23,3) node[above]{$l\hspace{-0.8mm}-\hspace{-0.8mm}1$}--(16,0.5)--(16,-0.5)--(23,-3);
\draw(24,3) node[above]{$l$}--(16.5,0.25)--(16.5,-0.25)--(20,-1.5)--(23,-2.2)--(25,-3);
\draw(25,3) node[above]{$l$}--(17,0)--(21,-1.5)--(22,-2.2)--(24,-3);
\draw(26,3) node[above]{$l\hspace{-0.8mm}-\hspace{-0.8mm}1$}--(18,0)--(26,-3);
\draw(27,3) node[above]{$\cdots$};
\draw(28,3) node[above]{$i\hspace{-0.8mm}+\hspace{-0.8mm}1$}--(20,0)--(28,-3);
\draw(29,3) node[above]{$1$}--(13,-3);
\draw(30,3) node[above]{$2$}--(14,-3);
\draw(31,3) node[above]{$\cdots$};
\draw(32,3) node[above]{$i$}--(16,-3);

\draw(16.5,-3) node[below]{\scriptsize $(-1)^{l+i}\epsilon_{02} \epsilon_{23} \cdots \epsilon_{i-1,i}(z-z')$};

\blackdot(8.85,-2);
\graydot(22.2,-2.3);

\end{braid}.
$$
\indent Now, focusing on \(B\) and moving the beads up to act on the tensor factors, we have that \(B=B_1 + B_2\), where 
$$
B_1=
\begin{braid}\tikzset{baseline=3mm}

\draw(4,3) node[above]{$0$}--(4,-3);
\draw(5,3) node[above]{$2$}--(5,-3);
\draw(6,3) node[above]{$\cdots$};
\draw(7,3) node[above]{$l\hspace{-0.8mm}-\hspace{-0.8mm}1$}--(7,-3);
\draw(8,3) node[above]{$l$}--(8,-3);
\draw(9,3) node[above]{$l$}--(25,-3);
\draw(10,3) node[above]{$l\hspace{-0.8mm}-\hspace{-0.8mm}1$}--(26,-3);
\draw(11,3) node[above]{$\cdots$};
\draw(12,3) node[above]{$i\hspace{-0.8mm}+\hspace{-0.8mm}1$}--(28,-3);
\draw(13,3) node[above]{$1$}--(29,-3);
\draw(14,3) node[above]{$2$}--(30,-3);
\draw(15,3) node[above]{$\cdots$};
\draw(16,3) node[above]{$i$}--(32,-3);

\draw(17,3) node[above]{$0$}--(9,0)--(17,-3);
\draw(18,3) node[above]{$2$}--(10,0)--(18,-3);
\draw(19,3) node[above]{$\cdots$};
\draw(20,3) node[above]{$i$}--(12,0)--(20,-3);
\draw(21,3) node[above]{$i\hspace{-0.8mm}+\hspace{-0.8mm}1$}--(13,0)--(21,-3);
\draw(22,3) node[above]{$\cdots$};
\draw(23,3) node[above]{$l\hspace{-0.8mm}-\hspace{-0.8mm}1$}--(15,0)--(23,-3);
\draw(24,3) node[above]{$l$}--(16,0)--(24,-3);
\draw(25,3) node[above]{$l$}--(9,-3);
\draw(26,3) node[above]{$l\hspace{-0.8mm}-\hspace{-0.8mm}1$}--(10,-3);
\draw(27,3) node[above]{$\cdots$};
\draw(28,3) node[above]{$i\hspace{-0.8mm}+\hspace{-0.8mm}1$}--(12,-3);
\draw(29,3) node[above]{$1$}--(13,-3);
\draw(30,3) node[above]{$2$}--(14,-3);
\draw(31,3) node[above]{$\cdots$};
\draw(32,3) node[above]{$i$}--(16,-3);

\draw(16.5,-3) node[below]{\scriptsize $\epsilon_{02} \epsilon_{23} \cdots \epsilon_{l-1,l}(z+z')$};

\end{braid}
$$
and
$$
B_2=
\begin{braid}\tikzset{baseline=3mm}

\draw(4,3) node[above]{$0$}--(4,-3);
\draw(5,3) node[above]{$2$}--(5,-3);
\draw(6,3) node[above]{$\cdots$};
\draw(7,3) node[above]{$l\hspace{-0.8mm}-\hspace{-0.8mm}1$}--(7,-3);
\draw(8,3) node[above]{$l$}--(8,-3);
\draw(9,3) node[above]{$l$}--(24,-3);
\draw(10,3) node[above]{$l\hspace{-0.8mm}-\hspace{-0.8mm}1$}--(26,-3);
\draw(11,3) node[above]{$\cdots$};
\draw(12,3) node[above]{$i\hspace{-0.8mm}+\hspace{-0.8mm}1$}--(28,-3);
\draw(13,3) node[above]{$1$}--(29,-3);
\draw(14,3) node[above]{$2$}--(30,-3);
\draw(15,3) node[above]{$\cdots$};
\draw(16,3) node[above]{$i$}--(32,-3);

\draw(17,3) node[above]{$0$}--(9,0)--(17,-3);
\draw(18,3) node[above]{$2$}--(10,0)--(18,-3);
\draw(19,3) node[above]{$\cdots$};
\draw(20,3) node[above]{$i$}--(12,0)--(20,-3);
\draw(21,3) node[above]{$i\hspace{-0.8mm}+\hspace{-0.8mm}1$}--(13,0)--(21,-3);
\draw(22,3) node[above]{$\cdots$};
\draw(23,3) node[above]{$l\hspace{-0.8mm}-\hspace{-0.8mm}1$}--(15,0)--(23,-3);
\draw(24,3) node[above]{$l$}--(16.5,0.25)--(25,-3);
\draw(25,3) node[above]{$l$}--(9,-3);
\draw(26,3) node[above]{$l\hspace{-0.8mm}-\hspace{-0.8mm}1$}--(10,-3);
\draw(27,3) node[above]{$\cdots$};
\draw(28,3) node[above]{$i\hspace{-0.8mm}+\hspace{-0.8mm}1$}--(12,-3);
\draw(29,3) node[above]{$1$}--(13,-3);
\draw(30,3) node[above]{$2$}--(14,-3);
\draw(31,3) node[above]{$\cdots$};
\draw(32,3) node[above]{$i$}--(16,-3);

\draw(16.5,-3) node[below]{\scriptsize $(-1)\epsilon_{02} \epsilon_{23} \cdots \epsilon_{l-1,l}$};

\end{braid}.
$$
We focus on \(B_1\). The braids 
$ 
\begin{braid}\tikzset{baseline=0mm}
\draw(0,0.3) node[above]{$l\hspace{-1mm}-\hspace{-1mm}1$}--(2,-0.3);
\draw(1,0.3) node[above]{$l$}--(0,0)--(1,-0.3);
\draw(2,0.3) node[above]{$l\hspace{-1mm}-\hspace{-1mm}1$}--(0,-0.3);
\end{braid}
$
through 
$ 
\begin{braid}\tikzset{baseline=0mm}
\draw(0,0.3) node[above]{$i\hspace{-1mm}+\hspace{-1mm}1$}--(2,-0.3);
\draw(1,0.3) node[above]{$i\hspace{-1mm}+\hspace{-1mm}2$}--(0,0)--(1,-0.3);
\draw(2,0.3) node[above]{$i\hspace{-1mm}+\hspace{-1mm}1$}--(0,-0.3);
\end{braid}
$
open in succession. Then
$ 
\begin{braid}\tikzset{baseline=0mm}
\draw(0,0.3) node[above]{$l$}--(2,-0.3);
\draw(1,0.3) node[above]{$l\hspace{-1mm}-\hspace{-1mm}1$}--(2,0)--(1,-0.3);
\draw(2,0.3) node[above]{$l$}--(0,-0.3);
\end{braid}
$
opens, introducing a quadratic term 
$ 
\begin{braid}\tikzset{baseline=0mm}
\draw(0,0.3) node[above]{$l$}--(1,0)--(0,-0.3);
\draw(1,0.3) node[above]{$l\hspace{-1mm}-\hspace{-1mm}1$}--(0,0)--(1,-0.3);
\end{braid}
$
which opens, introducing 
$ 
\begin{braid}\tikzset{baseline=0mm}
\draw(0,0.3) node[above]{$l\hspace{-1mm}-\hspace{-1mm}1$}--(1,0)--(0,-0.3);
\draw(1,0.3) node[above]{$l\hspace{-1mm}-\hspace{-1mm}1$}--(0,0)--(1,-0.3);
\blackdot(0.2,0);
\end{braid}
$,
which becomes 
$ 
\begin{braid}\tikzset{baseline=0mm}
\draw(0,0.3) node[above]{$l\hspace{-1mm}-\hspace{-1mm}1$}--(1,-0.3);
\draw(1,0.3) node[above]{$l\hspace{-1mm}-\hspace{-1mm}1$}--(0,-0.3);
\end{braid}
$.
This process repeats with the braids
$ 
\begin{braid}\tikzset{baseline=0mm}
\draw(0,0.3) node[above]{$l\hspace{-1mm}-\hspace{-1mm}1$}--(2,-0.3);
\draw(1,0.3) node[above]{$l\hspace{-1mm}-\hspace{-1mm}2$}--(2,0)--(1,-0.3);
\draw(2,0.3) node[above]{$l\hspace{-1mm}-\hspace{-1mm}1$}--(0,-0.3);
\end{braid}
$
through 
$ 
\begin{braid}\tikzset{baseline=0mm}
\draw(0,0.3) node[above]{$i\hspace{-1mm}+\hspace{-1mm}2$}--(2,-0.3);
\draw(1,0.3) node[above]{$i\hspace{-1mm}+\hspace{-1mm}1$}--(2,0)--(1,-0.3);
\draw(2,0.3) node[above]{$i\hspace{-1mm}+\hspace{-1mm}2$}--(0,-0.3);
\end{braid}
$.
Then
$ 
\begin{braid}\tikzset{baseline=0mm}
\draw(0,0.3) node[above]{$i\hspace{-1mm}+\hspace{-1mm}1$}--(2,-0.3);
\draw(1,0.3) node[above]{$i$}--(0,0)--(1,-0.3);
\draw(2,0.3) node[above]{$i\hspace{-1mm}+\hspace{-1mm}1$}--(0,-0.3);
\end{braid}
$
opens, leaving
$ 
\begin{braid}\tikzset{baseline=0mm}
\draw(0,0.3) node[above]{$l$}--(1,0)--(0,-0.3);
\draw(1,0.3) node[above]{$l$}--(0,0)--(1,-0.3);
\blackdot(0.2,0);
\end{braid}
$,
which becomes
$ 
\begin{braid}\tikzset{baseline=0mm}
\draw(0,0.3) node[above]{$l$}--(1,-0.3);
\draw(1,0.3) node[above]{$l$}--(0,-0.3);
\end{braid}
$, and we have 
$$
B_1=
\begin{braid}\tikzset{baseline=0mm}

\draw(4,3) node[above]{$0$}--(4,-3);
\draw(5,3) node[above]{$2$}--(5,-3);
\draw(6,3) node[above]{$\cdots$};
\draw(7,3) node[above]{$l\hspace{-0.8mm}-\hspace{-0.8mm}1$}--(7,-3);
\draw(8,3) node[above]{$l$}--(8,-3);
\draw(9,3) node[above]{$l$}--(9,-3);
\draw(10,3) node[above]{$l\hspace{-0.8mm}-\hspace{-0.8mm}1$}--(10,-3);
\draw(11,3) node[above]{$\cdots$};
\draw(12,3) node[above]{$i\hspace{-0.8mm}+\hspace{-0.8mm}1$}--(12,-3);
\draw(13,3) node[above]{$1$}--(29,-3);
\draw(14,3) node[above]{$2$}--(30,-3);
\draw(15,3) node[above]{$\cdots$};
\draw(16,3) node[above]{$i$}--(32,-3);

\draw(17,3) node[above]{$0$}--(13,2)--(13,-2)--(17,-3);
\draw(18,3) node[above]{$2$}--(13.5,1.75)--(13.5,-1.75)--(18,-3);
\draw(19,3) node[above]{$\cdots$};
\draw(20,3) node[above]{$i$}--(14.5,1.25)--(14.5,-1.25)--(20,-3);
\draw(21,3) node[above]{$i\hspace{-0.8mm}+\hspace{-0.8mm}1$}--(15,1)--(15,-1)--(21,-3);
\draw(22,3) node[above]{$\cdots$};
\draw(23,3) node[above]{$l\hspace{-0.8mm}-\hspace{-0.8mm}1$}--(16,0.5)--(16,-0.5)--(23,-3);
\draw(24,3) node[above]{$l$}--(16.5,0.25)--(16.5,-0.25)--(20,-1.5)--(23,-2.25)--(25,-3);
\draw(25,3) node[above]{$l$}--(17,0)--(21,-1.5)--(22,-2.25)--(24,-3);
\draw(26,3) node[above]{$l\hspace{-0.8mm}-\hspace{-0.8mm}1$}--(18,0)--(26,-3);
\draw(27,3) node[above]{$\cdots$};
\draw(28,3) node[above]{$i\hspace{-0.8mm}+\hspace{-0.8mm}1$}--(20,0)--(28,-3);
\draw(29,3) node[above]{$1$}--(13,-3);
\draw(30,3) node[above]{$2$}--(14,-3);
\draw(31,3) node[above]{$\cdots$};
\draw(32,3) node[above]{$i$}--(16,-3);

\draw(16.5,-3) node[below]{\scriptsize $(-1)^{l+i}\epsilon_{02} \epsilon_{23} \cdots \epsilon_{i-1,i}(z+z')$};

\end{braid}.
$$
Now we consider \(B_2\). The braids
$ 
\begin{braid}\tikzset{baseline=0mm}
\draw(0,0.3) node[above]{$l\hspace{-1mm}-\hspace{-1mm}1$}--(2,-0.3);
\draw(1,0.3) node[above]{$l$}--(0,0)--(1,-0.3);
\draw(2,0.3) node[above]{$l\hspace{-1mm}-\hspace{-1mm}1$}--(0,-0.3);
\end{braid}
$
through
$ 
\begin{braid}\tikzset{baseline=0mm}
\draw(0,0.3) node[above]{$i\hspace{-1mm}+\hspace{-1mm}1$}--(2,-0.3);
\draw(1,0.3) node[above]{$i\hspace{-1mm}+\hspace{-1mm}2$}--(0,0)--(1,-0.3);
\draw(2,0.3) node[above]{$i\hspace{-1mm}+\hspace{-1mm}1$}--(0,-0.3);
\end{braid}
$
open in succession. Then 
$ 
\begin{braid}\tikzset{baseline=0mm}
\draw(0,0.3) node[above]{$l$}--(2,-0.3);
\draw(1,0.3) node[above]{$l\hspace{-1mm}-\hspace{-1mm}1$}--(2,0)--(1,-0.3);
\draw(2,0.3) node[above]{$l$}--(0,-0.3);
\end{braid}
$
opens, leading to a quadratic term 
$ 
\begin{braid}\tikzset{baseline=0mm}
\draw(0,0.3) node[above]{$l$}--(1,0)--(0,-0.3);
\draw(1,0.3) node[above]{$l\hspace{-1mm}-\hspace{-1mm}1$}--(0,0)--(1,-0.3);
\end{braid}
$
which opens to give
$ 
\begin{braid}\tikzset{baseline=0mm}
\draw(0,0.3) node[above]{$l\hspace{-1mm}-\hspace{-1mm}1$}--(1,0)--(0,-0.3);
\draw(1,0.3) node[above]{$l\hspace{-1mm}-\hspace{-1mm}1$}--(0,0)--(1,-0.3);
\blackdot(0.2,0);
\end{braid}
$
which becomes 
$ 
\begin{braid}\tikzset{baseline=0mm}
\draw(0,0.3) node[above]{$l\hspace{-1mm}-\hspace{-1mm}1$}--(1,-0.3);
\draw(1,0.3) node[above]{$l\hspace{-1mm}-\hspace{-1mm}1$}--(0,-0.3);
\end{braid}
$.
This process repeats with braids
$ 
\begin{braid}\tikzset{baseline=0mm}
\draw(0,0.3) node[above]{$l\hspace{-1mm}-\hspace{-1mm}1$}--(2,-0.3);
\draw(1,0.3) node[above]{$l\hspace{-1mm}-\hspace{-1mm}2$}--(2,0)--(1,-0.3);
\draw(2,0.3) node[above]{$l\hspace{-1mm}-\hspace{-1mm}1$}--(0,-0.3);
\end{braid}
$
through 
$ 
\begin{braid}\tikzset{baseline=0mm}
\draw(0,0.3) node[above]{$i\hspace{-1mm}+\hspace{-1mm}2$}--(2,-0.3);
\draw(1,0.3) node[above]{$i\hspace{-1mm}+\hspace{-1mm}1$}--(2,0)--(1,-0.3);
\draw(2,0.3) node[above]{$i\hspace{-1mm}+\hspace{-1mm}2$}--(0,-0.3);
\end{braid}
$.
Finally, the 
$ 
\begin{braid}\tikzset{baseline=0mm}
\draw(0,0.3) node[above]{$i\hspace{-1mm}+\hspace{-1mm}1$}--(2,-0.3);
\draw(1,0.3) node[above]{$i$}--(0,0)--(1,-0.3);
\draw(2,0.3) node[above]{$i\hspace{-1mm}+\hspace{-1mm}1$}--(0,-0.3);
\end{braid}
$
braid opens, giving
$$
B_2=
\begin{braid}\tikzset{baseline=0mm}

\draw(4,3) node[above]{$0$}--(4,-3);
\draw(5,3) node[above]{$2$}--(5,-3);
\draw(6,3) node[above]{$\cdots$};
\draw(7,3) node[above]{$l\hspace{-0.8mm}-\hspace{-0.8mm}1$}--(7,-3);
\draw(8,3) node[above]{$l$}--(8,-3);
\draw(9,3) node[above]{$l$}--(9,-3);
\draw(10,3) node[above]{$l\hspace{-0.8mm}-\hspace{-0.8mm}1$}--(10,-3);
\draw(11,3) node[above]{$\cdots$};
\draw(12,3) node[above]{$i\hspace{-0.8mm}+\hspace{-0.8mm}1$}--(12,-3);
\draw(13,3) node[above]{$1$}--(29,-3);
\draw(14,3) node[above]{$2$}--(30,-3);
\draw(15,3) node[above]{$\cdots$};
\draw(16,3) node[above]{$i$}--(32,-3);

\draw(17,3) node[above]{$0$}--(13,2)--(13,-2)--(17,-3);
\draw(18,3) node[above]{$2$}--(13.5,1.75)--(13.5,-1.75)--(18,-3);
\draw(19,3) node[above]{$\cdots$};
\draw(20,3) node[above]{$i$}--(14.5,1.25)--(14.5,-1.25)--(20,-3);
\draw(21,3) node[above]{$i\hspace{-0.8mm}+\hspace{-0.8mm}1$}--(15,1)--(15,-1)--(21,-3);
\draw(22,3) node[above]{$\cdots$};
\draw(23,3) node[above]{$l\hspace{-0.8mm}-\hspace{-0.8mm}1$}--(16,0.5)--(16,-0.5)--(23,-3);
\draw(24,3) node[above]{$l$}--(16.5,0.25)--(16.5,-0.25)--(20,-1.5)--(23,-2.25)--(25,-3);
\draw(25,3) node[above]{$l$}--(17,0)--(21,-1.5)--(22,-2.25)--(24,-3);
\draw(26,3) node[above]{$l\hspace{-0.8mm}-\hspace{-0.8mm}1$}--(18,0)--(26,-3);
\draw(27,3) node[above]{$\cdots$};
\draw(28,3) node[above]{$i\hspace{-0.8mm}+\hspace{-0.8mm}1$}--(20,0)--(28,-3);
\draw(29,3) node[above]{$1$}--(13,-3);
\draw(30,3) node[above]{$2$}--(14,-3);
\draw(31,3) node[above]{$\cdots$};
\draw(32,3) node[above]{$i$}--(16,-3);

\blackdot(9,2);
\graydot(22.5,-2.45);

\draw(16.5,-3) node[below]{\scriptsize $(-1)^{l+i}\epsilon_{02} \epsilon_{23} \cdots \epsilon_{i-1,i}$};

\end{braid}
$$
Now we prove the following claim.

{\bf Claim.}
$$
\begin{braid}\tikzset{baseline=0mm}

\draw(4,3) node[above]{$0$}--(4,-3);
\draw(5,3) node[above]{$2$}--(5,-3);
\draw(6,3) node[above]{$\cdots$};
\draw(7,3) node[above]{$l\hspace{-0.8mm}-\hspace{-0.8mm}1$}--(7,-3);
\draw(8,3) node[above]{$l$}--(8,-3);
\draw(9,3) node[above]{$l$}--(9,-3);
\draw(10,3) node[above]{$l\hspace{-0.8mm}-\hspace{-0.8mm}1$}--(10,-3);
\draw(11,3) node[above]{$\cdots$};
\draw(12,3) node[above]{$i\hspace{-0.8mm}+\hspace{-0.8mm}1$}--(12,-3);
\draw(13,3) node[above]{$1$}--(29,-3);
\draw(14,3) node[above]{$2$}--(30,-3);
\draw(15,3) node[above]{$\cdots$};
\draw(16,3) node[above]{$i$}--(32,-3);

\draw(17,3) node[above]{$0$}--(13,2)--(13,-2)--(17,-3);
\draw(18,3) node[above]{$2$}--(13.5,1.75)--(13.5,-1.75)--(18,-3);
\draw(19,3) node[above]{$\cdots$};
\draw(20,3) node[above]{$i$}--(14.5,1.25)--(14.5,-1.25)--(20,-3);
\draw(21,3) node[above]{$i\hspace{-0.8mm}+\hspace{-0.8mm}1$}--(15,1)--(15,-1)--(21,-3);
\draw(22,3) node[above]{$\cdots$};
\draw(23,3) node[above]{$l\hspace{-0.8mm}-\hspace{-0.8mm}1$}--(16,0.5)--(16,-0.5)--(23,-3);
\draw(24,3) node[above]{$l$}--(16.5,0.25)--(16.5,-0.25)--(24,-3);
\draw(25,3) node[above]{$l$}--(17,0)--(25,-3);
\draw(26,3) node[above]{$l\hspace{-0.8mm}-\hspace{-0.8mm}1$}--(18,0)--(26,-3);
\draw(27,3) node[above]{$\cdots$};
\draw(28,3) node[above]{$i\hspace{-0.8mm}+\hspace{-0.8mm}1$}--(20,0)--(28,-3);
\draw(29,3) node[above]{$1$}--(13,-3);
\draw(30,3) node[above]{$2$}--(14,-3);
\draw(31,3) node[above]{$\cdots$};
\draw(32,3) node[above]{$i$}--(16,-3);

\end{braid}
$$
is equal to \((-1)\epsilon_{02}\epsilon_{23} \cdots \epsilon_{i-1,i} (z^2-z'^2)v_1^z \otimes v_1^{z'}\).

{\it Proof of Claim.} The braid
$ 
\begin{braid}\tikzset{baseline=0mm}
\draw(0,0.3) node[above]{$i$}--(2,-0.3);
\draw(1,0.3) node[above]{$i\hspace{-1mm}+\hspace{-1mm}1$}--(0,0)--(1,-0.3);
\draw(2,0.3) node[above]{$i$}--(0,-0.3);
\end{braid}
$
opens, then the quadratic factor 
$ 
\begin{braid}\tikzset{baseline=0mm}
\draw(0,0.3) node[above]{$i$}--(1,0)--(0,-0.3);
\draw(1,0.3) node[above]{$i\hspace{-1mm}+\hspace{-1mm}1$}--(0,0)--(1,-0.3);
\end{braid}
$
opens, introducing a factor of \((z^2-z'^2)\). Then the braids
$ 
\begin{braid}\tikzset{baseline=0mm}
\draw(0,0.3) node[above]{$i\hspace{-1mm}-\hspace{-1mm}1$}--(2,-0.3);
\draw(1,0.3) node[above]{$i$}--(2,0)--(1,-0.3);
\draw(2,0.3) node[above]{$i\hspace{-1mm}-\hspace{-1mm}1$}--(0,-0.3);
\end{braid}
$
through
$ 
\begin{braid}\tikzset{baseline=0mm}
\draw(0,0.3) node[above]{$1$}--(1,-0.3);
\draw(0.5,0.3) node[above]{$2$}--(1,0)--(0.5,-0.3);
\draw(1,0.3) node[above]{$1$}--(0,-0.3);
\end{braid}
$
open in succession, giving
$$
\begin{braid}\tikzset{baseline=0mm}

\draw(4,3) node[above]{$0$}--(4,-3);
\draw(5,3) node[above]{$2$}--(5,-3);
\draw(6,3) node[above]{$\cdots$};
\draw(7,3) node[above]{$l\hspace{-0.8mm}-\hspace{-0.8mm}1$}--(7,-3);
\draw(8,3) node[above]{$l$}--(8,-3);
\draw(9,3) node[above]{$l$}--(9,-3);
\draw(10,3) node[above]{$l\hspace{-0.8mm}-\hspace{-0.8mm}1$}--(10,-3);
\draw(11,3) node[above]{$\cdots$};
\draw(12,3) node[above]{$i\hspace{-0.8mm}+\hspace{-0.8mm}1$}--(12,-3);
\draw(13,3) node[above]{$1$}--(18,0)--(13,-3);
\draw(14,3) node[above]{$2$}--(19,0)--(14,-3);
\draw(15,3) node[above]{$\cdots$};
\draw(16,3) node[above]{$i$}--(20,0)--(16,-3);

\draw(17,3) node[above]{$0$}--(13,0)--(17,-3);
\draw(18,3) node[above]{$2$}--(14,0)--(18,-3);
\draw(19,3) node[above]{$\cdots$};
\draw(20,3) node[above]{$i$}--(16,0)--(20,-3);
\draw(21,3) node[above]{$i\hspace{-0.8mm}+\hspace{-0.8mm}1$}--(21,-3);
\draw(22,3) node[above]{$\cdots$};
\draw(23,3) node[above]{$l\hspace{-0.8mm}-\hspace{-0.8mm}1$}--(23,-3);
\draw(24,3) node[above]{$l$}--(24,-3);
\draw(25,3) node[above]{$l$}--(25,-3);
\draw(26,3) node[above]{$l\hspace{-0.8mm}-\hspace{-0.8mm}1$}--(26,-3);
\draw(27,3) node[above]{$\cdots$};
\draw(28,3) node[above]{$i\hspace{-0.8mm}+\hspace{-0.8mm}1$}--(28,-3);
\draw(29,3) node[above]{$1$}--(29,-3);
\draw(30,3) node[above]{$2$}--(30,-3);
\draw(31,3) node[above]{$\cdots$};
\draw(32,3) node[above]{$i$}--(32,-3);

\draw(16.5,-3) node[below]{\scriptsize $(-1)\epsilon_{12} \cdots \epsilon_{i-1,i}(z^2-z'^2)$};
\end{braid}.
$$
Next 
$ 
\begin{braid}\tikzset{baseline=0mm}
\draw(0,0.3) node[above]{$i\hspace{-1mm}-\hspace{-1mm}1$}--(1,0)--(0,-0.3);
\draw(1,0.3) node[above]{$i$}--(0,0)--(1,-0.3);
\end{braid}
$
opens, introducing 
$ 
\begin{braid}\tikzset{baseline=0mm}
\draw(0,0.3) node[above]{$i$}--(1,0)--(0,-0.3);
\draw(1,0.3) node[above]{$i$}--(0,0)--(1,-0.3);
\blackdot(0.2,0);
\end{braid}
$, 
which becomes 
$ 
\begin{braid}\tikzset{baseline=0mm}
\draw(0,0.3) node[above]{$i$}--(1,-0.3);
\draw(1,0.3) node[above]{$i$}--(0,-0.3);
\end{braid}
$.
The process repeats with the quadratic factors
$ 
\begin{braid}\tikzset{baseline=0mm}
\draw(0,0.3) node[above]{$i\hspace{-1mm}-\hspace{-1mm}2$}--(1,0)--(0,-0.3);
\draw(1,0.3) node[above]{$i\hspace{-1mm}-\hspace{-1mm}1$}--(0,0)--(1,-0.3);
\end{braid}
$
through
$ 
\begin{braid}\tikzset{baseline=0mm}
\draw(0,0.3) node[above]{$1$}--(0.5,0)--(0,-0.3);
\draw(0.5,0.3) node[above]{$2$}--(0,0)--(0.5,-0.3);
\end{braid}
$.
Next the braids
$ 
\begin{braid}\tikzset{baseline=0mm}
\draw(0,0.3) node[above]{$2$}--(1,-0.3);
\draw(0.5,0.3) node[above]{$0$}--(0,0)--(0.5,-0.3);
\draw(1,0.3) node[above]{$2$}--(0,-0.3);
\end{braid}
$
and
$ 
\begin{braid}\tikzset{baseline=0mm}
\draw(0,0.3) node[above]{$3$}--(1,-0.3);
\draw(0.5,0.3) node[above]{$2$}--(0,0)--(0.5,-0.3);
\draw(1,0.3) node[above]{$3$}--(0,-0.3);
\end{braid}
$
through
$ 
\begin{braid}\tikzset{baseline=0mm}
\draw(0,0.3) node[above]{$i$}--(2,-0.3);
\draw(1,0.3) node[above]{$i\hspace{-1mm}-\hspace{-1mm}1$}--(0,0)--(1,-0.3);
\draw(2,0.3) node[above]{$i$}--(0,-0.3);
\end{braid}
$
open in succession, proving the claim.

Then, applying the claim, we have that \(C=A+B_1 +B_2\) is equal to 
$$
\begin{braid}\tikzset{baseline=0mm}

\draw(4,1) node[above]{$0$}--(4,-1);
\draw(5,1) node[above]{$2$}--(5,-1);
\draw(6,1) node[above]{$\cdots$};
\draw(7,1) node[above]{$l\hspace{-0.8mm}-\hspace{-0.8mm}1$}--(7,-1);
\draw(8,1) node[above]{$l$}--(9,-1);
\draw(9,1) node[above]{$l$}--(8,-1);
\draw(10,1) node[above]{$l\hspace{-0.8mm}-\hspace{-0.8mm}1$}--(10,-1);
\draw(11,1) node[above]{$\cdots$};
\draw(12,1) node[above]{$i\hspace{-0.8mm}+\hspace{-0.8mm}1$}--(12,-1);
\draw(13,1) node[above]{$1$}--(13,-1);
\draw(14,1) node[above]{$2$}--(14,-1);
\draw(15,1) node[above]{$\cdots$};
\draw(16,1) node[above]{$i$}--(16,-1);

\draw(17,1) node[above]{$0$}--(17,-1);
\draw(18,1) node[above]{$2$}--(18,-1);
\draw(19,1) node[above]{$\cdots$};
\draw(20,1) node[above]{$i$}--(20,-1);
\draw(21,1) node[above]{$i\hspace{-0.8mm}+\hspace{-0.8mm}1$}--(21,-1);
\draw(22,1) node[above]{$\cdots$};
\draw(23,1) node[above]{$l\hspace{-0.8mm}-\hspace{-0.8mm}1$}--(23,-1);
\draw(24,1) node[above]{$l$}--(25,-1);
\draw(25,1) node[above]{$l$}--(24,-1);
\draw(26,1) node[above]{$l\hspace{-0.8mm}-\hspace{-0.8mm}1$}--(26,-1);
\draw(27,1) node[above]{$\cdots$};
\draw(28,1) node[above]{$i\hspace{-0.8mm}+\hspace{-0.8mm}1$}--(28,-1);
\draw(29,1) node[above]{$1$}--(29,-1);
\draw(30,1) node[above]{$2$}--(30,-1);
\draw(31,1) node[above]{$\cdots$};
\draw(32,1) node[above]{$i$}--(32,-1);

\blackdot(8.75,-0.5);
\graydot(24.25,-0.5);

\draw(17.5,-1) node[below]{\scriptsize $(-1)^{l+i+1}(z^2-z'^2)(z-z')$};
\end{braid}.
$$
$$
+
\begin{braid}\tikzset{baseline=0mm}

\draw(4,1) node[above]{$0$}--(4,-1);
\draw(5,1) node[above]{$2$}--(5,-1);
\draw(6,1) node[above]{$\cdots$};
\draw(7,1) node[above]{$l\hspace{-0.8mm}-\hspace{-0.8mm}1$}--(7,-1);
\draw(8,1) node[above]{$l$}--(8,-1);
\draw(9,1) node[above]{$l$}--(9,-1);
\draw(10,1) node[above]{$l\hspace{-0.8mm}-\hspace{-0.8mm}1$}--(10,-1);
\draw(11,1) node[above]{$\cdots$};
\draw(12,1) node[above]{$i\hspace{-0.8mm}+\hspace{-0.8mm}1$}--(12,-1);
\draw(13,1) node[above]{$1$}--(13,-1);
\draw(14,1) node[above]{$2$}--(14,-1);
\draw(15,1) node[above]{$\cdots$};
\draw(16,1) node[above]{$i$}--(16,-1);

\draw(17,1) node[above]{$0$}--(17,-1);
\draw(18,1) node[above]{$2$}--(18,-1);
\draw(19,1) node[above]{$\cdots$};
\draw(20,1) node[above]{$i$}--(20,-1);
\draw(21,1) node[above]{$i\hspace{-0.8mm}+\hspace{-0.8mm}1$}--(21,-1);
\draw(22,1) node[above]{$\cdots$};
\draw(23,1) node[above]{$l\hspace{-0.8mm}-\hspace{-0.8mm}1$}--(23,-1);
\draw(24,1) node[above]{$l$}--(25,-1);
\draw(25,1) node[above]{$l$}--(24,-1);
\draw(26,1) node[above]{$l\hspace{-0.8mm}-\hspace{-0.8mm}1$}--(26,-1);
\draw(27,1) node[above]{$\cdots$};
\draw(28,1) node[above]{$i\hspace{-0.8mm}+\hspace{-0.8mm}1$}--(28,-1);
\draw(29,1) node[above]{$1$}--(29,-1);
\draw(30,1) node[above]{$2$}--(30,-1);
\draw(31,1) node[above]{$\cdots$};
\draw(32,1) node[above]{$i$}--(32,-1);

\draw(17.5,-1) node[below]{\scriptsize $(-1)^{l+i+1}(z^2-z'^2)(z+z')$};
\end{braid}.
$$
$$
+
\begin{braid}\tikzset{baseline=0mm}

\draw(4,1) node[above]{$0$}--(4,-1);
\draw(5,1) node[above]{$2$}--(5,-1);
\draw(6,1) node[above]{$\cdots$};
\draw(7,1) node[above]{$l\hspace{-0.8mm}-\hspace{-0.8mm}1$}--(7,-1);
\draw(8,1) node[above]{$l$}--(8,-1);
\draw(9,1) node[above]{$l$}--(9,-1);
\draw(10,1) node[above]{$l\hspace{-0.8mm}-\hspace{-0.8mm}1$}--(10,-1);
\draw(11,1) node[above]{$\cdots$};
\draw(12,1) node[above]{$i\hspace{-0.8mm}+\hspace{-0.8mm}1$}--(12,-1);
\draw(13,1) node[above]{$1$}--(13,-1);
\draw(14,1) node[above]{$2$}--(14,-1);
\draw(15,1) node[above]{$\cdots$};
\draw(16,1) node[above]{$i$}--(16,-1);

\draw(17,1) node[above]{$0$}--(17,-1);
\draw(18,1) node[above]{$2$}--(18,-1);
\draw(19,1) node[above]{$\cdots$};
\draw(20,1) node[above]{$i$}--(20,-1);
\draw(21,1) node[above]{$i\hspace{-0.8mm}+\hspace{-0.8mm}1$}--(21,-1);
\draw(22,1) node[above]{$\cdots$};
\draw(23,1) node[above]{$l\hspace{-0.8mm}-\hspace{-0.8mm}1$}--(23,-1);
\draw(24,1) node[above]{$l$}--(25,-1);
\draw(25,1) node[above]{$l$}--(24,-1);
\draw(26,1) node[above]{$l\hspace{-0.8mm}-\hspace{-0.8mm}1$}--(26,-1);
\draw(27,1) node[above]{$\cdots$};
\draw(28,1) node[above]{$i\hspace{-0.8mm}+\hspace{-0.8mm}1$}--(28,-1);
\draw(29,1) node[above]{$1$}--(29,-1);
\draw(30,1) node[above]{$2$}--(30,-1);
\draw(31,1) node[above]{$\cdots$};
\draw(32,1) node[above]{$i$}--(32,-1);

\blackdot(9,0.5);
\graydot(24.25,-0.5);

\draw(17.5,-1) node[below]{\scriptsize $(-1)^{l+i+1}(z^2-z'^2)$};
\end{braid}.
$$
Finally, simplifying and applying Proposition \ref{InterReduce} gives the result.

Now we assume \(i=l\). Let
$$
C=
\begin{braid}\tikzset{baseline=0mm}

\draw(0,2) node[above]{$0$}--(0,-2);
\draw(1,2) node[above]{$2$}--(11,-2);
\draw(2,2) node[above]{$\cdots$};
\draw(3,2) node[above]{$l\hspace{-0.8mm}-\hspace{-0.8mm}1$}--(13,-2);
\draw(4,2) node[above]{$l$}--(14,-2);
\draw(5,2) node[above]{$1$}--(15,-2);
\draw(6,2) node[above]{$2$}--(16,-2);
\draw(7,2) node[above]{$\cdots$};
\draw(8,2) node[above]{$l\hspace{-0.8mm}-\hspace{-0.8mm}1$}--(18,-2);
\draw(9,2) node[above]{$l$}--(19,-2);

\draw(10,2) node[above]{$0$}--(5,0)--(10,-2);
\draw(11,2) node[above]{$2$}--(1,-2);
\draw(12,2) node[above]{$\cdots$};
\draw(13,2) node[above]{$l\hspace{-0.8mm}-\hspace{-0.8mm}1$}--(3,-2);
\draw(14,2) node[above]{$l$}--(4,-2);
\draw(15,2) node[above]{$1$}--(5,-2);
\draw(16,2) node[above]{$2$}--(6,-2);
\draw(17,2) node[above]{$\cdots$};
\draw(18,2) node[above]{$l\hspace{-0.8mm}-\hspace{-0.8mm}1$}--(8,-2);
\draw(19,2) node[above]{$l$}--(9,-2);

\end{braid}.
$$
We now attempt to simplify this term. The braids
$ 
\begin{braid}\tikzset{baseline=0mm}
\draw(0,0.3) node[above]{$2$}--(1,-0.3);
\draw(0.5,0.3) node[above]{$0$}--(0,0)--(0.5,-0.3);
\draw(1,0.3) node[above]{$2$}--(0,-0.3);
\end{braid}
$
and
$ 
\begin{braid}\tikzset{baseline=0mm}
\draw(0,0.3) node[above]{$3$}--(1,-0.3);
\draw(0.5,0.3) node[above]{$2$}--(0,0)--(0.5,-0.3);
\draw(1,0.3) node[above]{$3$}--(0,-0.3);
\end{braid}
$
through
$ 
\begin{braid}\tikzset{baseline=0mm}
\draw(0,0.3) node[above]{$l$}--(2,-0.3);
\draw(1,0.3) node[above]{$l\hspace{-1mm}-\hspace{-1mm}1$}--(0,0)--(1,-0.3);
\draw(2,0.3) node[above]{$l$}--(0,-0.3);
\end{braid}
$
open in succession. This last relation introduces a factor of \((z+z')\), giving
$$
\begin{braid}\tikzset{baseline=0mm}

\draw(0,2) node[above]{$0$}--(0,-2);
\draw(1,2) node[above]{$2$}--(1,-2);
\draw(2,2) node[above]{$\cdots$};
\draw(3,2) node[above]{$l\hspace{-0.8mm}-\hspace{-0.8mm}1$}--(3,-2);
\draw(4,2) node[above]{$l$}--(4,-2);
\draw(5,2) node[above]{$1$}--(15,-2);
\draw(6,2) node[above]{$2$}--(16,-2);
\draw(7,2) node[above]{$\cdots$};
\draw(8,2) node[above]{$l\hspace{-0.8mm}-\hspace{-0.8mm}1$}--(18,-2);
\draw(9,2) node[above]{$l$}--(19,-2);

\draw(10,2) node[above]{$0$}--(5,0)--(10,-2);
\draw(11,2) node[above]{$2$}--(6,0)--(11,-2);
\draw(12,2) node[above]{$\cdots$};
\draw(13,2) node[above]{$l\hspace{-0.8mm}-\hspace{-0.8mm}1$}--(8,0)--(13,-2);
\draw(14,2) node[above]{$l$}--(9,0)--(14,-2);
\draw(15,2) node[above]{$1$}--(5,-2);
\draw(16,2) node[above]{$2$}--(6,-2);
\draw(17,2) node[above]{$\cdots$};
\draw(18,2) node[above]{$l\hspace{-0.8mm}-\hspace{-0.8mm}1$}--(8,-2);
\draw(19,2) node[above]{$l$}--(9,-2);

\draw(9.5,-2) node[below]{\scriptsize $\epsilon_{02}\epsilon_{23} \cdots \epsilon_{l-1,l}(z+z')$};

\end{braid}.
$$
Now the braids
$ 
\begin{braid}\tikzset{baseline=0mm}
\draw(0,0.3) node[above]{$l\hspace{-1mm}-\hspace{-1mm}1$}--(2,-0.3);
\draw(1,0.3) node[above]{$l$}--(0,0)--(1,-0.3);
\draw(2,0.3) node[above]{$l\hspace{-1mm}-\hspace{-1mm}1$}--(0,-0.3);
\end{braid}
$
and 
$ 
\begin{braid}\tikzset{baseline=0mm}
\draw(0,0.3) node[above]{$l$}--(2,-0.3);
\draw(1,0.3) node[above]{$l\hspace{-1mm}-\hspace{-1mm}1$}--(0,0)--(1,-0.3);
\draw(2,0.3) node[above]{$l$}--(0,-0.3);
\end{braid}
$
open, introducing a factor of 
$ 
\begin{braid}\tikzset{baseline=0mm}
\draw(0,0.3) node[above]{$l$}--(1,0)--(0,-0.3);
\draw(1,0.3) node[above]{$l$}--(0,0)--(1,-0.3);
\blackdot(0.8,0);
\end{braid}
$,
which becomes
$ 
\begin{braid}\tikzset{baseline=0mm}
\draw(0,0.3) node[above]{$l$}--(1,-0.3);
\draw(1,0.3) node[above]{$l$}--(0,-0.3);
\end{braid}
$.
Then the braid 
$ 
\begin{braid}\tikzset{baseline=0mm}
\draw(0,0.3) node[above]{$l\hspace{-1mm}-\hspace{-1mm}1$}--(2,-0.3);
\draw(1,0.3) node[above]{$l\hspace{-1mm}-\hspace{-1mm}2$}--(0,0)--(1,-0.3);
\draw(2,0.3) node[above]{$l\hspace{-1mm}-\hspace{-1mm}1$}--(0,-0.3);
\end{braid}
$
opens, introducing a quadratic factor 
$ 
\begin{braid}\tikzset{baseline=0mm}
\draw(0,0.3) node[above]{$l\hspace{-1mm}-\hspace{-1mm}1$}--(1,0)--(0,-0.3);
\draw(1,0.3) node[above]{$l\hspace{-1mm}-\hspace{-1mm}2$}--(0,0)--(1,-0.3);
\end{braid}
$,
which opens, yielding a factor
$ 
\begin{braid}\tikzset{baseline=0mm}
\draw(0,0.3) node[above]{$l\hspace{-1mm}-\hspace{-1mm}2$}--(1,0)--(0,-0.3);
\draw(1,0.3) node[above]{$l\hspace{-1mm}-\hspace{-1mm}2$}--(0,0)--(1,-0.3);
\blackdot(0.8,0);
\end{braid}
$, 
which becomes 
$ 
\begin{braid}\tikzset{baseline=0mm}
\draw(0,0.3) node[above]{$l\hspace{-1mm}-\hspace{-1mm}2$}--(1,-0.3);
\draw(1,0.3) node[above]{$l\hspace{-1mm}-\hspace{-1mm}2$}--(0,-0.3);
\end{braid}
$.
This process repeats for the braids 
$ 
\begin{braid}\tikzset{baseline=0mm}
\draw(0,0.3) node[above]{$l\hspace{-1mm}-\hspace{-1mm}3$}--(2,-0.3);
\draw(1,0.3) node[above]{$l\hspace{-1mm}-\hspace{-1mm}2$}--(0,0)--(1,-0.3);
\draw(2,0.3) node[above]{$l\hspace{-1mm}-\hspace{-1mm}3$}--(0,-0.3);
\end{braid}
$
through
$ 
\begin{braid}\tikzset{baseline=0mm}
\draw(0,0.3) node[above]{$1$}--(1,-0.3);
\draw(0.5,0.3) node[above]{$2$}--(0,0)--(0.5,-0.3);
\draw(1,0.3) node[above]{$1$}--(0,-0.3);
\end{braid}
$, 
giving
$$
\begin{braid}\tikzset{baseline=0mm}

\draw(0,2) node[above]{$0$}--(0,-2);
\draw(1,2) node[above]{$2$}--(1,-2);
\draw(2,2) node[above]{$\cdots$};
\draw(3,2) node[above]{$l\hspace{-0.8mm}-\hspace{-0.8mm}1$}--(3,-2);
\draw(4,2) node[above]{$l$}--(4,-2);
\draw(5,2) node[above]{$1$}--(5,-2);
\draw(6,2) node[above]{$2$}--(11,-2);
\draw(7,2) node[above]{$\cdots$};
\draw(8,2) node[above]{$l\hspace{-0.8mm}-\hspace{-0.8mm}1$}--(13,-2);
\draw(9,2) node[above]{$l$}--(14,-2);

\draw(10,2) node[above]{$0$}--(7,0)--(10,-2);
\draw(11,2) node[above]{$2$}--(6,-2);
\draw(12,2) node[above]{$\cdots$};
\draw(13,2) node[above]{$l\hspace{-0.8mm}-\hspace{-0.8mm}1$}--(8,-2);
\draw(14,2) node[above]{$l$}--(9,-2);
\draw(15,2) node[above]{$1$}--(15,-2);
\draw(16,2) node[above]{$2$}--(16,-2);
\draw(17,2) node[above]{$\cdots$};
\draw(18,2) node[above]{$l\hspace{-0.8mm}-\hspace{-0.8mm}1$}--(18,-2);
\draw(19,2) node[above]{$l$}--(19,-2);

\draw(9.5,-2) node[below]{\scriptsize $(-1)\epsilon_{02}\epsilon_{23} \cdots \epsilon_{l-1,l}(z+z')$};

\end{braid}.
$$
Now, the braids 
$ 
\begin{braid}\tikzset{baseline=0mm}
\draw(0,0.3) node[above]{$2$}--(1,-0.3);
\draw(0.5,0.3) node[above]{$0$}--(0,0)--(0.5,-0.3);
\draw(1,0.3) node[above]{$2$}--(0,-0.3);
\end{braid}
$
and
$ 
\begin{braid}\tikzset{baseline=0mm}
\draw(0,0.3) node[above]{$3$}--(1,-0.3);
\draw(0.5,0.3) node[above]{$2$}--(0,0)--(0.5,-0.3);
\draw(1,0.3) node[above]{$3$}--(0,-0.3);
\end{braid}
$
through
$ 
\begin{braid}\tikzset{baseline=0mm}
\draw(0,0.3) node[above]{$l$}--(2,-0.3);
\draw(1,0.3) node[above]{$l\hspace{-1mm}-\hspace{-1mm}1$}--(0,0)--(1,-0.3);
\draw(2,0.3) node[above]{$l$}--(0,-0.3);
\end{braid}
$
open in succession. This last relation introduces a factor of \((-z+z')\), so we have \(C=(z^2-z'^2)v_1^z \otimes v_1^{z'}\), and applying Proposition \ref{InterReduce} completes the proof.
\end{proof}


\subsection{The map $R_{L_{\de,i}^{z'},L_{\de,i}^z}$ in type ${\tt C}_l^{(1)}$}\label{RCl}

Fix \(i \in I'\). Recall the construction of \(L_{\de,i}^z\) from section \ref{ClCons}. Take 
\begin{equation}\label{EITypeC}
\bi = (0,1,1,2,2, \ldots, i-1,i-1,i,i+1,\ldots,l-1,l,l-1,\ldots,i).
\end{equation}
Then \(\bi\) is an extremal word for \(L_{\de,i}^z\). Let
\begin{align*}
v_1^z:=(\psi_{2i-1} \cdots \psi_{2l-2}) \cdots (\psi_{2k+1} \cdots \psi_{2l-i+k-1}) \cdots (\psi_3 \cdots \psi_{2l-i})x.
\end{align*}
Then \(v_1^z\) is a vector of word \(\bi\) that generates \(L_{\de,i}^z\). To see this, note that
\begin{align*}
(\psi_{2i-3}\psi_{2i-2})\cdots (\psi_5 \psi_6)(\psi_3 \psi_4)v_1^z = \psi_a 1_{\bj}x,
\end{align*}
where  \(\bj = (0,1,\ldots,n-1,n,n-1,\ldots,i+1,1,2,\ldots,i)\) and \(\psi_a = \psi_{a_k} \cdots \psi_{a_1}\) is such that \(\psi_{a_m}1_{a_{m-1} \cdots a_1 \bj}\) is of degree zero for all \(1 \leq m \leq k\). But this implies that \(\psi_b\psi_a1_\bj = 1_{\bj}\) for some \(b\), so \(x\) is in the \(R_{\de}^z\)-span of \(v_1^z\), and \(x\) generates \(L_{\de,i}^z\), proving the claim.
\begin{Proposition}\label{Clintertwiner}
\begin{align*}R(v_1^{z'} \otimes v_1^{z}) \in (z^2-z'^2)^{l+1}\left[\sigma + (-1)^{l+i+1} + (z-z')R_{\de}^{z} \otimes R_{\de}^{z'}\right](v_1^z \otimes v_1^{z'}).
\end{align*}
\end{Proposition}
\begin{proof}
For \(1 \leq k \leq i\), let
$$
A_k=
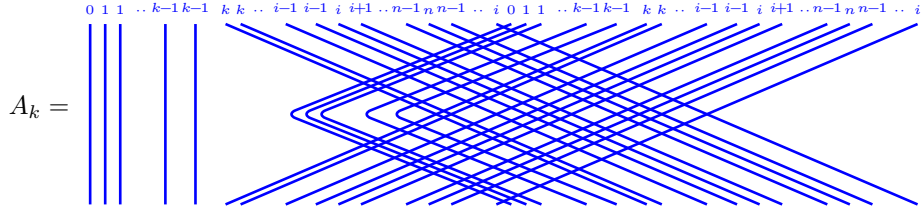
\begin{braid}\tikzset{baseline=0mm}

\draw(1.5,3) node[above]{$0$}--(1.5,-3);
\draw(2,3) node[above]{$1$}--(2,-3);
\draw(2.5,3) node[above]{$1$}--(2.5,-3);
\draw(3.2,3) node[above]{$\cdot\cdot$};
\draw(4,3) node[above]{$k\hspace{-0.9mm}-\hspace{-0.9mm}1$}--(4,-3);
\draw(5,3) node[above]{$k\hspace{-0.9mm}-\hspace{-0.9mm}1$}--(5,-3);
\draw(6,3) node[above]{$k$}--(20,-3);
\draw(6.5,3) node[above]{$k$}--(20.5,-3);
\draw(7.1,3) node[above]{$\cdot\cdot$};
\draw(8,3) node[above]{$i\hspace{-0.9mm}-\hspace{-0.9mm}1$}--(22,-3);
\draw(9,3) node[above]{$i\hspace{-0.9mm}-\hspace{-0.9mm}1$}--(23,-3);
\draw(9.75,3) node[above]{$i$}--(23.75,-3);
\draw(10.5,3) node[above]{$i\hspace{-0.9mm}+\hspace{-0.9mm}1$}--(24.5,-3);
\draw(11.25,3) node[above]{$\cdot\cdot$};
\draw(12,3) node[above]{$n\hspace{-0.9mm}-\hspace{-0.9mm}1$}--(26,-3);
\draw(12.75,3) node[above]{$n$}--(26.75,-3);
\draw(13.5,3) node[above]{$n\hspace{-0.9mm}-\hspace{-0.9mm}1$}--(27.5,-3);
\draw(14.4,3) node[above]{$\cdot\cdot$};
\draw(15,3) node[above]{$i$}--(29,-3);

\draw(15.5,3) node[above]{$0$}--(8,0)--(15.5,-3);
\draw(16,3) node[above]{$1$}--(8.5,0)--(16,-3);
\draw(16.5,3) node[above]{$1$}--(9,0)--(16.5,-3);
\draw(17.2,3) node[above]{$\cdot\cdot$};
\draw(18,3) node[above]{$k\hspace{-0.9mm}-\hspace{-0.9mm}1$}--(10.5,0)--(18,-3);
\draw(19,3) node[above]{$k\hspace{-0.9mm}-\hspace{-0.9mm}1$}--(11.5,0)--(19,-3);
\draw(20,3) node[above]{$k$}--(6,-3);
\draw(20.5,3) node[above]{$k$}--(6.5,-3);
\draw(21.1,3) node[above]{$\cdot\cdot$};
\draw(22,3) node[above]{$i\hspace{-0.9mm}-\hspace{-0.9mm}1$}--(8,-3);
\draw(23,3) node[above]{$i\hspace{-0.9mm}-\hspace{-0.9mm}1$}--(9,-3);
\draw(23.75,3) node[above]{$i$}--(9.75,-3);
\draw(24.5,3) node[above]{$i\hspace{-0.9mm}+\hspace{-0.9mm}1$}--(10.5,-3);
\draw(25.25,3) node[above]{$\cdot\cdot$};
\draw(26,3) node[above]{$n\hspace{-0.9mm}-\hspace{-0.9mm}1$}--(12,-3);
\draw(26.75,3) node[above]{$n$}--(12.75,-3);
\draw(27.5,3) node[above]{$n\hspace{-0.9mm}-\hspace{-0.9mm}1$}--(13.5,-3);
\draw(28.4,3) node[above]{$\cdot\cdot$};
\draw(29,3) node[above]{$i$}--(15,-3);

\end{braid}.
$$
We now focus on simplifying the term \(A_1\), or 
$$
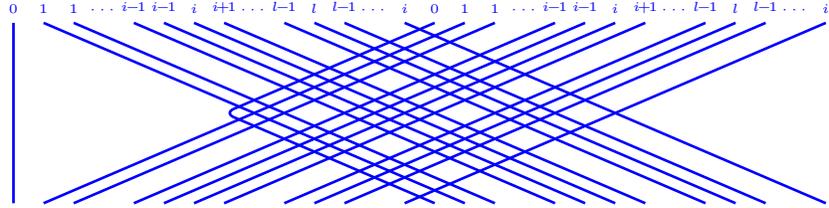
\begin{braid}\tikzset{baseline=0mm}

\draw(0,3) node[above]{$0$}--(0,-3);
\draw(1,3) node[above]{$1$}--(15,-3);
\draw(2,3) node[above]{$1$}--(16,-3);
\draw(3,3) node[above]{$\cdots$};
\draw(4,3) node[above]{$i\hspace{-0.9mm}-\hspace{-0.9mm}1$}--(18,-3);
\draw(5,3) node[above]{$i\hspace{-0.9mm}-\hspace{-0.9mm}1$}--(19,-3);
\draw(6,3) node[above]{$i$}--(20,-3);
\draw(7,3) node[above]{$i\hspace{-0.9mm}+\hspace{-0.9mm}1$}--(21,-3);
\draw(8,3) node[above]{$\cdots$};
\draw(9,3) node[above]{$l\hspace{-0.9mm}-\hspace{-0.9mm}1$}--(23,-3);
\draw(10,3) node[above]{$l$}--(24,-3);
\draw(11,3) node[above]{$l\hspace{-0.9mm}-\hspace{-0.9mm}1$}--(25,-3);
\draw(12,3) node[above]{$\cdots$};
\draw(13,3) node[above]{$i$}--(27,-3);

\draw(14,3) node[above]{$0$}--(7,0)--(14,-3);
\draw(15,3) node[above]{$1$}--(1,-3);
\draw(16,3) node[above]{$1$}--(2,-3);
\draw(17,3) node[above]{$\cdots$};
\draw(18,3) node[above]{$i\hspace{-0.9mm}-\hspace{-0.9mm}1$}--(4,-3);
\draw(19,3) node[above]{$i\hspace{-0.9mm}-\hspace{-0.9mm}1$}--(5,-3);
\draw(20,3) node[above]{$i$}--(6,-3);
\draw(21,3) node[above]{$i\hspace{-0.9mm}+\hspace{-0.9mm}1$}--(7,-3);
\draw(22,3) node[above]{$\cdots$};
\draw(23,3) node[above]{$l\hspace{-0.9mm}-\hspace{-0.9mm}1$}--(9,-3);
\draw(24,3) node[above]{$l$}--(10,-3);
\draw(25,3) node[above]{$l\hspace{-0.9mm}-\hspace{-0.9mm}1$}--(11,-3);
\draw(26,3) node[above]{$\cdots$};
\draw(27,3) node[above]{$i$}--(13,-3);

\end{braid}.
$$
This is in \(A_2 + (z-z')R_{\delta}^z \otimes R_{\delta}^{z'}A_2\), by a direct calculation, and \(A_k  \in A_{k+1} + (z-z')R_{\delta}^z \otimes R_{\delta}^{z'}A_{k+1}\) for all \(2 \leq k \leq i-1\), which one sees by `pulling the \((k-1)\)-strands to the right'. Using this fact recursively, we have that \(A_1 \in A_i + (z-z')R_{\de}^z \otimes R_{\de}^{z'} A_i\), where 
$$
A_i=
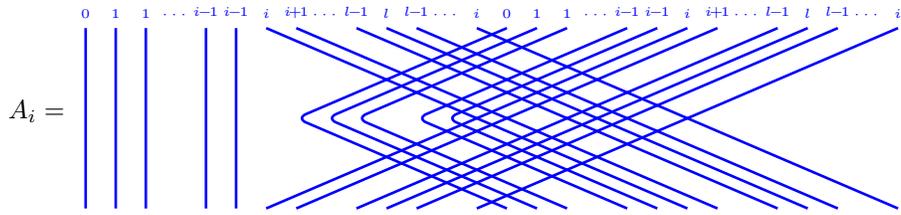
\begin{braid}\tikzset{baseline=0mm}

\draw(0,3) node[above]{$0$}--(0,-3);
\draw(1,3) node[above]{$1$}--(1,-3);
\draw(2,3) node[above]{$1$}--(2,-3);
\draw(3,3) node[above]{$\cdots$};
\draw(4,3) node[above]{$i\hspace{-0.9mm}-\hspace{-0.9mm}1$}--(4,-3);
\draw(5,3) node[above]{$i\hspace{-0.9mm}-\hspace{-0.9mm}1$}--(5,-3);
\draw(6,3) node[above]{$i$}--(20,-3);
\draw(7,3) node[above]{$i\hspace{-0.9mm}+\hspace{-0.9mm}1$}--(21,-3);
\draw(8,3) node[above]{$\cdots$};
\draw(9,3) node[above]{$l\hspace{-0.9mm}-\hspace{-0.9mm}1$}--(23,-3);
\draw(10,3) node[above]{$l$}--(24,-3);
\draw(11,3) node[above]{$l\hspace{-0.9mm}-\hspace{-0.9mm}1$}--(25,-3);
\draw(12,3) node[above]{$\cdots$};
\draw(13,3) node[above]{$i$}--(27,-3);

\draw(14,3) node[above]{$0$}--(7,0)--(14,-3);
\draw(15,3) node[above]{$1$}--(8,0)--(15,-3);
\draw(16,3) node[above]{$1$}--(9,0)--(16,-3);
\draw(17,3) node[above]{$\cdots$};
\draw(18,3) node[above]{$i\hspace{-0.9mm}-\hspace{-0.9mm}1$}--(11,0)--(18,-3);
\draw(19,3) node[above]{$i\hspace{-0.9mm}-\hspace{-0.9mm}1$}--(12,0)--(19,-3);
\draw(20,3) node[above]{$i$}--(6,-3);
\draw(21,3) node[above]{$i\hspace{-0.9mm}+\hspace{-0.9mm}1$}--(7,-3);
\draw(22,3) node[above]{$\cdots$};
\draw(23,3) node[above]{$l\hspace{-0.9mm}-\hspace{-0.9mm}1$}--(9,-3);
\draw(24,3) node[above]{$l$}--(10,-3);
\draw(25,3) node[above]{$l\hspace{-0.9mm}-\hspace{-0.9mm}1$}--(11,-3);
\draw(26,3) node[above]{$\cdots$};
\draw(27,3) node[above]{$i$}--(13,-3);

\end{braid}.
$$
Now, via \((i,i-1,i)\)-braid relations, moving crossed strands up to act on the individual tensor factors as zero when possible, we have that \(A_i = B_1 + B_2\), where
$$
B_1=
\begin{braid}\tikzset{baseline=0mm}

\draw(0,3) node[above]{$0$}--(0,-3);
\draw(1,3) node[above]{$1$}--(1,-3);
\draw(2,3) node[above]{$1$}--(2,-3);
\draw(3,3) node[above]{$\cdots$};
\draw(4,3) node[above]{$i\hspace{-0.9mm}-\hspace{-0.9mm}1$}--(4,-3);
\draw(5,3) node[above]{$i\hspace{-0.9mm}-\hspace{-0.9mm}1$}--(5,-3);
\draw(6,3) node[above]{$i$}--(12.1,0)--(6,-3);
\draw(7,3) node[above]{$i\hspace{-0.9mm}+\hspace{-0.9mm}1$}--(21,-3);
\draw(8,3) node[above]{$\cdots$};
\draw(9,3) node[above]{$l\hspace{-0.9mm}-\hspace{-0.9mm}1$}--(23,-3);
\draw(10,3) node[above]{$l$}--(24,-3);
\draw(11,3) node[above]{$l\hspace{-0.9mm}-\hspace{-0.9mm}1$}--(25,-3);
\draw(12,3) node[above]{$\cdots$};
\draw(13,3) node[above]{$i$}--(27,-3);

\draw(14,3) node[above]{$0$}--(7,0)--(14,-3);
\draw(15,3) node[above]{$1$}--(8,0)--(15,-3);
\draw(16,3) node[above]{$1$}--(9,0)--(16,-3);
\draw(17,3) node[above]{$\cdots$};
\draw(18,3) node[above]{$i\hspace{-0.9mm}-\hspace{-0.9mm}1$}--(11,0)--(18,-3);
\draw(19,3) node[above]{$i\hspace{-0.9mm}-\hspace{-0.9mm}1$}--(12,0)--(19,-3);
\draw(20,3) node[above]{$i$}--(12.5,0)--(20,-3);
\draw(21,3) node[above]{$i\hspace{-0.9mm}+\hspace{-0.9mm}1$}--(7,-3);
\draw(22,3) node[above]{$\cdots$};
\draw(23,3) node[above]{$l\hspace{-0.9mm}-\hspace{-0.9mm}1$}--(9,-3);
\draw(24,3) node[above]{$l$}--(10,-3);
\draw(25,3) node[above]{$l\hspace{-0.9mm}-\hspace{-0.9mm}1$}--(11,-3);
\draw(26,3) node[above]{$\cdots$};
\draw(27,3) node[above]{$i$}--(13,-3);

\draw(13.5,-3) node[below]{\scriptsize $\epsilon_{i-1,i}$};

\end{braid}
$$
and
$$
B_2=
\begin{braid}\tikzset{baseline=0mm}

\draw(0,3) node[above]{$0$}--(0,-3);
\draw(1,3) node[above]{$1$}--(1,-3);
\draw(2,3) node[above]{$1$}--(2,-3);
\draw(3,3) node[above]{$\cdots$};
\draw(4,3) node[above]{$i\hspace{-0.9mm}-\hspace{-0.9mm}1$}--(4,-3);
\draw(5,3) node[above]{$i\hspace{-0.9mm}-\hspace{-0.9mm}1$}--(5,-3);
\draw(6,3) node[above]{$i$}--(20,-3);
\draw(7,3) node[above]{$i\hspace{-0.9mm}+\hspace{-0.9mm}1$}--(21,-3);
\draw(8,3) node[above]{$\cdots$};
\draw(9,3) node[above]{$l\hspace{-0.9mm}-\hspace{-0.9mm}1$}--(23,-3);
\draw(10,3) node[above]{$l$}--(24,-3);
\draw(11,3) node[above]{$l\hspace{-0.9mm}-\hspace{-0.9mm}1$}--(25,-3);
\draw(12,3) node[above]{$\cdots$};
\draw(13,3) node[above]{$i$}--(27,-3);

\draw(14,3) node[above]{$0$}--(7,0)--(14,-3);
\draw(15,3) node[above]{$1$}--(8,0)--(15,-3);
\draw(16,3) node[above]{$1$}--(9,0)--(16,-3);
\draw(17,3) node[above]{$\cdots$};
\draw(18,3) node[above]{$i\hspace{-0.9mm}-\hspace{-0.9mm}1$}--(11,0)--(18,-3);
\draw(19,3) node[above]{$i\hspace{-0.9mm}-\hspace{-0.9mm}1$}--(21.5,1)--(19,0)--(21.5,-1)--(19,-3);
\draw(20,3) node[above]{$i$}--(6,-3);
\draw(21,3) node[above]{$i\hspace{-0.9mm}+\hspace{-0.9mm}1$}--(7,-3);
\draw(22,3) node[above]{$\cdots$};
\draw(23,3) node[above]{$l\hspace{-0.9mm}-\hspace{-0.9mm}1$}--(9,-3);
\draw(24,3) node[above]{$l$}--(10,-3);
\draw(25,3) node[above]{$l\hspace{-0.9mm}-\hspace{-0.9mm}1$}--(11,-3);
\draw(26,3) node[above]{$\cdots$};
\draw(27,3) node[above]{$i$}--(13,-3);

\end{braid}.
$$
We focus on \(B_1\). 
$ 
\begin{braid}\tikzset{baseline=0mm}
\draw(0,0.3) node[above]{$i$}--(1,0)--(0,-0.3);
\draw(1,0.3) node[above]{$i\hspace{-0.9mm}-\hspace{-0.9mm}1$}--(0,0)--(1,-0.3);
\end{braid}
$
opens, introducing a \((z-z')\) factor. Next the braids
$ 
\begin{braid}\tikzset{baseline=0mm}
\draw(0,0.3) node[above]{$i\hspace{-0.9mm}+\hspace{-0.9mm}1$}--(2,-0.3);
\draw(1,0.3) node[above]{$i$}--(0,0)--(1,-0.3);
\draw(2,0.3) node[above]{$i\hspace{-0.9mm}+\hspace{-0.9mm}1$}--(0,-0.3);
\end{braid}
$
through
$ 
\begin{braid}\tikzset{baseline=0mm}
\draw(0,0.3) node[above]{$l$}--(2,-0.3);
\draw(1,0.3) node[above]{$l\hspace{-0.9mm}-\hspace{-0.9mm}1$}--(0,0)--(1,-0.3);
\draw(2,0.3) node[above]{$l$}--(0,-0.3);
\end{braid}
$
open in succession. Then 
$ 
\begin{braid}\tikzset{baseline=0mm}
\draw(0,0.3) node[above]{$l\hspace{-0.9mm}-\hspace{-0.9mm}1$}--(2,-0.3);
\draw(1,0.3) node[above]{$l$}--(0,0)--(1,-0.3);
\draw(2,0.3) node[above]{$l\hspace{-0.9mm}-\hspace{-0.9mm}1$}--(0,-0.3);
\end{braid}
$
opens, yielding 
$ 
\begin{braid}\tikzset{baseline=0mm}
\draw(0,0.3) node[above]{$l\hspace{-0.9mm}-\hspace{-0.9mm}1$}--(1,0)--(0,-0.3);
\draw(1,0.3) node[above]{$l\hspace{-0.9mm}-\hspace{-0.9mm}1$}--(0,0)--(1,-0.3);
\blackdot(0.8,0);
\end{braid}
$
which becomes
$ 
\begin{braid}\tikzset{baseline=0mm}
\draw(0,0.3) node[above]{$l\hspace{-0.9mm}-\hspace{-0.9mm}1$}--(1,-0.3);
\draw(1,0.3) node[above]{$l\hspace{-0.9mm}-\hspace{-0.9mm}1$}--(0,-0.3);
\end{braid}
$.
Next, the braids 
$ 
\begin{braid}\tikzset{baseline=0mm}
\draw(0,0.3) node[above]{$l\hspace{-0.9mm}-\hspace{-0.9mm}2$}--(2,-0.3);
\draw(1,0.3) node[above]{$l\hspace{-0.9mm}-\hspace{-0.9mm}1$}--(0,0)--(1,-0.3);
\draw(2,0.3) node[above]{$l\hspace{-0.9mm}-\hspace{-0.9mm}2$}--(0,-0.3);
\end{braid}
$
through
$ 
\begin{braid}\tikzset{baseline=0mm}
\draw(0,0.3) node[above]{$i$}--(2,-0.3);
\draw(1,0.3) node[above]{$i\hspace{-0.9mm}+\hspace{-0.9mm}1$}--(0,0)--(1,-0.3);
\draw(2,0.3) node[above]{$i$}--(0,-0.3);
\end{braid}
$
open in succession, giving 
$$
B_1=
\begin{braid}\tikzset{baseline=0mm}

\draw(0.5,3) node[above]{$0$}--(0.5,-3);
\draw(1.5,3) node[above]{$1$}--(1.5,-3);
\draw(2.5,3) node[above]{$1$}--(2.5,-3);
\draw(3,3) node[above]{$\cdot\cdot$};
\draw(4,3) node[above]{$i\hspace{-0.9mm}-\hspace{-0.9mm}1$}--(4,-3);
\draw(5,3) node[above]{$i\hspace{-0.9mm}-\hspace{-0.9mm}1$}--(5,-3);
\draw(6,3) node[above]{$i$}--(6,-3);
\draw(7,3) node[above]{$i\hspace{-0.9mm}+\hspace{-0.9mm}1$}--(7,-3);
\draw(8,3) node[above]{$\cdots$};
\draw(9,3) node[above]{$l\hspace{-0.9mm}-\hspace{-0.9mm}1$}--(9,-3);
\draw(10,3) node[above]{$l$}--(10,-3);
\draw(11,3) node[above]{$l\hspace{-0.9mm}-\hspace{-0.9mm}1$}--(25,-3);
\draw(12,3) node[above]{$l\hspace{-0.9mm}-\hspace{-0.9mm}2$}--(19,0)--(12,-3);
\draw(13,3) node[above]{$\cdots$};
\draw(14,3) node[above]{$i$}--(21,0)--(14,-3);

\draw(15,3) node[above]{$0$}--(11,1.1)--(11,-1.1)--(15,-3);
\draw(16,3) node[above]{$1$}--(11.5,0.9)--(11.5,-0.9)--(16,-3);
\draw(17,3) node[above]{$1$}--(12,0.7)--(12,-0.7)--(17,-3);
\draw(18,3) node[above]{$\cdots$};
\draw(19,3) node[above]{$i\hspace{-0.9mm}-\hspace{-0.9mm}1$}--(12.5,0.3)--(12.5,-0.3)--(19,-3);
\draw(20,3) node[above]{$i\hspace{-0.9mm}-\hspace{-0.9mm}1$}--(13,0)--(20,-3);
\draw(21,3) node[above]{$i$}--(14,0)--(21,-3);
\draw(22,3) node[above]{$i\hspace{-0.9mm}+\hspace{-0.9mm}1$}--(15,0)--(22,-3);
\draw(23,3) node[above]{$\cdots$};
\draw(24,3) node[above]{$l\hspace{-0.9mm}-\hspace{-0.9mm}2$}--(17,0)--(24,-3);
\draw(25,3) node[above]{$l\hspace{-0.9mm}-\hspace{-0.9mm}1$}--(11,-3);
\draw(26,3) node[above]{$l$}--(26,-3);
\draw(27,3) node[above]{$l\hspace{-0.9mm}-\hspace{-0.9mm}1$}--(27,-3);
\draw(28,3) node[above]{$\cdot\cdot$};
\draw(28.5,3) node[above]{$i$}--(28.5,-3);

\draw(13.5,-3) node[below]{\scriptsize $(-1)^{n+i+1}(z-z')$};

\end{braid}.
$$
Next, 
$ 
\begin{braid}\tikzset{baseline=0mm}
\draw(0,0.3) node[above]{$l\hspace{-0.9mm}-\hspace{-0.9mm}1$}--(2,-0.3);
\draw(1,0.3) node[above]{$l\hspace{-0.9mm}-\hspace{-0.9mm}2$}--(2,0)--(1,-0.3);
\draw(2,0.3) node[above]{$l\hspace{-0.9mm}-\hspace{-0.9mm}1$}--(0,-0.3);
\end{braid}
$
opens, leading to 
$ 
\begin{braid}\tikzset{baseline=0mm}
\draw(0,0.3) node[above]{$l\hspace{-0.9mm}-\hspace{-0.9mm}1$}--(1,0)--(0,-0.3);
\draw(1,0.3) node[above]{$l\hspace{-0.9mm}-\hspace{-0.9mm}2$}--(0,0)--(1,-0.3);
\end{braid}
$,
which opens, yielding
$ 
\begin{braid}\tikzset{baseline=0mm}
\draw(0,0.3) node[above]{$l\hspace{-0.9mm}-\hspace{-0.9mm}2$}--(1,0)--(0,-0.3);
\draw(1,0.3) node[above]{$l\hspace{-0.9mm}-\hspace{-0.9mm}2$}--(0,0)--(1,-0.3);
\blackdot(0.2,0);
\end{braid}
$
which becomes
$ 
\begin{braid}\tikzset{baseline=0mm}
\draw(0,0.3) node[above]{$l\hspace{-0.9mm}-\hspace{-0.9mm}1$}--(1,-0.3);
\draw(1,0.3) node[above]{$l\hspace{-0.9mm}-\hspace{-0.9mm}1$}--(0,-0.3);
\end{braid}
$.
This process repeats for braids 
$ 
\begin{braid}\tikzset{baseline=0mm}
\draw(0,0.3) node[above]{$l\hspace{-0.9mm}-\hspace{-0.9mm}2$}--(2,-0.3);
\draw(1,0.3) node[above]{$l\hspace{-0.9mm}-\hspace{-0.9mm}3$}--(2,0)--(1,-0.3);
\draw(2,0.3) node[above]{$l\hspace{-0.9mm}-\hspace{-0.9mm}2$}--(0,-0.3);
\end{braid}
$
through
$ 
\begin{braid}\tikzset{baseline=0mm}
\draw(0,0.3) node[above]{$i\hspace{-0.9mm}+\hspace{-0.9mm}1$}--(2,-0.3);
\draw(1,0.3) node[above]{$i$}--(2,0)--(1,-0.3);
\draw(2,0.3) node[above]{$i\hspace{-0.9mm}+\hspace{-0.9mm}1$}--(0,-0.3);
\end{braid}
$,
giving
$$
B_1=
\begin{braid}\tikzset{baseline=0mm}

\draw(0.5,3) node[above]{$0$}--(0.5,-3);
\draw(1.5,3) node[above]{$1$}--(1.5,-3);
\draw(2.5,3) node[above]{$1$}--(2.5,-3);
\draw(3,3) node[above]{$\cdot\cdot$};
\draw(4,3) node[above]{$i\hspace{-0.9mm}-\hspace{-0.9mm}1$}--(4,-3);
\draw(5,3) node[above]{$i\hspace{-0.9mm}-\hspace{-0.9mm}1$}--(5,-3);
\draw(6,3) node[above]{$i$}--(6,-3);
\draw(7,3) node[above]{$i\hspace{-0.9mm}+\hspace{-0.9mm}1$}--(7,-3);
\draw(8,3) node[above]{$\cdots$};
\draw(9,3) node[above]{$l\hspace{-0.9mm}-\hspace{-0.9mm}1$}--(9,-3);
\draw(10,3) node[above]{$l$}--(10,-3);
\draw(11,3) node[above]{$l\hspace{-0.9mm}-\hspace{-0.9mm}1$}--(11,-3);
\draw(12,3) node[above]{$\cdots$};
\draw(13,3) node[above]{$i\hspace{-0.9mm}+\hspace{-0.9mm}1$}--(13,-3);
\draw(14,3) node[above]{$i$}--(21,-3);

\draw(15,3) node[above]{$0$}--(14.5,1.1)--(14.5,-1.1)--(15,-3);
\draw(16,3) node[above]{$1$}--(15,0.9)--(15,-0.9)--(16,-3);
\draw(17,3) node[above]{$1$}--(15.5,0.7)--(15.5,-0.7)--(17,-3);
\draw(18,3) node[above]{$\cdots$};
\draw(19,3) node[above]{$i\hspace{-0.9mm}-\hspace{-0.9mm}1$}--(16,0.3)--(16,-0.3)--(19,-3);
\draw(20,3) node[above]{$i\hspace{-0.9mm}-\hspace{-0.9mm}1$}--(16.5,0)--(20,-3);
\draw(21,3) node[above]{$i$}--(14,-3);
\draw(22,3) node[above]{$i\hspace{-0.9mm}+\hspace{-0.9mm}1$}--(22,-3);
\draw(23,3) node[above]{$\cdots$};
\draw(24,3) node[above]{$l\hspace{-0.9mm}-\hspace{-0.9mm}2$}--(24,-3);
\draw(25,3) node[above]{$l\hspace{-0.9mm}-\hspace{-0.9mm}1$}--(25,-3);
\draw(26,3) node[above]{$l$}--(26,-3);
\draw(27,3) node[above]{$l\hspace{-0.9mm}-\hspace{-0.9mm}1$}--(27,-3);
\draw(28,3) node[above]{$\cdot\cdot$};
\draw(28.5,3) node[above]{$i$}--(28.5,-3);

\draw(13.5,-3) node[below]{\scriptsize $(-1)^{n+i+1}(z-z')$};

\end{braid}.
$$
Which simplifies to \((-1)^{n+i+1}(z-z')^2(v_1^z \otimes v_2^z)\). Now we focus on \(B_2\). The rightmost 
$ 
\begin{braid}\tikzset{baseline=0mm}
\draw(0,0.3) node[above]{$i$}--(2,-0.3);
\draw(1,0.3) node[above]{$i\hspace{-0.9mm}-\hspace{-0.9mm}1$}--(0,0)--(1,-0.3);
\draw(2,0.3) node[above]{$i$}--(0,-0.3);
\end{braid}
$
opens, then on the left side the braids
$ 
\begin{braid}\tikzset{baseline=0mm}
\draw(0,0.3) node[above]{$i$}--(2,-0.3);
\draw(1,0.3) node[above]{$i\hspace{-0.9mm}-\hspace{-0.9mm}1$}--(0,0)--(1,-0.3);
\draw(2,0.3) node[above]{$i$}--(0,-0.3);
\end{braid}
$
through
$ 
\begin{braid}\tikzset{baseline=0mm}
\draw(0,0.3) node[above]{$l$}--(2,-0.3);
\draw(1,0.3) node[above]{$l\hspace{-0.9mm}-\hspace{-0.9mm}1$}--(0,0)--(1,-0.3);
\draw(2,0.3) node[above]{$l$}--(0,-0.3);
\end{braid}
$
open in succession. Then from the right, the braids
$ 
\begin{braid}\tikzset{baseline=0mm}
\draw(0,0.3) node[above]{$i\hspace{-0.9mm}+\hspace{-0.9mm}1$}--(2,-0.3);
\draw(1,0.3) node[above]{$i$}--(2,0)--(1,-0.3);
\draw(2,0.3) node[above]{$i\hspace{-0.9mm}+\hspace{-0.9mm}1$}--(0,-0.3);
\end{braid}
$
through
$ 
\begin{braid}\tikzset{baseline=0mm}
\draw(0,0.3) node[above]{$l\hspace{-0.9mm}-\hspace{-0.9mm}1$}--(2,-0.3);
\draw(1,0.3) node[above]{$l\hspace{-0.9mm}-\hspace{-0.9mm}2$}--(2,0)--(1,-0.3);
\draw(2,0.3) node[above]{$l\hspace{-0.9mm}-\hspace{-0.9mm}1$}--(0,-0.3);
\end{braid}
$
open in succession. Then 
$ 
\begin{braid}\tikzset{baseline=0mm}
\draw(0,0.3) node[above]{$l\hspace{-0.9mm}-\hspace{-0.9mm}1$}--(1,0)--(0,-0.3);
\draw(1,0.3) node[above]{$l$}--(0,0)--(1,-0.3);
\end{braid}
$
opens, introducing a factor of \((-z+z')\), giving
$$
B_2=
\begin{braid}\tikzset{baseline=0mm}

\draw(0.5,3) node[above]{$0$}--(0.5,-3);
\draw(1.5,3) node[above]{$1$}--(1.5,-3);
\draw(2.5,3) node[above]{$1$}--(2.5,-3);
\draw(3,3) node[above]{$\cdot\cdot$};
\draw(4,3) node[above]{$i\hspace{-0.9mm}-\hspace{-0.9mm}1$}--(4,-3);
\draw(5,3) node[above]{$i\hspace{-0.9mm}-\hspace{-0.9mm}1$}--(5,-3);
\draw(6,3) node[above]{$i$}--(6,-3);
\draw(7,3) node[above]{$i\hspace{-0.9mm}+\hspace{-0.9mm}1$}--(7,-3);
\draw(8,3) node[above]{$\cdots$};
\draw(9,3) node[above]{$l\hspace{-0.9mm}-\hspace{-0.9mm}1$}--(9,-3);
\draw(10,3) node[above]{$l$}--(10,-3);
\draw(11,3) node[above]{$l\hspace{-0.9mm}-\hspace{-0.9mm}1$}--(25,-3);
\draw(12,3) node[above]{$l\hspace{-0.9mm}-\hspace{-0.9mm}2$}--(19,0)--(12,-3);
\draw(13,3) node[above]{$\cdots$};
\draw(14,3) node[above]{$i$}--(21,0)--(14,-3);

\draw(15,3) node[above]{$0$}--(11,1.1)--(11,-1.1)--(15,-3);
\draw(16,3) node[above]{$1$}--(11.5,0.9)--(11.5,-0.9)--(16,-3);
\draw(17,3) node[above]{$1$}--(12,0.7)--(12,-0.7)--(17,-3);
\draw(18,3) node[above]{$\cdots$};
\draw(19,3) node[above]{$i\hspace{-0.9mm}-\hspace{-0.9mm}1$}--(12.5,0.3)--(12.5,-0.3)--(19,-3);
\draw(20,3) node[above]{$i\hspace{-0.9mm}-\hspace{-0.9mm}1$}--(22,0)--(20,-3);
\draw(21,3) node[above]{$i$}--(14,0)--(21,-3);
\draw(22,3) node[above]{$i\hspace{-0.9mm}+\hspace{-0.9mm}1$}--(15,0)--(22,-3);
\draw(23,3) node[above]{$\cdots$};
\draw(24,3) node[above]{$l\hspace{-0.9mm}-\hspace{-0.9mm}2$}--(17,0)--(24,-3);
\draw(25,3) node[above]{$l\hspace{-0.9mm}-\hspace{-0.9mm}1$}--(11,-3);
\draw(26,3) node[above]{$l$}--(26,-3);
\draw(27,3) node[above]{$l\hspace{-0.9mm}-\hspace{-0.9mm}1$}--(27,-3);
\draw(28,3) node[above]{$\cdot\cdot$};
\draw(28.5,3) node[above]{$i$}--(28.5,-3);

\draw(13.5,-3) node[below]{\scriptsize $(-1)^{n+i}(z-z')$};

\end{braid}.
$$
Now 
$ 
\begin{braid}\tikzset{baseline=0mm}
\draw(0,0.3) node[above]{$l\hspace{-0.9mm}-\hspace{-0.9mm}1$}--(2,-0.3);
\draw(1,0.3) node[above]{$l\hspace{-0.9mm}-\hspace{-0.9mm}2$}--(0,0)--(1,-0.3);
\draw(2,0.3) node[above]{$l\hspace{-0.9mm}-\hspace{-0.9mm}1$}--(0,-0.3);
\end{braid}
$
opens, giving
$ 
\begin{braid}\tikzset{baseline=0mm}
\draw(0,0.3) node[above]{$l\hspace{-0.9mm}-\hspace{-0.9mm}2$}--(1,0)--(0,-0.3);
\draw(1,0.3) node[above]{$l\hspace{-0.9mm}-\hspace{-0.9mm}1$}--(0,0)--(1,-0.3);
\end{braid}
$, 
which opens yielding
$ 
\begin{braid}\tikzset{baseline=0mm}
\draw(0,0.3) node[above]{$l\hspace{-0.9mm}-\hspace{-0.9mm}2$}--(1,0)--(0,-0.3);
\draw(1,0.3) node[above]{$l\hspace{-0.9mm}-\hspace{-0.9mm}2$}--(0,0)--(1,-0.3);
\blackdot(0.8,0);
\end{braid}
$, 
which becomes
$ 
\begin{braid}\tikzset{baseline=0mm}
\draw(0,0.3) node[above]{$l\hspace{-0.9mm}-\hspace{-0.9mm}2$}--(1,-0.3);
\draw(1,0.3) node[above]{$l\hspace{-0.9mm}-\hspace{-0.9mm}2$}--(0,-0.3);
\end{braid}
$. 
This process repeats with braids 
$ 
\begin{braid}\tikzset{baseline=0mm}
\draw(0,0.3) node[above]{$l\hspace{-0.9mm}-\hspace{-0.9mm}2$}--(2,-0.3);
\draw(1,0.3) node[above]{$l\hspace{-0.9mm}-\hspace{-0.9mm}3$}--(0,0)--(1,-0.3);
\draw(2,0.3) node[above]{$l\hspace{-0.9mm}-\hspace{-0.9mm}2$}--(0,-0.3);
\end{braid}
$
through
$ 
\begin{braid}\tikzset{baseline=0mm}
\draw(0,0.3) node[above]{$i\hspace{-0.9mm}+\hspace{-0.9mm}1$}--(2,-0.3);
\draw(1,0.3) node[above]{$i$}--(0,0)--(1,-0.3);
\draw(2,0.3) node[above]{$i\hspace{-0.9mm}+\hspace{-0.9mm}1$}--(0,-0.3);
\end{braid}
$.
Then
$ 
\begin{braid}\tikzset{baseline=0mm}
\draw(0,0.3) node[above]{$i$}--(2,-0.3);
\draw(1,0.3) node[above]{$i\hspace{-0.9mm}-\hspace{-0.9mm}1$}--(0,0)--(1,-0.3);
\draw(2,0.3) node[above]{$i$}--(0,-0.3);
\end{braid}
$
opens, yielding 
$ 
\begin{braid}\tikzset{baseline=0mm}
\draw(0,0.3) node[above]{$i\hspace{-0.9mm}-\hspace{-0.9mm}1$}--(1,0)--(0,-0.3);
\draw(1,0.3) node[above]{$i$}--(0,0)--(1,-0.3);
\end{braid}
$
which opens, introducing a factor of \((-2z')\), so we have \(B_2 = (-1)^{n+i+1}(z-z')(2z')(v_1^z \otimes v_1^{z'})\). Then \(A_i = B_1 + B_2 = (-1)^{n+i+1}(z^2 - z'^2)(v_1^z \otimes v_1^{z'})\), so \(A_1 \in (z^2-z'^2)\left[(-1)^{n+i+1} + (z-z')R_\de^z \otimes R_\de^{z'} \right] (v_1^z \otimes v_1^{z'})\). Then Proposition \ref{InterReduce} provides the result.
\end{proof}


\subsection{The map $R_{L_{\de,i}^{z'},L_{\de,i}^z}$ in type ${\tt F}_4$}\label{RF4}
Fix \(i \in I'\). Recall the construction of \(L_{\de,i}^z\) from section \ref{F4Cons}. Let
\begin{equation}\label{EITypeF}
\bi = \begin{cases}
(0,1,2,3,3,2,4,4,3,3,2,1), & \textup{if } i=1;\\
(0,1,2,3,3,2,4,4,3,3,1,2), & \textup{if } i=2;\\
(0,1,2,3,4,3,2,3,4,1,2,3), & \textup{if } i=3;\\
(0,1,2,3,4,3,2,3,1,2,3,4), & \textup{if } i=4.
\end{cases}
\end{equation}
Then \(\bi\) is an extremal word for \(L_{\de,i}^z\). Let 
\begin{align*}
v_1^z= \begin{cases}
v_{(1,3,2,5,4,6),\bi} & \textup{if }i \in \{1,2\}\\
v_{(1,2,3,4,5,6), \bi} & \textup{if }i =3\\
v_{(1,2,3,4,6,5), \bi} & \textup{if }i=4\\
\end{cases}
\end{align*}
Then \(v_1^z\) is a vector of word \(\bi\), and it is easily seen from the action of generators that \(v_1^z\) generates \(L_{\de,i}^z\).
\begin{Proposition}\label{F4intertwiner}
\begin{align*}R(v_1^{z'} \otimes v_1^{z}) \in (z^2-z'^2)^{24}\left[\sigma + c + (z-z')R_{\de}^{z} \otimes R_{\de}^{z'}\right](v_1^z \otimes v_1^{z'}),
\end{align*}
where \(c=1\) if \(i \in \{1,4\}\) and \(c=-1\) if \(i \in \{2,3\}\).
\end{Proposition}
\begin{proof}
The proof of this proposition is a straightforward but lengthy calculation made with the aid of a computer, using Proposition \ref{InterReduce} and the same techniques as in the \({\tt B}_l^{(1)}\) and \({\tt C}_l^{(1)}\) cases.
\end{proof}


\subsection{The map $R_{L_{\de,i}^{z'},L_{\de,i}^z}$ in type ${\tt G}_2^{(1)}$}\label{RG2}
Fix \(i \in I'\). Recall the construction of \(L_{\de,i}^z\) from section \ref{G2Cons}. Let
\begin{equation}\label{EITypeG}
\bi = \begin{cases}
(0,1,2,2,2,1), & \textup{if } i=1;\\
(0,1,2,2,1,2), & \textup{if } i=2.
\end{cases}
\end{equation}
Then \(\bi\) is an extremal word for \(L_{\de,i}^z\). Let 
\begin{align*}
v_1^z:= v_{(1,2,3)}.
\end{align*}
Then \(v_1^z\) is a vector of word \(\bi\) that generates \(L_{\de,i}^z\).
\begin{Proposition}\label{G2intertwiner}
\begin{align*}R_{L_{\de,i}^{z'},L_{\de,i}^z}(v_1^{z'} \otimes v_1^{z}) \in (z^3-z'^3)^{8}\left[\sigma + (-1)^i + (z-z')R_{\de}^{z} \otimes R_{\de}^{z'}\right](v_1^z \otimes v_1^{z'}).
\end{align*}
\end{Proposition}
\begin{proof}
As in the \({\tt F}_4^{(1)}\) case, the proof of this proposition is a straightforward but lengthy calculation made with the aid of a computer, using Proposition \ref{InterReduce} and the same techniques as in the \({\tt B}_l^{(1)}\) and \({\tt C}_l^{(1)}\) cases.
\end{proof}

\subsection{Constructing $\tau_r$ from $R_{L_{\de,i}^{z'}, L_{\de,i}^z}$}\label{TauNSL}

Let \(\Car\) be a Cartan matrix of type \({\tt B}_l^{(1)}, {\tt C}_l^{(1)}, {\tt F}_4^{(1)}\), or \({\tt G}_2^{(1)}\). Fix \(i \in I'\). In Sections \ref{RBl}-\ref{RG2} we have chosen a word vector \(v_1^z\) of extremal word \(\bi\) that generates \(L_{\de,i}^z\), and shown that
\begin{align*}
R_{L_{\de,i}^{z'},L_{\de,i}^z}(v_1^{z'} \otimes v_1^{z}) \in f(z,z')\left[\sigma + c + (z-z')R_{\de}^{z} \otimes R_{\de}^{z'}\right](v_1^z \otimes v_1^{z'}),
\end{align*}
for some \(c = \pm 1\) and nonzero \(f(z,z') \in \O[z,z']\). As \(v_1^{z'} \otimes v_1^z\) generates \(L_{\de,i}^{z'} \circ L_{\de,i}^z\) and \(R_{L_{\de,i}^{z'},L_{\de,i}^z}\) is an \(R_{2\de}^{z,z'}\)-linear endomorphism, this implies that every element in the image of \(R_{L_{\de,i}^{z'},L_{\de,i}^z}\) is divisible by \(f\). Let \(\pi: L_{\de,i}^{z} \circ L_{\de,i}^{z'} \twoheadrightarrow L_{\de,i} \circ L_{\de,i}\) and \(\pi': L_{\de,i}^{z'} \circ L_{\de,i}^{z} \twoheadrightarrow L_{\de,i} \circ L_{\de,i}\) be the quotient maps given by setting \(z=z'=0\). Then since  \(L_{\de,i}^{z'} \circ L_{\de,i}^z\) is free as an \(\O[z,z']\)-module, there is a well-defined map \(\tilde{\tau}:=\pi \circ f^{-1}R_{L_{\de,i}^{z'},L_{\de,i}^z}\) which factors through \(\pi'\) to give a \(R_{2\de}\)-homomorphism
\begin{align*}\label{TauFctnNSL}
\tau: L_{\de,i} \circ L_{\de,i} \to L_{\de,i} \circ L_{\de,i}
\end{align*}
such that \(\tau(v_2)=(\sigma + c)v_2\), where \(v_2=\pi(v_1^z \otimes v_1^{z'})=\pi'(v_1^{z'} \otimes v_1^z)\).\\
\indent Write \(v_1\) for the image of \(v_1^z\) under the quotient \(L_{\de,i}^z \twoheadrightarrow  L_{\de,i}\) of Proposition \ref{zTwistMod}. Then \(v_1\) spans the 1-dimensional top degree component \((1_{\bi}L_{\de,i})_N\) of the  (extremal) word space \(1_{\bi}L_{\de,i}\) in \(L_{\de,i}\), and \(v_2 = v_1 \otimes v_1\). Write \(v_n\) for \(v_1 \otimes \cdots \otimes v_1 \in M_n = L_{\de,i}^{\circ n}\). Inserting the endomorphism \(\tau\) into the \(r\)-th and \((r+1)\)-th positions in \(M_n\) yields endomorphisms \(\tau_r : M_n \to M_n\), \(v_n \mapsto (\sigma_r + c)v_n\).

\begin{Proposition}\label{NSLbraidrels}
The endomorphisms \(\tau_{r}\) satisfy the usual Coxeter relations of the standard generators of the symmetric group \(\frak{S}_n\), i.e., \(\tau_r^2 = 1\), \(\tau_r \tau_s = \tau_s \tau_r\) for \(|r-s|>1\) and \(\tau_r\tau_{r+1} \tau_r = \tau_{r+1} \tau_r \tau_{r+1}\).
\end{Proposition}
 \begin{proof}
 The braid relations follow from the fact that the map \(R\) satisfies these relations (see \cite[Chapter 1]{KKK}). By Proposition \ref{InterReduce} we have that \(f(z,z')=\prod_{\substack{k,m\\i_k=i_m}} (c_{k}z^{a_k}-c_{m}z'^{a_m})\), where \(y_kv_1^{z} = c_k^{a_k}v_1^z\). Then it follows from \cite[Lemma 1.3.1]{KKK} that \(R_{L_{\de,i}^z,L_{\de,i}^{z'}} \circ R_{L_{\de,i}^{z'},L_{de,i}^z}(v_1^{z'} \otimes v_1^z) = f^2(v_1^{z'} \otimes v_1^z)\), so \(\tau^2(v_1 \otimes v_1) = v_1 \otimes v_1\), and thus \(\tau_r^2 = 1\).
\end{proof}


\end{document}